\documentclass[reqno,11pt]{amsart}
\usepackage[margin=1.0in]{geometry}
\usepackage{amsthm,amsmath,amssymb,dsfont}
\usepackage{indentfirst, latexsym, enumerate,amsmath,graphicx}

\usepackage{relsize}
\usepackage{marginnote}
\usepackage{stmaryrd}
\usepackage{esint}
\usepackage{graphicx}
\usepackage{bm}
\usepackage{caption}
\usepackage{subfigure}
\usepackage{cite}
\usepackage{graphicx}
\usepackage{subfigure}
\usepackage{paralist}
\usepackage{indentfirst}
 \allowdisplaybreaks[4]
\theoremstyle{plain}
\newtheorem{thm}{Theorem}[section]

\newtheorem{lem}[thm]{Lemma}
\newtheorem{prop}[thm]{Proposition}
\newtheorem{rem}{Remark}[section]

\numberwithin{equation}{section}

\usepackage{appendix}
\usepackage{xcolor}

\newcommand{\eps}{{\varepsilon}}

\begin{document}

\title[The non-isentropic compressible fluid-particle interaction model]
{Enhanced stability and asymptotic limits to the non-isentropic compressible 
fluid-particle interaction model with thermal  effects}

\author[F. Li]{Fucai Li}
\address{School of Mathematics, Nanjing University, Nanjing
 210093, P. R. China}
\email{fli@nju.edu.cn}

\author[J. Ni]{Jinkai Ni} 
\address{School of Mathematics, Nanjing University, Nanjing
 210093, P. R. China}
\email{jinkaini123@gmail.com}
 
\author[Z. Xin]{Zhouping Xin} 
\address{The Institute of Mathematical Sciences, The Chinese University of Hong Kong,  
Shatin, New Territories, Hong Kong}
\email{zpxin@ims.cuhk.edu.hk}

\begin{abstract}

In Einstein's seminal work [Ann. Physik, 17 (1905), 549-560], he pointed out that the temperature of a fluid influences the motion of suspended particles dramatically. To describe the effect of the temperature in this physical process more precisely, Boudin et al. [ESAIM Proc., 28 (2009), 195-210] introduced a new fluid-particle interaction model containing of the non-isentropic compressible Euler equations for the fluid and a nonlinear Vlasov-Fokker-Planck type equation for the particles. By adding some viscous and heat conductive terms to the fluid part of this model, Mu and Wang [Calc. Var. Partial Differential Equations, 59 (2020), Paper no. 110] established the global existence of classical solutions near an equilibrium state.
  
 In this paper, through establishing the uniform a priori estimates with respect to the viscosity and heat conductivity coefficients and taking the combined zero viscosity and heat conductivity limits, we show that the model introduced by Boudin et al. still admits a global classical solution and enjoys optimal decay rates thereby improving Mu and Wang's results and confirming Einstein's predications. Our work indicates that the presence of particles indeed emanates new dissipation effects on the non-isentropic compressible fluid-particle model via the differences between the macroscopic velocity of the particles and the fluid velocity, and the macroscopic temperature of the particles and the fluid temperature, which is significantly different from the case of pure non-isentropic compressible Euler equations. To achieve these goals, we have developed new ideas and techniques to surmount substantial obstacles caused by the absence of viscosity and heat conductivity, and the nonlinear interactions between the fluid and particles.

\end{abstract}

  % enhanced dissipation}


\maketitle

\tableofcontents

\setcounter{equation}{0}
 \indent
  % \allowdisplaybreaks

\section{Introduction}

\subsection{Brief reviews on  compressible fluid-particle models}
Fluid-particle models have been widely applied across various fields, 
such as spray dynamics \cite{BBBDLLT-irma-2005}, diesel engines \cite{RM-1952,RM-1952-a}, 
combustion theory \cite{Wfa-1958,Wfa-1985}, sedimentation processes \cite{BWC-zamm-2000}, 
and medical treatments \cite{BBJM-esaim-2005}.
Owing to their widespread applications and the  significant role of 
fluid-particle coupling models in   mathematics and physics, these 
models have garnered significant research attention. For studies on
 compressible fluid-particle models,  interested   readers  may consult,
  for example, \cite{CKL-JHDE-2013,CJ-MMMAS-2021,LMW-SIAM-2017,LNW-arXiv-2025,
  LWW-ARMA-2022,MV-CMP-2008,Ww-CMS-2024,LNW-1}, while research on incompressible 
  fluid-particle models can be found in \cite{CDM-krm-2011,GHMZ-siamj-2009,LNW-SAPM-2025, LNX-arXiv,
  LLY22,Wang-Yu2015,Yc-JMPA-2013,LNSW-arxiv} and the references cited therein. Below, we briefly
 review some results concerning compressible fluid-particle models which are related to our 
concerns in this paper.

\subsubsection{Fluid-particle models with momentum exchanges}
For a typical  compressible fluid-particle coupled model,  the  dynamics of particles at
 the level of statistical mechanics are governed by a  distribution function $F=F(t,x,v) \geq 0$,
  where $t > 0$ denotes the time, $x \in \Omega \subset \mathbb{R}^3$ represents spatial coordinates, 
  and $v \in \mathbb{R}^3$ represents the  velocity of particles. This distribution function satisfies 
  the well-known Vlasov-Fokker-Planck (VFP) equation (\cite{CG-2006,MV-MMMAS-2007}): 
\begin{align}  \label{C1}
\partial_{t}F + v \cdot \nabla_x  F - \mathrm{div}_v (F_d F + \nabla_v F) = 0.
\end{align}
Here, the drag force term $F_d$ is defined as
\begin{align}\nonumber%\label{C2}
F_d = F_0 (v-u  ),
\end{align}
where $F_0 > 0$ represents the drag coefficient (usually  chosen as one for brevity), 
and the unknown function $u = (u_1, u_2, u_3)^{\top} \in \mathbb{R}^3$ corresponds to 
the given fluid velocity field.
If we consider  {that} the fluid is    compressible,    viscous and isentropic,  
we obtain the following   compressible 
  fluid-particle interaction model (usually named as  the  compressible \emph{isentropic} 
  Navier-Stokes-Vlasov-Fokker-Planck (NS-VFP) system (\cite{CG-2006,CKL-JHDE-2013,MV-MMMAS-2007,MV-CMP-2008}):  
    \begin{equation}\label{H2}
\left\{\begin{aligned}
&\partial_t \varrho+{\rm div}(\varrho u)=0,\\
&\partial_{t}(\varrho u)+\mathrm{div}(\varrho u\otimes u)+\nabla P-\mu\Delta u %-(\mu+\lambda)\nabla{\rm div}u 
=-\int_{\mathbb{R}^{3}}F_d F {\rm d}v,\\
&\partial_{t}F+v\cdot\nabla F-\mathrm{div}_{v}(F_d F+\nabla_{v}F)=0. 
 \end{aligned} 
 \right.
\end{equation} 
Here the pressure  $P(\rho)=a \rho^\gamma$ with $a>0 $ and $ \gamma \geq 1$ is a given 
function of the fluid density $\rho$.
There  is  a  substantial body of mathematical results  on this model. 
Mellet and Vasseur \cite{MV-MMMAS-2007} established  the  global existence of weak 
solutions to \eqref{H2} in a bounded domain $\Omega \subset \mathbb{R}^3$ with homogeneous 
Dirichlet boundary conditions. Furthermore, by employing the relative entropy method, they 
\cite{MV-CMP-2008} performed an asymptotic analysis of the weak solutions. Moreover, 
the global existence and   exponential decay of  classical solutions to \eqref{H2}
were   explored in \cite{CKL-JHDE-2013} in $\mathbb{T}^3$. 
% Under certain additional conditions on 
Duan and Liu \cite{DL-KRM-2013} 
established the global well-posedness of classical solutions to the inviscid version of \eqref{H2}  
 near the equilibrium state $(1,0,M)$ in $\mathbb{R}^3$ when the initial data is a   small    perturbation of   it, 
  where $M:=M(v)=(2\pi)^{-{3}/{2}}e^{- {|v|^2}/{2}}.  $  
Recently, Li, Wang, and Wang \cite{LWW-ARMA-2022} conducted a rigorous analysis of the asymptotic 
stability of solutions and justified the vanishing viscosity limit from   \eqref{H2} to the compressible Euler-VFP system within a finite time  {interval}, achieving a 
convergence rate of $\mathcal{O}(\mu^\frac{1}{2})$. 
% Furthermore, by incorporating the term $(\mu+\lambda)\nabla{\rm div} u$ in \eqref{H2}$_3$, 
Wang \cite{Ww-CMS-2024} established the global strong solutions  {in   the $H^N$ norm $(N \geq 2)$}
 and derived the optimal decay rates for all derivatives of the solutions to \eqref{H2} by employing 
 energy methods and low-high frequency decomposition techniques.
Li, Mu and Wang \cite{LMW-SIAM-2017} studied  a   more complex  isentropic fluid-particle model and 
obtained 
the global existence of classical solutions in  the  $H^4$ framework (see also the recent improvement
 by Li, Ni and Wu \cite{LNW-arXiv} in the $H^2$ framework and by Li, Ni and Wang \cite{LNW-1} on optimal decay estimates and inviscid limit).
For further results on the compressible isentropic Euler/NS-VFP system, interested readers can 
consult \cite{CG-2006,LLY22,BD-JHDE-2006,CJ-MMMAS-2021,li-li} and the references cited therein.

\subsubsection{Fluid-particle models with momentum and energy exchanges}
It should be noted   that in the   
 aforementioned     results 
on  compressible fluid-particle models,   the   thermal effect  of the fluid  has been 
 ignored by assuming that the temperature of  {the} fluid is a constant for simplification.
However, the  Einstein's  {seminal} work \cite{Ei-1905} showed   that the dynamics of  
 temperature, especially spatial-temporal variations of temperature,  play  a  crucial
  role in governing   dispersive two-phase flow systems. 
Motivated by \cite{Ei-1905}, Boudin et al. \cite{BBBGLLM-esaim-2009} introduced a
 fluid-particle model  including a temperature equation in the fluid part. 
More precisely, in their model, the behavior of the fluid is described by the well-known 
 non-isentropic compressible Euler equations, while the movement of particles follows a 
 VFP-type equation involving the temperature of the fluid. These two parts interact 
 with each other through the exchanges of momentum and energy
between the particles and the fluid. Boudin et al. \cite{BBBGLLM-esaim-2009} also presented some numerical 
simulation examples to their model (in $\mathbb{R}^1$ case) but no well-posedenss result was given. 
 Goudon, Jin  {and} Yan \cite{GJY-cms-2012}
developed an asymptotic-preserving numerical scheme for the model introduced in \cite{BBBGLLM-esaim-2009}.
Afterwards, by adding additional   viscous and heat  conductive   terms  into the model, 
Mu and Wang \cite{Mu-Wang-20} demonstrated the global well-posedness and large time 
behavior of classical solutions
to this model when the initial data are   small perturbations around the equilibrium 
state in the  $H^4$ framework.
Recently,  by incorporating  both shear and bulk viscosities of the fluid   into  
the model of \cite{BBBGLLM-esaim-2009}, Li, Ni, and  Wu  \cite{LNW-arXiv-2025} 
diminished the regularity requirement of the initial data from $H^4$ in \cite{Mu-Wang-20} 
 to $H^2$. Moreover, they derived the time decay rates of global strong solutions and their gradients in the 
 $L^p$-norm for $2 \leq p \leq \infty$. 
%(see also  \cite{LNW-SAPM-2025} for some discussions   on  
% an incompressible inhomogeneous fluid-particle interaction model with momentum and energy exchanges).
 
We remark that the   additional    viscous  and  heat conductivity terms  play  a 
crucial role in the arguments of \cite{Mu-Wang-20,LNW-arXiv-2025} because they 
provide dissipation effects to the velocity and the temperature 
of the fluid. A  natural  question  arises: Is it possible to  obtain   global 
well-posedness of classical solutions to the model introduced in  \cite{BBBGLLM-esaim-2009} 
without adding extra viscosity  or  heat conductivity? If so, what is the new dissipation 
mechanism in this situation? 
This is the main motivation of our work. We shall show that  the model  introduced by  
Boudin et al.,  without requiring any additional  
  viscous or heat conductivity terms, still admits a global classical solution and enjoys  
   optimal decay rates, thereby improving  Mu and Wang's results \cite{Mu-Wang-20}. 
   Our findings indicate that the 
  interactions between the macroscopic velocity $b$ of the particles and the fluid 
  velocity $u$, and   between  the macroscopic temperature $\omega$ of the particles and 
  the fluid temperature $\theta$ emanate enhanced dissipation effects on the whole system.

%
%In this paper, we investigate the global well-posedness and the zero 
%heat conductivity limit for the compressible non-isentropic Navier–Stokes–Vlasov–Fokker–Planck system 
%\eqref{LN-1}, by exploiting the enhanced stability induced by the quantity $\sqrt{2}\omega - \sqrt{3}\theta$. 
%Our analysis primarily focuses on the thermal effects in fluid-particle interaction models. 
%Indeed, our results uncover novel 
%dissipation mechanisms arising from momentum and energy exchanges between the fluid and particles.

\subsection{Our model}
Now, we introduce our non-isentropic compressible NS-VFP system of the fluid-particle interaction 
model (cf. \cite{BBBGLLM-esaim-2009,LNW-arXiv-2025,Mu-Wang-20}):
\begin{equation}\label{LN-1}\left\{
\begin{aligned}
&\partial_t \varrho+{\rm div}  (\varrho u)=0,\\
&\partial_t (\varrho u)+{\rm div}  (\varrho u\otimes u)-{\rm div} (\mathds{T})+\nabla  P=\mathcal{M} ,\\
&\partial_t (\varrho E)+{\rm div}  \left((\varrho E+P)u\right)
-{\rm div} (\kappa \nabla  \vartheta ) -{\rm div} (\mathds{T}u)=\mathcal{F},\\
&\partial_t F+v\cdot\nabla  F={\tilde{\mathcal{L}}_{u,\vartheta }}F.
\end{aligned}\right.\end{equation}
Here, $\varrho=\varrho(t,x)\geq0$, 
$u=u(t,x)=(u_1(t,x),u_2(t,x),u_3(t,x))\in\mathbb{R}^3$, 
$P=P(t,x)\geq0$, $E=e+\frac{1}{2}|u|^2\geq0$, $e=e(t,x)\geq0$, 
and $\vartheta =\vartheta (t,x)\geq0$ for $(t,x)\in\mathbb{R}^+\times\mathbb{R}^3$ 
represent the density, velocity, pressure, total energy, internal energy, 
and temperature of the fluid, respectively. Furthermore, $F=F(t,x,v)\geq0$ for 
$(t,x,v)\in\mathbb{R}^+\times\mathbb{R}^3\times\mathbb{R}^3$ denotes 
the particle density distribution function in
 the phase space $\mathbb{R}^3_x\times\mathbb{R}^3_v$. The stress tensor $\mathds{T}$ is defined as  
\begin{align}\label{LN-T}  
\mathds{T}=2\mu D(u)+\lambda {\rm div} u \mathbb{I}_3,  
\end{align}  
where $D(u)=\frac{1}{2}(\nabla u+(\nabla u)^\top)$ is the deformation tensor, $(\nabla u)^\top$ is the transpose 
of $\nabla u$, and $\mathbb{I}_3$ is the $3\times 3$ identity matrix. 
Here, $\mu > 0$ represents the shear viscosity coefficient, $\lambda+\frac{2}{3}\mu$ 
represents the bulk viscosity coefficient, and $\kappa>0$ represents 
the thermal conductivity coefficient of the fluid.

Physically, the coefficients $\mu$, $\lambda$ and $\kappa$ may 
vary as functions of the density $\rho$ and the temperature $\vartheta $. In this paper,  
we assume that these coefficients are constants independent of 
$\rho$ and $\vartheta $. For the case of a perfect gas, the pressure 
$P$, temperature $\vartheta $, and internal energy $e$ obey the following relationships:  
\begin{align}  \label{LN-pe}
& P = Rn\vartheta , \quad  e = \frac{R}{\gamma - 1}\vartheta  =: C_\textsl{v} \vartheta ,  
\end{align}  
where $R > 0$ denotes the gas constant, $\gamma > 1$ represents the 
adiabatic constant, and $C_\textsl{v} > 0$ is referred to as the specific heat capacity at constant volume.

In the system \eqref{LN-1}, the fluid and particles interact mutually via the Fokker-Planck type operator  
\begin{align}\label{LN-Lf}  
{\tilde{\mathcal{L}}_{u,\vartheta }}F={\rm div}_v\big((v-u)F+\vartheta \nabla_v F\big),  
\end{align}  
which accounts for the friction force exerted on the particles by the surrounding 
fluid and the Brownian motion of the particles, in the same spirit  of 
 Einstein  \cite{Ei-1905}. Specifically, the friction force is assumed 
 to be proportional to the relative velocity $v - u$, and the Brownian 
 motion, which depends on the temperature of the fluid, induces diffusion 
 with respect to the velocity variable (This is different from the isentropic case, 
  where 
the Fokker-Planck  operator takes the form $\mathrm{div}_v [(v-u) F + \nabla_v F]$, 
see \eqref{C1}). Furthermore, the coupling terms $\mathcal{M}$ in \eqref{LN-1}$_2$ 
and $\mathcal{F}$ in \eqref{LN-1}$_3$ describe the different interactions between the dispersed 
phase and the dense phase, defined as  
\begin{align}\label{LN-mf}  
\mathcal{M}=\int_{\mathbb{R}^3} (v-u)F\mathrm{d}v , \quad &\mathcal{F}
=\int_{\mathbb{R}^3} [v\cdot(v-u)-3\vartheta ]F\mathrm{d}v,  
\end{align}  
representing the momentum exchanges and energy exchanges between the particles and the fluid, respectively.

Substituting \eqref{LN-T}--\eqref{LN-mf} into the system \eqref{LN-1} leads to 
\begin{equation}\label{LN-2}  
\left\{  
\begin{aligned}  
&\partial_t \varrho + {\rm div} (\varrho u) = 0, \\  
&\partial_t (\varrho u) + {\rm div} (\varrho u \otimes u) - \mu \Delta u 
- (\mu + \lambda)\nabla {\rm div} u + \nabla P = \int_{\mathbb{R}^3} (v-u)F\mathrm{d}v, \\  
&C_\textsl{v}(\partial_t (\varrho \vartheta ) + {\rm div} (\varrho u \vartheta )) 
+ P {\rm div} u - \kappa \Delta \vartheta \\
& \quad  - 2\mu |D(u)|^2 - \lambda|{\rm div} u|^2 = \int_{\mathbb{R}^3} (|v-u|^2 - 3\vartheta )F\mathrm{d}v, \\  
&\partial_t F + v \cdot \nabla F = {\rm div}_v\big((v-u)F + \vartheta \nabla_v F\big).  
\end{aligned}  
\right.  
\end{equation}  
We supplement the above system  with the initial data:  
\begin{align}  \label{LN-3}
(\varrho, u, \vartheta , F)|_{t=0} = (\varrho_0(x), u_0(x), \vartheta _0(x), F_0(x)), 
\quad (x,v) \in \mathbb{R}^3 \times \mathbb{R}^3.  
\end{align}  

It is readily to verify that \eqref{LN-2}--\eqref{LN-3} has an  equilibrium state 
\begin{align}\label{maxwell}
 (\varrho, u, \vartheta , F) = (1, 0, 1, M),  
\end{align}
  where the global Maxwellian $M$ is defined by
$
M=M(v)=(2\pi)^{- {3}/{2}}e^{- {|v|^2}/{2}}.    
$

Setting $F=0$ in the system \eqref{LN-2}, then it reduces to the well-known non-isentropic compressible
Navier-Stokes system (see \cite{MN-1980}):  
\begin{equation}\label{LN-5}
\left\{\begin{aligned}
& \partial_t \varrho+{\rm div}(\varrho u)=0, \\&\partial_t (\varrho u) + {\rm div} (\varrho u \otimes u) 
- \mu \Delta u - (\mu + \lambda)\nabla {\rm div} u + \nabla P=0, \\
& C_\textsl{v}(\partial_t (\varrho \vartheta ) 
+ {\rm div} (\varrho u \vartheta )) + P {\rm div} u  - \kappa \Delta \vartheta  - 2\mu |D(u)|^2 - \lambda|{\rm div} u|^2 =0.
\end{aligned}\right.\end{equation}  
There are many results on global well-posedess and time decay of classical solutions to the system \eqref{LN-5}, among others, we mention \cite{DJ-2018,MN-1980,MN-1983,KS-1999}.

 If we   set $\mu = \lambda=\kappa = 0$ in the system  \eqref{LN-2}, we obtain the 
 following  non-isentropic compressible Euler-VFP system  (initially  proposed by Boudin 
 et al. \cite{BBBGLLM-esaim-2009}):
\begin{equation}\label{LNX-1}  
\left\{  
\begin{aligned}  
&\partial_t \varrho + {\rm div} (\varrho u) = 0, \\  
&\partial_t (\varrho u) + {\rm div} (\varrho u \otimes u) + \nabla P = \int_{\mathbb{R}^3}
 (v-u)F\mathrm{d}v, \\  
&C_\textsl{v}(\partial_t (\varrho \vartheta ) + {\rm div} (\varrho u \vartheta )) + P {\rm div} u 
  = \int_{\mathbb{R}^3} (|v-u|^2 - 3\vartheta )F\mathrm{d}v, \\  
&\partial_t F + v \cdot \nabla F = {\rm div}_v\big((v-u)F + \vartheta \nabla_v F\big).  
\end{aligned}  
\right.  
\end{equation} 
Note that   the system    \eqref{LNX-1} consists of purely hyperbolic equations
  in the $x$-space. As can be observed from its form, this implies that it is hard to obtain a
   global classical solution. In fact, 
 the arguments in \cite{Mu-Wang-20,LNW-arXiv-2025} heavily depend on the dissipation effects
  provided by the viscosity  term  $-\mu \Delta u$ and the heat conductivity term 
   $-\kappa \Delta  \vartheta$ and \emph{cannot} be applied  to the system   
    \eqref{LNX-1} directly.

When taking $F = 0$ in \eqref{LNX-1}, we have the well-known non-isentropic compressible Euler system:
\begin{equation}\label{LNX-2}
\left\{\begin{aligned}& \partial_t \varrho+{\rm div}(\varrho u)=0, \\&\partial_t (\varrho u) 
+ {\rm div} (\varrho u \otimes u)  + \nabla P=0, \\
& C_\textsl{v}(\partial_t (\varrho \vartheta )
 + {\rm div} (\varrho u \vartheta )) + P {\rm div} u   =0.\end{aligned}\right.
\end{equation}  
It is well-known that, generally, the classical solutions of  the system \eqref{LNX-2} shall 
blows up in finite time, for example, see   \cite{ BSV-CPAM-23, Chen-JHDE-2011,CD-2024, MRRS-2022,Yin-2004}. 
And a good way to obtain the global existence of classical solutions is to add a damping term on the velocity
 equation \eqref{LNX-2}$_2$, see \cite{STW-2003-cpde, Zhang-Wu-2024}. 

%
%\textcolor{red}
%{As mentioned in \cite{Tr-NS-2001}, it is difficult to establish the global solution for 
%such a compressible non-isentropic Euler system \eqref{LNX-2}, and the solution will blow up.} 

To the best of our knowledge, until now, there is no mathematical result on well-posedness of the 
 system \eqref{LNX-1} since it was proposed by Boudin et al. in 2009. A natural and interesting 
 question arises: Can we utilize the weak dissipative property of $\mathcal{L}$
  (see the definition in \eqref{lglg} below) with respect to $v$ and the drag force
   on the velocity equation and the energy equation to establish the existence of 
   global classical solutions for the system \eqref{LNX-1}?
In this paper we shall give an affirmative answer to this question. 
% We aim to establish the global stability of solutions to the Cauchy problem \eqref{LN-2}--\eqref{LN-3} near
%The main  purpose  of this paper is to
We shall  investigate both the large-time behavior and the enhanced stability mechanism
 of the non-isentropic compressible NS-VFP  system \eqref{LN-2} and the
 non-isentropic compressible  Euler-VFP system \eqref{LNX-1} in   $\mathbb{R}^3$.
% \emph{without heat conductivity}, i.e., the system,} 
Assuming that the initial data perturbation around the equilibrium state
  $(1, 0, 1, M)$    is sufficiently small in the   $H^4$ framework, we establish the global well-posedness 
 of classical solutions to the system \eqref{LN-2}, \emph{uniformly in}  $\mu>0$, $\lambda>0$ and 
  $\kappa>0$  in $\mathbb{R}^3$.   Furthermore, we rigorously justify the global-in-time 
  vanishing viscosity limits $\mu\rightarrow0$ and $\lambda\rightarrow 0$, and
vanishing heat conductivity limit $\kappa \rightarrow 0$ with a convergence rate of 
$\mathcal{O}(\max\{\mu,\lambda,\kappa\})$ and hence improve the results in \cite{Mu-Wang-20}. 
 For the case $\mu=\lambda=\kappa=0$ in  \eqref{LN-2}, i.e., the  
 system \eqref{LNX-1}, we derive the optimal time-decay rates of 
 the global classical solutions and their gradients in   $L^2$-norm. 
 The key lies in capturing the new dissipation modes $b-u$ and $\sqrt{2}\omega-\sqrt{3}\theta$ 
 (see the definitions of $b$ in \eqref{LNX-b} and of $\omega$ in \eqref{omegaa} below). 
 Additionally, we obtain analogous results regarding global existence and the heat
  conductivity limit in the case of the periodic domain, which demonstrates an exponential
   time-decay rate of the solution. 
   
 %  Moreover,    it is worth mentioning that our results also address the global stability of the system \eqref{LN-4}. 

 Compared with the non-isentropic compressible Euler equations  \eqref{LNX-2}, when the 
 coupling effect of particles is considered, that is, for the system \eqref{LNX-1},
  our work reveals additional dissipation structures for the velocity and temperature 
  (i.e., $b - u$ and $\sqrt{2}\omega-\sqrt{3}\theta$). By utilizing these newly 
  identified dissipative structures, we derive the optimal decay rates of classical 
  solutions and their gradients.
%These results align with, in some sense,  those obtained for 
%the compressible non-isentropic Navier-Stokes equations  without heat conductivity \eqref{LN-5}.
%% NS-VFP  without heat conductivity \eqref{LN-4} and the. 
%In stark contrast to the purely transport-based dissipation mechanisms observed in 
%compressible Navier-Stokes equations without heat conductivity, 
Our results highlight 
the new dissipation mechanisms for the velocity and the temperature of the fluid driven by the  momentum and energy 
exchanges of the particle and the fluid.

%
%Below, we present our main results for the compressible non-isentropic NS-VFP system  with  thermal effect.

\subsection{Main results}

We first focus on the stability of global classical solutions near the equilibrium 
state $(\varrho, u, \vartheta , F) = (1, 0, 1, M)$ to the system \eqref{LN-2}. 
To  this end, we first introduce the following transformations:  
\begin{align*}  
\varrho^{\mu,\lambda,\kappa} = 1 + \rho^{\mu,\lambda,\kappa}, 
\quad \vartheta ^{\mu,\lambda,\kappa} = \theta^{\mu,\lambda,\kappa} + 1, 
\quad F^{\mu,\lambda,\kappa} = M + \sqrt{M}f^{\mu,\lambda,\kappa}.  
\end{align*}  
 Here we have added the superscripts  $\mu$, $\lambda$ and 
$\kappa$ on the unknowns $(\rho,u, \theta,f)$ to emphasise the dependence of  $\mu$, $\lambda$ and 
$\kappa$. Without loss of generality, we assume that the physical constants $ R = C_\textsl{v} =1$.
 Consequently, the Cauchy problem \eqref{LN-2}--\eqref{LN-3} can be rewritten as follows:
\begin{equation}\label{Q1}
\left\{\begin{aligned}
&\partial_t \rho^{\mu,\lambda,\kappa} +u^{\mu,\lambda,\kappa} 
\cdot \nabla \rho^{\mu,\lambda,\kappa}+(1+\rho^{\mu,\lambda,\kappa}) 
{\rm div} u^{\mu,\lambda,\kappa}=0, \\
&\partial_t u^{\mu,\lambda,\kappa} +u^{\mu,\lambda,\kappa} 
\cdot \nabla u^{\mu,\lambda,\kappa}
+\frac{1+\theta^{\mu,\lambda,\kappa}}{1+\rho^{\mu,\lambda,\kappa}} \nabla \rho^{\mu,\lambda,\kappa}
+\nabla \theta^{\mu,\lambda,\kappa}\\
&\quad=\frac{ \mu\Delta  u^{\mu,\lambda,\kappa}
+(\mu+\lambda)\nabla{\rm div}  u^{\mu,\lambda,\kappa} }{1+\rho^{\mu,\lambda,\kappa}}  
+\frac{b^{\mu,\lambda,\kappa}-u^{\mu,\lambda,\kappa}
-a^{\mu,\lambda,\kappa} u^{\mu,\lambda,\kappa}}{1+\rho^{\mu,\lambda,\kappa}}, \\
&\partial_t \theta^{\mu,\lambda,\kappa} +u^{\mu,\lambda,\kappa} 
\cdot \nabla \theta^{\mu,\lambda,\kappa}+\theta^{\mu,\lambda,\kappa} 
 {\rm div} u^{\mu,\lambda,\kappa}+{\rm div} u^{\mu,\lambda,\kappa}
 -\sqrt{3}(\sqrt{2} \omega^{\mu,\lambda,\kappa}-\sqrt{3} \theta^{\mu,\lambda,\kappa})  \\
&\quad =\frac{ \kappa\Delta  \theta^{\mu,\lambda,\kappa}
+|u^{\mu,\lambda,\kappa}|^2-2 b^{\mu,\lambda,\kappa} \cdot u^{\mu,\lambda,
\kappa}+a^{\mu,\lambda,\kappa}|u^{\mu,\lambda,\kappa}|^2-3 a^{\mu,\lambda,\kappa}
 \theta^{\mu,\lambda,\kappa}  }{1+\rho^{\mu,\lambda,\kappa}}  \\
&\qquad + \frac{2\mu |D(u^{\mu,\lambda,\kappa})|^2+\lambda|{\rm div}  
u^{\mu,\lambda,\kappa}|^2 }{1+\rho^{\mu,\lambda,\kappa}}
-\frac{\sqrt{3}\rho^{\mu,\lambda,\kappa}(\sqrt{2} 
\omega^{\mu,\lambda,\kappa}-\sqrt{3} \theta^{\mu,\lambda,\kappa})}{1+\rho^{\mu,\lambda,\kappa}},\\
&\partial_t f^{\mu,\lambda,\kappa} +v \cdot \nabla 
f^{\mu,\lambda,\kappa}+u^{\mu,\lambda,\kappa} \cdot \nabla_v 
f^{\mu,\lambda,\kappa}-\frac{1}{2} u^{\mu,\lambda,\kappa} \cdot v f^{\mu,\lambda,\kappa}-
u^{\mu,\lambda,\kappa} \cdot v \sqrt{M}  \\
&\quad = (|v|^2-3 ) \sqrt{M} \theta^{\mu,\lambda,\kappa}
+\mathcal{L} f^{\mu,\lambda,\kappa}+\frac{\theta^{\mu,\lambda,\kappa}}{\sqrt{M}} 
\Delta_v(\sqrt{M} f^{\mu,\lambda,\kappa}), \\
\end{aligned}\right.
\end{equation}
with the initial data
\begin{align}\label{Q1-1}
&(\rho^{\mu,\lambda,\kappa},u^{\mu,\lambda,\kappa},
\theta^{\mu,\lambda,\kappa},f^{\mu,\lambda,\kappa})|_{t=0}  \nonumber\\
= \,&(\rho^{\mu,\lambda,\kappa}_0(x),u^{\mu,\lambda,\kappa}_0(x),
\theta^{\mu,\lambda,\kappa}_0(x),f^{\mu,\lambda,\kappa}_0(x,v))\nonumber\\
:=\,&\Big( \varrho_0^{\mu,\lambda,\kappa}(x)-1,u^{\mu,\lambda,
\kappa}_0(x), {\vartheta}^{\mu,\lambda,\kappa}_0(x)-1,
\frac{F^{\mu,\lambda,\kappa}_0(x,v)-M}{\sqrt{M}}\Big),
 \ \  x \in \mathbb{R}^3, \, v\in\mathbb{R}^3. \qquad 
\end{align}

 Similarly, for the case $\mu=\lambda=\kappa=0$ (i. e.,  the system \eqref{LNX-1}), 
 using the transformations:   
\begin{align*}  
\varrho = 1 + \rho , \quad \vartheta  = \theta  + 1, \quad F = M + \sqrt{M}f,   
\end{align*}  
and taking $R = C_\textsl{v} = 1$,
we reformulate the Cauchy problem for it as  (in other words, we obtain 
it by   dropping  the {superscripts $\mu$, $\lambda$ and $\kappa$} on
 the unknowns $(\rho^{\mu,\lambda,\kappa}, u^{\mu,\lambda,\kappa},
  \theta^{\mu,\lambda,\kappa},  f^{\mu,\lambda,\kappa})$, and deleting 
{the linear terms $\mu  \Delta u^{\mu,\lambda,\kappa}$, $(\mu+\lambda)
  \nabla{\rm div} u^{\mu,\lambda,\kappa}$
and $\kappa \Delta \theta^{\mu,\lambda,\kappa}$,
and the nonlinear terms $2\mu |D(u^{\mu,\lambda,\kappa})|^2$ and
$\lambda |{\rm div}u^{\mu,\lambda,\kappa}|^2$
in the system \eqref{Q1}):}
\begin{equation}\label{Q2}
\left\{\begin{aligned}
&\partial_t \rho + u\cdot\nabla\rho+(1+\rho){\rm div} u=0, \\
&\partial_t u +u \cdot \nabla u+\frac{1+\theta}{1+\rho} \nabla 
\rho+\nabla \theta=\frac{b-u-au}{1+\rho}, \\
&\partial_t \theta +u \cdot \nabla \theta+\theta  {\rm div} u
+ {\rm div} u-\sqrt{3}(\sqrt{2} \omega-\sqrt{3} \theta)  \\
&\quad =\frac{ |u|^2-2 u \cdot b+a|u|^2-3 a \theta }{1+\rho}  
 -\frac{\sqrt{3}\rho(\sqrt{2} \omega-\sqrt{3} \theta)}{1+\rho},\\
&\partial_t f +v \cdot \nabla f+u \cdot \nabla_v f-\frac{1}{2} u \cdot v f-
u \cdot v \sqrt{M}- (|v|^2-3 ) \sqrt{M} \theta  \\
&\quad =\mathcal{L} f+\frac{\theta}{\sqrt{M}} \Delta_v(\sqrt{M} f),\\
\end{aligned}\right.
\end{equation} 
with the initial data
\begin{align}\label{Q2-2}
(\rho ,u ,\theta ,f )|_{t=0}=&\,(\rho _0(x),u _0(x),\theta _0(x),f _0(x,v))\nonumber\\
:=&\,\Big( \varrho_0 (x)-1,u _0(x), {\vartheta} _0(x)-1,\frac{F _0(x,v)-M}{\sqrt{M}}\Big),
 \quad   x \in \mathbb{R}^3, \, v\in\mathbb{R}^3.  \quad
\end{align}
In the systems \eqref{Q1} and \eqref{Q2}, the linearized Fokker-Planck operator $\mathcal{L}$ is defined as  
\begin{align} \label{lglg}
\mathcal{L}g := \frac{1}{\sqrt{M}} \nabla_v \cdot \Big[ M \nabla_v \Big( \frac{g}{\sqrt{M}} \Big) \Big], 
\quad g=f^{\mu,\lambda,\kappa} \ \  \textrm{or}\ \ g=f.
\end{align}  
And $a^{\mu,\lambda,\kappa} = a^{f^{\mu,\lambda,\kappa}}(t,x)$,  
$b^{\mu,\lambda,\kappa} = b^{f^{\mu,\lambda,\kappa}}(t,x)$,
$\omega^{\mu,\lambda,\kappa} = \omega^{f^{\mu,\lambda,\kappa}}(t,x)$,  $a = a^{f}(t,x)$, $b = b^f(t,x)$,
 and $\omega = \omega^f(t,x)$ are given by,
for $g=f^{\mu,\lambda,\kappa}$ or $f$,    
\begin{align} 
a^g(t,x) & := \int_{\mathbb{R}^3} \sqrt{M} g(t,x,v) \, {\rm d}v, \\ 
b^g(t,x) & := \int_{\mathbb{R}^3} v \sqrt{M} g(t,x,v) \, {\rm d}v, \label{LNX-b} \\ 
\omega^g(t,x) & := \int_{\mathbb{R}^3} \frac{|v|^2 - 3}{\sqrt{6}} \sqrt{M} g(t,x,v) \, {\rm d}v. \label{omegaa}
\end{align}

 We define the following dissipation functional to \eqref{Q1}:  
%\begin{equation} 
\begin{align}  \label{D0} 
&\mathcal{D}(\rho^{\mu,\lambda,\kappa} ,u^{\mu,\lambda,\kappa} ,
\theta^{\mu,\lambda,\kappa} ,f^{\mu,\lambda,\kappa} )(t)\nonumber\\
\,&\quad :=\|b^{\mu,\lambda,\kappa} -u^{\mu,\lambda,\kappa} 
\|_{H^4}^2+\|\sqrt{2}\omega^{\mu,\lambda,\kappa}
 -\sqrt{3}
\theta^{\mu,\lambda,\kappa}\|_{H^4}^2  +\|\nabla (a^{\mu,\lambda,\kappa} ,
b^{\mu,\lambda,\kappa} ,\omega^{\mu,\lambda,\kappa} ,\rho^{\mu,\lambda,\kappa} )\|_{H^3}^2  
 \nonumber \\
&\qquad +\sum_{|\alpha|\leq 4}\|\{\mathbf{I}-\mathbf{P}\}\partial^\alpha 
f^{\mu,\lambda,\kappa} \|_{\nu}^2  +\sum_{\substack{1 \leq | \beta| 
\leq 4\\|\alpha|+|\beta| \leq 4}} \|\partial_\beta^\alpha\{\mathbf{I}-\mathbf{P}\} 
f^{\mu,\lambda,\kappa} \|_\nu^2,  
\end{align}  
where $\{\mathbf{I}-\mathbf{P}\}f^{\mu,\lambda,\kappa}$ represents 
the microscopic component of $f^{\mu,\lambda,\kappa}$, as defined in Section 2.3.

We first establish the global existence and uniqueness of classical solutions to 
the problem \eqref{Q1}--\eqref{Q1-1} in $\mathbb{R}^3$, \emph{uniformly with respect to 
 the viscosity coefficients $\mu,\lambda>0$, and
the heat conductivity coefficient $\kappa > 0$.}
\begin{thm}\label{T1.1}
Let % $\Omega=\mathbb{R}^3$, and 
$\mu,\lambda,\kappa>0$. 
Assume that the initial data  $(\rho_0^{\mu,\lambda,\kappa},u_0^{\mu,\lambda,\kappa},
\theta_0^{\mu,\lambda,\kappa},f_0^{\mu,\lambda,\kappa})$
satisfies $F_0^{\mu,\lambda,\kappa}=M+\sqrt{M}f_0^{\mu,\lambda,\kappa}\geq 0$, 
$(\rho_0^{\mu,\lambda,\kappa},u_0^{\mu,\lambda,\kappa},\theta_0^{\mu,\lambda,\kappa})\in 
H^4(\mathbb{R}^3)$ and $f_0^{\mu,\lambda,\kappa}\in H^4_{x,v}(\mathbb{R}^3\times\mathbb{R}^3)$.
There exists a constant $\varepsilon_0>0$ independent of 
$\mu$, $\lambda$ and $\kappa$ such that if
\begin{align}\label{a-1}
\mathcal{E}_0^{\mu,\lambda,\kappa}:
=\|(\rho^{{\mu,\lambda,\kappa}}_0,u^{{\mu,\lambda,\kappa}}_0,
\theta^{{\mu,\lambda,\kappa}}_0 )\|_{H^4}^2+\|f_0^{\mu,\lambda,\kappa}\|_{H_{x,v}^4}^2\leq \varepsilon_0,
\end{align}
then the Cauchy problem \eqref{Q1}--\eqref{Q1-1} 
admits a unique global classical solution $(\rho^{{\mu,\lambda,\kappa}}$, 
$u^{{\mu,\lambda,\kappa}}$, $\theta^{{\mu,\lambda,\kappa}},f^{{\mu,\lambda,\kappa}})$,  
 satisfying $\varrho^{{\mu,\lambda,\kappa}}=1+\rho^{\mu,\lambda,\kappa}>0$, 
 $F^{\mu,\lambda,\kappa}=M+\sqrt{M}f^{\mu,\lambda,\kappa}\geq 0$, and
\begin{align}\label{a-2}
&\sup_{t\geq 0}\big(\|(\rho^{{\mu,\lambda,\kappa}},u^{{\mu,\lambda,\kappa}},
\theta^{{\mu,\lambda,\kappa}} )(t)\|_{H^{4}}^2+\|f^{\mu,\lambda,\kappa}(t)\|_{H_{x,v}^4}^2\big)
+\mu\int_0^{t}\|\nabla u^{\mu,\lambda,\kappa} (\tau)\|_{H^4}^2{\rm d}\tau  \nonumber\\
& \quad+ (\mu+\lambda) \int_0^{t}\|{\rm div} u^{\mu,\lambda,\kappa} (\tau)\|_{H^4}^2{\rm d}\tau 
 +\kappa\int_0^{t}\|\nabla \theta^{\mu,\lambda,\kappa} (\tau)\|_{H^4}^2{\rm d}\tau  \nonumber\\
  & \quad\quad+\int_{0}^{t} \mathcal{D}(\rho^{\mu,\lambda,\kappa},u^{\mu,\lambda,\kappa},
  \theta^{\mu,\lambda,\kappa},f^{\mu,\lambda,\kappa})(\tau){\rm d}\tau 
   \leq C_0\mathcal{E}_0^{\mu,\lambda,\kappa},
  \quad \forall \, t\in \mathbb{R}^{+},
\end{align}
where $C_0 > 0$ is a constant   independent of $\mu$, $\lambda$,  $\kappa$ and   $t$.
\end{thm}

Based on  the uniform estimates established in Theorem \ref{T1.1}, next we analyze 
the vanishing viscosity limits
as $\mu,\lambda\rightarrow0$ and
  heat conductivity limit as $\kappa \to 0$ and demonstrate the global existence of 
  classical solutions for the non-isentropic compressible Euler-VFP system \eqref{LNX-1}.

\begin{thm}\label{T1.2}
% Let $\Omega=\mathbb{R}^3$. 
Assume that $(\rho_0 ,u_0 ,\theta_0,f_0)$
satisfy $F_0=M+\sqrt{M}f_0\geq 0$, $(\rho_0,u_0,\theta_0)\in H^4(\mathbb{R}^3)$
and $f_0\in H_{x,v}^4(\mathbb{R}^3\times\mathbb{R}^3)$, and $\{(\rho_0^{\mu,\lambda,\kappa},
u_0^{\mu,\lambda,\kappa},\theta_0^{\mu,\lambda,\kappa},f_0^{\mu,\lambda,\kappa})\}_{0<\mu,\lambda,\kappa<1}$ is
a sequence such that $(\rho^{\mu,\lambda,\kappa}_0$, $u^{\mu,\lambda,\kappa}_0 $, 
$\theta^{\mu,\lambda,\kappa}_0) \rightarrow (\rho_0,u_0,\theta_0)$ in 
$H^4(\mathbb{R}^3)$ and $f_0^{\mu,\lambda,\kappa}\rightarrow f_0$ in
$H^4_{x,v}(\mathbb{R}^3\times\mathbb{R}^3)$, as $\mu,\lambda,\kappa \rightarrow 0$.
 And there exists a constant $\eps_1>0$
independent of $\mu$, $\lambda$ and $\kappa$ such that if
\begin{align}\label{b-1}
\mathcal{E}_0:=\|(\rho _0,u _0,\theta _0 )\|_{H^4}^2+\|f_0 \|_{H_{x,v}^4}^2\leq \eps_1,
\end{align}
then for the global solution $(\rho^{\mu,\lambda,\kappa},u^{\mu,\lambda,\kappa},
\theta^{\mu,\lambda,\kappa},f^{\mu,\lambda,\kappa})$ of the Cauchy problem 
\eqref{Q1}--\eqref{Q1-1} subject to the initial data  $(\rho_0^{\mu,\lambda,\kappa},
u_0^{\mu,\lambda,\kappa},\theta_0^{\mu,\lambda,\kappa},f_0^{\mu,\lambda,\kappa})$ 
obtained in Theorem \ref{T1.1}, as $\mu,\lambda,\kappa\rightarrow 0$,
there exists a limit $(\rho,u,\theta,f)$ such that, up to a subsequence, it holds 
\begin{equation}\label{b2}
\left\{
\begin{aligned}
(\rho^{\mu,\lambda,\kappa},u^{\mu,\lambda,\kappa},\theta^{\mu,\lambda,\kappa}) 
&\rightharpoonup (\rho,u,\theta )\quad  \,\, \text{weakly-$*$~  ~in ~~~} 
&& L^{\infty}(\mathbb{R}^{+};H^4(\mathbb{R}^3)),\\
f^{\mu,\lambda,\kappa} &\rightharpoonup  f\quad  \,\, \quad\quad\,\,\,\, 
\text{weakly-$*$~ ~in ~~~} &&L^{\infty}(\mathbb{R}^{+};H_{x,v}^4(\mathbb{R}^3\times\mathbb{R}^3)),\\
(\rho^{\mu,\lambda,\kappa},u^{\mu,\lambda,\kappa},\theta^{\mu,\lambda,\kappa} ) 
&\rightarrow (\rho,u,\theta )\quad  \,\, \text{strong ~~  in ~~~} 
&&C_{\rm{loc}}(\mathbb{R}^{+};H^3_{\rm{loc}}(\mathbb{R}^3)).
\end{aligned}
\right.
\end{equation}
 The limit $(\rho,u,\theta,f)$ is the unique global solution to the 
 Cauchy problem \eqref{Q2}--\eqref{Q2-2}
 associated with the initial data $(\rho_0,u_0,\theta_0,f_0)$. 
 Furthermore, $(\rho,u,\theta,f)$ satisfy $\varrho=1+\rho>0$,
 $F=M+\sqrt{M}f \geq 0$ and
\begin{align}\nonumber% \label{b3}
&\sup_{t\geq 0}\big(\|(\rho,u,\theta)(t)\|_{H^{4}}^2+\|f(t)\|_{H_{x,v}^4}^2\big)
+\int_{0}^{t}\mathcal{D}(\rho,u,\theta,f)(\tau){\rm d}\tau \leq C_1\mathcal{E}_0,  
\quad \forall\,  t\in \mathbb{R}^{+},
\end{align}
where $\mathcal{D}(\rho,u,\theta,f)$ is defined in the same way as in \eqref{D0}, and 
 $C_1>0$ is a   constant independent of $\mu$, $\lambda$, $\kappa$ and  $t$.
\end{thm}

To establish the  global convergence in time as $\mu  \rightarrow 0$, $\lambda\rightarrow0$ and
$\kappa \rightarrow 0$ simultaneously, below we derive   error estimates between the global solutions of the 
systems \eqref{Q1} and \eqref{Q2}.  

\begin{thm}\label{T1.3}
% Let $\Omega=\mathbb{R}^3$.
Let $(\rho^{\mu,\lambda,\kappa}, u^{\mu,\lambda,\kappa}, \theta^{\mu,\lambda,\kappa}, 
f^{\mu,\lambda,\kappa})$ and $(\rho, u, \theta, f)$ 
be global solutions to the Cauchy problems of \eqref{Q1}--\eqref{Q1-1} and \eqref{Q2}--\eqref{Q2-2}, 
as obtained in
 Theorem \ref{T1.1} and Theorem \ref{T1.2}  respectively, associated with the initial data 
 $(\rho_0^{\mu,\lambda,\kappa}, 
 u_0^{\mu,\lambda,\kappa}, \theta_0^{\mu,
\lambda,\kappa}, f_0^{\mu,\lambda,\kappa})$ and $(\rho_0, u_0, \theta_0, f_0)$.
 Assume further that 
\begin{align}\nonumber %\label{c1}
\|(\rho_0^{\mu,\lambda,\kappa}-\rho_0,u_0^{\mu,\lambda,\kappa}-u_0,
\theta_0^{\mu,\lambda,\kappa}-\theta_0)\|_{H^1}+\|f_0^{\mu,\lambda,\kappa}-f_0\|_{H_{x,v}^1}\leq 
 \max\{\mu,\lambda,\kappa\}.
\end{align}
Then, for any $t\in \mathbb{R}^{+} $, it holds that,
\begin{align}\label{c2}
&\sup_{t\geq 0}\big(\|(\rho^ {\mu,\lambda,\kappa}-\rho,u^ {\mu,\lambda,\kappa}-u,
\theta^ {\mu,\lambda,\kappa}-\theta)(t)\|_{H^1}^2
+\|(f^ {\mu,\lambda,\kappa}-f)(t)\|_{H_{x,v}^1}^2\big)\nonumber\\
& +\int_0^t \big(\big\|\big((b^{\mu,\lambda,\kappa}-u^{\mu,\lambda,\kappa})-(b-u)\big)(\tau)\big\|_{H^1}^2 
 + \big\|\big((\sqrt{2}\omega^{\mu,\lambda,\kappa}-\sqrt{3}\theta^{\mu,\lambda,\kappa})
-(\sqrt{2}\omega-\sqrt{3}\theta)\big)(\tau)\big\|_{H^1}^2 \big){\rm d}\tau \nonumber\\
& +\int_0^t\Big( \|\nabla(\rho^{\mu,\lambda,\kappa}-\rho)(\tau)\|_{L^2}^2
+ \sum_{|\alpha|\leq 1}\|\{\mathbf{I}-\mathbf{P}\}\partial^\alpha 
(f^{\mu,\lambda,\kappa}-f)(\tau)\|_{\nu}^2
+\|\nabla (f^{\mu,\lambda,\kappa}-f)(\tau)\|_{L_{x,v}^2}^2\Big){\rm d}\tau\nonumber\\
&+\int_{0}^t \|\nabla (u^{\mu,\lambda,\kappa}-u,\theta^{\mu,\lambda,\kappa}
-\theta)(\tau)\|_{L^2}^2{\rm d}\tau+\int_0^t
 \|{\rm div}(u^{\mu,\lambda,\kappa}-u)(\tau)\|_{L^2}^2{\rm d}\tau  
  \leq  C_2  \max\{\mu,\lambda,\kappa\}^2,
%\quad t\in \mathbb{R}^{+},
\end{align}
where    $C_2>0$ is a constant independent of $\mu$, $\lambda$, $\kappa$ and    $t$. 
\end{thm}

\begin{rem} To the best of our knowledge, our work 
is the first systematic investigation of the vanishing
 viscosity limits and vanishing heat conductivity limit to
  non-isentropic fluid-particle models. The key element of our
  analysis is the uniform dissipation estimates of $u^{\mu,\lambda,\kappa}$ 
  and $\theta^{\mu,\lambda,\kappa}$, which primarily stem from the damping 
  effect of the new effective modes $b^{\mu,\lambda,\kappa}-u^{\mu,\lambda,\kappa}$, 
  $\sqrt{2}\omega^{\mu,\lambda,\kappa}-\sqrt{3}\theta^{\mu,\lambda,\kappa}$, and the 
diffusion effect of the microscopic quantities
 $b^{\mu,\lambda,\kappa}$ and $\omega^{\mu,\lambda,\kappa}$. 
\end{rem}
 
\begin{rem}
Theorem \ref{T1.3} shows that the convergence 
rate of the classical solution of the non-isentropic compressible NS-VFP 
system \eqref{Q1} to that of the non-isentropic compressible Euler-VFP 
system \eqref{Q2} is $\mathcal{O}(\max\{\mu,\lambda,\kappa\})$ which  
reflects that there is a certain competitive 
relationship among the viscosity coefficients $\mu, \lambda$ and the conductivity coefficient $\kappa$.
\end{rem} 

\begin{rem}
Obviously, our Theorem \ref{T1.3} still holds for the cases: \emph{(i)} 
 $\mu$ and $\lambda$ are fixed, and $ \kappa\rightarrow0$; \emph{(ii)} 
 $\mu \rightarrow 0,  \lambda\rightarrow 0$, and $\kappa $ is fixed. 
 For the case \emph{(i)}, we obtain the convergence of the system \eqref{LN-2} 
 to the system %\eqref{LN-4}
 \begin{equation}\label{LN-4}  
\left\{  
\begin{aligned}  
&\partial_t \varrho + {\rm div} (\varrho u) = 0, \\  
&\partial_t (\varrho u) + {\rm div} (\varrho u \otimes u) - \mu \Delta u - (\mu + \lambda)\nabla
 {\rm div} u + \nabla P = \int_{\mathbb{R}^3} (v-u)F\mathrm{d}v, \\  
&C_\textsl{v}(\partial_t (\varrho \vartheta ) + {\rm div} (\varrho u \vartheta )) + P {\rm div} u 
- 2\mu |D(u)|^2 - \lambda|{\rm div} u|^2 = \int_{\mathbb{R}^3} (|v-u|^2 - 3\vartheta )F\mathrm{d}v, \\  
&\partial_t F + v \cdot \nabla F = {\rm div}_v\big((v-u)F + \vartheta \nabla_v F\big), 
\end{aligned}  
\right.  
\end{equation}  
   via the 
  vanishing heat conductivity limit   $\kappa \to 0$. This limit is new even 
  for the pure fluid equations where we obtain the non-isentropic compressible Navier-Stokes 
  equations  without  heat conductivity  (i.e., \eqref{LN-5} without the term $-\kappa \Delta \theta$), see 
  \cite{DM-2008-IUMJ, TW-2025-JMAA} on the global well-posedess and time decay estimates to this system.
  For the case (ii), if we further taking $F=0$, we shall obtain the so-call compressible Euler-Fourier system  which was 
   initially introduced  in  \cite{CFK-15}. 
   
\end{rem} 

Next, we investigate the large-time behavior of classical 
solutions to the Cauchy problem \eqref{Q2}--\eqref{Q2-2}.
 Assuming that the initial data $(\rho_0,u_0,\theta_0,f_0)$ 
 is uniformly bounded in  $L^1$-norm, we derive optimal
 time-decay estimates for both the solutions and their 
 spatial gradients  for the system \eqref{Q2}, as presented below.

\begin{thm}\label{T1.4}
%Let $\Omega=\mathbb{R}^3$ and 
% $\kappa=0$.
 Suppose that the assumption
\eqref{b-1} holds and $\|(\rho_0,u_0,\theta_0)\|_{L^1}+\|f_0\|_{ {Z}_1}$ is bounded, then
the solution $(\rho,u,\theta,f)$ to the Cauchy problem \eqref{Q2}--\eqref{Q2-2} satisfying
 \begin{equation}\label{d1} 
\left\{
\begin{aligned}
 \|\nabla^k(\rho, u, \theta)(t)\|_{L^2 } +\|\nabla^kf(t)\|_{{L_v^2(L^{2})}}
&\leq C_3(1+t)^{-\frac{3}{4}-\frac{k}{2}},\quad k=0,1, \\
\|\nabla^k(\rho, u, \theta)(t)\|_{L^2 } +\|\nabla^kf(t)\|_{{L_v^2(L^{2})}}
&\leq C_3(1+t)^{-\frac{5}{4}},\quad\,\,\,\,\,\,\, k=2,3,4
\end{aligned}
\right.
\end{equation}
and
 \begin{equation} \label{d2}
\left\{
\begin{aligned}
\| (\rho, u, \theta)(t)\|_{L^p}+\|  f(t)\|_{L_v^2(L^p)}
 &\leq C_3(1+t)^{-\frac{3}{2}(1-\frac{1}{p})},\,\,\, 2\leq p\leq 6, \\  
\| (\rho, u, \theta)(t)\|_{L^p}+\|  f(t)\|_{L_v^2(L^p)}
 &\leq C_3(1+t)^{-\frac{5}{4} },\,\,\,\quad\quad\,  6\leq p\leq \infty. 
\end{aligned}
\right.
\end{equation}
Furthermore, we have 
\begin{align}\label{NJKd3}
 \|  (b-u)(t)\|_{L^2}+\|  (\sqrt{2}\omega-\sqrt{3}\theta)(t)\|_{L^2}
 +\|  \{\mathbf{I}-\mathbf{P}\}f(t)\|_{L_v^2(L^2)}  \leq C_3(1+t)^{-\frac{5}{4} }.
\end{align}
where    $C_3>0$ is a constant independent of    $t$. 

\begin{rem}
 We observe that the decay rate of   $L^2$-norm for the dissipative 
 quantities $ b - u $ and $ \sqrt{2}\omega - \sqrt{3}\theta $ is
  $ (1 + t)^{-\frac{5}{4}} $, which is  $\frac{1}{2}$-order faster 
  than the decay rate of the solution itself. This  new phenomenon is induced by the coupling effects between 
 the particles and the fluid and is constituted   first time for non-isentropic compressible fluid-particle models. 
\end{rem}

\end{thm}

\begin{rem}
Since $L^1(\mathbb{R}^3) \hookrightarrow \dot{B}^{-\frac{3}{2}}_{2,\infty}(\mathbb{R}^3)$, 
the $L^1$-boundedness condition in this work can be extended to the
 Besov space $\dot{B}^{-\frac{3}{2}}_{2,\infty}$.  In fact, it suffices to employ the 
 Littlewood-Paley decomposition techniques \emph{(}cf. \cite[Section 3.1]{DX-ARMA-2017}
  and \cite{Xu15}\emph{)}. Specifically, the inequality \eqref{G5.7} in Theorem \ref{T5.1} 
  can be rewritten as  
\begin{align*}  
\|\nabla^k \mathbb{A}U_0^L\|_{L^2} \lesssim (1+t)^{-\frac{s+k}{2}} \|U_0\|_{\dot{B}^{-s}_{2,\infty}}.  
\end{align*}  
This $\dot{B}^{-\frac{3}{2}}_{2,\infty}(\mathbb{R}^3)$ condition is strongly connected to 
the classical decay character theorem and is optimal for initial data to achieve the 
best possible decay rates of the solution to the heat equation \emph{(}see \cite{BL,BSXZ}\emph{)}.
\end{rem}

\begin{rem}
If we follow \cite{CDM-krm-2011,Mu-Wang-20}, the standard approach will involve defining 
a non-increasing functional with respect to time as follows:
\begin{align*}  
\Bar{\mathcal{E}}_{\infty}(t) = \sup_{0 \leq s \leq t} (1+s)^{\frac{3}{2}} \big( \|(\rho, u, 
\theta)(s)\|_{H^4}^2 + \|f(s)\|_{L_v^2(H^4)}^2 \big),  
\end{align*}  
and deriving the critical estimate:  
\begin{align*}  
\Bar{\mathcal{E}}_{\infty}(t) \lesssim \|(\rho_0,u_0,\theta_0,f_0)\|_{\mathcal{H}^4 \cap 
\mathcal{Z}_1}^2 + \Bar{\mathcal{E}}_{\infty}^2(t).  
\end{align*}  
Note that in this situation  the smallness of $\|(\rho_0,u_0,\theta_0,f_0)\|_{\mathcal{H}^4 \cap \mathcal{Z}_1}^2$ 
is essential. In the present work, by employing the low-high frequency decomposition method,
 we successfully address the challenges posed by 
 zero-order terms such as $\frac{au}{1+\rho}$, 
 $\frac{bu}{1+\rho}$, $\frac{a\theta}{1+\rho}$, and $-\frac{1}{2}u \cdot v f$. 
 Furthermore, without imposing any smallness restrictions on the $\mathcal{Z}_1$-norm 
 of $(\rho_0, u_0, \theta_0, f_0)$, we establish the optimal time decay rates 
 for $(\rho, u, \theta , f )$ and their spatial derivatives of first order.
\end{rem}

\begin{rem}
By carrying out analogous computations, we can establish the same decay rates for the 
classical solutions to the Cauchy problem \eqref{Q1}--\eqref{Q1-1} as presented in Theorem \ref{T1.4}, 
uniformly in $\mu>0$, $\lambda>0$ and $\kappa > 0$. In this paper, in stark contrast to the prior work 
  \cite{Mu-Wang-20}, which requires an additional $L^1$-smallness assumption on the initial data and 
  $H^4$-regularity, our findings are valid under the weaker assumption \emph{(}i.e.,  $L^1$-boundedness\emph{)} 
  and we further obtain the optimal time decay rates of the first-order spatial derivatives of $(\rho,u,\theta,f)$
which improve the results  in \cite{Mu-Wang-20}.
\end{rem}

\subsection{Main strategies of our  proofs}

To begin with, we first establish a priori estimates 
 uniformly in  $\mu,\lambda$ and $\kappa$  for the solution $(\rho^{\mu,\lambda,\kappa}, 
 u^{\mu,\lambda,\kappa}, \theta^{\mu,\lambda,\kappa}, f^{\mu,\lambda,\kappa})$ 
to the Cauchy problem   \eqref{Q1}--\eqref{Q1-1}. In contrast to 
 the approach used in \cite{Mu-Wang-20}, we first analyze the structural properties
  of the equation (see \eqref{G3.2}) before proceeding to estimate the energy norms of
   $(\rho^{\mu,\lambda,\kappa}, u^{\mu,\lambda,\kappa}, \theta^{\mu,\lambda,\kappa},
    f^{\mu,\lambda,\kappa})$.  This linearization treatment allows us to derive 
more precise and manageable energy estimates for both zeroth-order and higher-order
 terms (see Lemma \ref{L3.1}).

However, in the process of dealing with the higher-order estimates 
of $(\rho^{\mu,\lambda,\kappa}, u^{\mu,\lambda,\kappa}, \theta^{\mu,\lambda,\kappa},
 f^{\mu,\lambda,\kappa})$, the conventional energy method fails to provide effective estimates. 
 The main difficulty arises when dealing with the integral terms such as 
 \begin{align*}\int_{\mathbb{R}^3}\partial^\alpha(\rho^{\mu,\lambda,\kappa} 
 {\rm div}u^{\mu,\lambda,\kappa})\partial^\alpha\rho^{\mu,\lambda,\kappa}
 {\rm d}x \quad \text{and}\quad \int_{\mathbb{R}^3}\partial^\alpha(\theta^{\mu,\lambda,\kappa} 
 {\rm div}u^{\mu,\lambda,\kappa})\partial^\alpha\theta^{\mu,\lambda,\kappa}{\rm d}x\end{align*} 
 during the derivation
 of high-order energy estimates, because they cannot be directly controlled  
  by using standard energy estimates. 
 
To surround   this difficulty, we first rewrite the system \eqref{Q1} into a new form by leveraging 
system \eqref{LNXG3.13}--\eqref{LNXG3.16}. Inspired of the approach developed in \cite{At-ARMA-2006,JJLX-2014}, 
the most critical aspect is that we need to identify the symmetrical structure of 
the system \eqref{LNXG3.13}--\eqref{LNXG3.16}. We   divide \eqref{LNXG3.13} 
by $(1+\rho^{\mu,\lambda,\kappa})^2$, \eqref{LNXG3.14} by $ 1+\theta^{\mu,\lambda,\kappa} $, 
and \eqref{LNXG3.15} by $(1+\theta^{\mu,\lambda,\kappa})^2$ respectively, then  the challenging terms
 such as $(1+\rho^{\mu,\lambda,\kappa})\partial^\alpha{\rm div}{u}^{\mu,\lambda,\kappa}$, 
 $\frac{1+\theta^{\mu,\lambda,\kappa}}{1+\rho^{\mu,\lambda,\kappa}}
 \nabla\partial^\alpha\rho^{\mu,\lambda,\kappa}+\nabla\partial^\alpha\theta^{\mu,\lambda,\kappa}$, 
 and $(1+\theta^{\mu,\lambda,\kappa})\partial^\alpha{\rm div}u^{\mu,\lambda,\kappa}$ 
 are transformed into $\frac{\partial^\alpha{\rm div}{u}^{\mu,\lambda,\kappa}}{1+\rho^{\mu,\lambda,\kappa}}$, 
 $\frac{\nabla\partial^\alpha\rho^{\mu,\lambda,\kappa}}{1+\rho^{\mu,\lambda,\kappa}}
 +\frac{\nabla\partial^\alpha\theta^{\mu,\lambda,\kappa}}{1+\theta^{\mu,\lambda,\kappa}}$, 
 and $\frac{\partial^\alpha{\rm div}u^{\mu,\lambda,\kappa}}{1+\theta^{\mu,\lambda,\kappa}}$. 
 Subsequently, through integration by parts, we   exploit this symmetric structure to avoid 
 the loss of regularity and tackle these challenging terms, thereby obtaining more accurate 
 estimates (see Lemma \ref{L3.2}).
  Here, adding a weight of $\frac{1}{1+\theta^{\mu,\lambda,\kappa}}$ to the velocity 
  equation of $u^{\mu,\lambda,\kappa}$ 
demands meticulous observation, which leads to a certain symmetry property. 
 
 Unlike the results obtained in \cite{Mu-Wang-20}, which require $H^4$ regularity of the initial 
 data and 
 the coefficients $\mu, \lambda, $ and $  \kappa$ are restricted to $\mu=\kappa = 1$ (They also
  took $\lambda=0$ for simplicity), here  our estimates hold for all 
$\mu$, $\lambda$ and $\kappa > 0$ which reveals the dissipation effect on the velocity and the temperature of the
 fluid.  
Consequently, the terms involving $u^{\lambda,\mu,\kappa}$ 
and $\theta^{\mu,\lambda,\kappa}$ in the system \eqref{Q1} demand more refined estimates. Notably, 
the terms $-a u$ in the  equation \eqref{G3.2}$_2$, $-b \cdot u$ and $-3a\theta$ in the 
equation \eqref{G3.2}$_3$ are all zero-order nonlinear terms (Notice that in \eqref{G3.2}, 
the superscripts $\mu,\lambda$ and $\kappa$ on the unknowns is suppressed for presentation simplicity). 
Direct application of energy estimates is ineffective in this context. 
Therefore, it becomes necessary to combine these terms with the nonlinear terms in \eqref{G3.2}$_4$.
 By employing the macro-micro decomposition technique, we successfully address these challenges 
 (see the estimates \eqref{G3.6}--\eqref{G3.7}). Moreover, by considering the dissipation of
  $b^{\mu,\lambda,\kappa}-u^{\mu,\lambda,\kappa}$ and $\sqrt{2}\omega^{\mu,\lambda,\kappa} 
  - \sqrt{3}\theta^{\mu,\lambda,\kappa}$, 
we observe the uniform dissipation of $u^{\mu,\lambda,\kappa}$ and $\theta^{\mu,\lambda,\kappa}$:  
\begin{align*}  
\|\nabla u^{\mu,\lambda,\kappa}\|_{H^3} \lesssim&\, \| b^{\mu,\lambda,\kappa} 
- u^{\mu,\lambda,\kappa} \|_{H^4} + \|\nabla b^{\mu,\lambda,\kappa}\|_{H^3},\\
\|\nabla \theta^{\mu,\lambda,\kappa}\|_{H^3} \lesssim &\,\| \sqrt{2}\omega^{\mu,\lambda,\kappa} 
- \sqrt{3}\theta^{\mu,\lambda,\kappa} \|_{H^4} + \|\nabla\omega^{\mu,\lambda,\kappa}\|_{H^3}, 
\end{align*}  
where the combination $\sqrt{2}\omega^{\mu,\lambda,\kappa} 
- 
\sqrt{3}\theta^{\mu,\lambda,\kappa}$ acts as a new effective mode.  
This uniformity allows us to rigorously establish the stability of
 global classical solutions and analyze the limiting behavior as the viscosity coefficients and 
the heat conductivity coefficient vanish ($\mu,\lambda,\kappa \to 0$).  
By employing the continuous argument and  Aubin-Lions's lemma, we systematically 
construct global classical solutions to the Cauchy problem \eqref{Q2}--\eqref{Q2-2} via 
taking the  limits $\mu\rightarrow0$, $\mu\rightarrow0$ and  $\kappa \rightarrow  0$ simultaneously.

To rigorously demonstrate the global-in-time convergence of the limiting processes,
 it is crucial to establish a uniform-in-time  convergence rate by analyzing the error 
 between the global classical solutions of the systems \eqref{Q1} and \eqref{Q2}. 
 To achieve this goal, we introduce the following new energy functional (see \eqref{G4.4}):  
\begin{align*}  
\Xi^{\mu,\lambda,\kappa}(t) :=&\sup_{\tau\in[0,t]}
\big(\|(\widetilde{\rho}^{\mu,\lambda,\kappa},
\widetilde{u}^{\mu,\lambda,\kappa} ,
\widetilde{\theta}^{\mu,\lambda,\kappa})(\tau)\|_{H^1}^2 
+\|\widetilde{f}^{\mu,\lambda,\kappa}(\tau)\|_{H_{x,v}^1}^2\big) \nonumber\\
&+\int_{0}^t\big( \|\nabla(\widetilde{a}^{\mu,\lambda,\kappa},
\widetilde{b}^{\mu,\lambda,\kappa},\widetilde{\omega}^{\mu,\lambda,\kappa},
\widetilde{\rho}^{\mu,\lambda,\kappa})(\tau)\|_{L^2}^2 {\rm d}\tau\nonumber\\  
&+\int_{0}^t\big(   \|(\widetilde{b}^{\mu,\lambda,\kappa}-\widetilde{u}^{\mu,
\lambda,\kappa})(\tau)\|_{H^1}^2 + \|\big(\sqrt{2}\widetilde{\omega}^{\mu,
\lambda,\kappa}-\sqrt{3}\widetilde{\theta}^{\mu,\lambda,\kappa}\big)(\tau)\|_{H^1}^2 
\big){\rm d}\tau\nonumber\\  
&+\int_0^t\sum_{|\alpha|\leq 1}\|\{\mathbf{I}
-\mathbf{P}\}\partial^\alpha\widetilde{f}^{\mu,\lambda,\kappa}(\tau)\|_{\nu}^2{\rm d}\tau
+\int_0^t \|\nabla_v \{\mathbf{I}-\mathbf{P}\}\widetilde{f}^{\mu,\lambda,\kappa}(\tau)\|_{\nu}^2{\rm d}\tau.  
\end{align*}  
Note that our VFP equations in \eqref{Q1} and \eqref{Q2} incorporate  the macro-micro 
decomposition of particles. Consequently, estimating $(a^{\mu,\lambda,\kappa},
 b^{\mu,\lambda,\kappa}, \omega^{\mu,\lambda,\kappa})$ 
becomes more intricate and challenging, thereby complicating our energy estimate 
(see Lemma \ref{L4.3}). By employing a refined energy method, we derive 
Lemmas \ref{L4.1}--\ref{L4.5}, which allow us to rigorously justify the 
global-in-time vanishing viscosity limits and
heat conductivity limit with a convergence rate of $\mu$, $\lambda$ and $ {\kappa}$
 via the following uniform estimate:  
\begin{align*}  
{\Xi} ^{\mu,\lambda,\kappa}(t) \leq C \big(\eps_0^{\frac{1}{2}} + \eps_1^{\frac{1}{2}}\big){\Xi} ^{\mu,\lambda,\kappa}(t) + C\max\{\mu,\lambda,\kappa\}^2.  
\end{align*}  
Thus, Theorem \ref{T1.3} is rigorously established.

Next, we investigate the optimal time-decay rates of classical solutions for the Cauchy problem \eqref{Q2}--\eqref{Q2-2}
of the non-isentropic compressible Euler-VFP system. To begin with, we linearize the system \eqref{Q2} 
(as shown in \eqref{G5.1}) and subsequently derive sharp pointwise estimates for the solution to
 the linearized system. These estimates indicate that the 
 solutions display heat-like diffusion behavior at low frequencies 
 and exhibit damping effects at high frequencies. Inspired by this insight, 
 we apply low-high frequency decomposition techniques to construct a general decay framework. 
 This framework allows us to establish the  optimal decay rates under precise $L^1$ assumptions. 
 As a direct result, we obtain $L^p-L^q$ estimates for the linearized equations (see Theorem \ref{T5.1}).

Based on the result of the linearized system obtained above, 
we can  analyze the decay estimates of the nonlinear Cauchy problem \eqref{Q2}--\eqref{Q2-2}. By applying  Duhamel's principle, the solution $U(t)$ of the nonlinear problem can be rewritten as (see \eqref{G5.54}):  
\begin{align*} 
U(t)=\mathbb{A}(t)U_{0}+\int_{0}^{t}\mathbb{A}(t-s)\big( S_\rho(s),S_u(s), S_\theta(s), S_f(s)\big) {\rm d}s,
\end{align*}    
The presence of the zero-order nonlinear terms $\frac{b-u-au}{1+\rho}$ in the equation \eqref{Q2}$_2$, and $\frac{|u|^2-2u\cdot b-3a\theta}{1+\rho}$ and $\frac{\sqrt{3}\rho(\sqrt{2}\omega-\sqrt{3}\theta)}{1+\rho}$ in the equation \eqref{Q2}$_3$ complicates the derivation of higher-order spatial derivatives and time-decay rates for $(\rho, u, \theta, f)$. 
The conventional energy method only provides the optimal decay rates for zero-order derivatives within the $L^2$ framework. To circumvent this limitation, we establish estimates for second and third-order derivatives using frequency decomposition techniques and a refined energy method. This leads to the critical inequality (see Proposition \ref{P5.2}):  
\begin{align}  
\label{LN-inequality}  
&\frac{{\rm d}}{{\rm d}t}\big( \|\nabla^k(\rho,u,\theta)\|_{L^2}^2 + \|\nabla^k f\|_{L_{x,v}^2}^2\big)\nonumber\\ &\quad + \lambda_{12}\big( \|\nabla^k(\rho,u,\theta)\|_{L^2}^2 + \|\nabla^k f\|_{L_{x,v}^2}^2\big) 
\lesssim \|\nabla^k(a^L,b^L,\omega^L,\rho^L,\theta^L)\|_{L^2}^2,  
\end{align}  
for $k=2,3,4$. This step is crucial for obtaining the optimal time-decay rates of classical solutions 
and their spatial gradients. Since the nonlinear term $\frac{\theta}{\sqrt{M}} \Delta_v(\sqrt{M} f)$
 in \eqref{Q2}$_4$ involves second-order velocity 
derivatives such as $\Delta_v f$, it is nontrivial to transfer the decay estimates of
 the nonlinear terms fully into the energy $\mathcal{E}_2(t)$ 
during the estimation process. This could potentially lead to energy loss, 
which differs from the findings reported in \cite{Ww-CMS-2024} for the isentropic NS-VFP system \eqref{H2}.
For instance, achieving optimal time-decay rates for $\|\nabla^2(\rho,u,\theta,f)\|_{\mathcal{H}^2}$ 
faster than $(1+t)^{-\frac{5}{4}}$ 
is infeasible (see \eqref{G5.73}). In the context of nonlinear estimation, 
when analyzing the optimal decay rates of the solutions, it is essential to  carefully handle  
the terms $L_6(t)$ and $L_7(t)$ carefully by separating the zeroth-order terms 
from higher-order  {ones} for distinct treatment (see \eqref{G5.74}). 
Subsequently, by   combining  Theorem \ref{T5.1}, \eqref{LN-inequality}, 
and the nonlinear estimates of $(\rho,u,\theta,f)$, 
we deduce the optimal time-decay rates of classical solutions and their spatial gradients. Furthermore,
 through the application of the interpolation inequality, we establish the optimal decay rate 
$(1+t)^{-\frac{3}{2} (1-\frac{1}{p} )}$ in $L^p$-norm with $2\leq p\leq 6$, 
surpassing the method presented in \cite{Mu-Wang-20} .
 However, since the term $\frac{\theta}{\sqrt{M}}\Delta_v(\sqrt{M}f)$ cannot be transformed fully 
into an absolute energy $\|(\rho,u,\theta)(t)\|_{H^4}^2+\|f(t)\|_{L_v^2(H^4)}^2$, 
and consequently the corresponding dissipation $\mathcal{D}(t)$ cannot be derived, 
it seems impossible to determine the optimal time-decay rates for the second-order 
and third-order spatial derivatives of $(\rho,u,\theta,f)$ directly. 

 By leveraging the damping effect inherent in the equations of the  combinations $b-u$, $\sqrt{2}\omega - \sqrt{3}\theta$, 
 and  the microscopic part $\{\mathbf{I} - \mathbf{P}\}f$, we estimate their nonlinear 
 components in the framework of semigroup theory (see \eqref{NJKG3.162}–\eqref{NJKG3.163} and
  \eqref{NJKG3.165}). Consequently, using the decay rates we have established, we further 
  demonstrate that the dissipation decay rates of 
 the quantities $b - u$, $\sqrt{2}\omega - \sqrt{3}\theta$, and $\{\mathbf{I} - 
 \mathbf{P}\}f$ are $\frac{1}{2}$-order faster than those of the solution. 
 Therefore, the proof of Theorem \ref{T1.4} is completed. 

Finally, by leveraging the  conservation quantities and Poincar\'{e} inequality, 
we demonstrate that the decay rate of  classical solutions to the system \eqref{Q1} 
on $\mathbb{T}^3$ is exponential,  uniformly in  $\mu$,  $\lambda$  and
$\kappa$.

\subsection{Outline of this paper}
The remainder of this paper is structured as follows. In Section 2, we present some notations, 
the macro-micro decomposition of $f(t,x,v)$, some properties of the Fokker-Planck operator
 $\mathcal{L}$, and several key lemmas that are extensively employed throughout this work. 
 In Section 3, we focus on the well-posedness  of the non-isentropic  compressible NS-VFP  
 system \eqref{Q1}. Specifically, we   establish the global existence of classical solutions
  for the Cauchy problem  \eqref{Q1}--\eqref{Q1-1} in $\mathbb{R}^3$,  uniformly in 
   $\mu>0$, $\lambda>0$ and $\kappa > 0$. 
Subsequently, in Section 4, we analyze the  vanishing   viscosity limits
  $\mu,\lambda\rightarrow0$ and
  heat conductivity limit   $\kappa \to 0$   and then establish a convergence rate of 
  $\mathcal{O}(\max\{\mu,\lambda,\kappa\})$ for the solution $(\rho^{\mu,\lambda,\kappa}, u^{\mu,\lambda,\kappa}, \theta^{\mu,\lambda,\kappa}, f^{\mu,\lambda,\kappa})$. 
  Building upon these results, we further prove the global stability of the Cauchy problem 
  \eqref{Q2}--\eqref{Q2-2}  in the $H^4$ framework.
In Section 5, for the case $\mu=\lambda=\kappa = 0$, we employ the method of low- and high-frequency 
decomposition, followed by a systematic development of linear and nonlinear analyses and the construction
 of suitable energy functionals to derive the optimal decay rates for $(\rho, u, \theta, f)$ and their 
 spatial gradients. Finally, in Section 6, we extend the results obtained in Sections 3–4 to the case
  of a periodic domain.
In this situation, we   achieve the exponential decay rates for classical solutions via the intrinsic 
 conservation   quantities of the systems and the well-known Poincar\'{e} inequality.

%%%%%%%%%%%%%%%%%%%%%%%%%%%%%%%%%%%
 \section{Preliminaries}
In this section, we introduce some notations, definitions, foundational facts, 
and useful lemmas that are extensively utilized throughout this paper.

\subsection{Notations}
The symbols $C$ and $C_i$ denote  a generic positive constant that is independent of the time 
$t$ and may vary across different instances. 
Moreover, the notation $A \lesssim B$ signifies that $A \leq CB$ for some constant $C>0$. 
The symbol $A\backsim B$ represents the relationship $\frac{A}{C}\leq B\leq CA$, 
where $C>1$ is some general positive constant. For simplicity, we define $\|(g, h)\|_X := \|g\|_X + \|h\|_X$ 
for a Banach space $X$, with $g = g(\cdot)$ and $h = h(\cdot)$ being two elements of $X$.
And we use $\Delta$ to denote $\Delta_x$ and $\nabla$ to denote $\nabla_x$.

We employ $\langle\cdot,\cdot\rangle$ to denote the inner product in 
the Hilbert space $L^2_v=L^2(\mathbb{R}^3_v)$, which is defined as  
\begin{align*}  
\langle g,h\rangle := \int_{\mathbb{R}^3} g(v) h(v) \, \mathrm{d}v, \quad \forall \, g,h \in L^2_v,  
\end{align*}  
equipped with its associated norm $ |\cdot |_{L^2}$.
We define the weight function $\nu(v):=1+|v|^2$ and introduce the norm $|\cdot|_{\nu}$ as 
% follows:  
\begin{align*}  
|g|_{\nu}^{2} := \int_{\mathbb{R}^3} \big(|\nabla_v g(v)|^2 
+ \nu(v) |g(v)|^2\big) {\rm d}v, \quad \text{where } g = g(v).  
\end{align*}  
For the norm associated with the dissipation of the following Fokker-Planck 
operator $\mathcal{L}$ 
\begin{align} \label{lglga}
\mathcal{L}g := \frac{1}{\sqrt{M}} \nabla_v \cdot \Big[ M \nabla_v \Big( \frac{g}{\sqrt{M}} \Big) \Big], 
\quad g=g(t,x,v) ,
\end{align}  
we denote it by  
\begin{align*}  
\|g(t)\|_{\nu}^{2} := \iint_{\mathbb{R}^3 \times \mathbb{R}^3} 
\big(|\nabla_v g(t,x,v)|^2 + \nu(v) |g(t,x,v)|^2\big) {\rm d}x {\rm d}v.  
\end{align*}

When incorporating the spatial variable, we write $L^2_{x,v} = L^2(\mathbb{R}^3_x 
\times \mathbb{R}^3_v)$ associated with its norm $\|\cdot\|_{L^2_{x,v}}$. 
The norm in the space $L^p = L^p(\mathbb{R}^3_x)$ (resp. $L^2_v(L^p )
 = L^2(\mathbb{R}^3_v; L^p(\mathbb{R}^3_x))$) is denoted by $\|\cdot\|_{L^p}$ 
 (resp. $\|\cdot\|_{L^2_v(L^p )}$). For $q \geq 1$, the standard spatial-velocity mixed
 Lebesgue space $ {Z}_q = L_v^2(L ^q) = L^2(\mathbb{R}_v^3; L^q(\mathbb{R}_x^3))$ is defined as  
\begin{align*}  
\|g\|_{Z_q}^2 := \int_{\mathbb{R}^3} \Big( \int_{\mathbb{R}^3} |g(x,v)|^q \, \mathrm{d}x \Big)^{\frac{2}{q}} \mathrm{d}v.  
\end{align*}
Furthermore, we introduce the norms $\|\cdot\|_{\mathcal{Z}_{q}}$ and $\|\cdot\|_{\mathcal{H}^{m}}$, where $m$ is a non-negative integer ($m \geq 0$) and $q \geq 1$. These norms are defined as  
\begin{align*}  
\|(\rho, u,\theta,f)\|_{\mathcal{Z}_{q}} & := 
 \|\rho\|_{L^{q}} + \|u\|_{L^{q}}+\|\theta\|_{L^q}+\|f\|_{ {Z}_{q}} , \\  
\|( \rho, u,\theta,f)\|_{\mathcal{H}^{m}} & :=  \|\rho\|_{H^{m}} + \|u\|_{H^{m}}+\|\theta\|_{H^m}+\|f\|_{L^{2}_{v}(H^m)}.  
\end{align*}

For an integrable function $g(x): \mathbb{R}^3 \to \mathbb{R}$, its Fourier transform is defined as  
\begin{equation*}  
\widehat{g}(\xi) = \mathcal{F}[g](\xi) = \int_{\mathbb{R}^3} e^{-i x \cdot \xi} g(x) \, {\rm d}x, \quad \text{where } x \cdot \xi = \sum_{j=1}^3 x_j \xi_j,  
\end{equation*}  
for $\xi \in \mathbb{R}^3$. 
Here, $i = \sqrt{-1} \in \mathbb{C}$ denotes the imaginary unit. For two complex-valued functions $f$ and $g$, the inner product $(g|h) := g \cdot \Bar{h}$ represents the dot product of $g$ with the complex conjugate of $h$.  

For any multi-indices $\alpha = (\alpha_1, \alpha_2, \alpha_3)$ and $\beta = (\beta_1, \beta_2, \beta_3)$, we denote    
\begin{align*}  
\partial_\beta^\alpha :=  \partial_{x_1}^{\alpha_1}\partial_{x_2}^{\alpha_2}\partial_{x_3}^{\alpha_3}   \partial_{v_1}^{\beta_1}\partial_{v_2}^{\beta_2}\partial_{v_3}^{\beta_3},  
\end{align*}  
representing the partial derivatives with respect to $x = (x_1, x_2, x_3)$ and $v = (v_1, v_2, v_3)$. The lengths of $\alpha$ and $\beta$ are defined as $|\alpha| = \alpha_1 + \alpha_2 + \alpha_3$ and $|\beta| = \beta_1 + \beta_2 + \beta_3$, respectively. Furthermore, we define  
\begin{align*}  
\|g\|_{H^s} := \sum_{|\alpha| \leq s} \|\partial^\alpha g\|_{L^2}, \quad \|g\|_{H_{x, v}^s} := \sum_{|\alpha| + |\beta| \leq s} \|\partial_\beta^\alpha g\|_{L^2_{x,v}}.  
\end{align*} 

\subsection{Low-high frequency   decomposition}

Choose $\phi_0(\xi)$ and $\phi_1(\xi)$ be two smooth cutoff functions satisfying  
\begin{align*}  
0 \leq \phi_0(\xi),\, \phi_1(\xi) \leq 1, \quad \phi_1(\xi) = 1 - \phi_0(\xi),  
\end{align*}  
with  
\begin{align}\label{G2.2}  
\phi_0(\xi) = \begin{cases}  
1, & |\xi| < \frac{r_0}{2}, \\  
0, & |\xi| > r_0,  
\end{cases}  
\end{align}  
for some constant $r_0 > 0$ to be determined later.  

Define $D_x = \frac{1}{i} (\partial_{x_1}, \partial_{x_2}, \partial_{x_3})$.
 For any function $g(x) \in L^2(\mathbb{R}^3)$, we define  
\begin{align}\label{G2.1}  
g^{L}(x) := \phi_0(D_x)g(x), \quad g^{H}(x) := \phi_1(D_x)g(x),  
\end{align}  
where $\phi_0(D_x)$ and $\phi_1(D_x)$ are the pseudo-differential operators associated with $\phi_0(\xi)$ and $\phi_1(\xi)$, respectively.   By virtue of \eqref{G2.2}--\eqref{G2.1}, we have
\begin{align*}  
g(x) = g^{L}(x) + g^{H}(x).  
\end{align*}  

Applying the well-known Fourier-Plancherel's  theorem, it is readily to prove  the follows facts: 

\begin{lem}
For any $g \in H^2(\mathbb{R}^3)$, it hold that 
\begin{align}  \label{G2.3}
\|g^H\|_{L^2} \leq \frac{C}{r_0} \|\nabla g\|_{L^2}, \quad \|g^H\|_{L^2} \leq \frac{C}{r_0^2} \|\nabla^2 g\|_{L^2}, \quad \|\nabla^2 g^L\|_{L^2} \leq C r_0 \|\nabla g^L\|_{L^2},  
\end{align}
 for some $C>0$.  
\end{lem}

\subsection{Macro-micro decomposition and properties of   {\(\mathcal{L}\)} }
We shall apply the macro-micro decomposition to $f^{\mu,\lambda,\kappa}(t,x,v)$ and $f(t,x,v)$, which acts as a crucial tool for deriving the subsequent estimates. This methodology was initially introduced in \cite{GY-iumj-2004} for the Boltzmann equation and later generalized in \cite{DFT-2010-CMP} for the Fokker-Planck type equations (see also \cite{Mu-Wang-20,LNW-SAPM-2025,LNW-arXiv-2025}).

Recall the global  Maxwellian $M =M(v)=(2\pi)^{- {3}/{2}}e^{- {|v|^2}/{2}}$ defined in \eqref{maxwell}.     
We introduce   the following velocity orthogonal projection $\mathbf{P}$ of $L_{v}^{2}(\mathbb{R}^3)$:  
\begin{equation*}  
\mathbf{P}: L_{v}^{2}(\mathbb{R}^3) \to \mathrm{Span}\{\sqrt{M}, v_{1}\sqrt{M}, v_{2}\sqrt{M}, v_{3}\sqrt{M}, |v|^2\sqrt{M}\},  
\end{equation*}  
and  for any $g(t,x,v) $,
\begin{align*}  
\mathbf{P} := \mathbf{P}_{0} \oplus \mathbf{P}_{1} \oplus \mathbf{P}_{2},  
\end{align*}  
where $\mathbf{P}_{0}g := a\sqrt{M}$, $\mathbf{P}_{1}g := b \cdot v\sqrt{M}$, and $\mathbf{P}_{2}g := \omega\frac{|v|^2-3}{\sqrt{6}}\sqrt{M}$, with 
  $a = a^g(t,x)$, $b = b^g(t,x)$, and $\omega = \omega^g(t,x)$ are given by  
\begin{gather*} 
a^g(t,x)   := \int_{\mathbb{R}^3} \sqrt{M} g(t,x,v) \, {\rm d}v,  \quad 
b^g(t,x)   := \int_{\mathbb{R}^3} v \sqrt{M} g(t,x,v) \, {\rm d}v, \\ 
\omega^g(t,x)   := \int_{\mathbb{R}^3} \frac{|v|^2 - 3}{\sqrt{6}} \sqrt{M} g(t,x,v) \, {\rm d}v. 
\end{gather*}

%Note that  the  linearized Fokker-Planck operator $\mathcal{L}$  is defined as  
%\begin{align*} 
%\mathcal{L}g := \frac{1}{\sqrt{M}} \nabla_v \cdot \Big[ M \nabla_v \Big( \frac{g}{\sqrt{M}} \Big) \Big], 
%\end{align*}  
%for any $g(t,x,v) $.

Using the projection $\mathbf{P}$, it is readily to check that $g(t,x,v) $ 
  can be decomposed as 
  % follows:  
\begin{align}\label{G2.4}  
g(t,x,v) = \mathbf{P}g + \{\mathbf{I}-\mathbf{P}\}g.  
\end{align}  
Here, $\mathbf{P}g$ and $\{\mathbf{I}-\mathbf{P}\}g$ represent the macroscopic part and   microscopic part
of $g$, respectively. Meanwhile, $\mathcal{L}g$ has the following decomposition:  
\begin{align}\label{G2.5}  
\mathcal{L}g = \mathcal{L}\{\mathbf{I}-\mathbf{P}\}g + \mathcal{L}\mathbf{P}g = \mathcal{L}\{\mathbf{I}-\mathbf{P}\}g - \mathbf{P}_{1}g - 2\mathbf{P}_{2}g.  
\end{align}  
Noting that $\mathcal{L}$ is self-adjoint, there exists a positive constant $\bar\lambda > 0$ such that (cf. \cite{Mu-Wang-20,DFT-2010-CMP,LNW-SAPM-2025,CDM-krm-2011}):  
\begin{align} 
-\langle \mathcal{L}g, g\rangle &\geq \bar\lambda|\{\mathbf{I}-\mathbf{P}_0\}g|_{\nu}^{2}, \label{G2.6a}  \\  
-\langle \mathcal{L}\{\mathbf{I}-\mathbf{P}\}g,g\rangle &\geq \bar\lambda|\{\mathbf{I}-\mathbf{P}\}g|_{\nu}^{2}, \label{G2.6b}  \\  
-\langle \mathcal{L}g, g\rangle &\geq \bar\lambda |\{\mathbf{I}-\mathbf{P}\}g|_{\nu}^{2} + |b|^{2} + 2|\omega|^2.\label{G2.6}  
\end{align}

\subsection{Analytic tools}

In this subsection we 
  present several useful facts that will be frequently employed throughout this paper.

\begin{lem}  [\!{\! {\cite{AF-Pa-2003}}} and {\cite[Lemma 2.1]{CDM-krm-2011}}]
\label{LA.1}   
There exists a   positive constant $C$, 
such that for any $g,h\in H^4(\mathbb{R}^3)$ and any multi-index $\gamma$  with $1\leq|\gamma|\leq4$, it holds
\begin{align*}
\|g\|_{L^{\infty}(\mathbb{R}^{3})} \leq\,& C\|\nabla g\|_{L^{2}(\mathbb{R}^{3})}^{1/2}
\|\nabla^{2}g\|_{L^{2}(\mathbb{R}^{3})}^{1/2}, \\
\|gh\|_{H^{3}(\mathbb{R}^{3})} \leq\,& C\|g\|_{H^{3}(\mathbb{R}^{3})}\|\nabla h\|_{H^{3}(\mathbb{R}^{3})}, \\
\|\partial^{\gamma}(gh)\|_{L^{2}(\mathbb{R}^{3})}
 \leq\,& C\|\nabla  g\|_{H^{3}(\mathbb{R}^{3})}\|\nabla  h\|_{H^{3}(\mathbb{R}^{3})},\\
\|g\|_{L^6(\mathbb{R}^3)} \leq\,& C\|\nabla  g\|_{L^2(\mathbb{R}^3)}\leq C\|g\|_{H^1(\mathbb{R}^3)},\\
 \|g\|_{L^q(\mathbb{R}^3)} \leq\,& C\|g\|_{H^1(\mathbb{R}^3)}, 
\end{align*}
for some $2\leq q\leq 6$.
\end{lem}

\begin{lem}[\! \!{\cite{commutator1,commutator2}}]\label{LA.2}
Let $h$ and $g$ be two Schwarz functions. For $k\geq 0$, one has
\begin{align*}
\|\nabla^{k}(gh) \|_{L^r} \leq &\, C\|g\|_{L^{r_1} }\|\nabla^{k}h\|_{L^{r_2} }+C\|h\|_{L^{r_3} }\|\nabla^{k}g\|_{L^{r_4} },\\
\|\nabla^{k}(gh)-g\nabla^k h \|_{L^r} \leq  &\, C\|\nabla g\|_{L^{r_1}}\|\nabla^{k-1}h\|_{L^{r_2}}+C\|h\|_{L^{r_3}}\|\nabla^{k}g\|_{L^{r_4}},    
\end{align*}
where $1<r,r_2,r_4<\infty$ and $r_i(1\leq i\leq 4)$ satisfy 
\begin{align*}
\frac{1}{r_1}+\frac{1}{r_2}=\frac{1}{r_3}+\frac{1}{r_4}=\frac{1}{r}.   
\end{align*}
\end{lem}

\begin{lem} [{\!\! \cite[Lemma 3.2]{CDM-krm-2011}}]\label{LA.3}   
Given any $0<\beta_1\ne 1$ and $\beta_2>1$, then
\begin{align*}
\int_0^t (1+t-s)^{-\beta_1}(1+s)^{-\beta_2} {\rm d}s \leq C(1+t)^{-\min\{\beta_1,\beta_2\}}, 
\end{align*}
for all $t\geq 0$.
\end{lem} 

\begin{lem} [{\!\! \cite[Lemma 3.3]{CDM-krm-2011}}]\label{LA.4}   
Let $\gamma>1$ and $g_1,g_2\in C( \mathbb{R}^+, \mathbb{R}^+)$ with
$g_1(0)=0$. For $A\in  \mathbb{R}^+$, define $ \mathds{B}_{A}:=\{y\in C( \mathbb{R}^+,( \mathbb{R}^+)|\, \,y\leq A+g_1(A)y+g_2(A)y^{\gamma},\,\, y(0)\leq A\}$.
Then, there exists a constant $A_0\in (0,\,\min\{A_1,A_2\})$ such that for
any $0<A\leq A_0$, 
\begin{align*}
y\in \mathds{B}_A \Rightarrow \sup_{t\geq 0}y(t)\leq 2A.
\end{align*}
\end{lem}

%\section{Global existence of classical solutions  for \texorpdfstring{\(\kappa>0\)}{kappa>0}}
\section{Global existence of classical solutions  for  fixed  {\(\mu,\lambda, \kappa>0\)} }
Throughout this  section, we simplify the notations by suppressing the superscripts $\mu$, $\lambda$ and $\kappa$ in the Cauchy problem \eqref{Q1}--\eqref{Q1-1}. Our objective is to establish the global existence of classical solutions to the  problem \eqref{Q1}--\eqref{Q1-1} in the whole space $\mathbb{R}^3$, uniformly in $\mu>0$, $\lambda>0$ and $\kappa>0$.

\subsection{A priori estimates}  
Let $(\rho, u, \theta, f)$ denote the classical solution to the problem \eqref{Q1}--\eqref{Q1-1} on $0 \leq t < T$ for some $T > 0$.
To derive a uniform-in-time a priori estimate for $(\rho, u, \theta, f)$ for any fixed \(\mu,\lambda, \kappa>0\), we assume that  
\begin{align}\label{G3.1}  
\sup_{0 \leq t < T} \big\{\|(\rho,u,\theta)\|_{H^4} + \|f\|_{H^4_{x,v}} \big\} \leq \delta,  
\end{align}  
where the constant $0 < \delta < 1$, which is independent of $\mu,\lambda$ and $\kappa$ and will be carefully determined later.

For the sake of presentation,  we reformulate the problem \eqref{Q1}--\eqref{Q1-1}  as  
follows: 
\begin{equation}\label{G3.2}
\left\{\begin{aligned}
&\partial_t \rho  +{\rm div} u=\mathcal{N}_1, \\
&\partial_t u + \nabla \rho+\nabla \theta-\mu\Delta  u-(\mu+\lambda)\nabla{\rm div} u-(b-u)=\mathcal{N}_2-au, \\
&\partial_t \theta + {\rm div} u-\sqrt{3}(\sqrt{2} \omega-\sqrt{3} \theta) -\kappa\Delta \theta = \mathcal{N}_3-b\cdot u-3a\theta,\\
&\partial_t f +v \cdot \nabla f-
u \cdot v \sqrt{M}- (|v|^2-3 ) \sqrt{M} \theta -\mathcal{L} f=\mathcal{N}_4-\frac{1}{2}u\cdot v f+\frac{\theta}{\sqrt{M}}\Delta_v(\sqrt{M}f),\\
\end{aligned}\right.
\end{equation} 
with the initial data
\begin{align}\label{G3.3}
(\rho ,u ,\theta ,f )|_{t=0}=&\,(\rho _0(x),u _0(x),\theta _0(x),f _0(x,v))\nonumber\\
:=&\,\Big( \varrho_0 (x)-1,u _0(x), {\vartheta} _0(x)-1,\frac{F _0(x,v)-M}{\sqrt{M}}\Big), \quad   x \in \mathbb{R}^3, \, v\in\mathbb{R}^3. 
\end{align} 
Here, the nonlinear terms $\mathcal{N}_i(1\leq i\leq 4)$ are defined by
\begin{equation*}
\left\{
\begin{aligned}
\mathcal{N}_1:=&\,-\rho{\rm div} u-\nabla\rho\cdot u, \\
\mathcal{N}_2:=&\,-u\cdot\nabla u-h(\rho) \theta \nabla\rho +g(\rho)\big(\mu\Delta  u+(\mu+\lambda)\nabla{\rm div} u-\nabla\rho+(b-u)-au \big), \\
\mathcal{N}_3:=&\,h(\rho)\big(2\mu|D(u)|^2+\lambda|{\rm div} u|^2 -u\cdot(b-u)+a|u|^2\big)-\theta{\rm div}u\\
&-u\cdot\nabla\theta+ g(\rho)\big(\kappa\Delta \theta+\sqrt{3}(\sqrt{2}w-\sqrt{3}\theta)-b\cdot u-3a\theta\big)   ,\\
\mathcal{N}_4:=&\, u\cdot\nabla_v f, 
\end{aligned}\right.
\end{equation*}
where
\begin{equation}\label{define:g,h}
%\left\{
%\begin{aligned}
g(\rho):=  \frac{1}{1+\rho}-1 \ \ \ \text{and} \ \ \
h(\rho):= \frac{1}{1+\rho}.
%\end{aligned}\right.
\end{equation}
It is evident that  
\begin{align*}
|g(\rho)|\leq C|\rho|,\ \ \quad |h(\rho)|\leq C, 
\end{align*}
and for any integer $k\geq 1$, it holds that
\begin{align*}
|g^{(k)}(\rho)|\leq C,\  \ \quad |h^{(k)}(\rho)|\leq C,
\end{align*}
for some constant $C>0$. Moreover, by invoking the embedding $H^2(\mathbb{R}^3_x) \hookrightarrow L^{\infty}(\mathbb{R}^3_x)$ and 
the assumption inequality \eqref{G3.1}, we deduce that  
 \begin{align}\label{LNXG3.2}  
\frac{1}{2} \leq \rho + 1 \leq \frac{3}{2}, \quad \frac{1}{2}\leq\theta+1\leq \frac{3}{2}.  
\end{align} 

We first provide the $L^2$-estimate of     $(\rho, u, \theta, f)$.

\begin{lem}\label{L3.1}
For the  classical solution $(\rho, u,\theta,f)$ to the Cauchy problem \eqref{Q1}--\eqref{Q1-1}, it holds that 
\begin{align}\label{G3.4}
&\frac{\rm d}{{\rm d}t}\big(\|(\rho,u,\theta)\|_{L^2}^{2} +\|f\|_{L_{x,v}^{2}}^2\big)+\lambda_1\big(\|b-u\|_{L^2}^{2}
+\|\sqrt{2}\omega-\sqrt{3}\theta\|_{L^2}^2+\|\{\mathbf{I}-\mathbf{P}\}f\|_{\nu}^{2}\big)\nonumber\\
&%\quad\quad\quad
\quad+\lambda_1\big(\mu\|\nabla u\|_{L^2}^2+(\mu+\lambda)\| {\rm div}u\|_{L^2}^2+\kappa\|\nabla\theta\|_{L^2}^2\big)
 \leq C\delta\|\nabla(a,b,\omega,\rho,u,\theta)\|_{L^2}^2,
\end{align}
for any $0 \leq t < T$ and $\lambda_1>0$,   independent of  $T$, $\mu$, $\lambda$ and
$\kappa$. 
\end{lem}
\begin{proof}
Multiplying  the equations  \eqref{G3.2}$_1$--\eqref{G3.2}$_4$  by $\rho$, $u$, $\theta$, and $f$ respectively, followed by integration and summation, we get  the following equality:
\begin{align}\label{G3.5}
&\frac{1}{2}\frac{\rm d}{{\rm d}t}\big(\|(\rho,u,\theta)\|_{L^2}^{2} +\|f\|_{L_{x,v}^{2}}^2\big) +\|b-u\|_{L^2}^{2}+\|\sqrt{2}\omega-\sqrt{3}\theta\|_{L^2}^2 \nonumber\\
&+\mu\|\nabla u\|_{L^2}^2+(\mu+\lambda)\|{\rm div} u\|_{L^2}^2+\kappa\|\nabla\theta\|_{L^2}^2+\int_{\mathbb{R}^3} \langle-\mathcal{L}\{\mathbf{I}-\mathbf{P}\}f,f\rangle {\rm d}x \nonumber\\
=\,&\frac{1}{2}\int_{\mathbb{R}^3}u\cdot\langle vf,f\rangle{\rm d}x-\int_{\mathbb{R}^3}a|u|^2{\rm d}x-\int_{\mathbb{R}^3}b\cdot u\theta{\rm d}x+\int_{\mathbb{R}^3}\Big\langle\frac{\theta}{\sqrt{M}}\Delta_v(\sqrt{M}f),f\Big\rangle{\rm d}x
\nonumber\\
&-3\int_{\mathbb{R}^3}a\theta^2{\rm d}x+\int_{\mathbb{R}^3}\rho\mathcal{N}_1{\rm d}x+\int_{\mathbb{R}^3}u\cdot\mathcal{N}_2{\rm d}x+\int_{\mathbb{R}^3}\theta\mathcal{N}_3{\rm d}x+\int_{\mathbb{R}^3}\langle f, \mathcal{N}_4\rangle{\rm d}x\nonumber\\
\equiv:\,& \sum_{i=1}^9 I_i.
 \end{align}
Here, we have utilized the decomposition of $\mathcal{L}f$  in \eqref{G2.5}.  

Thanks to the  macro-micro decomposition given by \eqref{G2.4} and Lemma \ref{LA.1}, we have
\begin{align}\label{G3.6}
  |I_1+I_2+I_3|%\nonumber\\
 =&\,-\int_{\mathbb{R}^3}au\cdot(u-b){\rm d}x+\frac{\sqrt{3}}{3}\int_{\mathbb{R}^3}(\sqrt{2}\omega-\sqrt{3}\theta)u\cdot b{\rm d}x\nonumber\\
 &+\int_{\mathbb{R}^3}\langle v\mathbf{P}f,\{\mathbf{I}-\mathbf{P}\}f\rangle{\rm d}x+\frac{1}{2}\int_{\mathbb{R}^3}u\cdot \langle v,|\{\mathbf{I}-\mathbf{P}\}f|^2\rangle{\rm d}x\nonumber\nonumber\\
 \lesssim&\, \|a\|_{L^6}\|u\|_{L^3}\|b-u\|_{L^2}+\|b\|_{L^6}\|u\|_{L^3}\|\sqrt{2}\omega-\sqrt{3}\theta\|_{L^2}\nonumber\\
 &+\|u\|_{L^{\infty}} \|\{\mathbf{I}-\mathbf{P}\}f\|_{\nu}^{2}+\|(a,b,\omega)\|_{L^6}\|u\|_{L^3} \|\{\mathbf{I}-\mathbf{P}\}f\|_{\nu}\nonumber\\
 \lesssim&\,\|u\|_{H^2}\|\nabla(a,b,\omega)\|_{L^2}
 \big(\|b-u\|_{L^2} 
+\|\sqrt{2}\omega-\sqrt{3}\theta\|_{L^2} +\|\{\mathbf{I}-\mathbf{P}\}f\|_{\nu} \big)\nonumber\\
&+\|u\|_{H^2}\|\{\mathbf{I}-\mathbf{P}\}f\|_{\nu}^2\nonumber\\
\lesssim&\,\delta\big(\|b-u\|_{L^2}^{2}
+\|\sqrt{2}\omega-\sqrt{3}\theta\|_{L^2}^2+\|\{\mathbf{I}-\mathbf{P}\}f\|_{\nu}^{2}+\|\nabla(a,b,\omega)\|_{L^2}^2\big).
\end{align}
The terms $I_4$ and $I_5$ can   be estimated in a similar way by  
\begin{align}\label{G3.7}
 |I_4+I_5| 
=&\,\Big|\int_{\mathbb{R}^3}\theta\Big\langle  \nabla_v(\sqrt{M}f),\nabla_v\Big(\frac{f}{\sqrt{M}}\Big)\Big\rangle{\rm d}x+3\int_{\mathbb{R}^3}a\theta^2{\rm d}x\Big|\nonumber\\
\leq&\, \Big|\iint_{\mathbb{R}^3\times\mathbb{R}^3}\theta
\Big(|\nabla_v\{\mathbf{I}-\mathbf{P}\}f|^2-\frac{|v|^2}{4}|\{\mathbf{I}-\mathbf{P}\}f|^2\Big){\rm d}x{\rm d}v-\sqrt{3}\int_{\mathbb
{R}^3}a\theta(\sqrt{2}\omega-\sqrt{3}\theta){\rm d}x\Big|\nonumber\\
&+2\Big|\iint_{\mathbb{R}^3\times\mathbb{R}^3}\theta\Big(\nabla_v\{\mathbf{I}
-\mathbf{P}\}f\cdot\nabla_v\mathbf{P}f-\frac{|v|^2}{4}\{\mathbf{I}-\mathbf{P}\}f\cdot   \mathbf{P}f         \Big){\rm d}x{\rm d}v\Big|\nonumber\\
\lesssim&\, \|\theta\|_{L^3}\|(a,b,\omega)\|_{L^6}\|\{\mathbf{I}-\mathbf{P}\}f\|_{\nu}+\|\theta\|_{L^3}\| a\|_{L^6}\|\sqrt{2}\omega-\sqrt{3}\theta\|_{L^2}+\|\theta\|_{L^\infty}\|\{\mathbf{I}-\mathbf{P}\}f\|_{\nu}^2\nonumber\\
\lesssim&\,\delta\big( \|\sqrt{2}\omega-\sqrt{3}\theta\|_{L^2}^2+\|\{\mathbf{I}-\mathbf{P}\}f\|_{\nu}^{2}+\|\nabla(a,b,\omega)\|_{L^2}^2\big).
\end{align}

By applying  integration by parts, H\"{o}lder’s inequality and Lemma \ref{LA.1}, one has
\begin{align}
|I_6|\lesssim&\, \|\rho \|_{L^3}\big(\|\rho\|_{L^6}\|\nabla u\|_{L^2}+ \|u\|_{L^6}\|\nabla\rho\|_{L^2} \big)\nonumber\\
\lesssim&\, \delta  \|\nabla(\rho,u)\|_{L^2}^2,  \label{G3.8-1}\\
|I_7|\lesssim&\, \|u\|_{L^3}\|u\|_{L^6}\|\nabla u\|_{L^2}+\|h(\rho)\|_{L^\infty}\|\theta\|_{L^3}\|u\|_{L^6}\|\nabla\rho\|_{L^2}+ \|g(\rho)\|_{L^2}\|u\|_{L^6}^2 \|a\|_{L^6}  \nonumber\\
&+\big( \mu\|\nabla u\|_{L^2}+ (\mu+\lambda) \|{\rm div} u\|_{L^2} \big)\big( \|\nabla g(\rho)\|_{L^3}\|u\|_{L^6}  +\|g(\rho)\|_{L^\infty}\|\nabla u\|_{L^2}    \big) \nonumber\\
&+\|g(\rho)\|_{L^3} \|u\|_{L^6}(\|\nabla\rho\|_{L^2}+\|b-u\|_{L^2} )\nonumber\\
\lesssim&\, \delta \big( \|b-u\|_{L^2}^2+\|\nabla(a,b,\rho,u,\theta)\|_{L^2}^2+\mu\|\nabla u\|_{L^2}^2+(\mu+\lambda)\|{\rm div} u\|_{L^2}^2       \big),\label{G3.8-2}\\
|I_8|\lesssim&\, \|h(\rho)\|_{L^\infty} \|\theta\|_{L^\infty}     \big(\mu\|\nabla u\|_{L^2}^2+(\mu+\lambda)\|{\rm div} u\|_{L^2}  \big)+\|h(\rho)\|_{L^\infty} \|\theta\|_{L^2}\|a\|_{L^6}\|u\|_{L^6}^2 \nonumber\\
&+\|h(\rho)\|_{L^\infty} \|b-u\|_{L^2}\|u\|_{L^3}\|\theta\|_{L^6}+\|\theta\|_{L^3}\|\theta\|_{L^6}\|\nabla u\|_{L^2}+\kappa \|g(\rho)\|_{L^\infty}\|\nabla\theta\|_{L^2}^2  \nonumber\\
&+\kappa \|\nabla\rho\|_{L^3}\|\theta\|_{L^6}\|\nabla\theta\|_{L^2}+\|g(\rho)\|_{L^6} \big( \|\theta\|_{L^3}\|\sqrt{2}\omega-\sqrt{3}\theta\|_{L^2} +\|b\|_{L^6}\|u\|_{L^2}\|\theta\|_{L^6}       \big)\nonumber\\
&+\|g(\rho)\|_{L^6}\|a\|_{L^6}\|\theta\|_{L^6}^2\nonumber\\
\lesssim&\, \delta \big(  \mu \|\nabla u\|_{L^2}^2+(\mu+\lambda)\|{\rm div} u\|_{L^2}^2+\kappa\|\nabla\theta\|_{L^2}^2+\|b-u\|_{L^2}^{2}
+\|\sqrt{2}\omega-\sqrt{3}\theta\|_{L^2}^2 \big) \nonumber\\
&+\delta  \|\nabla(a,b,\rho,u, \theta)\|_{L^2}^2.\label{G3.8-3}
\end{align} 
By applying  integration by parts to the term $I_9$, one has
\begin{align} 
I_9=&\, -\frac{1}{2}\iint_{\mathbb{R}^3\times\mathbb{R}^3} {\rm div}_v u |f|^2{\rm d}x{\rm d}v=0.\label{G3.8}
\end{align}
Substituting \eqref{G3.6}--\eqref{G3.8} into \eqref{G3.5} and utilizing \eqref{G2.6b}, we  obtain the desired  estimate \eqref{G3.4}.
\end{proof}

Next, we give the   estimate of   high-order  derivatives of     $(\rho, u, \theta, f)$.
Before handling the estimates on the higher-order spatial derivatives of it, for the convenience of subsequent calculations, we now analyze the estimate of $(\partial_t\rho,\partial_t u,\partial_t\theta)$ in  $L^\infty$-norm. This also implies that our {\it a priori estimates} need to be considered in the $H^4$ framework. 
 \begin{lem}\label{LNXL3.2}
For the  classical solution $(\rho, u,\theta,f)$ to the Cauchy problem \eqref{Q1}--\eqref{Q1-1}, it holds that 
\begin{align}\label{LNXG3.17}
\|(\partial_t\rho,\partial_t u,\partial_t\theta)\|_{L^\infty}\lesssim \|(\rho,u,\theta)\|_{H^4}\lesssim \delta,
\end{align}
for any $0 \leq t < T$.
\end{lem}
\begin{proof}
It follows from \eqref{G3.2}$_1$--\eqref{G3.2}$_3$ and Lemma \ref{LA.1} that
\begin{align*}
\|(\partial_t\rho,\partial_t u,\partial_t\theta)\|_{L^\infty}\lesssim&\, \|\nabla (\rho,u,\theta)\|_{L^\infty}+\|\nabla^2(u,\theta)\|_{L^\infty}+\|au\|_{L^\infty}+\|b\cdot u\|_{L^\infty}+\|a\theta\|_{L^\infty}\nonumber\\
&+\|(b,\omega,u,\theta)\|_{L^\infty}+\|(\mathcal{N}_1,\mathcal{N}_2,\mathcal{N}_3)\|_{L^\infty}\nonumber\\
\lesssim&\,  \|(\rho,u,\theta)\|_{H^4}+ \|\nabla(\rho,u,\theta)\|_{L^\infty}\|(\rho,u,\theta)\|_{L^\infty}+\|g
(\rho)\|_{L^\infty}  \|\nabla^2 (u,\theta)\|_{L^\infty} \nonumber\\
&+\|g(\rho)\|_{L^\infty}(\|\nabla\rho\|_{L^\infty}+\|(b,\omega,u,\theta)\|_{L^\infty}+\|au\|_{L^\infty}+\|a\theta\|_{L^\infty})\nonumber\\
&+(1+\delta)\|h(\rho)\|_{L^\infty}\big( \|\theta\|_{L^\infty}\|\nabla\rho\|_{L^\infty}+\|(a,b,u)\|_{L^\infty}^2 +\|\nabla u\|_{L^\infty}^2     \big) \nonumber\\
\lesssim &\,(1+\delta )\|(\rho,u,\theta)\|_{H^4},
\end{align*}
which together with \eqref{G3.1} leads to \eqref{LNXG3.17}.   
\end{proof}

\begin{lem}\label{L3.2}
For the  classical solution $(\rho, u,\theta,f)$ to the Cauchy problem \eqref{Q1}--\eqref{Q1-1}, it holds that 
\begin{align}\label{G3.9}
&\frac{1}{2}\frac{\rm d}{{\rm d}t}\sum_{1\leq|\alpha|\leq
4}\Big(\Big\|\frac{\partial^\alpha\rho}{1+\rho}\Big\|_{L^2}^2+\Big\|\frac{\partial^\alpha u}{\sqrt{1+\theta}}\Big\|_{L^2}^{2}+\Big\|\frac{\partial^\alpha\theta}{1+\theta}\Big\|_{L^2}^2 +\|\partial^\alpha f\|_{L_{x,v}^{2}}^2\Big)\nonumber\\
&\quad+\lambda_2\sum_{1\leq |\alpha|\leq 4}\big(\|\partial^\alpha(b-u)\|_{L^2}^{2}
+\|\partial^\alpha(\sqrt{2}\omega-\sqrt{3}\theta)\|_{L^2}^2+\|\{\mathbf{I}-\mathbf{P}\}\partial^\alpha f\|_{\nu}^{2}\big)\nonumber\\
&\quad \quad+\lambda_2\sum_{1\leq |\alpha|\leq 4} \big(\mu \|\nabla\partial^\alpha u\|_{L^2}^2+(\mu+\lambda) \|\partial^\alpha {\rm div} u\|_{L^2}^2+\kappa\|\nabla\partial^\alpha \theta\|_{L^2}^2\big)\nonumber\\ 
 &\quad\quad\quad \lesssim \delta\|\nabla(a,b,\omega,\rho)\|_{H^3}^2,
\end{align}
for any $0 \leq t < T$ and $\lambda_2>0$,   independent of $T$ and $\mu$, $\lambda$ and
$\kappa$. 
\end{lem}

Before beginning the proof of Lemma \ref{L3.2}, we give some preparations. 
Applying the differentiation operator $\partial^\alpha$ with $1\leq |\alpha|\leq 4$ to the equations \eqref{Q1}  yields
\begin{align}\label{LNXG3.13}
& \partial_t  \partial^\alpha \rho  +u \cdot \nabla \partial^\alpha \rho+(1+\rho) \partial^\alpha {\rm div}   u=[-\partial^\alpha, \rho {\rm div}  ] u+ [-\partial^\alpha, u \cdot \nabla   ] \rho ,\\ \label{LNXG3.14}
& \partial_t  \partial^\alpha u  +u \cdot \nabla  \partial^\alpha u +\frac{1+\theta}{1+\rho}\nabla\partial^\alpha\rho+\nabla  \partial^\alpha \theta-\partial^\alpha(b-u) -\mu\partial^\alpha\Big(\frac{\Delta u}{1+\rho}\Big) -(\mu+\lambda)\partial^\alpha \Big(\frac{\nabla{\rm div}\, u}{1+\rho}\Big) \nonumber \\
&\quad = -\partial^\alpha(au)+[-\partial^\alpha, u \cdot \nabla] u+ \Big[\partial^\alpha,  \frac{1+\theta}{1+\rho}\nabla\Big] \rho-\partial^\alpha\Big ( \frac{\rho(b-u)}{1+\rho} \Big)  +\partial^\alpha \Big(  \frac{\rho  au}{1+\rho}  \Big), \\
\label{LNXG3.15}
& \partial_t\partial^\alpha \theta+ u\cdot\nabla\partial^\alpha \theta+(1+\theta ) \partial^\alpha{\rm div}u-\kappa\partial^\alpha \Big( \frac{\Delta \theta}{1+\rho}\Big)-\sqrt{3}\partial^\alpha\big(\sqrt{2}\omega-\sqrt{3}\theta  \big)  +\partial^\alpha(b\cdot u)+\partial^\alpha(3a\theta)\nonumber\\
&\quad=[-\partial^\alpha, u\cdot\nabla]\theta-[\partial^\alpha, \theta{\rm div}]u -\sqrt{3}\partial^\alpha\Big( \frac{\rho(\sqrt{2}\omega-\sqrt{3}\theta)}{1+\rho}  \Big)+\partial^\alpha\Big(  \frac{3\rho a\theta}{1+\rho} \Big)+\partial^\alpha\Big(  \frac{\rho b\cdot u}{1+\rho} \Big)\nonumber\\
&\quad\quad-\partial^\alpha \Big( \frac{(b-u)\cdot u}{1+\rho}  \Big)+\partial^\alpha\Big(\frac{a|u|^2+2\mu |D(u)|^2+\lambda |{\rm div}u|^2}{1+\rho}\Big), \\\label{LNXG3.16}
&\partial_t \partial^\alpha f +v \cdot \nabla \partial^\alpha f +u \cdot \nabla_v \partial^\alpha f -\partial^\alpha u \cdot v \sqrt{M}- (|v|^2-3)  \sqrt{M} \partial^\alpha \theta-\mathcal{L} \partial^\alpha f\nonumber \\
 &\quad =\frac{1}{2} \partial^\alpha(u \cdot v f)+ [-\partial^\alpha, u \cdot \nabla_v ] f+\partial^\alpha\Big[\frac{\theta}{\sqrt{M}}    \Delta_v\big(\sqrt{M} f\big)\Big],  
\end{align}
where $[A,B]:=AB - BA$ denote the commutators for the   operators $A$ and $B$. 

Unlike the works   \cite{LNW-arXiv-2025,Mu-Wang-20}, none of them  
 observed that the equation of $u$ in \eqref{LNXG3.14} actually possesses 
 a symmetrical structure,   we develop a 
 novel method here,  inspired by \cite{At-ARMA-2006, JJLX-2014}, by dividing \eqref{LNXG3.14} by $1+\theta$, which 
 yields new symmetrical terms like
  $\frac{\nabla\partial^\alpha \rho}{1+\rho}$ and $\frac{\nabla\partial^\alpha 
  \theta}{1+\theta}$. Subsequently, 
  if we divide \eqref{LNXG3.13} by $\frac{1}{(1+\rho)^2}$ and \eqref{LNXG3.15} 
  by $\frac{1}{(1+\theta)^2}$, the newly-emerging terms, 
  such as $\frac{\partial^\alpha{\rm div}u}{1+\rho}$ and
   $\frac{\partial^\alpha{\rm div}u}{1+\theta}$,
    can be matched with $\frac{\nabla\partial^\alpha \rho}{1+\rho}$ 
    and $\frac{\nabla\partial^\alpha \theta}{1+\theta}$. 
    Finally, through integration by parts, we 
     utilize this symmetric structure to prevent the loss
      of regularity and thus complete the energy estimates of  
       $(\partial^\alpha\rho,\partial^\alpha u,\partial^\alpha\theta,\partial^\alpha f)$ 
       with $1\leq|\alpha|\leq 4$ in the $L^2$-norm. 

\begin{proof}[Proof of Lemma \ref{L3.2}]
 Multiplying \eqref{LNXG3.13}--\eqref{LNXG3.16} by $\frac{\partial^\alpha\rho}{(1+\rho)^2}$, $\frac{\partial^\alpha u}{(1+\theta) }$, $\frac{\partial^\alpha\theta}{(1+\theta)^2 }$, and $\partial^\alpha f$ respectively, and then taking integration and summation, we obtain 
\begin{align}\label{G3.10}
&\frac{1}{2}\frac{{\rm d}}{{\rm d}t}   \Big(\Big\|\frac{\partial^\alpha\rho}{1+\rho}\Big\|_{L^2}^2+\Big\|\frac{\partial^\alpha u}{\sqrt{1+\theta}}\Big\|_{L^2}^{2}+\Big\|\frac{\partial^\alpha\theta}{1+\theta}\Big\|_{L^2}^2 +\|\partial^\alpha f\|_{L_{x,v}^{2}}^2\Big) +\|\partial^\alpha(b-u)\|_{L^2}^{2}\nonumber\\
&+\|\partial^\alpha(\sqrt{2}\omega-\sqrt{3}\theta)\|_{L^2}^2 +\int_{\mathbb{R}^3} \langle-\mathcal{L}\{\mathbf{I}-\mathbf{P}\}\partial^\alpha f,\partial^\alpha f\rangle {\rm d}x+\mu \int_{\mathbb{R}^3}\frac{|\nabla\partial^\alpha u|^2}{(1+\rho)(1+\theta)} {\rm d}x  \nonumber\\
&+(\mu +\lambda) \int_{\mathbb{R}^3}\frac{| \partial^\alpha{\rm div}\, u|^2}{(1+\rho)(1+\theta)} {\rm d}x+\kappa\int_{\mathbb{R}^3}\frac{| \nabla\partial^\alpha  \theta|^2}{(1+\rho)(1+\theta)^2} {\rm d}x   \nonumber\\
=\,& \frac{1}{2}\int_{\mathbb{R}^3}\langle \partial^\alpha(u\cdot vf),\partial^\alpha f\rangle{\rm d}x+\int_{\mathbb{R}^3}\Big\langle\partial^\alpha\Big(\frac{\theta}{\sqrt{M}}\Delta_v(\sqrt{M}f)\Big),\partial^\alpha f\Big\rangle{\rm d}x\nonumber\\
&+\int_{\mathbb{R}^3}\langle [-\partial^\alpha,u\cdot\nabla_v ]f    ,\partial^\alpha f\rangle {\rm d}x+\int_{\mathbb{R}^3}\nabla\Big( \frac{1}{1+\rho}\Big)  \partial^\alpha\rho\cdot\partial^\alpha u{\rm d}x+\int_{\mathbb{R}^3}\nabla\Big( \frac{1}{1+\theta}\Big)  \partial^\alpha\theta\cdot\partial^\alpha u{\rm d}x\nonumber\\
&-\int_{\mathbb{R}^3} \frac{\theta}{1+\theta}\partial^\alpha (b-u)\cdot\partial^\alpha u {\rm d}x-\sqrt{3}\int_{\mathbb{R}^3} \frac{\theta^2+2\theta}{(1+\theta)^2}\partial^\alpha (\sqrt{2}\omega-\sqrt{3}\theta) \partial^\alpha \theta {\rm d}x\nonumber\\
&+\mu\sum_{0\leq|\beta|<|\alpha|} C_{\alpha,\beta}\int_{\mathbb{R}^3}\frac{1}{1+\theta}\partial^{\alpha-\beta}\Big( \frac{1}{1+\rho} \Big)\partial^\beta\Delta u\partial^\alpha u {\rm d}x\nonumber\\
&-\mu\int_{\mathbb{R}^3} \nabla\Big(\frac{1}{(1+\theta)(1+\rho)} \Big)\cdot\nabla\partial^\alpha u \cdot \partial^\alpha u {\rm d}x \nonumber\\
&+(\mu+\lambda)\sum_{0\leq|\beta|<|\alpha|} C_{\alpha,\beta}\int_{\mathbb{R}^3}\frac{1}{1+\theta}\partial^{\alpha-\beta}\Big( \frac{1}{1+\rho} \Big)\partial^\beta\nabla{\rm div}  u\partial^\alpha u {\rm d}x\nonumber\\
&-(\mu+\lambda)\int_{\mathbb{R}^3} \nabla\Big(\frac{1}{(1+\theta)(1+\rho)} \Big)\cdot \partial^\alpha u  \partial^\alpha{\rm div}  u {\rm d}x \nonumber\\
&+\kappa\sum_{0\leq|\beta|<|\alpha|} C_{\alpha,\beta}\int_{\mathbb{R}^3}\frac{1}{(1+\theta)^2}\partial^{\alpha-\beta}\Big( \frac{1}{1+\rho} \Big)\partial^\beta\Delta \theta \partial^\alpha \theta {\rm d}x\nonumber\\
&-\kappa\int_{\mathbb{R}^3} \nabla\Big(\frac{1}{(1+\theta)^2(1+\rho)} \Big)\cdot\nabla\partial^\alpha \theta  \partial^\alpha \theta {\rm d}x  +J_\rho+J_u+J_\theta \nonumber\\
\equiv:&\,\sum_{i=1}^{13}J_{i}+ J_\rho+J_u+J_\theta,
 \end{align}
where the terms $J_\rho$, $J_u$ and $J_\theta$ are denoted by
\begin{align*} 
J_\rho:=&\,\int_{\mathbb{R}^3}\frac{1}{2}\partial_t \Big(\frac{1}{(1+\rho)^2} \Big) |\partial^\alpha\rho|^2{\rm d}x+\int_{\mathbb{R}^3} {\rm div}\Big( \frac{u}{(1+\rho)^2}\Big)   |\partial^\alpha\rho|^2{\rm d}x  \nonumber\\
&+\int_{\mathbb{R}^3} \frac{1}{(1+\rho)^2} \big([-\partial^\alpha, \rho {\rm div}  ] u+ [-\partial^\alpha, u \cdot \nabla   ] \rho \big)\partial^\alpha\rho{\rm d}x, \\
J_u:=&\, \int_{\mathbb{R}^3}\frac{1}{2}\partial_t \Big(\frac{1}{1+\theta } \Big) |\partial^\alpha u|^2{\rm d}x+\int_{\mathbb{R}^3} {\rm div}\Big( \frac{u}{1+\theta }\Big)   |\partial^\alpha u|^2{\rm d}x  \nonumber\\
&+\int_{\mathbb{R}^3} \frac{1}{1+\theta} \Big\{[-\partial^\alpha,u\cdot\nabla] u+  \Big[\partial^\alpha,  \frac{1+\theta}{1+\rho}\nabla\Big] \rho-\partial^\alpha\Big [ \frac{\rho(b-u)}{1+\rho} \Big]  -\partial^\alpha \Big(  \frac{   au}{1+\rho}  \Big)\Big\}\cdot\partial^\alpha u{\rm d}x,\\ 
J_\theta:=&\,\int_{\mathbb{R}^3}\frac{1}{2}\partial_t \Big(\frac{1}{(1+\theta)^2} \Big) |\partial^\alpha\theta|^2{\rm d}x+\int_{\mathbb{R}^3} {\rm div}\Big( \frac{u}{(1+\theta)^2}\Big)   |\partial^\alpha\theta|^2{\rm d}x  \nonumber\\
&+\int_{\mathbb{R}^3}\frac{1}{(1+\theta)^2}\bigg\{[-\partial^\alpha, u\cdot\nabla]\theta-[\partial^\alpha, \theta{\rm div}]u -\sqrt{3}\partial^\alpha\bigg( \frac{\rho(\sqrt{2}\omega-\sqrt{3}\theta)}{1+\rho}   \bigg)\bigg\}\partial^\alpha\theta{\rm d}x\nonumber\\
&+\int_{\mathbb{R}^3} \frac{1}{(1+\theta)^2}\bigg\{\partial^\alpha \bigg( \frac{(b-u)\cdot u}{1+\rho}  \Big)+\partial^\alpha\Big(\frac{a|u|^2+2\mu |D(u)|^2+\lambda |{\rm div} u|^2}{1+\rho}\bigg)\bigg\}\partial^\alpha\theta {\rm d}x \nonumber\\
&-\int_{\mathbb{R}^3}\frac{1}{(1+\theta)^2}\Big\{\partial^\alpha\Big(  \frac{3  a\theta}{1+\rho} \Big)+\partial^\alpha\Big(  \frac{  b\cdot u}{1+\rho} \Big)\Big\} \partial^\alpha\theta{\rm d}x.
\end{align*} 
 Now we deal with the terms $J_i$ $(i=1,\dots, 13)$, $J_{\rho}$, $J_u$ and $J_{\theta}$ in turn. 
It follows from the macro-micro decomposition  \eqref{G2.4} and Lemma \ref{LA.1} that  
\begin{align}\label{G3.11}
|J_1+J_2| \lesssim&\, \|\partial^\alpha(u f)\|_{L_{x,v}^2}\|v\partial^\alpha f\|_{L_{x,v}^2}+\big\|\partial^\alpha\big[\theta(2\nabla_v f-vf) \big]\big\|_{L_{x,v}^2} \|\partial^\alpha  (2\nabla_v f+vf)   \|_{L_{x,v}^2} \nonumber\\
\lesssim&\,\|\nabla u\|_{H^3}\|\nabla f\|_{L_v^2(H^3)}\Big(\|\nabla(a,b,\omega)\|_{H^3}+\sum_{1\leq |\alpha|\leq 4}\|\{\mathbf{I}-\mathbf{P}\}\partial^\alpha f\|_{\nu}\Big)\nonumber\\
&+\|\nabla\theta\|_{H^3}\|(2\nabla_v-v)\nabla f\|_{L_v^2(H^3)} \|(2\nabla_v+v)\nabla f\|_{L_v^2(H^3)}\nonumber\\
\lesssim&\,\delta\Big(\|\nabla(a,b,\omega)\|_{H^3}^2+\sum_{1\leq |\alpha|\leq 4}\|\{\mathbf{I}-\mathbf{P}\}\partial^\alpha f\|_{\nu}^2\Big).
\end{align}

For the term $J_3$, by applying integration by parts, Lemmas \ref{LA.1} and \ref{LA.2}, and utilizing the macro-micro decomposition \eqref{G2.4}, we arrive at  
\begin{align}\label{G3.12}
|J_3|\lesssim&\, \|\nabla_v \partial^\alpha f\|_{L_{x,v}^2}\|\partial^\alpha (uf)-u\partial^\alpha f\|_{L_{x,v}^2} +\iint_{\mathbb{R}^3\times\mathbb{R}^3} {\rm div}_v u |\partial^\alpha f|^2 {\rm d}x\nonumber\\
\lesssim&\,\|\nabla u\|_{H^3}\|\nabla f\|_{L_v^2(H^3)}\Big(\|\nabla(a,b,\omega)\|_{H^3}+\sum_{1\leq |\alpha|\leq 4}\|\{\mathbf{I}-\mathbf{P}\}\partial^\alpha f\|_{\nu}\Big)\nonumber\\
&+\|\nabla\theta\|_{H^3}\|(2\nabla_v-v)\nabla f\|_{L_v^2(H^3)} \|(2\nabla_v+v)\nabla f\|_{L_v^2(H^3)}\nonumber\\
\lesssim&\,\delta \Big(\|\nabla(a,b,\omega)\|_{H^3}^2+\sum_{1\leq |\alpha|\leq 4}\|\{\mathbf{I}-\mathbf{P}\}\partial^\alpha f\|_{\nu}^2\Big).
\end{align} 

For the terms $J_4$, $J_5$, $J_6$ and $J_7$, by applying H\"{o}lder's and Sobolev’s inequalities, we   derive that
\begin{align}\label{G3.13}
 &|J_4+J_5+J_6+J_7|\nonumber\\
 \lesssim &\,    \big(\|\nabla(\rho,\theta)\|_{L^\infty} +\|\nabla(\rho,\theta)\|_{L^\infty}^2\big) \big(\|\nabla (\rho,u,\theta)\|_{H^3}^2+\|\nabla(b-u)\|_{H^3}^2+\|\nabla(\sqrt{2}\omega-\sqrt{3}\theta)\|_{H^3}^2 \big)\nonumber\\
 \lesssim&\, \delta \big(\|\nabla (\rho,u,\theta)\|_{H^3}^2+\|\nabla(b-u)\|_{H^3}^2+\|\nabla(\sqrt{2}\omega-\sqrt{3}\theta)\|_{H^3}^2 \big). 
\end{align} 

Note that for any $|\beta|<|\alpha|$,
\begin{equation*}
\int_{\mathbb{R}^3}\partial^{\alpha-\beta}
\Big(\frac{1}{1+\rho}\Big)\partial^{\beta}\Delta u\cdot\partial^{\alpha}u{\rm d}x
\leq \left\{\begin{aligned} 
&\Big\|\partial^{\alpha}\Big(\frac{1}{1+\rho}\Big)\Big\|_{L^2}
\|\Delta u\|_{L^{\infty}}\|\partial^{\alpha}u\|_{L^2},
& (|\beta|=0)\\
&\Big\|\partial^{\alpha-\beta}\Big(\frac{1}{1+\rho}\Big)\Big\|_{L^{3}}\|\partial^{\beta}\Delta u\|_{L^{6}}
\|\partial^{\alpha}u\|_{L^2},& (|\beta|=1)\\
&\Big\|\partial^{\alpha-\beta}\Big(\frac{1}{1+\rho}\Big)\Big\|_{L^{\infty}}\|\partial^{\beta}\Delta u\|_{L^2}
\|\partial^{\alpha}u\|_{L^2},&  (|\beta|\geq 2)
\end{aligned}\right.\nonumber\\
\end{equation*}
which, combined with \eqref{LNXG3.2}, leads to 
\begin{align}\label{G3.14}
|J_8|\lesssim \mu \|u\|_{H^4} \|\nabla \rho\|_{H^3} \|\nabla^2 u\|_{H^3}  \lesssim \delta \mu \|\nabla^2 u\|_{H^3}^2+\delta \|\nabla\rho\|_{H^3}^2.
\end{align}
Similarly, we infer that
\begin{align}\label{G3.15}
|J_{10}+J_{12}|\lesssim \delta(\mu+\lambda)\|\nabla{\rm div} u\|_{H^3}^2+\delta\kappa \|\nabla^2\theta\|_{H^3}^2+\delta \|\nabla\rho\|_{H^3}^2.    
\end{align}

For the terms $J_{9}, J_{11}$, and $J_{13}$, 
using H\"{o}lder's inequality and Young's inequality yields
\begin{align}\label{G3.16}
|J_{9}+J_{11}+J_{13}|\lesssim &\, \mu \big(\|\nabla(\rho,\theta)\|_{L^\infty}  +\|(\rho,\theta)\|_{L^\infty}\nabla(\rho,\theta)\|_{L^\infty}\big) \|u\|_{H^4}\|\nabla^2 u\|_{H^3} \nonumber\\
&+(\mu+\lambda)\big(\|\nabla(\rho,\theta)\|_{L^\infty}  +\|(\rho,\theta)\|_{L^\infty}\nabla(\rho,\theta)\|_{L^\infty}\big) \|u\|_{H^4}\|\nabla {\rm div}u\|_{H^3} \nonumber\\
&+\kappa  \big(1+   \|(\rho,\theta)\|_{L^\infty}^2\big) \nabla(\rho,\theta)\|_{L^\infty} \|\theta \|_{H^4}\|\nabla^2\theta\|_{H^3}^2 \nonumber\\
\lesssim&\, \delta\mu \|\nabla^2 u\|_{H^3}^2+\delta(\mu+\lambda)\|\nabla{\rm div}u\|_{H^3}^2+\delta\kappa\|\nabla^2\theta\|_{H^3}^2+\delta\|\nabla(\rho,\theta)\|_{H^3}^2.
\end{align}  
 
Finally, we estimate the remaining terms $J_{\rho}$, $J_u$, and $J_{\theta}$ individually. It follows from \eqref{LNXG3.2}, Lemma \ref{LNXL3.2}, and Lemmas \ref{LA.1} and \ref{LA.2} that 
\begin{align}\label{G3.17}
|J_{\rho}|\lesssim&\,  \big( \delta+(1+\|(\rho,u)\|_{L^\infty})\|\nabla (\rho,u)\|_{L^\infty} \big)  \|\nabla\rho\|_{H^3}^2 + \|\nabla(\rho,u)\|_{H^3}\|\nabla(\rho,u)\|_{L^\infty}\|(\rho,u)\|_{H^4}\nonumber\\
\lesssim&\, \delta \|\nabla(\rho,u)\|_{H^3}^2,\\ \label{G3.18}
|J_u|\lesssim&\,\big( \delta+\|(u,\theta)\|_{L^\infty}\|\nabla ( u,\theta)\|_{L^\infty} \big)  \|\nabla u\|_{H^3}^2+\|\nabla u\|_{H^3}\|\nabla  u\|_{L^\infty}\|u\|_{H^4}\nonumber\\
&+\Big(\Big\|\nabla\Big(\frac{1+\rho}{1+\theta}\Big)\Big\|_{L^\infty} \|\nabla^4\rho\|_{L^2}+\|\nabla\rho\|_{L^\infty} \Big\|\nabla^4\Big(\frac{1+\rho}{1+\theta}\Big) \Big\|_{L^2} \Big)\|\nabla u\|_{H^3} \nonumber\\
&+\|\nabla(b-u)\|_{H^3}\Big\|\nabla\Big(  \frac{\rho}{1+\rho} \Big)\Big\|_{H^3}\|\nabla u\|_{H^3}+\Big\|\nabla\Big(  \frac{a}{1+\rho} \Big)\Big\|_{H^3}\|\nabla u\|_{H^3}^2 \nonumber\\
\lesssim&\, \delta \big(\|\nabla (a,\rho,u,\theta)\|_{H^3}^2+\|\nabla(b-u)\|_{H^3}^2\big),\\ \label{G3.19}
|J_{\theta}|\lesssim&\,    \big( \delta+(1+\|( u,\theta)\|_{L^\infty})\|\nabla ( u,\theta)\|_{L^\infty} \big)  \|\nabla \theta\|_{H^3}^2+\|\nabla(u,\theta)\|_{L^\infty}\|(u,\theta)\|_{H^4}\|\nabla\theta\|_{H^3} \nonumber\\
&+\big(\|\nabla(b-u)\|_{H^3}+\|\nabla(\sqrt{2}\omega-\sqrt{3}\theta)\|_{H^3} \big)
\Big\|\nabla\Big(  \frac{\rho}{1+\rho} \Big)\Big\|_{H^3} \|\nabla\theta\|_{H^3} \nonumber\\
&+\mu\Big\|\nabla\Big( \frac{D(u)}{1+\rho} \Big)\Big\|_{H^3} \|\nabla D(u)\|_{H^3} \|\theta\|_{H^4}+\lambda \Big\|\nabla\Big( \frac{{\rm div} u}{1+\rho} \Big)\Big\|_{H^3} \|\nabla {\rm div} u\|_{H^3} \|\theta\|_{H^4}\nonumber\\
&+\Big(\Big\|\nabla\Big( \frac{a}{1+\rho} \Big) \Big\|_{H^3}+\Big\|\nabla\Big( \frac{b}{1+\rho} \Big) \Big\|_{H^3}+\Big\|\nabla\Big( \frac{au}{1+\rho} \Big) \Big\|_{H^3}\Big)\|\nabla(u,\theta)\|_{H^3}^2 \nonumber\\
\lesssim&\, \delta   \big(\|\nabla (a,b,\rho,u,\theta)\|_{H^3}^2+\|\nabla(b-u)\|_{H^3}^2+\|\nabla(\sqrt{2}\omega-\sqrt{3}\theta)\|_{H^3}^2\big).
\end{align} 
Substituting \eqref{G3.11}--\eqref{G3.19}  into \eqref{G3.10}, and using 
the facts  \eqref{G2.6} and the inequalities
\begin{align*}
  \|\nabla u\|_{L^2}^2\lesssim&\,\|\nabla( b-u)\|_{L^2}^2+\|\nabla b\|_{L^2}^2,  \\
    \|\nabla \theta\|_{L^2}^2\lesssim&\,\|\nabla(\sqrt{2}\omega-\sqrt{3}\theta)\|_{L^2}^2+\|\nabla \omega\|_{L^2}^2, 
\end{align*}
we ultimately obtain the desired estimate \eqref{G3.9}. 
\end{proof}

In order to estimate the energy dissipation of $ \|\nabla (a, b, \omega) \|_{H^3}$, we analyze the equations for $a$, $b$, and $\omega$, which involve the following operators:  
\begin{align*}  
\Gamma_{i, j}(g) := \big\langle (v_i v_j - 1) \sqrt{M(v)}, g \big\rangle, \quad \Upsilon_i (g) := \Big\langle \frac{1}{\sqrt{6}} v_i (|v|^2 - 3) \sqrt{M(v)}, g \Big\rangle,  
\end{align*}  
for any $g = g(v)$, where $1 \leq i, j \leq 3$. Through direct computations, we derive the following system for $(a, b, \omega)$ (see \cite{LNW-arXiv-2025,Mu-Wang-20}):

\begin{equation}\label{G3.21}
\left\{\begin{aligned}
&\partial_t a+{\rm div} b=0,\\
& \partial_t b_i+\partial_i a+\frac{2}{\sqrt{6}} \partial_i \omega+\sum_{j=1}^3 \partial_{j} \Gamma_{i, j}(\{\mathbf{I}-\mathbf{P}\} f)=-b_i+u_i+u_i a, \\
& \partial_t \omega+\sqrt{2}(\sqrt{2} \omega-\sqrt{3} \theta)-\sqrt{6} a \theta+\frac{2}{\sqrt{6}} {\rm div}  b-\frac{2}{\sqrt{6}} u \cdot b    +\sum_{i=1}^3 \partial_{i} \Upsilon_i(\{\mathbf{I}-\mathbf{P}\} f)=0, \\
& \partial_j b_i+\partial_i b_j- (u_i b_j+u_j b_i )  -\frac{2}{\sqrt{6}} \delta_{i j}\bigg(\frac{2}{\sqrt{6}} {\rm div}   b-\frac{2}{\sqrt{6}} u \cdot b+\sum_{i=1}^3 \partial_{ i} \Upsilon_i(\{\mathbf{I}-\mathbf{P}\} f)\bigg) \\
&\quad =-\partial_t \Gamma_{i, j}(\{\mathbf{I}-\mathbf{P}\} f)+\Gamma_{i, j}(\mathfrak{l}+\mathfrak{r}+\mathfrak{s}), \\
& \frac{5}{3} (\partial_i \omega-\omega u_i-\sqrt{6} \theta b_i )-\frac{2}{\sqrt{6}} \sum_{j=1}^3 \partial_{j} \Gamma_{i, j}(\{\mathbf{I}-\mathbf{P}\} f)   \\
&\quad=-\partial_t \Upsilon_i(\{\mathbf{I}-\mathbf{P}\} f)+\Upsilon_i(\mathfrak{l}+\mathfrak{r}+\mathfrak{s}).
\end{aligned}\right.
\end{equation}
Here, $\mathfrak{l}$, $\mathfrak{r}$, and $\mathfrak{s}$ in \eqref{G3.21} are respectively defined as  
\begin{equation}\label{G3.22}
\left\{\begin{aligned}
& \mathfrak{l}:=-v \cdot \nabla\{\mathbf{I}-\mathbf{P}\} f+\mathcal{L}\{\mathbf{I}-\mathbf{P}\} f,  \\
& \mathfrak{r}:=-u \cdot \nabla_v\{\mathbf{I}-\mathbf{P}\} f+\frac{1}{2} u \cdot v\{\mathbf{I}-\mathbf{P}\} f,  \\
& \mathfrak{s}:=\frac{1}{2}\frac{\theta}{\sqrt{M}} {\rm div}_v(\sqrt{M} (2\nabla_v-v)\{\mathbf{I}-\mathbf{P}\} f) . 
\end{aligned}\right.
\end{equation}
 By using the macro-micro decomposition \eqref{G2.4}, we rewrite the equation for $f$ in \eqref{Q1} as 
\begin{align}\label{G3.23}  
& \partial_t \mathbf{P} f + v \cdot \nabla \mathbf{P} f + u \cdot \nabla_v \mathbf{P} f - \frac{1}{2} u \cdot v \mathbf{P} f - u \cdot v \sqrt{M} - \theta (|v|^2 - 3) \sqrt{M} \nonumber \\  
& \quad+ \mathbf{P}_1 f + 2 \mathbf{P}_2 f - \frac{1}{2}\frac{\theta}{\sqrt{M}}  {\rm div}_v  ( \sqrt{M}  ( 2\nabla_v - v  ) \mathbf{P} f )  = -\partial_t \{\mathbf{I} - \mathbf{P}\} f + \mathfrak{l} + \mathfrak{r} + \mathfrak{s}.  
\end{align}
\begin{rem}
In fact, multiplying \eqref{G3.23} by $\sqrt{M}$, $v_i\sqrt{M}\, (1\leq i\leq 3)$, and $\frac{|v|^2-3}{\sqrt{6}}\sqrt{M}$ respectively, and then integrating the results over $\mathbb{R}^3_v$, we obtain \eqref{G3.21}$_1$--\eqref{G3.21}$_3$. Moreover, applying the operators $\Gamma_{i j}$ and $\Upsilon_i$ to \eqref{G3.23}, we deduce the remaining two equations in \eqref{G3.21}.
\end{rem}
Next, we  define the temporal functional
\begin{align}\nonumber% \label{define:E}
 \mathfrak{E}_0(t):= \, &\frac{2}{21} \sum_{|\alpha| \leq 3} \int_{\mathbb{R}^3} \partial ^\alpha a \partial^\alpha\bigg(\frac{\sqrt{6}}{5} \sum_{i=1}^3 \partial_i \Upsilon_i(\{\mathbf{I}-\mathbf{P}\} f)-{\rm div}   b\bigg) \mathrm{d} x \nonumber\\
& +\sum_{|\alpha| \leq 3} \sum_{i=1}^3 \int_{\mathbb{R}^3} \partial^\alpha \partial_i \omega \partial^\alpha \Upsilon_i(\{\mathbf{I}-\mathbf{P}\} f) \mathrm{d} x  \nonumber\\
& + \sum_{|\alpha| \leq 3} \sum_{i, j=1}^3 \int_{\mathbb{R}^3} \partial^\alpha (\partial_j b_i+\partial_i b_j ) \partial ^\alpha \Gamma_{i, j}(\{\mathbf{I}-\mathbf{P}\} f) \mathrm{d} x .
\end{align} 
By adopting a similar approach to that in \cite{LNW-arXiv-2025,Mu-Wang-20}, we obtain the following lemma (the proof is omitted for brevity).
\begin{lem}\label{L3.3}
For the  classical solution $(\rho, u,\theta,f)$ to the Cauchy problem \eqref{Q1}--\eqref{Q1-1}, it holds that 
\begin{align}\label{G3.24}
& \frac{\mathrm{d}}{\mathrm{d} t} \mathfrak{E}_0(t)+\lambda_3 \|\nabla (a, b, \omega) \|_{H^3}^2  \lesssim \big(\|\{\mathbf{I}-\mathbf{P}\} f\|_{L_v^2 (H^4 )}^2+\|b-u\|_{H^3}^2+\|\sqrt{2} \omega-\sqrt{3} \theta\|_{H^3}^2 \big),
\end{align}
for any $0 \leq t < T$ and $\lambda_3>0$, independent of  $T$, $\mu$, $\lambda$ and
$\kappa$. 
\end{lem}

Next, we give the estimate on $\|\nabla\rho\|_{H^3}$ as follows. 
\begin{lem}\label{L3.4}
For  the classical solution $(\rho, u,\theta,f)$ to the Cauchy problem \eqref{Q1}--\eqref{Q1-1}, it holds that 
\begin{align}\label{G3.25}
& \frac{\mathrm{d}}{\mathrm{d} t} \sum_{|\alpha| \leq 3} \int_{\mathbb{R}^3} \partial^\alpha u \cdot \partial^\alpha \nabla \rho \mathrm{d} x+\lambda_4\|\nabla \rho\|_{H^3}^2 \nonumber\\
&\quad\lesssim  \big(\|b-u\|_{H^3}^2+\|\nabla (u,\theta)\|_{H^3}^2+ \mu\|\nabla^2u\|_{H^3}^2+(\mu+\lambda)\|\nabla {\rm div}u\|_{H^3}^2\big),
\end{align} 
for any $0 \leq t <T$ and $\lambda_4>0$, independent of   $T$, $\mu$, $\lambda$ and
$\kappa$. 
\end{lem}
\begin{proof}
Applying $\partial^\alpha (|\alpha| \leq 3$) to the equation for $u$ in \eqref{G3.2}, then multiplying the resulting equation by $\nabla\partial^\alpha\rho$, and integrating over $\mathbb{R}^3_x$, we obtain 
\begin{align}\label{G3.26}
 \|\nabla \partial^\alpha \rho \|_{L^2}^2= &\,-\int_{\mathbb{R}^3} \nabla \partial^\alpha \rho \cdot\partial^\alpha \partial_t u \mathrm{d} x-\int_{\mathbb{R}^3}\partial^\alpha\nabla\theta\cdot\partial^\alpha\nabla\rho{\rm d}x+\int_{\mathbb{R}^3}\partial^\alpha(b-u)\cdot\partial^\alpha\nabla\rho {\rm d}x\nonumber\\
 &+\mu\int_{\mathbb{R}^3}\partial^\alpha\Delta u\partial^\alpha\nabla\rho {\rm d}x+(\mu+\lambda)\int_{\mathbb{R}^3}\partial^\alpha\nabla{\rm div
 } u\cdot
\partial^\alpha\nabla\rho {\rm d}x\nonumber\\
&-\int_{\mathbb{R}^3}\partial^\alpha(a
 u)\cdot\partial^\alpha\nabla\rho {\rm d}x+\int_{\mathbb{R}^3}\partial^\alpha\mathcal{N}_2\cdot\partial^\alpha\nabla\rho {\rm d}x\nonumber\\
\equiv:&\,\sum_{j={1}}^7 K_j. 
\end{align}

Now, we estimate the terms $K_j(j=1, \dots, 7)$ in turn. For the term $K_1$, 
using \eqref{Q1}$_1$ together with H\"{o}lder's inequality, and Lemma \ref{LA.1}, we derive that 
\begin{align}\label{G3.27}
K_1= & -\frac{\mathrm{d}}{\mathrm{d} t} \int_{\mathbb{R}^3} \nabla \partial^\alpha \rho \cdot\partial^\alpha u \mathrm{d} x+\int_{\mathbb{R}^3} \partial^\alpha {\rm div}  u \partial^\alpha\big((1+\rho) {\rm div}  u+u \cdot \nabla \rho\big) \mathrm{d} x \nonumber\\
\leq& -\frac{\mathrm{d}}{\mathrm{d} t} \int_{\mathbb{R}^3} \nabla \partial^\alpha \rho\cdot \partial^\alpha u \mathrm{d} x+C\|\partial^\alpha {\rm div}  u \|_{L^2}^2+C\|\rho\|_{H^4}^2\|\nabla u\|_{H^3}^2\nonumber\\
\leq&  -\frac{\mathrm{d}}{\mathrm{d} t} \int_{\mathbb{R}^3} \nabla \partial^\alpha  \rho \cdot\partial^\alpha u \mathrm{d} x+C\|\nabla u\|_{H^3}^2.
\end{align}
 For the term $K_2, \dots, K_6$, with the help of H\"{o}lder's and Young's inequalities, as well as Lemma \ref{LA.1}, we get
\begin{align}%\label{G3.28}
|K_2|\leq&\, C\|\nabla\theta\|_{H^3}^2+ \frac{1}{6}\|\nabla \partial^\alpha \rho\|_{L^2}^2, 
\label{G3.28-1}\\
   |K_3|\leq \,& C\| b-u\|_{H^3}^2+ \frac{1}{6}\|\nabla \partial^\alpha \rho\|_{L^2}^2,  \label{G3.28-2}\\
|K_4+K_5|\leq&\,C\mu\|\nabla^2 u \|_{H^3}^2+ C(\mu+\lambda)\|\nabla {\rm div} u \|_{H^3}^2+ \frac{1}{6}\|\nabla \partial^\alpha \rho\|_{L^2}^2, \label{G3.28} \\
|K_6|\leq&\, C\delta \|\nabla u \|_{H^3}^2+  \frac{1}{6}\|\nabla \partial^\alpha \rho\|_{L^2}^2. \label{G3.29}
\end{align}
For the last term $K_7$, based on the definition of $\mathcal{N}_2$ and Young's inequality, we consequently have  \begin{align}\label{G3.30}
|K_7|\leq&\, \frac{1}{6}\|\nabla \partial^\alpha \rho\|_{L^2}^2+C\delta\|\nabla (\rho,u)\|_{H^3}^2+  C\mu\|\nabla^2 u \|_{H^3}^2+ C(\mu+\lambda)\|\nabla {\rm div}\, u \|_{H^3}^2 \nonumber\\
&+C\|\nabla(b-u)\|_{H^2}^2.
\end{align} 
By inserting the estimates \eqref{G3.27}--\eqref{G3.30} into \eqref{G3.26}, it follows that \eqref{G3.25} holds. 
 \end{proof}

Finally, motivated by \cite{DFT-2010-CMP,GY-iumj-2004,LMW-SIAM-2017}, we establish the energy estimates concerning the mixed space-velocity derivatives of $f$, specifically $\partial_\beta^\alpha f$. By utilizing the properties outlined in \eqref{G2.4} and \eqref{G2.5}, it is readily to check that  
\begin{align*}
& \{\mathbf{I}-\mathbf{P}\}(v \sqrt{M})=0,  
\quad\{\mathbf{I}-\mathbf{P}\}   \big((|v|^2-3 ) \sqrt{M} \big)= 0, \quad
\mathcal{L}\{\mathbf{I}-\mathbf{P}\} f-\{\mathbf{I}-\mathbf{P}\} \mathcal{L} f= 0,\\
&\quad\quad\quad\,\,   g \cdot\{\mathbf{I}-\mathbf{P}\} f-\mathbf{P}(g \cdot\{\mathbf{I}-\mathbf{P}\} f)+\{\mathbf{I}-\mathbf{P}\}(g \cdot \mathbf{P} f)= \{\mathbf{I}-\mathbf{P}\}(g \cdot f),
\end{align*}
 which, together with applying $\{\mathbf{I}-\mathbf{P}\}$ to both sides of \eqref{Q1}$_1$, yields 
\begin{align}\label{G3.31}
&\partial_t\{\mathbf{I} -\mathbf{P}\} f+v \cdot \nabla\{\mathbf{I}-\mathbf{P}\} f+u \cdot \nabla_v\{\mathbf{I}-\mathbf{P}\} f-\frac{1}{2} u \cdot v\{\mathbf{I}-\mathbf{P}\} f-\frac{\theta }{\sqrt{M}} \Delta_v (\sqrt{M}\{\mathbf{I}-\mathbf{P}\} f ) \nonumber\\
&\quad = \mathbf{P}\Big(v \cdot \nabla\{\mathbf{I}-\mathbf{P}\} f+u \cdot \nabla_v\{\mathbf{I}-\mathbf{P}\} f-\frac{1}{2} u \cdot v\{\mathbf{I}-\mathbf{P}\} f -\frac{\theta}{\sqrt{M}} \Delta_v(\sqrt{M}\{\mathbf{I}-\mathbf{P}\} f)\Big) \nonumber\\
&\quad\quad-\{\mathbf{I}-\mathbf{P}\}\Big(v \cdot \nabla \mathbf{P} f+u \cdot \nabla_v \mathbf{P} f -\frac{1}{2} u \cdot v \mathbf{P} f-\frac{\theta }{\sqrt{M}} \Delta_v(\sqrt{M} \mathbf{P} f)\Big)+\mathcal{L}\{\mathbf{I}-\mathbf{P}\} f.
\end{align}

Performing energy estimates on \eqref{G3.31}, we can derive the following mixed spatial-velocity derivatives of $f$. For the sake of brevity, the detailed proof is omitted here (see \cite{LNW-arXiv-2025,Mu-Wang-20}).
\begin{lem}\label{L3.5}
For the  classical solution $(\rho, u,\theta,f)$ to the Cauchy problem \eqref{Q1}--\eqref{Q1-1}, it holds that 
\begin{align}\label{G3.32}
& \frac{\mathrm{d}}{\mathrm{d} t} \sum_{1 \leq m \leq 4} \tilde C_m \sum_{ {|\beta|=m ,
|\alpha|+|\beta| \leq 4}} \|\partial_\beta^\alpha\{\mathbf{I}-\mathbf{P}\} f \|_{L_{x,v}^2}^2+\lambda_5 \sum_{ {1 \leq|\beta| \leq 4,
|\alpha|+|\beta| \leq 4}} \|\partial_\beta^\alpha\{\mathbf{I}-\mathbf{P}\} f \|_\nu^2\nonumber \\
&\quad \lesssim \sum_{\left|\alpha\right| \leq 4}\left\|\partial^{\alpha}\{\mathbf{I}-\mathbf{P}\} f\right\|_\nu^2+ \|\nabla(a,b,\omega)\|_{H^3}^2.
\end{align}
for any $0 \leq t < T$ and $\lambda_5>0$, independent of  $T$, $\mu$, $\lambda$ and $\kappa$.  
Here, $\tilde C_m$ are some appropriately chosen constants that depend  only on $m$.  
\end{lem}

\subsection{Proof of Theorem \ref{T1.1}}
First, through a standard iteration argument and combining with the local-in-time energy estimates, we can deduce the following lemma. The detailed proof  is similar to that  in \cite{LNW-arXiv-2025,Mu-Wang-20}, and thus we omit it here for simplicity.
\begin{lem}\label{L3.6}
Assume that the initial data $(\rho_0,u_0,\theta_0,f_0)$ satisfies $(\rho_0,u_0,\theta_0)\in H^4$, $f_0\in H_{x,v}^4$,
and $\inf _{x\in\mathbb{R}^3}\rho_0(x)>0$.
There exists a constant $\eps_0^*>0$ independent of $\kappa$
such that if
\begin{align*}
\|(\rho _0,u _0,\theta _0)\|_{H^4}^2+\|f_0 \|_{H_{x,v}^4}^2\leq \varepsilon_0^*,
\end{align*}
then there exists a time $T^{*}>0$ such that the Cauchy problem \eqref{Q1}--\eqref{Q1-1} admits a unique classical solution $(\rho,u,\theta,f)$ on $[0,T^{*})$. Moreover, it satisfies $\inf\limits_{(t,x)\in [0,T^*)\times\mathbb{R}^3}\rho(t,x)>0$ and  
\begin{align*} 
&\|(\rho,u,\theta)(t)\|_{H^{4}}^2+\|f(t)\|_{H_{x,v}^4}^2+\int_{0}^{t} \mathcal{D}(\rho,u,\theta,f)(\tau){\rm d}\tau+\mu\int_{0}^{t}\|\nabla u(\tau)\|_{H^4}^2 {\rm d}\tau\nonumber\\ 
&\quad +(\mu+\lambda) \int_{0}^{t}\|{\rm div} u(\tau)\|_{H^4}^2 {\rm d}\tau  +\kappa\int_{0}^{t}\|\nabla\theta(\tau)\|_{H^4}^2 {\rm d}\tau 
\leq  2\big(\|(\rho_{0},u_{0},\theta_{0})\|_{H^{4}}^2+\|f_0\|_{H_{x,v}^4}^2\big), 
\end{align*}  
for all $t\in [0,T^{*})$. And  $\|(\rho,u,\theta)(t)\|_{H^{4}}^2+\|f(t)\|_{H^4_{x,v}}^2$ is continuous over $[0,T^{*})$.
\end{lem}
 
With the aid of Lemma \ref{L3.6}, we now proceed to establish Theorem \ref{T1.1}. Suppose that \eqref{a-1} holds with $\eps_0 \leq \eps_0^*$, where   $\eps_0 $ will be determined later. Define  
\begin{align*}  
T^{**} = \sup  \left\{
  T \in [0,T^{*})\left| \begin{aligned}
 &\text{\eqref{Q1}--\eqref{Q1-1} admits a solution}    \   (\rho,u,\theta,f)     \  \text{on}  \     [0,T)  \    \\  
 &\text{such that}\ \| (\rho,u,\theta)(t)\|_{H^4}+\|f(t)\|_{H^4_{x,v}} \leq \delta,  \     \forall \,     t \in [0,T)
\end{aligned}\right. \right\} .
\end{align*}  
Here, the constant $\delta > 0$ is explicitly defined in \eqref{G3.1}. By invoking Lemma \ref{L3.6}, it follows that $T^{**} \in (0,T^*]$ is well-defined. Thus, \eqref{G3.1} holds on $(0,T^{**})$.

With the uniform estimate   \eqref{G3.1} and Lemmas \ref{L3.1}--\ref{L3.5} in hand, for any $t \in [0, T^{**})$, we aim to construct the energy functional $\mathcal{E}(t)$ as 
% follows:  
\begin{align} \label{G3.33} 
\mathcal{E}(t) :=&\,\|(\rho,u,\theta)\|_{L^2}^2+\|f\|_{L_{x,v}^2}^2\nonumber\\
&+ \sum_{1\leq|\alpha|\leq
4}\bigg(\Big\|\frac{\partial^\alpha\rho}{1+\rho}\Big\|_{L^2}^2+\Big\|\frac{\partial^\alpha u}{\sqrt{1+\theta}}\Big\|_{L^2}^{2}+\Big\|\frac{\partial^\alpha\theta}{1+\theta}\Big\|_{L^2}^2 +\|\partial^\alpha f\|_{L_{x,v}^{2}}^2\bigg) + \tau_1 \mathfrak{E}_0(t)\nonumber \\  
&+ \tau_2 \sum_{|\alpha| \leq 3} \int_{\mathbb{R}^3} \partial^\alpha u \cdot \partial^\alpha \nabla \rho \, \mathrm{d}x + \tau_3 \sum_{1 \leq m \leq 4} \tilde C_m \sum_{ {|\beta| = m,  |\alpha| + |\beta| \leq 4}} \|\partial_\beta^\alpha \{\mathbf{I} - \mathbf{P}\} f\|_{L_{x,v}^2}^2,  
\end{align}  
where the constants $0 < \tau_2 \ll 1$ and $0 < \tau_3 \ll \tau_1 \ll 1$. Given that $\tau_1, \tau_2, \tau_3 > 0$ are sufficiently small, under the  uniform estimate  \eqref{G3.1}, and by applying Young's inequality and 
\eqref{LNXG3.2}, we conclude that  
\begin{align}  \label{G3.34}
\mathcal{E}(t) \sim \|(\rho, u, \theta)(t)\|_{H^4}^2 + \|f(t)\|_{H_{x,v}^4}^2.  
\end{align}  
By leveraging the newly developed effective dissipation modes, we obtain   
\begin{align*}  
\|\nabla u\|_{H^3}^2 \lesssim&\, \| b-u\|_{H^4}^2 + \|\nabla b\|_{H^3}^2, \\
\|\nabla \theta\|_{H^3}^2 \lesssim&\, \|\sqrt{2}\omega - \sqrt{3}\theta\|_{H^4}^2 + \|\nabla \omega\|_{H^3}^2,  
\end{align*} 
which implies that 
\begin{align*}
\|\nabla u(t)\|_{H^3}^2 +\|\nabla \theta(t)\|_{H^3}^2 \leq \mathcal{D}(\rho, u, \theta, f)(t).    
\end{align*}  
Thus, adding \eqref{G3.4}, \eqref{G3.9}, $\tau_1\times$\eqref{G3.24}, $\tau_2\times$\eqref{G3.25}, and $\tau_3\times$\eqref{G3.32} yields  
\begin{align*}  
\frac{{\rm d}}{{\rm d}t} \mathcal{E}(t) + \lambda_6 \mathcal{D}(\rho,u,\theta,f)(t) \leq C\delta \mathcal{D}(\rho,u,\theta,f)(t),  
\end{align*}  
for any $0 \leq t < T^{**}$, where $\lambda_6 < \min\{\lambda_1, \lambda_2, \tau_1\lambda_3, \tau_2\lambda_4, \tau_3\lambda_5\}$.
Owing to \eqref{G3.1}, we further get
\begin{align}\label{G3.35}
 \frac{{\rm d}}{{\rm d}t} \mathcal{E}(t) + \lambda_7\mathcal{D}(\rho,u,\theta,f)(t) \leq 0, 
\end{align}
for any $0 \leq t < T^{**}$ and   some $\lambda_7>0$.
Then, integrating \eqref{G3.35} over   $[0,t]$ leads to  
\begin{align}  \label{G3.36}
\mathcal{E}(t) + \lambda_7 \int_0^t \mathcal{D}(\rho, u, \theta, f)(\tau) \, {\rm d}\tau \leq \mathcal{E}(0),  
\end{align}  
for any $0 \leq t < T$.
Choosing $\eps_0 = \frac{1}{2C_*}\delta ^2$, we deduce from \eqref{G3.36} that  
\begin{align*}  
\|(\rho, u, \theta)(t)\|_{H^4}^2 + \|f(t)\|_{H_{x,v}^4}^2 \leq C^*\big(\|(\rho_0, u_0, \theta_0)\|_{H^4}^2 + \|f_0\|_{H_{x,v}^4}^2\big) \leq \frac{1}{2}\delta^2,  
\end{align*}  
for all $t \in [0, T^{**})$, where $C^*>0$ is a constant uniformly with respect to $\mu, \lambda, \kappa$ and $T^{**}$. Consequently, by employing the standard continuity argument, we conclude that $T^{**} = T^*$.
    
Finally, we assert that $T^{**} = T^{*} = +\infty$. In the spirit of \cite{CDM-krm-2011}, to demonstrate this claim, we assume % for the sake of contradiction 
that $T^{*} < +\infty$. By invoking the local existence theory and the uniform estimate 
established earlier, the solution can be extended to at least an interval
 $[0, T^{*} + \eta]$ for some $\eta > 0$. This directly contradicts
  the definition of $T^{*}$ as the maximal time of existence. 
  Therefore, $(\rho, u, \theta, f)$ constitutes a global 
  solution to the Cauchy problem \eqref{Q1}--\eqref{Q1-1}. 
  Moreover, Lemma \ref{L3.6} guarantees the validity of 
  \eqref{a-2}. Consequently, by invoking the maximum principle (cf. \cite{GY-iumj-2004}), we obtain
 $F = M + \sqrt{M}f \linebreak \geq 0$. Hence, the proof of Theorem \ref{T1.1} is completed.
 \hfill $\Box$

\section{{Justification of the vanishing viscosity  and    heat conductivity limits }} 

This section is dedicated to establishing Theorems \ref{T1.2} 
and \ref{T1.3}, which addresses the vanishing viscosity limits and 
 heat conductivity limit and their convergence rate.

\subsection{Proof of Theorem \ref{T1.2}}
Let the initial data $(\rho_0, u_0, \theta_0, f_0)$ satisfy 
\eqref{b-1}. By the density argument, one can construct a 
sequence $\{(\rho_0^{\mu,\lambda,\kappa}, u_0^{\mu,\lambda,\kappa}, 
\theta_0^{\mu,\lambda,\kappa}, f_0^{\mu,\lambda,\kappa})\}_{\mu,\lambda,\kappa \in (0,1)}$ 
such that $(\rho_0^{\mu,\lambda,\kappa}$, $ u_0^{\mu,\lambda,\kappa}$, 
$\theta_0^{\mu,\lambda,\kappa}) \to (\rho_0, u_0, \theta_0)$ in $H^4$ 
and $f_0^{\mu,\lambda,\kappa} \to f_0$ in $H_{x,v}^4$ as $\mu\rightarrow0$, 
$\lambda\rightarrow 0$ and $\kappa \to 0$ simultaneously. Consequently, 
recalling that $\eps_0 > 0$ given by \eqref{a-1}, 
we chose $\mu_0>0$, $\lambda_0>0$, and $\kappa_0 > 0$ such  that  
\begin{align*}  
\|(\rho_0^{\mu,\lambda,\kappa} - \rho_0, u_0^{\mu,\lambda,\kappa} - u_0, 
\theta_0^{\mu,\lambda,\kappa} - \theta_0)\|_{H^4}^2 + \|f_0^{\mu,\lambda,\kappa} 
- f_0\|_{H_{x,v}^4}^2 < \frac{\eps_0}{2},  
\end{align*}  
for all $0<\mu<\mu_0$, $0<\lambda<\lambda_0$ and $0 < \kappa < \kappa_0$. 
Combining this with \eqref{b-1}, we deduce that 
\begin{align*}  
\mathcal{E}_0^{\mu,\lambda,\kappa} =&\, \|(\rho_0^{\mu,\lambda,\kappa}, u_0^{\mu,\lambda,\kappa}, \theta_0^{\mu,\lambda,\kappa})\|_{H^4}^2
 + \|f_0^{\mu,\lambda,\kappa}\|_{H_{x,v}^4}^2  \nonumber\\
<&\, \frac{\eps_0}{2} + \|(\rho_0, u_0, \theta_0)\|_{H^4}^2 + \|f_0\|_{H_{x,v}^4}^2 \nonumber\\
\leq&\, \frac{\eps_0}{2} + \eps_1 \leq \eps_0,  
\end{align*}  
provided that $\eps_1 \leq \frac{\eps_0}{2}$. According
 to Theorem \ref{T1.1}, for any  $0<\mu<\mu_0$, $0<\lambda<\lambda_0$ and
  $0 < \kappa < \kappa_0$, we can construct a unique global solution 
  $(\rho^{\mu,\lambda,\kappa}, u^{\mu,\lambda,\kappa}, \theta^{\mu,\lambda,\kappa},
   f^{\mu,\lambda,\kappa})$ to the Cauchy problem \eqref{Q1}--\eqref{Q1-1}. 
   This solution satisfies the regularity estimate \eqref{a-2}, 
   which is independent of $\mu$, $\lambda$,  $\kappa$ and 
   the time $t$. As a direct consequence of \eqref{a-2}, there exists a limit $(\rho, u, \theta, f)$
such that as $\mu\rightarrow0$, $\lambda\rightarrow 0$ and $\kappa \to 0$, up to a subsequence,  
\begin{align*}  
&(\rho^{\mu,\lambda,\kappa}, u^{\mu,\lambda,\kappa}, \theta^{\mu,\lambda,\kappa} ) \rightharpoonup (\rho, u, \theta ) \quad \text{weakly-$*$ in } \quad L^\infty( \mathbb{R}^+; H^4(\mathbb{R}^3_x)),\\
&\quad\quad \quad \qquad\,\,\,\,\,\,\,\,f^{\mu,\lambda,\kappa} \rightharpoonup f \quad\quad\quad\,\,\,\,  \text{weakly-$*$ in } \quad L^\infty( \mathbb{R}^+; H_{x,v}^4(\mathbb{R}^3_x\times\mathbb{R}^3_v)).
\end{align*}  
Then, \eqref{b2} follows.

To rigorously justify the convergence of all nonlinear terms in the sense of distributions, it is essential to analyze the time derivatives terms: $\partial_t \rho^{\mu,\lambda,\kappa}$, $\partial_t u^{\mu,\lambda,\kappa}$  and $\partial_t \theta^{\mu,\lambda,\kappa}$.
%  and $\partial_t f^\kappa$. 
Indeed, from \eqref{Q1} and \eqref{a-2}, we  deduce that  
\begin{align*}  
\|\partial_t\rho^{\mu,\lambda,\kappa}\|_{L^{\infty}_t(H^3)} \lesssim&\, \|\nabla u^{\mu,\lambda,\kappa}\|_{L^{\infty}_t(H^3)}\|\nabla\rho^{\mu,\lambda,\kappa}\|_{L^{\infty}_t(H^3)}+\|\nabla u^{\mu,\lambda,\kappa}\|_{L^{\infty}_t(H^3)}\nonumber\\
\lesssim&\,\big(1+\|\rho^{\mu,\lambda,\kappa}\|_{L^{\infty}_t(H^4)}\big)\|u^{\mu,\lambda,\kappa}
\|_{L^{\infty}_t(H^4)}%\nonumber\\
\lesssim  C,
\end{align*}  
and  
\begin{align*}  
&\|\partial_t (u^{\mu,\lambda,\kappa}, \theta^{{\mu,\lambda,\kappa}})\|_{L^{\infty}_t(H^2)}\nonumber\\ 
\lesssim&\, \|u^{\mu,\lambda,\kappa}\|_{L_t^{\infty}(H^4)}\|\nabla u^{\mu,\lambda,\kappa}\|_{L_t^\infty(H^3)}+\big(1+\|(\rho^{\mu,\lambda,\kappa},\theta^{\mu,\lambda,\kappa})\|_{L_t^\infty(H^4)}^2\big)\|\nabla\rho^{\mu,\lambda,\kappa}\|_{L_t^\infty(H^3)}\nonumber\\
&+\|\nabla (u^{\mu,\lambda,\kappa},\theta^{\mu,\lambda,\kappa})\|_{L_t^\infty(H^3)}+\big(1+\|\rho^{\mu,\lambda,\kappa}\|_{L_t^\infty(H^3)}\big)\|\nabla(u^{\mu,\lambda,\kappa},\theta^{\mu,\lambda,\kappa})\|_{L_t^\infty(H^3)} \nonumber\\
&+\big(1+\|\rho^{\mu,\lambda,\kappa}\|_{L_t^\infty(H^4)}\big)\big( \|b^ {\mu,\lambda,\kappa}-u^{\mu,\lambda,\kappa}\|_{L_t^\infty(H^3)}+\|\sqrt{2}\omega ^{\mu,\lambda,\kappa}-\sqrt{3}\theta ^{\mu,\lambda,\kappa}\|_{L_t^\infty(H^3)} \big)\nonumber\\
&+\|  u^{\mu,\lambda,\kappa}\|_{L_t^\infty(H^4)}\|\nabla\theta^{\mu,\lambda,\kappa}\|_{L_t^\infty(H^3)}+\|\nabla\theta^{\mu,\lambda,\kappa}\|_{L_t^\infty(H^3)}+\|\nabla\rho^{\mu,\lambda,\kappa}\|_{L_t^\infty(H^2)}\|\nabla\theta^{\mu,\lambda,\kappa}\|_{L_t^\infty(H^3)}\nonumber\\
&+\big(1+\|\rho^{\mu,\lambda,\kappa}\|_{L_t^\infty(H^4)}\big)\|(a^{\mu,\lambda,\kappa},b^{\mu,\lambda,\kappa},u^{\mu,\lambda,\kappa})\|_{L_t^\infty(H^4)}\|\nabla(u^{\mu,\lambda,\kappa},\theta^{\mu,\lambda,\kappa},u^{\mu,\lambda,\kappa})\|_{L_t^\infty(H^3)}\nonumber\\
\lesssim&\, C,
\end{align*} 
where $C>0$ represents a constant independent of $\mu$, $\lambda$ and $\kappa$. Here, we have utilized the smallness of $\mu$, $\lambda$ and $\kappa$. By applying  Aubin-Lions's lemma, we conclude, up to extracting a subsequence,  that
\begin{align*}  
(\rho^{\mu,\lambda,\kappa}, u^{\mu,\lambda,\kappa}, \theta^{\mu,\lambda,\kappa}) \rightarrow (\rho, u, \theta) \quad \text{strongly in}\quad  C_{\rm loc}(\mathbb{R}^{+}; H^3_{\rm loc}(\mathbb{R}^3_x)).  
\end{align*}  
As a result, it is straightforward to verify that $(\rho, u, \theta, f)$, with $\rho =  \varrho-1$ and $f =\frac{F-M}{\sqrt{M}}  $, solves the problem \eqref{Q2}--\eqref{Q2-2} in the distributional sense. \hfill $\Box$

\subsection{Proof of Theorem \ref{T1.3}}
We are now in a position to rigorously establish the
 global-in-time convergence of the limits processes.
  For  $\mu,\lambda,\kappa \in (0,1)$, let $(\rho^{\mu,
  \lambda,\kappa}, u^{\mu,\lambda,\kappa}, \theta^{\mu,\lambda,\kappa},
   f^{\mu,\lambda,\kappa})$ denote the global solution to the Cauchy 
   problem \eqref{Q1}--\eqref{Q1-1}, as constructed in Theorem \ref{T1.1}. 
    We recall the equations governing the evolution of $(\rho^{\mu,\lambda,\kappa},
 u^{\mu,\lambda,\kappa}, \theta^{\mu,\lambda,\kappa}, f^{\mu,\lambda,\kappa})$ can be expressed as follows:
\begin{equation}\label{G4.1}
\left\{\begin{aligned}
&\partial_t \rho^{\mu,\lambda,\kappa} +u^{\mu,\lambda,\kappa} \cdot \nabla \rho^{\mu,\lambda,\kappa}+(1+\rho^{\mu,\lambda,\kappa}) {\rm div} u^{\mu,\lambda,\kappa}=0, \\
&\partial_t u^{\mu,\lambda,\kappa} +u^{\mu,\lambda,\kappa} \cdot \nabla u^{\mu,\lambda,\kappa}+\frac{1+\theta^{\mu,\lambda,\kappa}}{1+\rho^{\mu,\lambda,\kappa}} \nabla \rho^{\mu,\lambda,\kappa}+\nabla \theta^{\mu,\lambda,\kappa}\\
&\quad=\frac{ \mu\Delta  u^{\mu,\lambda,\kappa}+(\mu+\lambda)\nabla{\rm div}  u^{\mu,\lambda,\kappa} }{1+\rho^{\mu,\lambda,\kappa}}  
+\frac{b^{\mu,\lambda,\kappa}-u^{\mu,\lambda,\kappa}-a^{\mu,\lambda,\kappa} u^{\mu,\lambda,\kappa}}{1+\rho^{\mu,\lambda,\kappa}}, \\
&\partial_t \theta^{\mu,\lambda,\kappa} +u^{\mu,\lambda,\kappa} \cdot \nabla \theta^{\mu,\lambda,\kappa}+\theta^{\mu,\lambda,\kappa}  {\rm div} u^{\mu,\lambda,\kappa}+{\rm div} u^{\mu,\lambda,\kappa}-\sqrt{3}(\sqrt{2} \omega^{\mu,\lambda,\kappa}-\sqrt{3} \theta^{\mu,\lambda,\kappa})  \\
&\quad =\frac{ \kappa\Delta  \theta^{\mu,\lambda,\kappa}+|u^{\mu,\lambda,\kappa}|^2-2 b^{\mu,\lambda,\kappa} \cdot u^{\mu,\lambda,\kappa}+a^{\mu,\lambda,\kappa}|u^{\mu,\lambda,\kappa}|^2-3 a^{\mu,\lambda,\kappa} \theta^{\mu,\lambda,\kappa}  }{1+\rho^{\mu,\lambda,\kappa}}  \\
&\qquad + \frac{2\mu |D(u^{\mu,\lambda,\kappa})|^2+\lambda|{\rm div}  u^{\mu,\lambda,\kappa}|^2 }{1+\rho^{\mu,\lambda,\kappa}}-\frac{\sqrt{3}\rho^{\mu,\lambda,\kappa}(\sqrt{2} \omega^{\mu,\lambda,\kappa}-\sqrt{3} \theta^{\mu,\lambda,\kappa})}{1+\rho^{\mu,\lambda,\kappa}},\\
&\partial_t f^{\mu,\lambda,\kappa} +v \cdot \nabla f^{\mu,\lambda,\kappa}+u^{\mu,\lambda,\kappa} \cdot \nabla_v f^{\mu,\lambda,\kappa}-\frac{1}{2} u^{\mu,\lambda,\kappa} \cdot v f^{\mu,\lambda,\kappa}-
u^{\mu,\lambda,\kappa} \cdot v \sqrt{M}  \\
&\quad = (|v|^2-3 ) \sqrt{M} \theta^{\mu,\lambda,\kappa}+\mathcal{L} f^{\mu,\lambda,\kappa}+\frac{\theta^{\mu,\lambda,\kappa}}{\sqrt{M}} \Delta_v(\sqrt{M} f^{\mu,\lambda,\kappa}), \\
\end{aligned}\right.
\end{equation} 
For the global solution $(\rho, u, \theta, f)$ to the Cauchy problem \eqref{Q2}--\eqref{Q2-2} as presented in Theorem \ref{T1.2}, it  satisfies the system: 
\begin{equation}\label{G4.2}
\left\{\begin{aligned}
&\partial_t \rho + u\cdot\nabla\rho+(1+\rho){\rm div} u=0, \\
&\partial_t u +u \cdot \nabla u+\frac{1+\theta}{1+\rho} \nabla \rho+\nabla \theta=\frac{b-u-au}{1+\rho}, \\
&\partial_t \theta +u \cdot \nabla \theta+\theta  {\rm div} u+ {\rm div} u-\sqrt{3}(\sqrt{2} \omega-\sqrt{3} \theta)  \\
&\quad =\frac{ |u|^2-2 u \cdot b+a|u|^2-3 a \theta}{1+\rho}   -\frac{\sqrt{3}\rho(\sqrt{2} \omega-\sqrt{3} \theta)}{1+\rho},\\
&\partial_t f +v \cdot \nabla f+u \cdot \nabla_v f-\frac{1}{2} u \cdot v f-
u \cdot v \sqrt{M}- (|v|^2-3 ) \sqrt{M} \theta   =\mathcal{L} f+\frac{\theta}{\sqrt{M}} \Delta_v(\sqrt{M} f).
\end{aligned}\right.
\end{equation}  
Define the error variables  $(\widetilde{\rho}^{\mu,\lambda,\kappa},\widetilde{u}^{\mu,\lambda,\kappa},\widetilde{\theta}^{\mu,\lambda,\kappa},\widetilde{f}^{\mu,\lambda,\kappa})$ as 
\begin{align*}
(\widetilde{\rho}^{\mu,\lambda,\kappa},\widetilde{u}^{\mu,\lambda,\kappa},\widetilde{\theta}^{\mu,\lambda,\kappa},\widetilde{f}^{\mu,\lambda,\kappa}):=(\rho^{\mu,\lambda,\kappa}-\rho, u^{\mu,\lambda,\kappa}-u, \theta^{\mu,\lambda,\kappa}-\theta, f^{\mu,\lambda,\kappa}-f),    
\end{align*}
and the macroscopic quantities $(\widetilde{a}^{\mu,\lambda,\kappa},\widetilde{b}^{\mu,\lambda,\kappa},\widetilde{\omega}^{\mu,\lambda,\kappa})$ as 
\begin{equation*}
\begin{aligned}
(\widetilde{a}^{\mu,\lambda,\kappa},\widetilde{b}^{\mu,\lambda,\kappa},\widetilde{\omega}^{\mu,\lambda,\kappa}):=(a^{\mu,\lambda,\kappa}-a , b^{\mu,\lambda,\kappa}-b,\omega^{\mu,\lambda,\kappa}-\omega ).
\end{aligned}
\end{equation*}    

By combining the systems \eqref{G4.1} and \eqref{G4.2}, we obtain the following error system:
\begin{equation}\label{G4.3}
\left\{
\begin{aligned}
&\partial_t \widetilde{\rho}^{\mu,\lambda,\kappa}+{\rm div}\widetilde{u}^{\mu,\lambda,\kappa} =\widetilde{F}_1,\\
&\partial_{t}\widetilde{u}^{\mu,\lambda,\kappa}+\nabla \widetilde{\rho}^{\mu,\lambda,\kappa}
+\nabla\widetilde{\theta}^{\mu,\lambda,\kappa}-(\widetilde{b}^{\mu,\lambda,\kappa}-\widetilde{u}^{\mu,\lambda,\kappa})
-\frac{\mu \Delta u^{\mu,\lambda,\kappa}}{1+\rho^{\mu,\lambda,\kappa}}-\frac{(\mu+\lambda) 
\nabla {\rm div}u^{\mu,\lambda,\kappa}}{1+\rho^{\mu,\lambda,\kappa}} \\
&\quad=\frac{2\mu |D(u^{\mu,\lambda,\kappa})|^2}{1+\rho^{\mu,\lambda,\kappa}}
+\frac{ \lambda |{\rm div}u^{\mu,\lambda,\kappa}|^2}{1+\rho^{\mu,\lambda,\kappa}}
-a^{\mu,\lambda,\kappa}\widetilde{u}^{\mu,\lambda,\kappa}-\widetilde{a}^{\mu,\lambda,\kappa} u
+ \widetilde{F}_2+\widetilde{F}_3+\widetilde{F}_4,\\
&\partial_t\widetilde{\theta}^{{\mu,\lambda,\kappa}}+{\rm div}\widetilde{u}^{{\mu,\lambda,\kappa}}
-\sqrt{3}(\sqrt{2}\widetilde{\omega}^{{\mu,\lambda,\kappa}}-\sqrt{3}\widetilde{\theta}^{{\mu,\lambda,\kappa}}) \\
&\quad=\frac{\kappa\Delta\theta^{{\mu,\lambda,\kappa}}}{1+\rho^{{\mu,\lambda,\kappa}}}
-b^{\mu,\lambda,\kappa}\cdot \widetilde{u}^{\mu,\lambda,\kappa}-\widetilde{b}^{\mu,\lambda,\kappa}\cdot u
-3\widetilde{a}^{\mu,\lambda,\kappa}\theta-3a^{\mu,\lambda,\kappa}\widetilde{\theta}^{\mu,\lambda,\kappa}
+\widetilde{F}_5+\widetilde{F}_6+\widetilde{F}_7,\\
&\partial_{t}\widetilde{f}^{\mu,\lambda,\kappa}+v\cdot\nabla \widetilde{f}^{\mu,\lambda,\kappa} 
+ {u}\cdot\nabla_v\widetilde{f}^{\mu,\lambda,\kappa}+\widetilde{u}^{\mu,\lambda,\kappa}\cdot\nabla_v f
-\widetilde{u}^{\mu,\lambda,\kappa}\cdot v\sqrt{M}-(|v|^2-3)\sqrt{M}\widetilde\theta^{\mu,\lambda,\kappa} \\
&\quad= \mathcal{L}\widetilde{f}^{\mu,\lambda,\kappa}+\frac{1}{2}u\cdot v\widetilde{f}^{\mu,\lambda,\kappa}
+\frac{1}{2}\widetilde{u}^{\mu,\lambda,\kappa}\cdot vf^{\mu,\lambda,\kappa}
+\frac{\theta}{\sqrt{M}}\Delta_v(\sqrt{M}\widetilde{f}^{\mu,\lambda,\kappa})
+\frac{\widetilde\theta^{\mu,\lambda,\kappa}}{\sqrt{M}}\Delta_v(\sqrt{M} {f} )  ,
\end{aligned}\right.
\end{equation} 
where $\widetilde{F}_i(i=1,\dots,7)$  are given by
\begin{align*}
\widetilde{F}_1:=&\,-{\rho}^{\mu,\lambda,\kappa}{\rm div}\widetilde{u}^{{\mu,\lambda,\kappa}}-\widetilde{\rho}^{{\mu,\lambda,\kappa}}{\rm div}u
-\widetilde{u}^{{\mu,\lambda,\kappa}}\cdot\nabla\rho-u^{{\mu,\lambda,\kappa}}\cdot\nabla\widetilde{\rho}^{{\mu,\lambda,\kappa}} ,\\
\widetilde{F}_2:=&\,\Big(\frac{\rho^{\mu,\lambda,\kappa}}{1+\rho^{\mu,\lambda,\kappa}}-\frac{\rho }{1+\rho }\Big)
 \nabla\rho^{{\mu,\lambda,\kappa}}-\Big(\frac{1}{1+\rho^{\mu,\lambda,\kappa}}-\frac{1}{1+\rho}\Big)\theta^{{\mu,\lambda,\kappa}} \nabla\rho^{{\mu,\lambda,\kappa}} \\
&-\Big(\frac{\rho^{\mu,\lambda,\kappa}}{1+\rho^{\mu,\lambda,\kappa}}-\frac{\rho }{1+\rho }\Big) ( b^{\mu,\lambda,\kappa}
-u^{\mu,\lambda,\kappa}-a^{\mu,\lambda,\kappa} u^{\mu,\lambda,\kappa}),\nonumber\\
\widetilde{F}_3:=&\,-\frac{1}{1+\rho}(\widetilde{\theta}^{{\mu,\lambda,\kappa}}\nabla\rho
+\theta^{{\mu,\lambda,\kappa}}\nabla\widetilde{\rho}^{{\mu,\lambda,\kappa}})
+\frac{\rho}{1+\rho}(\widetilde{a}^{\mu,\lambda,\kappa} u+a^{\mu,\lambda,\kappa}\widetilde{u}^{\mu,\lambda,\kappa})\nonumber\\
&+\frac{\rho}{1+\rho}\nabla\widetilde{\rho}^{{\mu,\lambda,\kappa}}-\frac{\rho}{1+\rho} (\widetilde{b}^{\mu,\lambda,\kappa}-\widetilde{u}^{\mu,\lambda,\kappa})  ,\\
\widetilde{F}_4:=&-\widetilde{u}^{\mu,\lambda,\kappa}\cdot\nabla u-u^{\mu,\lambda,\kappa}\cdot\nabla\widetilde{u}^{\mu,\lambda,\kappa},\nonumber\\
\widetilde{F}_5:=&\,-\widetilde{\theta}^{{\mu,\lambda,\kappa}}{\rm div}u^{{\mu,\lambda,\kappa}}- {\theta}^{{\mu,\lambda,\kappa}}{\rm div}\widetilde{u}^{{\mu,\lambda,\kappa}}
-\widetilde{u}^{{\mu,\lambda,\kappa}}\cdot\nabla\theta-{u}^{{\mu,\lambda,\kappa}}\cdot\nabla\widetilde{\theta}^{{\mu,\lambda,\kappa}},\\
\widetilde{F}_6:=&\,\frac{1}{1+\rho}\big((u^{\mu,\lambda,\kappa}-b^{\mu,\lambda,\kappa})\cdot\widetilde{u}^{\mu,\lambda,\kappa}
+u\cdot(\widetilde{u}^{\mu,\lambda,\kappa}-\widetilde{b}^{\mu,\lambda,\kappa}) \big) \nonumber\\
&+\frac{1}{1+\rho}\big( \widetilde{a}^{\mu,\lambda,\kappa}\theta^2+a^{\mu,\lambda,\kappa}(\theta\widetilde{\theta}^{\mu,\lambda,\kappa}
+\theta^{\mu,\lambda,\kappa}\widetilde{\theta}^{\mu,\lambda,\kappa})\big) +\sqrt{3}\frac{\rho}{1+\rho}(\sqrt{2}\widetilde{\omega}^{\mu,\lambda,\kappa}
-\sqrt{3}\widetilde\theta^{\mu,\lambda,\kappa})\nonumber\\
&+\frac{\rho}{1+\rho}(\widetilde{b}^{\mu,\lambda,\kappa}\cdot u+b^{\mu,\lambda,\kappa}\cdot \widetilde{u}^{\mu,\lambda,\kappa}
+3\widetilde{a}^{\mu,\lambda,\kappa}\theta+3 {a}^{\mu,\lambda,\kappa}\widetilde{\theta}^{\mu,\lambda,\kappa}),\nonumber\\
\widetilde{F}_7:=&\,-\sqrt{3}\Big(\frac{\rho^{\mu,\lambda,\kappa}}{1+\rho^{\mu,\lambda,\kappa}}
-\frac{\rho}{1+\rho}\Big)(\sqrt{2}\omega^{\mu,\lambda,\kappa}-\sqrt{3}\theta^{\mu,\lambda,\kappa})\nonumber\\
&+\Big(\frac{1}{1+\rho^{\mu,\lambda,\kappa}}-\frac{1}{1+\rho}\Big) a^{\mu,\lambda,\kappa}(\theta^{\mu,\lambda,\kappa})^2\nonumber\\
&+\Big(\frac{1}{1+\rho^{\mu,\lambda,\kappa}}-\frac{1}{1+\rho}\Big)\big(| {u}^{\mu,\lambda,\kappa}|^2-2b^{\mu,\lambda,\kappa}\cdot u^{\mu,\lambda,\kappa}-3a^{\mu,\lambda,\kappa}\theta^{\mu,\lambda,\kappa} \big).
\end{align*} 
It is important to note that the convergence rate is predominantly 
determined by the  terms $\frac{\mu\Delta u^{{\mu,\lambda,\kappa}}}{1+\rho^{{\mu,\lambda,\kappa}}}$,
 $\frac{(\mu+\lambda)\nabla{\rm div}u^{{\mu,\lambda,\kappa}}}{1+\rho^{{\mu,\lambda,\kappa}}}$, 
  $\frac{\kappa\Delta\theta^{\mu,\lambda,\kappa}}{1+\rho^{{\mu,\lambda,\kappa}}}$, 
  $\frac{2\mu |D(u^{\mu,\lambda,\kappa})|^2}{1+\rho^{\mu,\lambda,\kappa}}$ 
and $\frac{ \lambda |{\rm div}u^{\mu,\lambda,\kappa}|^2}{1+\rho^{\mu,\lambda,\kappa}}$ 
appearing in \eqref{G4.3}. 

Now, we proceed to estimate the error $\|(\widetilde{\rho}^{\mu,\lambda,\kappa},\widetilde{u}^{\mu,\lambda,\kappa},\widetilde{\theta}^{\mu,\lambda,\kappa})\|_{H^1}$ and $\|\widetilde{f}^{\mu,\lambda,\kappa}\|_{H_{x,v}^1}$ respectively. 
To achieve this, we introduce the following functional:  
\begin{align}\label{G4.4}  
{\Xi} ^{\mu,\lambda,\kappa}(t) :=&\sup_{\tau\in[0,t]}\big(\|(\widetilde{\rho}^{\mu,\lambda,\kappa},\widetilde{u}^{\mu,\lambda,\kappa} ,\widetilde{\theta}^{\mu,\lambda,\kappa})(\tau)\|_{H^1}^2 +\|\widetilde{f}^{\mu,\lambda,\kappa}(\tau)\|_{H_{x,v}^1}^2\big) \nonumber\\
&+\int_{0}^t\big( \|\nabla(\widetilde{a}^{\mu,\lambda,\kappa},\widetilde{b}^{\mu,\lambda,\kappa},\widetilde{\omega}^{\mu,\lambda,\kappa},\widetilde{\rho}^{\mu,\lambda,\kappa})(\tau)\|_{L^2}^2 {\rm d}\tau\nonumber\\  
&+\int_{0}^t\big(   \|(\widetilde{b}^{\mu,\lambda,\kappa}-\widetilde{u}^{\mu,\lambda,\kappa})(\tau)\|_{H^1}^2 + \|\big(\sqrt{2}\widetilde{\omega}^{\mu,\lambda,\kappa}-\sqrt{3}\widetilde{\theta}^{\mu,\lambda,\kappa}\big)(\tau)\|_{H^1}^2 \big){\rm d}\tau\nonumber\\  
&+\int_0^t\sum_{|\alpha|\leq 1}\|\{\mathbf{I}-\mathbf{P}\}\partial^\alpha\widetilde{f}^{\mu,\lambda,\kappa}(\tau)\|_{\nu}^2{\rm d}\tau+\int_0^t \|\nabla_v \{\mathbf{I}-\mathbf{P}\}\widetilde{f}^{\mu,\lambda,\kappa}(\tau)\|_{\nu}^2{\rm d}\tau.  
\end{align}  
It is evident that  
\begin{align}\label{LNXG4.5}  
\int_0^t \|\nabla \widetilde{u}^{\mu,\lambda,\kappa}(\tau)\|_{L^2}^2 {\rm d}\tau
\lesssim&\, \int_0^t\big( \|\nabla \widetilde{b}^{\mu,\lambda,\kappa}(\tau)\|_{L^2}^2+\|\nabla(  \widetilde{b}^{\mu,\lambda,\kappa}- \widetilde{u}^{\mu,\lambda,\kappa})(\tau)\|_{L^2}^2 \big){\rm d}\tau \nonumber\\
\lesssim&\, {\Xi} ^{\mu,\lambda,\kappa}(t), \\ \label{G4.5}
 \int_0^t \|\nabla \widetilde{\theta}^{\mu,\lambda,\kappa}(\tau)\|_{L^2}^2 {\rm d}\tau  
\lesssim &\,
\int_0^t\big( \|\nabla \widetilde{\omega}^{\mu,\lambda,\kappa}(\tau)\|_{L^2}^2+\|\nabla(\sqrt{2} \widetilde{\omega}^{\mu,\lambda,\kappa}-\sqrt{3}\widetilde{\theta}^{\mu,\lambda,\kappa})(\tau)\|_{L^2}^2 \big){\rm d}\tau\nonumber\\
\lesssim&\, {\Xi} ^{\mu,\lambda,\kappa}(t),  
\end{align} 
and  
\begin{align}\label{G4.6}  
\int_0^t \|\nabla \widetilde{f}^{\mu,\lambda,\kappa}(\tau)\|_{L_{x,v}^2}^2 {\rm d}\tau \lesssim&\, \int_0^t\big( \|\nabla (\widetilde{a}^{\mu,\lambda,\kappa},\widetilde{b}^{\mu,\lambda,\kappa},\widetilde{\omega}^{\mu,\lambda,\kappa})(\tau)\|_{L^2}^2+ \|\nabla \{\mathbf{I}-\mathbf{P}\}\widetilde{f}^{\mu,\lambda,\kappa}(\tau)\|_{\nu}^2\big){\rm d}\tau\nonumber\\
\lesssim&\,{\Xi} ^{\mu,\lambda,\kappa}(t).  
\end{align}   

Now we estimate the terms appearing in the right hand-side of \eqref{G4.4} .
\begin{lem}\label{L4.1}
It holds that
\begin{align}\label{G4.7}
&\sup_{\tau\in[0,t]}\big\{\|(\widetilde{\rho}^{\mu,\lambda,\kappa},\widetilde{u}^{\mu,\lambda,\kappa},
\widetilde{\theta}^{\mu,\lambda,\kappa})(\tau)\|_{L^2}^2
+\|\widetilde{f}^{\mu,\lambda,\kappa}(\tau)\|_{L_{x,v}^2}^2\big\}   \nonumber\\
&\quad +\int_{0}^t\big( \|(\widetilde{b}^{\mu,\lambda,\kappa}-\widetilde{u}^{\mu,\lambda,\kappa})(\tau)\|_{L^2}^2 +\|(\sqrt{2}\widetilde{\omega}^{\mu,\lambda,\kappa}-\sqrt{3}\widetilde{\theta}^{\mu,\lambda,\kappa})(\tau)\|_{L^2}^2+\|  \{\mathbf{I}-\mathbf{P}\}\widetilde{f}^{\mu,\lambda,\kappa}(\tau)\|_{\nu}^2 \big) {\rm d}\tau \nonumber\\
&\quad\quad \leq {\Xi} ^{\mu,\lambda,\kappa}(0)+C\max\{\mu,\lambda,\kappa\}^2+C \big(\eps_0^\frac{1}{2}  +\eps_1^\frac{1}{2}\big){\Xi} ^{\mu,\lambda,\kappa}(t),
\end{align}
where $C>0$ is a constant independent of  $\lambda$, $\mu$,  $\kappa$ and   $t$.
\end{lem}
\begin{proof}
By performing energy estimates on \eqref{G4.3} and utilizing the properties of decomposition given in \eqref{G2.4} and \eqref{G2.5} and the dissipation property of $\mathcal{L}$ in \eqref{G2.6}, we obtain  
\begin{align}\label{G4.8}
&\frac{1}{2}\big( \|(\widetilde{\rho}^{\mu,\lambda,\kappa},\widetilde{u}^{\mu,\lambda,\kappa},\widetilde\theta^{\mu,\lambda,\kappa})(t)\|_{L^2}^2+  \|\widetilde{f}^{\mu,\lambda,\kappa}(t)\|_{L_{x,v}^2}^2 \big) {\rm d}\tau\nonumber\\
&+\int_0^t\big( \|(\widetilde{b}^{\mu,\lambda,\kappa}-\widetilde{u}^{\mu,\lambda,\kappa})(\tau)\|_{L^2}^2 +\|(\sqrt{2}\widetilde{\omega}^{\mu,\lambda,\kappa}-\sqrt{3}\widetilde{\theta}^{\mu,\lambda,\kappa})(\tau)\|_{L^2}^2+\bar\lambda\|  \{\mathbf{I}-\mathbf{P}\}\widetilde{f}^{\mu,\lambda,\kappa}(\tau)\|_{\nu}^2   \big){\rm d}\tau\nonumber\\
\leq\,&\int_0^t\!   \int_{\mathbb{R}^3 }\frac{1}{2} u\cdot\langle v\widetilde{f}^{\mu,\lambda,\kappa},\widetilde{f}^{\mu,\lambda,\kappa}\rangle {\rm d}x{\rm d}\tau-\int_0^t  \int_{\mathbb{R}^3 } u\cdot\langle \nabla_v\widetilde{f}^{\mu,\lambda,\kappa},\widetilde{f}^{\mu,\lambda,\kappa}\rangle {\rm d}x{\rm d}\tau\nonumber\\
&-\int_0^t \widetilde{b}^{\mu,\lambda,\kappa}\cdot {u} \widetilde{\theta}^{\mu,\lambda,\kappa}{\rm d}\tau-\int_0^t \widetilde{a}^{\mu,\lambda,\kappa} u\cdot\widetilde{u}^{\mu,\lambda,\kappa} {\rm d}{\tau} +\int_0^t \! \int_{\mathbb{R}^3 }\frac{1}{2} \widetilde{u}^{\mu,\lambda,\kappa}\cdot\langle v{f}^{\mu,\lambda,\kappa},\widetilde{f}^{\mu,\lambda,\kappa}\rangle {\rm d}x{\rm d}\tau           \nonumber\\
&-\int_0^t  \int_{\mathbb{R}^3 } \widetilde{u}^{\mu,\lambda,\kappa}\cdot\langle \nabla_v {f} ,\widetilde{f}^{\mu,\lambda,\kappa}\rangle {\rm d}x{\rm d}\tau-\int_0^t {b}^{\mu,\lambda,\kappa}\cdot  \widetilde{u}^{\mu,\lambda,\kappa}\widetilde{\theta}^{\mu,\lambda,\kappa}{\rm d}\tau     -\int_0^t{a}^{\mu,\lambda,\kappa} |\widetilde{u}^{\mu,\lambda,\kappa}|^2 {\rm d}{\tau}       \nonumber\\
&+\int_0^t  \int_{\mathbb{R}^3}\Big\langle\frac{\widetilde\theta^{\mu,\lambda,\kappa}}{\sqrt{M}}\Delta_v(\sqrt{M}f),\widetilde{f}^{\mu,\lambda,\kappa}\Big\rangle{\rm d}x{\rm d}{\tau}-3\int_0^t\!\int_{\mathbb{R}^3}a^{\mu,\lambda,\kappa}(\widetilde\theta^{\mu,\lambda,\kappa})^2{\rm d}x{\rm d}
\tau\nonumber\\
&+\int_0^t \int_{\mathbb{R}^3}\Big\langle\frac{\theta}{\sqrt{M}}\Delta_v(\sqrt{M}\widetilde{f}^{\mu,\lambda,\kappa}),\widetilde{f}^{\mu,\lambda,\kappa}\Big\rangle{\rm d}x{\rm d}{\tau}-3\int_0^t\!\int_{\mathbb{R}^3}\widetilde{a}^{\mu,\lambda,\kappa} \theta\widetilde{\theta}^{\mu,\lambda,\kappa} {\rm d}x{\rm d}
\tau\nonumber\\
&+\int_0^t\!\int_{\mathbb{R}^3} \widetilde{\rho}^{{\mu,\lambda,\kappa}}\widetilde{F}_1 {\rm d}x{\rm d}\tau+\int_0^t\!\int_{\mathbb{R}^3} \widetilde{u}^{{\mu,\lambda,\kappa}}\cdot\widetilde{F}_2 {\rm d}x{\rm d}\tau+\int_0^t\!\int_{\mathbb{R}^3} \widetilde{u}^{{\mu,\lambda,\kappa}}\cdot\widetilde{F}_3 {\rm d}x{\rm d}\tau\nonumber\\
&+\int_0^t\!\int_{\mathbb{R}^3} \widetilde{u}^{{\mu,\lambda,\kappa}}\cdot\widetilde{F}_4 {\rm d}x{\rm d}\tau+\int_0^t\!\int_{\mathbb{R}^3} \frac{\mu\Delta u^{\mu,\lambda,\kappa} \widetilde{u}^{\mu,\lambda,\kappa} }{1+\rho^{\mu,\lambda,\kappa}} {\rm d}x{\rm d}\tau+\int_0^t\!\int_{\mathbb{R}^3} \frac{(\mu+\lambda)\nabla{\rm div} u^{\mu,\lambda,\kappa} \widetilde{u}^{\mu,\lambda,\kappa} }{1+\rho^{\mu,\lambda,\kappa}} {\rm d}x{\rm d}\tau
\nonumber\\
&+\int_0^t\!\int_{\mathbb{R}^3} \frac{2\mu |D(u^{\mu,\lambda,\kappa})|^2 \widetilde{u}^{\mu,\lambda,\kappa} }{1+\rho^{\mu,\lambda,\kappa}} {\rm d}x{\rm d}\tau+\int_0^t\!\int_{\mathbb{R}^3} \frac{\lambda |{\rm div}u^{\mu,\lambda,\kappa}|^2 \widetilde{u}^{\mu,\lambda,\kappa} }{1+\rho^{\mu,\lambda,\kappa}} {\rm d}x{\rm d}\tau\nonumber\\
&+\int_0^t\!\int_{\mathbb{R}^3} \frac{\kappa\Delta \theta^{\mu,\lambda,\kappa} \widetilde{\theta}^{\mu,\lambda,\kappa} }{1+\rho^{\mu,\lambda,\kappa}} {\rm d}x{\rm d}\tau + \int_0^t\!\int_{\mathbb{R}^3} \widetilde{\theta}^{{\mu,\lambda,\kappa}} \widetilde{F}_5 {\rm d}x{\rm d}\tau+ \int_0^t\!\int_{\mathbb{R}^3} \widetilde{\theta}^{{\mu,\lambda,\kappa}} \widetilde{F}_6 {\rm d}x{\rm d}\tau\nonumber\\
&+ \int_0^t\!\int_{\mathbb{R}^3} \widetilde{\theta}^{{\mu,\lambda,\kappa}} \widetilde{F}_7 {\rm d}x{\rm d}\tau  +\frac{1}{2}\big( \|(\widetilde{\rho}^{\mu,\lambda,\kappa},\widetilde{u}^{\mu,\lambda,\kappa},\widetilde\theta^{\mu,\lambda,\kappa})(0)\|_{L^2}^2+  \|\widetilde{f}^{\mu,\lambda,\kappa}(0)\|_{L_{x,v}^2}^2 \big)\nonumber\\
\equiv:\,&\sum_{i=1}^{24}\widetilde{I}_i+\frac{1}{2}\big( \|(\widetilde{\rho}^{\mu,\lambda,\kappa},\widetilde{u}^{\mu,\lambda,\kappa},\widetilde\theta^{\mu,\lambda,\kappa})(0)\|_{L^2}^2+  \|\widetilde{f}^{\mu,\lambda,\kappa}(0)\|_{L_{x,v}^2}^2 \big).
\end{align}

Now we handle the terms $\widetilde{I}_i (i=1,\dots, 24)$ in turn. 
Similar to \eqref{G3.6}, applying the macro-micro decomposition method together with the estimate $v^{k}\sqrt{M} \lesssim 1$ for $k \geq 0$, we obtain   
\begin{align}\label{G4.9}
&|\widetilde{I}_1+\widetilde{I}_2+\widetilde{I}_3+\widetilde{I}_4|\nonumber\\
\lesssim&\, \Big|\int_0^t
\int_{\mathbb{R}^3}\widetilde{a}^{\mu,\lambda,\kappa} u\cdot(\widetilde{b}^{\mu,\lambda,\kappa}-\widetilde{u}^{\mu,\lambda,\kappa}){\rm d}x{\rm d}\tau\Big|+ \Big|\int_0^t(\sqrt{2}\widetilde{\omega}^{\mu,\lambda,\kappa}-\sqrt{3}\widetilde{\theta}^{\mu,\lambda,\kappa})u\cdot\widetilde{b}^{\mu,\lambda,\kappa}{\rm d}x{\rm d}\tau\Big|\nonumber\\
&+ \Big|\int_0^t\!\int_{\mathbb{R}^3}u\cdot\langle v\{\mathbf{I}-\mathbf{P}\}\widetilde{f}^{\mu,\lambda,\kappa}, \{\mathbf{I}-\mathbf{P}\}\widetilde{f}^{\mu,\lambda,\kappa}\rangle{\rm d}x{\rm d}\tau \Big|\nonumber\\
&+\int_0^t \int_{\mathbb{R}^3}|u||(\widetilde{a}^{\mu,\lambda,\kappa},\widetilde{b}^{\mu,\lambda,\kappa},\widetilde{\omega}^{\mu,\lambda,\kappa})||\{\mathbf{I}-\mathbf{P}\}\widetilde{f}^{\mu,\lambda,\kappa}|_2{\rm d}x{\rm d}\tau\nonumber\\
&+\Big|\int_0^t\!\int_{\mathbb{R}^3}u\cdot\langle \nabla_v\{\mathbf{I}-\mathbf{P}\}\widetilde{f}^{\mu,\lambda,\kappa},\{\mathbf{I}-\mathbf{P}\}\widetilde{f}^{\mu,\lambda,\kappa}\rangle{\rm d}x{\rm d}\tau\Big|\nonumber\\
\lesssim&\, \int_0^t\big(\|\widetilde{a}^{\mu,\lambda,\kappa}\|_{L^3}\|u\|_{L^6}\|\widetilde{b}^{\mu,\lambda,\kappa}-\widetilde{u}^{\mu,\lambda,\kappa}\|_{L^2}+\|\widetilde{b}^{\mu,\lambda,\kappa}\|_{L^3}\|u\|_{L^6}\|\sqrt{2}\widetilde\omega^{\mu,\lambda,\kappa}-\sqrt{3}\widetilde\theta^{\mu,\lambda,\kappa}\|_{L^2}\big){\rm d}\tau\nonumber\\
 &+\int_0^t\big(\|u\|_{L^{\infty}} \|\{\mathbf{I}-\mathbf{P}\}\widetilde{f}^{\mu,\lambda,\kappa}\|_{\nu}^{2}+\|(\widetilde{a}^{\mu,\lambda,\kappa},\widetilde{b}^{\mu,\lambda,\kappa},\widetilde\omega^{\mu,\lambda,\kappa})\|_{L^3}\|u\|_{L^6} \|\{\mathbf{I}-\mathbf{P}\}\widetilde{f}^{\mu,\lambda,\kappa}\|_{\nu}\big){\rm d}\tau\nonumber\\
 \lesssim&\, \|(\widetilde{a}^{\mu,\lambda,\kappa},\widetilde{b}^{\mu,\lambda,\kappa}
 )\|_{L_t^\infty(H^1)}\|\nabla u\|_{L_t^2(L^2)}\big(\|\widetilde{b}^{\mu,\lambda,\kappa}-\widetilde{u}^{\mu,\lambda,\kappa}\|_{L_t^2(L^2)}+\|\sqrt{2}\widetilde{\omega}^{\mu,\lambda,\kappa}-\sqrt{3}\widetilde{\theta}^{\mu,\lambda,\kappa}\|_{L_t^2(L^2)}   \big) \nonumber\\
 &+\|u\|_{L_t^\infty(H^2)} \|\{\mathbf{I}-\mathbf{P}\}\widetilde{f}^{\mu,\lambda,\kappa}\|_{L_t^2(\nu)}^2 +\|(\widetilde{a}^{\mu,\lambda,\kappa},\widetilde{b}^{\mu,\lambda,\kappa}
 ,\widetilde{\omega}^{\mu,\lambda,\kappa})\|_{L_t^\infty(H^1)}\|\nabla u\|_{L_t^2(L^2)}\{\mathbf{I}-\mathbf{P}\}\widetilde{f}^{\mu,\lambda,\kappa}\|_{L_t^2(\nu)}\nonumber\\
 \lesssim&\, \big(\eps_0^\frac{1}{2}+\eps_1^\frac{1}{2}\big){\Xi} ^{\mu,\lambda,\kappa}(t),
\end{align}
and
\begin{align}\label{G4.10}
&|\widetilde{I}_5+\widetilde{I}_6+\widetilde{I}_7+\widetilde{I}_8|\nonumber\\
\lesssim&\, \int_0^t\big(\| {a}^{\mu,\lambda,\kappa}\|_{L^6}\|\widetilde{u}^{\mu,\lambda,\kappa}\|_{L^3}\|\widetilde{b}^{\mu,\lambda,\kappa}-\widetilde{u}^{\mu,\lambda,\kappa}\|_{L^2}+\| {b}^{\mu,\lambda,\kappa}\|_{L^6}\|\widetilde{u}^{\mu,\lambda,\kappa}\|_{L^3}\|\sqrt{2}\widetilde\omega^{\mu,\lambda,\kappa}-\sqrt{3}\widetilde\theta^{\mu,\lambda,\kappa}\|_{L^2}\big){\rm d}\tau\nonumber\\
&+\int_0^t \|\widetilde{u}^{\mu,\lambda,\kappa}\|_{L^{3}} \|\{\mathbf{I}-\mathbf{P}\}\widetilde{f}^{\mu,\lambda,\kappa}\|_{\nu}\|\nabla\{\mathbf{I}-\mathbf{P}\} {f}^{\mu,\lambda,\kappa}\|_{\nu} {\rm d}\tau\nonumber\\
&+\int_0^t\|( {a}^{\mu,\lambda,\kappa}, {b}^{\mu,\lambda,\kappa}, \omega^{\mu,\lambda,\kappa})\|_{L^6}\|\widetilde{u}^{\mu,\lambda,\kappa}\|_{L^3} \|\{\mathbf{I}-\mathbf{P}\}\widetilde{f}^{\mu,\lambda,\kappa}\|_{\nu} {\rm d}\tau\nonumber\\
 \lesssim&\, \|\widetilde{u}^{\mu,\lambda,\kappa} \|_{L_t^\infty(H^1)}\|\nabla  a^{\mu,\lambda,\kappa}  \|_{L_t^2(L^2)}\big(\|\widetilde{b}^{\mu,\lambda,\kappa}-\widetilde{u}^{\mu,\lambda,\kappa}\|_{L_t^2(L^2)}+\|\sqrt{2}\widetilde{\omega}^{\mu,\lambda,\kappa}-\sqrt{3}\widetilde{\theta}^{\mu,\lambda,\kappa}\|_{L_t^2(L^2)}   \big) \nonumber\\
 &+\|\widetilde{u}^{\mu,\lambda,\kappa} \|_{L_t^\infty(H^1)} \|\{\mathbf{I}-\mathbf{P}\}\widetilde{f}^{\mu,\lambda,\kappa}\|_{L_t^2(\nu)}\|\nabla\{\mathbf{I}-\mathbf{P}\} {f}^{\mu,\lambda,\kappa}\|_{L_t^2(\nu)}   \nonumber\\
 &+\|\widetilde{u}^{\mu,\lambda,\kappa} \|_{L_t^\infty(H^1)}\|\nabla (a^{\mu,\lambda,\kappa},b^{\mu,\lambda,\kappa},\omega^{\mu,\lambda,\kappa})\|_{L_t^2(L^2)}\|\{\mathbf{I}-\mathbf{P}\}\widetilde{f}^{\mu,\lambda,\kappa}\|_{L_t^2(\nu)}\nonumber\\
 \lesssim&\, \big(\eps_0^\frac{1}{2}+\eps_1^\frac{1}{2}\big){\Xi} ^{\mu,\lambda,\kappa}(t).    
\end{align}
Similar to the estimate    in \eqref{G3.7}, we can further deduce that  
\begin{align}\label{G4.11}
|\widetilde{I}_9+\widetilde{I}_{10}|    \lesssim&\, \int_0^t \|\widetilde\theta^{\mu,\lambda,\kappa}\|_{L^3}\|(a,b,\omega)\|_{L^6}\|\{\mathbf{I}-\mathbf{P}\}\widetilde{f}^{\mu,\lambda,\kappa}\|_{\nu}{\rm d}\tau\nonumber\\
&+\int_0^t\|\widetilde\theta^{\mu,\lambda,\kappa}\|_{L^3}\| a\|_{L^6}\|\sqrt{2}\widetilde\omega^{\mu,\lambda,\kappa}-\sqrt{3}\widetilde\theta^{\mu,\lambda,\kappa}\|_{L^2} {\rm d}\tau\nonumber\\
&+\int_0^t\|\widetilde\theta^{\mu,\lambda,\kappa}\|_{L^3}\|\nabla\{\mathbf{I}-\mathbf{P}\}f\|_{\nu}\|\{\mathbf{I}-\mathbf{P}\}\widetilde{f}^{\mu,\lambda,\kappa}\|_{\nu}{\rm d}\tau\nonumber\\
\lesssim&\, \|\widetilde{\theta}^{\mu,\lambda,\kappa}\|_{L_t^\infty(H^1)}\|\{\mathbf{I}-\mathbf{P}\}\widetilde{f}^{\mu,\lambda,\kappa}\|_{L_t^2(\nu)}\big(\|\nabla(a,b,\omega)\|_{L_t^2(L^2)}+\|\nabla\{\mathbf{I}-\mathbf{P}\} {f} \|_{L_t^2(\nu)}   \big)\nonumber\\
&+\|\widetilde{\theta}^{\mu,\lambda,\kappa}\|_{L_t^\infty(H^1)}\|\nabla a\|_{L_t^2(L^2)}\|\sqrt{2}\widetilde\omega^{\mu,\lambda,\kappa}-\sqrt{3}\widetilde\theta^{\mu,\lambda,\kappa}\|_{L_t^2(L^2)}\nonumber\\
\lesssim&\, \big(\eps_0^\frac{1}{2}+\eps_1^\frac{1}{2}\big){\Xi} ^{\mu,\lambda,\kappa}(t),
\end{align}
and the term $\widetilde{I}_{11}+\widetilde{I}_{12}$ can be derived in a similar manner:
\begin{align}\label{G4.12}
|\widetilde{I}_{11}+\widetilde{I}_{12}|\lesssim  \big(\eps_0^\frac{1}{2}+\eps_1^\frac{1}{2}\big){\Xi} ^{\mu,\lambda,\kappa}(t).   
\end{align}

For the terms $\widetilde{I}_{13}, \dots,  \widetilde{I}_{16}$, by employing the definition of ${\Xi} ^{\mu,\lambda,\kappa}(t)$, the inequalities  \eqref{LNXG4.5}, \eqref{G4.5} and \eqref{G4.6}, H\"{o}lder's inequality, and Lemma \ref{LA.1}, it follows that  
\begin{align}
|\widetilde{I}_{13}|\lesssim&\, \int_0^t\|\widetilde{\rho}^{\mu,\lambda,\kappa}\|_{L^6}(\| {\rho}^{\mu,\lambda,\kappa}\|_{L^2} \|\nabla\widetilde{u}^{{\mu,\lambda,\kappa}}\|_{L^3}+\| \widetilde{\rho}^{\mu,\lambda,\kappa}\|_{L^2}\|\nabla u\|_{L^3} ){\rm d}\tau\nonumber\\
&+\int_0^t\|\widetilde{\rho}^{{\mu,\lambda,\kappa}}\|_{L^2}(\|\widetilde{u}^{{\mu,\lambda,\kappa}}\|_{L^6}\|\nabla \rho\|_{L^3}+\|u^{{\mu,\lambda,\kappa}}\|_{L^{\infty}}\|\nabla\widetilde{\rho}^{{\mu,\lambda,\kappa}}\|_{L^2})   {\rm d}\tau \nonumber\\
\lesssim&\,\|\widetilde{\rho}^{{\mu,\lambda,\kappa}}\|_{L_t^{\infty}(H^1)}\|\nabla\widetilde{u}^{{\mu,\lambda,\kappa}}\|_{L_t^2(H^1)}\big(\|\nabla\rho^{\mu,\lambda,\kappa}\|_{L_t^{2}(H^1)} +\|\nabla\rho \|_{L_t^{2}(H^1)}\big)\nonumber\\
&+\|\widetilde{\rho}^{{\mu,\lambda,\kappa}}\|_{L_t^{\infty}(H^1)}\|\nabla\widetilde{\rho}^{\mu,\lambda,\kappa}\|_{L_t^2(L^2)}\big(\|\nabla u\|_{L_t^2(H^1)}+\|\nabla {u}^{\mu,\lambda,\kappa}\|_{L_t^2(H^1)}\big)\nonumber\\
\lesssim&\, \big(\eps_0^\frac{1}{2}+\eps_1^\frac{1}{2}\big){\Xi} ^{\mu,\lambda,\kappa}(t),\label{G4.13-1}\\
|\widetilde{I}_{14}|\lesssim&\,\int_0^t\big((1+\|\theta^{\mu,\lambda,\kappa}\|_{L^\infty})\|\widetilde{u}^{\mu,\lambda,\kappa}\|_{L^2}\|\widetilde{\rho}^{\mu,\lambda,\kappa}\|_{L^6}\|\nabla\rho^{\mu,\lambda,\kappa}\|_{L^3}\big){\rm d}\tau\nonumber\\
&+\int_{0}^t\big(\|\widetilde{u}^{\mu,\lambda,\kappa}\|_{L^2}\|\widetilde{\rho}^{\mu,\lambda,\kappa}\|_{L^6}\|a^{\mu,\lambda,\kappa}\|_{L^6}\|u^{\mu,\lambda,\kappa}\|_{L^6}+\|\widetilde{u}^{\mu,\lambda,\kappa}\|_{L^2}\|\widetilde{\rho}^{\mu,\lambda,\kappa}\|_{L^3}\|b^{\mu,\lambda,\kappa}-u^{\mu,\lambda,\kappa}\|_{L^6}\big){\rm d}\tau\nonumber\\
\lesssim&\, \big(1+\eps_1^\frac{1}{2}\big)\|\widetilde{u}^{{\mu,\lambda,\kappa}}\|_{L_t^{\infty}(H^1)}\|\nabla\widetilde{\rho}^{{\mu,\lambda,\kappa}}\|_{L_t^{2}(L^2)}\big(\|\nabla {\rho}^{{\mu,\lambda,\kappa}}\|_{L_t^{2}(H^1)}+\|\nabla u^{\mu,\lambda,\kappa}\|_{L_t^2(L^2)}\|a^{\mu,\lambda,\kappa}\|_{L_t^\infty(H^1)}\big)\nonumber\\
&+\|\widetilde{u}^{{\mu,\lambda,\kappa}}\|_{L_t^{\infty}(H^1)}\|\nabla\widetilde{\rho}^{{\mu,\lambda,\kappa}}\|_{L_t^{2}(L^2)} \|b^{\mu,\lambda,\kappa}-u^{\mu,\lambda,\kappa}\|_{L_t^2(H^1)}  \nonumber\\
\lesssim&\, \big(\eps_0^\frac{1}{2}+\eps_1^\frac{1}{2}\big){\Xi} ^{\mu,\lambda,\kappa}(t),\label{G4.13-2}\\
|\widetilde{I}_{15}|\lesssim&\, \int_0^t \|{\widetilde{u}^{\mu,\lambda,\kappa}}\|_{L^2}(\|\widetilde\theta^{\mu,\lambda,\kappa}\|_{L^6}\|\nabla\rho\|_{L^3}+\|\theta^{\mu,\lambda,\kappa}\|_{L^\infty}\|\nabla\widetilde{\rho}^{\mu,\lambda,\kappa}\|_{L^2}) {\rm d}\tau\nonumber\\
&+\int_{0}^t \|\widetilde{u}^{\mu,\lambda,\kappa}\|_{L^2}\|\rho\|_{L^6}(\|\widetilde{a}^{\mu,\lambda,\kappa}\|_{L^6}\|u\|_{L^6}+\|\widetilde{u}^{\mu,\lambda,\kappa}\|_{L^6}\|a^{\mu,\lambda,\kappa}\|_{L^6}) {\rm d}\tau \nonumber\\
&+\int_0^t \big(\|\widetilde{u}^{\mu,\lambda,\kappa}\|_{L^3}\|\widetilde{\rho}^{\mu,\lambda,\kappa}\|_{L^6}\|b^{\mu,\lambda,\kappa}-u^{\mu,\lambda,\kappa}\|_{L^2}+\|{\widetilde{u}^{\mu,\lambda,\kappa}}\|_{L^2}\|\nabla\widetilde{\rho}^{\mu,\lambda,\kappa}\|_{L^2}\|\rho \|_{L^\infty}\big){\rm d}\tau\nonumber\\
\lesssim&\, \|\widetilde{u}^{\mu,\lambda,\kappa}\|_{L_t^\infty(H^1)}\|\nabla(\widetilde{a}^{\mu,\lambda,\kappa},\widetilde{\rho}^{\mu,\lambda,\kappa},\widetilde{u}^{\mu,\lambda,\kappa},\widetilde{\theta}^{\mu,\lambda,\kappa})\|_{L_t^2(L^2)}\big(\|\nabla(\rho,u)\|_{L_t^2(H^1)}+\|\nabla \theta^{\mu,\lambda,\kappa} \|_{L_t^2(H^1)}\big)\nonumber\\
&+\|\widetilde{u}^{\mu,\lambda,\kappa}\|_{L_t^\infty(H^1)}\|\nabla \widetilde{\rho}^{\mu,\lambda,\kappa} \|_{L_t^2(L^2)}\|b^{\mu,\lambda,\kappa}-u^{\mu,\lambda,\kappa}\|_{L_t^2(L^2)}\nonumber\\
\lesssim&\, \big(\eps_0^\frac{1}{2}+\eps_1^\frac{1}{2}\big){\Xi} ^{\mu,\lambda,\kappa}(t), \label{G4.13-3}\\
|\widetilde{I}_{16}|\lesssim&\, \int_0^t\|\widetilde{u}^{\mu,\lambda,\kappa}\|_{L^3}(\|\widetilde{u}^{\mu,\lambda,\kappa}\|_{L^6}\|\nabla u\|_{L^2}+\| {u}^{\mu,\lambda,\kappa}\|_{L^6}\|\nabla \widetilde{u}^{\mu,\lambda,\kappa}\|_{L^2}){\rm d}\tau\nonumber\\
\lesssim&\, \|\widetilde{u}^{\mu,\lambda,\kappa}\|_{L_t^\infty(H^1)} \|\nabla\widetilde{u}^{\mu,\lambda,\kappa}\|_{L_t^2(L^2)}\big( \|\nabla u\|_{L_t^2(L^2)}+\|\nabla u^{\mu,\lambda,\kappa}\|_{L_t^2(L^2)}\big)\nonumber\\
\lesssim&\, \big(\eps_0^\frac{1}{2}+\eps_1^\frac{1}{2}\big){\Xi} ^{\mu,\lambda,\kappa}(t).\label{G4.13}
\end{align}
 
 Note that the terms $\widetilde{I}_{17},\dots,\widetilde{I}_{21} $ play  a pivotal role in determining the convergence rate of the limits $\mu\rightarrow0$, $\lambda\rightarrow0$ and $\kappa\rightarrow 0$.  By employing integration by parts and Young's inequality, we can effectively address the term $\widetilde{I}_{17}$ as follows
\begin{align}\label{LNXG4.18}
|\widetilde{I}_{17}|=&\,\Big| -\mu\int_0^t \nabla u^{\mu,\lambda,\kappa}\cdot\nabla\Big(\frac{\widetilde{u}^{{\mu,\lambda,\kappa}}}{1+\rho^{{\mu,\lambda,\kappa}}}\Big)  {\rm d}\tau\Big|\nonumber\\ 
\lesssim&\, \mu \int_0^t(\|\nabla\widetilde{u}^{\mu,\lambda,\kappa}\|_{ L^2}\|\nabla {u}^{\mu,\lambda,\kappa}\|_{L^2} + \|\nabla\rho^{\mu,\lambda,\kappa}\|_{L^3}\|\widetilde{u}^{\mu,\lambda,\kappa}\|_{L^2}\|\nabla u^{\mu,\lambda,\kappa}\|_{L^6}){\rm d}\tau\nonumber\\
\lesssim&\, \mu\big( \|\nabla\widetilde{u}^{\mu,\lambda,\kappa}\|_{L_t^2(L^2)}   \|\nabla {u}^{{\mu,\lambda,\kappa}}\|_{L_t^2(L^2)}+  \|\nabla\rho^{{\mu,\lambda,\kappa}}\|_{L_t^2(H^1)}\|\nabla u^{{\mu,\lambda,\kappa}}\|_{L_t^2(H^1)}\|\widetilde{u}^{{\mu,\lambda,\kappa}}\|_{L_t^{\infty}(L^2)}\big)\nonumber\\
\lesssim& \,\mu  \eps_0^\frac{1}{2}  {\Xi} ^{\mu,\lambda,\kappa}(t)^{\frac{1}{2}}+\mu \big(\eps_0^\frac{1}{2} +\eps_1^\frac{1}{2}\big)  {\Xi} ^{\mu,\lambda,\kappa}(t)^{\frac{1}{2}}\nonumber\\
\lesssim&\,\big(\eps_0^\frac{1}{2}+\eps_1^\frac{1}{2}\big){\Xi} ^{\mu,\lambda,\kappa}(t)+ \mu^2.  
\end{align}
Similarly, we have
\begin{align}\label{LNXG4.14}
|\widetilde{I}_{18}|+|\widetilde{I}_{21}|
\lesssim&\,\big(\eps_0^\frac{1}{2}+\eps_1^\frac{1}{2}\big){\Xi} ^{\mu,\lambda,\kappa}(t)+ (\mu+\lambda)^2+\kappa^2,
\end{align}
and
\begin{align}\label{LNXG4.20}
|\widetilde{I}_{19}|+|\widetilde{I}_{20}|\lesssim&\,  \int_0^t \|\widetilde{u}^{\mu,\lambda,\kappa}\|_{L^2}(\mu\|\nabla u^{\mu,\lambda,\kappa}\|_{L^3}\|\nabla u^{\mu,\lambda,\kappa}\|_{L^6}+\lambda\|{\rm div} u^{\mu,\lambda,\kappa}\|_{L^3}\|{\rm div} u^{\mu,\lambda,\kappa}\|_{L^6}){\rm d}\tau\nonumber\\ 
\lesssim&\,\mu\|\nabla u^{\mu,\lambda,\kappa}\|_{L_t^2(H^1)}^2\|\widetilde{u}^{\mu,\lambda,\kappa}\|_{L_t^\infty(L^2)}+\lambda\| {\rm div}u^{\mu,\lambda,\kappa}\|_{L_t^2(H^1)}^2\|\widetilde{u}^{\mu,\lambda,\kappa}\|_{L_t^\infty(L^2)} \nonumber\\
\lesssim&\, (\mu+\lambda)\eps_0{\Xi} ^{\mu,\lambda,\kappa}(t)^{\frac{1}{2}} \nonumber\\
\lesssim&\, \big(\eps_0^\frac{1}{2}+\eps_1^\frac{1}{2}\big){\Xi} ^{\mu,\lambda,\kappa}(t)+ (\mu+\lambda)^2.
\end{align}

For the remaining terms $\widetilde{I}_{22}, \widetilde{I}_{23}$, and $\widetilde{I}_{24}$, 
it follows from Lemma \ref{LA.1},  H\"{o}lder's inequality, the definition of ${\Xi} ^{\mu,\lambda,\kappa}(t)$, and the inequalities \eqref{G4.5} and \eqref{G4.6} that
\begin{align}
|\widetilde{I}_{22}|\lesssim&\, \int_0^t \|\widetilde{\theta}^{\mu,\lambda,\kappa}\|_{L^3}\big(\|\widetilde{\theta}^{\mu,\lambda,\kappa}\|_{L^6}\|\nabla u^{\mu,\lambda,\kappa}\|_{L^2}+ \|\nabla(\widetilde{u}^{\mu,\lambda,\kappa},\widetilde{\theta}^{\mu,\lambda,\kappa})\|_{L^2}\|(u^{\mu,\lambda,\kappa},\theta^{\mu,\lambda,\kappa})\|_{L^6}\big){\rm d}\tau   \nonumber\\
&+\int_0^t \|\widetilde{\theta}^{\mu,\lambda,\kappa}\|_{L^3}\|\widetilde{u}^{\mu,\lambda,\kappa}\|_{L^6}\|\nabla\theta\|_{L^2}{\rm d}\tau\nonumber\\
\lesssim&\, \|\widetilde{\theta}^{\mu,\lambda,\kappa}\|_{L_t^\infty(H^1)} \|\nabla(\widetilde{u}^{\mu,\lambda,\kappa},\widetilde{\theta}^{\mu,\lambda,\kappa})\|_{L_t^2(L^2)}\big(\|\nabla(u^{\mu,\lambda,\kappa},\theta^{\mu,\lambda,\kappa})\|_{L_t^2(L^2)}+\|\nabla \theta \|_{L_t^2(L^2)}\big)\nonumber\\
\lesssim&\,\big(\eps_0^\frac{1}{2}+\eps_1^\frac{1}{2}\big){\Xi} ^{\mu,\lambda,\kappa}(t),\label{G4.15-1}\\
|\widetilde{I}_{23}|\lesssim&\, \int_0^t \|\widetilde{\theta}^{\mu,\lambda,\kappa}\|_{L^3}\big(\|{u}\|_{L^6}\|\widetilde{b}^{\mu,\lambda,\kappa}-\widetilde{u}^{\mu,\lambda,\kappa}\|_{L^2}+\|\widetilde{u}^{\mu,\lambda,\kappa}\|_{L^6}\| {b}^{\mu,\lambda,\kappa}- {u}^{\mu,\lambda,\kappa} \|_{L^2}\big){\rm d}\tau\nonumber\\
&+\int_0^t \|\widetilde{\theta}^{\mu,\lambda,\kappa}\|_{L^2}\big(\|\widetilde{\theta}^{\mu,\lambda,\kappa}\|_{L^6}\|a^{\mu,\lambda,\kappa}\|_{L^6}(\|\theta\|_{L^6}+\|\theta^{\mu,\lambda,\kappa}\|_{L^6})+\|  \widetilde{a}^{\mu,\lambda,\kappa}\|_{L^6}\| \theta\|_{L^6}^2\big){\rm d}\tau\nonumber\\
&+ \int_0^t  \big(\|\widetilde{\theta}^{\mu,\lambda,\kappa}\|_{L^2}\|\widetilde{b}^{\mu,\lambda,\kappa}\|_{L^6}\|u\|_{L^6}\|\rho\|_{L^6}+\|\widetilde{\theta}^{\mu,\lambda,\kappa}\|_{L^3}\|\rho\|_{L^6} \|\sqrt{2}\widetilde{\omega}^{\mu,\lambda,\kappa}-\sqrt{3}\widetilde{\theta}^{\mu,\lambda,\kappa}\|_{L^2}\big)  {\rm d}\tau \nonumber\\
&+\int_0^t\|\widetilde{\theta}^{\mu,\lambda,\kappa}\|_{L^2}\|\rho\|_{L^6}\big(\|b^{\mu,\lambda,\kappa}\|_{L^6}\|\widetilde{u}^{\mu,\lambda,\kappa}\|_{L^6}+\|\theta^{\mu,\lambda,\kappa}\|_{L^6}\|\widetilde{a}^{\mu,\lambda,\kappa}\|_{L^6}+\|a^{\mu,\lambda,\kappa}\|_{L^6}\|\widetilde{\theta}^{\mu,\lambda,\kappa}\|_{L^6}\big){\rm d}\tau \nonumber\\
\lesssim&\, \|\widetilde{\theta}^{\mu,\lambda,\kappa}\|_{L_t^\infty(H^1)}\big(\|\nabla u\|_{L_t^2(L^2)}\|\widetilde{b}^{\mu,\lambda,\kappa}-\widetilde{u}^{\mu,\lambda,\kappa}\|_{L_t^2(L^2)}+\|\nabla\widetilde{u}^{\mu,\lambda,\kappa}\|_{L_t^2(L^2)}\|b^{\mu,\lambda,\kappa}-u^{\mu,\lambda,\kappa}\|_{L_t^2(L^2)}\big) \nonumber\\
&+\big(\eps_0^\frac{1}{2}+\eps_1^\frac{1}{2}\big)\|\widetilde{\theta}^{\mu,\lambda,\kappa}\|_{L_t^\infty(H^1)} \|\nabla(\widetilde{a}^{\mu,\lambda,\kappa},\widetilde{b}^{\mu,\lambda,\kappa},\widetilde{u}^{\mu,\lambda,\kappa},\widetilde{\theta}^{\mu,\lambda,\kappa})\|_{L_t^2(L^2)}\|\nabla (a^{\mu,\lambda,\kappa},b^{\mu,\lambda,\kappa})\|_{L_t^2(L^2)}\nonumber\\
&+\big(\eps_0^\frac{1}{2}+\eps_1^\frac{1}{2}\big)\|\widetilde{\theta}^{\mu,\lambda,\kappa}\|_{L_t^\infty(H^1)} \|\nabla(\widetilde{a}^{\mu,\lambda,\kappa},\widetilde{b}^{\mu,\lambda,\kappa},\widetilde{u}^{\mu,\lambda,\kappa},\widetilde{\theta}^{\mu,\lambda,\kappa})\|_{L_t^2(L^2)}\|\nabla (\rho,u,\theta)\|_{L_t^2(L^2)} \nonumber\\
&+\|\widetilde{\theta}^{\mu,\lambda,\kappa}\|_{L_t^\infty(H^1)}\|\nabla\rho\|_{L_t^2(L^2)}\|\sqrt{2}\widetilde\omega^{\mu,\lambda,\kappa}-\sqrt{3}\widetilde{\theta}^{\mu,\lambda,\kappa}\|_{L_t^2(L^2)}\nonumber\\
\lesssim&\,\big(\eps_0^\frac{1}{2}+\eps_1^\frac{1}{2}\big){\Xi} ^{\mu,\lambda,\kappa}(t),\label{G4.15-2}\\
|\widetilde{I}_{24}|\lesssim&\,  \int_0^t\big(\|\widetilde{\theta}^{\mu,\lambda,\kappa}\|_{L^3}\|\widetilde{\rho}^{\mu,\lambda,\kappa}\|_{L^6}\|\sqrt{2}\omega-\sqrt{3}\theta\|_{L^2}+\|\widetilde{\theta}^{\mu,\lambda,\kappa}\|_{L^2}\|\widetilde{\rho}^{\mu,\lambda,\kappa}\|_{L^6}\|\theta^{\mu,\lambda,\kappa}\|_{L^6}^2\|a^{\mu,\lambda,\kappa}\|_{L^\infty}\big){\rm d}\tau\nonumber\\
&+\|\widetilde{\theta}^{\mu,\lambda,\kappa}\|_{L^2}\|\widetilde{\rho}^{\mu,\lambda,\kappa}\|_{L^6}\|(a^{\mu,\lambda,\kappa},b^{\mu,\lambda,\kappa},u^{\mu,\lambda,\kappa},\theta^{\mu,\lambda,\kappa})\|_{L^6}^2\nonumber\\
\lesssim&\,\big(1+\eps_0^\frac{1}{2}+\eps_1^\frac{1}{2}\big)\|\widetilde{\theta}^{\mu,\lambda,\kappa}\|_{L_t^\infty(H^1)}\|\nabla\widetilde{\rho}^{\mu,\lambda,\kappa}\|_{L_t^2(L^2)}\|\sqrt{2}\omega-\sqrt{3}\theta\|_{L_t^2(L^2)}\nonumber\\
&+\big(1+\eps_0^\frac{1}{2}+\eps_1^\frac{1}{2}\big)\|\widetilde{\theta}^{\mu,\lambda,\kappa}\|_{L_t^\infty(H^1)}\|\nabla\widetilde{\rho}^{\mu,\lambda,\kappa}\|_{L_t^2(L^2)}\|\nabla(a^{\mu,\lambda,\kappa},b^{\mu,\lambda,\kappa},u^{\mu,\lambda,\kappa},\theta^{\mu,\lambda,\kappa})\|_{L_t^2(L^2)}\nonumber\\
\lesssim&\, \big(\eps_0^\frac{1}{2}+\eps_1^\frac{1}{2}\big){\Xi} ^{\mu,\lambda,\kappa}(t).\label{G4.15}
\end{align} 
By substituting the estimates   \eqref{G4.9}--\eqref{G4.15} into \eqref{G4.8}, we derive the desired estimate 
\eqref{G4.7}.
\end{proof}

 We  continue to handle the first-order spatial derivatives of $(\widetilde{\rho}^{\mu,\lambda,\kappa},\widetilde{u}^{\mu,\lambda,\kappa},\widetilde{\theta}^{\mu,\lambda,\kappa},\widetilde{f}^{\mu,\lambda,\kappa})$. Similar to Lemma \ref{L3.2}, in order to capture the symmetric structure of the equations for $ \widetilde{\rho}^{\mu,\lambda,\kappa}$, $\widetilde{u}^{\mu,\lambda,\kappa}$ and $\widetilde{\theta}^{\mu,\lambda,\kappa} $, we rewrite the error equations \eqref{G4.3} as  
 
\begin{equation}\label{LNXG4.24}
\left\{
\begin{aligned}
&\partial_t \widetilde{\rho}^{\mu,\lambda,\kappa}+(1+\rho^{\mu,\lambda,\kappa}){\rm div}\widetilde{u}^{\mu,\lambda,\kappa} =\widetilde{G}_1,\\
&\partial_{t}\widetilde{u}^{\mu,\lambda,\kappa}+\frac{1+\theta^{\mu,\lambda,\kappa}}{1+\rho^{\mu,\lambda,\kappa}}\nabla \widetilde{\rho}^{\mu,\lambda,\kappa}+\nabla\widetilde{\theta}^{\mu,\lambda,\kappa}-(\widetilde{b}^{\mu,\lambda,\kappa}-\widetilde{u}^{\mu,\lambda,\kappa}) \\
&\quad=\frac{\mu \Delta u^{\mu,\lambda,\kappa}}{1+\rho^{\mu,\lambda,\kappa}}+\frac{(\mu+\lambda) \nabla {\rm div}u^{\mu,\lambda,\kappa}}{1+\rho^{\mu,\lambda,\kappa}}+\frac{2\mu |D(u^{\mu,\lambda,\kappa})|^2}{1+\rho^{\mu,\lambda,\kappa}}+\frac{ \lambda |{\rm div}u^{\mu,\lambda,\kappa}|^2}{1+\rho^{\mu,\lambda,\kappa}} \\
&\quad\quad-a^{\mu,\lambda,\kappa}\widetilde{u}^{\mu,\lambda,\kappa}-\widetilde{a}^{\mu,\lambda,\kappa} u+ \widetilde{G}_2+\widetilde{G}_3+\widetilde{G}_4,\\
&\partial_t\widetilde{\theta}^{{\mu,\lambda,\kappa}}+(1+\theta^{\mu,\lambda,\kappa}){\rm div}\widetilde{u}^{{\mu,\lambda,\kappa}}-\sqrt{3}(\sqrt{2}\widetilde{\omega}^{{\mu,\lambda,\kappa}}-\sqrt{3}\widetilde{\theta}^{{\mu,\lambda,\kappa}}) \\
&\quad=\frac{\kappa\Delta\theta^{{\mu,\lambda,\kappa}}}{1+\rho^{{\mu,\lambda,\kappa}}}-b^{\mu,\lambda,\kappa}\cdot \widetilde{u}^{\mu,\lambda,\kappa}-\widetilde{b}^{\mu,\lambda,\kappa}\cdot u-3\widetilde{a}^{\mu,\lambda,\kappa}\theta-3a^{\mu,\lambda,\kappa}\widetilde{\theta}^{\mu,\lambda,\kappa}+\widetilde{G}_5+\widetilde{G}_6+\widetilde{G}_7,\\
&\partial_{t}\widetilde{f}^{\mu,\lambda,\kappa}+v\cdot\nabla \widetilde{f}^{\mu,\lambda,\kappa} + {u}\cdot\nabla_v\widetilde{f}^{\mu,\lambda,\kappa}+\widetilde{u}^{\mu,\lambda,\kappa}\cdot\nabla_v f-\widetilde{u}^{\mu,\lambda,\kappa}\cdot v\sqrt{M}-(|v|^2-3)\sqrt{M}\widetilde\theta^{\mu,\lambda,\kappa} \\
&\quad= \mathcal{L}\widetilde{f}^{\mu,\lambda,\kappa}+\frac{1}{2}u\cdot v\widetilde{f}^{\mu,\lambda,\kappa}+\frac{1}{2}\widetilde{u}^{\mu,\lambda,\kappa}\cdot vf^{\mu,\lambda,\kappa}+\frac{\theta}{\sqrt{M}}\Delta_v(\sqrt{M}\widetilde{f}^{\mu,\lambda,\kappa})+\frac{\widetilde\theta^{\mu,\lambda,\kappa}}{\sqrt{M}}\Delta_v(\sqrt{M} {f} ) ,
\end{aligned}\right.
\end{equation} 
where the terms $\widetilde{G}_i(i=1,\dots,7)$  are given by
\begin{equation*}
\begin{aligned}
\widetilde{G}_1:=&\, -\widetilde{\rho}^{{\mu,\lambda,\kappa}}{\rm div}u-\widetilde{u}^{{\mu,\lambda,\kappa}}\cdot\nabla\rho-u^{\mu,\lambda,\kappa}\cdot\nabla\widetilde{\rho}^{\mu,\lambda,\kappa}  ,\\
\widetilde{G}_2:=&\,\Big(\frac{\rho^{\mu,\lambda,\kappa}}{1+\rho^{\mu,\lambda,\kappa}}-\frac{\rho }{1+\rho }\Big) \nabla\rho^{{\mu,\lambda,\kappa}}-\Big(\frac{1}{1+\rho^{\mu,\lambda,\kappa}}-\frac{1}{1+\rho}\Big)\theta^{{\mu,\lambda,\kappa}} \nabla\rho^{{\mu,\lambda,\kappa}} \\
&-\Big(\frac{\rho^{\mu,\lambda,\kappa}}{1+\rho^{\mu,\lambda,\kappa}}-\frac{\rho }{1+\rho }\Big) ( b^{\mu,\lambda,\kappa}-u^{\mu,\lambda,\kappa}-a^{\mu,\lambda,\kappa} u^{\mu,\lambda,\kappa}),\nonumber\\
\widetilde{G}_3:=&\,-\frac{1}{1+\rho} \widetilde{\theta}^{{\mu,\lambda,\kappa}}\nabla\rho +\frac{\rho}{1+\rho}(\widetilde{a}^{\mu,\lambda,\kappa} u+a^{\mu,\lambda,\kappa}\widetilde{u}^{\mu,\lambda,\kappa}) \frac{\rho}{1+\rho}\nabla\widetilde{\rho}^{{\mu,\lambda,\kappa}}-\frac{\rho}{1+\rho} (\widetilde{b}^{\mu,\lambda,\kappa}-\widetilde{u}^{\mu,\lambda,\kappa})  ,\\
\widetilde{G}_4:=&-\widetilde{u}^{\mu,\lambda,\kappa}\cdot\nabla u-u^{\mu,\lambda,\kappa}\cdot\nabla\widetilde{u}^{\mu,\lambda,\kappa},\nonumber\\
\widetilde{G}_5:=&\,-\widetilde{\theta}^{{\mu,\lambda,\kappa}}{\rm div}u^{{\mu,\lambda,\kappa}} -\widetilde{u}^{{\mu,\lambda,\kappa}}\cdot\nabla\theta-{u}^{{\mu,\lambda,\kappa}}\cdot\nabla\widetilde{\theta}^{{\mu,\lambda,\kappa}},\\
\widetilde{G}_6:=&\,\frac{1}{1+\rho}\big((u^{\mu,\lambda,\kappa}-b^{\mu,\lambda,\kappa})\cdot\widetilde{u}^{\mu,\lambda,\kappa}+u\cdot(\widetilde{u}^{\mu,\lambda,\kappa}-\widetilde{b}^{\mu,\lambda,\kappa}) \big) \nonumber\\
&+\frac{1}{1+\rho}\big( \widetilde{a}^{\mu,\lambda,\kappa}\theta^2+a^{\mu,\lambda,\kappa}(\theta\widetilde{\theta}^{\mu,\lambda,\kappa}+\theta^{\mu,\lambda,\kappa}\widetilde{\theta}^{\mu,\lambda,\kappa})\big) +\sqrt{3}\frac{\rho}{1+\rho}(\sqrt{2}\widetilde{\omega}^{\mu,\lambda,\kappa}-\sqrt{3}\widetilde\theta^{\mu,\lambda,\kappa})\nonumber\\
&+\frac{\rho}{1+\rho}(\widetilde{b}^{\mu,\lambda,\kappa}\cdot u+b^{\mu,\lambda,\kappa}\cdot \widetilde{u}^{\mu,\lambda,\kappa}+3\widetilde{a}^{\mu,\lambda,\kappa}\theta+3 {a}^{\mu,\lambda,\kappa}\widetilde{\theta}^{\mu,\lambda,\kappa}),\nonumber\\
\widetilde{G}_7:=&\,-\sqrt{3}\Big(\frac{\rho^{\mu,\lambda,\kappa}}{1+\rho^{\mu,\lambda,\kappa}}-\frac{\rho}{1+\rho}\Big)(\sqrt{2}\omega^{\mu,\lambda,\kappa}-\sqrt{3}\theta^{\mu,\lambda,\kappa})+\Big(\frac{1}{1+\rho^{\mu,\lambda,\kappa}}-\frac{1}{1+\rho}\Big) a^{\mu,\lambda,\kappa}(\theta^{\mu,\lambda,\kappa})^2\nonumber\\
&+\Big(\frac{1}{1+\rho^{\mu,\lambda,\kappa}}-\frac{1}{1+\rho}\Big)\big(| {u}^{\mu,\lambda,\kappa}|^2-2b^{\mu,\lambda,\kappa}\cdot u^{\mu,\lambda,\kappa}-3a^{\mu,\lambda,\kappa}\theta^{\mu,\lambda,\kappa} \big).
\end{aligned}
\end{equation*}

We have 

\begin{lem}\label{L4.2}
It holds that
\begin{align}\label{G4.16}
&\sup_{\tau\in[0,t]}\big\{\|\nabla(\widetilde{\rho}^{\mu,\lambda,\kappa},\widetilde{u}^{\mu,\lambda,\kappa},
\widetilde{\theta}^{\mu,\lambda,\kappa})(\tau)\|_{L^2}^2
+\|\nabla\widetilde{f}^{\mu,\lambda,\kappa}(\tau)\|_{L_{x,v}^2}^2\big\}  \nonumber\\
&\quad +\int_{0}^t\big( \|\nabla(\widetilde{b}^{\mu,\lambda,\kappa}-\widetilde{u}^{\mu,\lambda,\kappa})(\tau)\|_{L^2}^2 +\|\nabla(\sqrt{2}\widetilde{\omega}^{\mu,\lambda,\kappa}-\sqrt{3}\widetilde{\theta}^{\mu,\lambda,\kappa})(\tau)\|_{L^2}^2+\|  \nabla\{\mathbf{I}-\mathbf{P}\}\widetilde{f}^{\mu,\lambda,\kappa}(\tau)\|_{\nu}^2 \big) {\rm d}\tau \nonumber\\
&\quad\quad \leq {\Xi} ^{\mu,\lambda,\kappa}(0)+C\max\{\mu,\lambda,\kappa\}^2+C \big(\eps_0^\frac{1}{2}  +\eps_1^\frac{1}{2}\big){\Xi} ^{\mu,\lambda,\kappa}(t),
\end{align}
where $C>0$ is a constant independent of $\mu$, $\lambda$,   $\kappa$ and   $t$.
\end{lem}
\begin{proof}
Applying the operator $\nabla$ to the equations in \eqref{LNXG4.24}, and using \eqref{G2.6} and the energy method, one gets 
\begin{align}\label{G4.17}
&\frac{1}{2}   \Big(\Big\|\frac{ \nabla\widetilde\rho^{\mu,\lambda,\kappa}}{1+\rho^{\mu,\lambda,\kappa}}(t)\Big\|_{L^2}^2+\Big\|\frac{\nabla \widetilde u^{\mu,\lambda,\kappa}}{\sqrt{1+\theta^{\mu,\lambda,\kappa}}}(t)\Big\|_{L^2}^{2}+\Big\|\frac{\nabla\widetilde\theta^{\mu,\lambda,\kappa}}{1+\theta^{\mu,\lambda,\kappa}}(t)\Big\|_{L^2}^2 +\|\nabla \widetilde{f}^{\mu,\lambda,\kappa}(t)\|_{L_{x,v}^{2}}^2\Big)  \nonumber\\
&+\int_0^t \big(  \bar\lambda\| \nabla\{\mathbf{I}-\mathbf{P}\}\widetilde{f}^{\mu,\lambda,\kappa}(\tau)\|_{\nu}^2 +\|\nabla(\widetilde{b}^{\mu,\lambda,\kappa}-\widetilde{u}^{\mu,\lambda,\kappa})(\tau)\|_{L^2}^2  \big){\rm d}\tau\nonumber\\
&+\int_0^t \|\nabla(\sqrt{2}\widetilde{\omega}^{\mu,\lambda,\kappa}-\sqrt{3}\widetilde{\theta}^{\mu,\lambda,\kappa})(\tau)\|_{L^2}^2  {\rm d}\tau\nonumber\\
\leq\,&\frac{1}{2}\int_0^t  \int_{\mathbb{R}^3 }\langle \nabla( u\cdot v\widetilde{f}^{\mu,\lambda,\kappa}),\nabla\widetilde{f}^{\mu,\lambda,\kappa}\rangle {\rm d}x{\rm d}\tau-\int_0^t  \int_{\mathbb{R}^3 }\langle \nabla( u\cdot \nabla_v\widetilde{f}^{\mu,\lambda,\kappa}),\nabla\widetilde{f}^{\mu,\lambda,\kappa}\rangle {\rm d}x{\rm d}\tau\nonumber\\
&-\int_0^t \nabla(\widetilde{b}^{\mu,\lambda,\kappa}\cdot {u} )\cdot\nabla\widetilde{\theta}^{\mu,\lambda,\kappa}{\rm d}\tau       -\int_0^t \nabla(\widetilde{a}^{\mu,\lambda,\kappa} u)\cdot\nabla\widetilde{u}^{\mu,\lambda,\kappa}{\rm d}{\tau}  \nonumber\\
&+\frac{1}{2} \int_0^t  \int_{\mathbb{R}^3 }\langle\nabla(\widetilde{u}^{\mu,\lambda,\kappa}\cdot v{f}^{\mu,\lambda,\kappa}),\nabla\widetilde{f}^{\mu,\lambda,\kappa}\rangle {\rm d}x{\rm d}\tau -\int_0^t  \int_{\mathbb{R}^3 }\langle\nabla(\widetilde{u}^{\mu,\lambda,\kappa}\cdot \nabla_v{f}^{\mu,\lambda,\kappa}),\nabla\widetilde{f}^{\mu,\lambda,\kappa}\rangle {\rm d}x{\rm d}\tau    \nonumber\\
&-\int_0^t \nabla({b}^{\mu,\lambda,\kappa}\cdot  \widetilde{u}^{\mu,\lambda,\kappa} )\cdot\nabla\widetilde{\theta}^{\mu,\lambda,\kappa}{\rm d}\tau-\int_0^t\nabla({a}^{\mu,\lambda,\kappa}  \widetilde{u}^{\mu,\lambda,\kappa})\cdot\nabla\widetilde{u}^{\mu,\lambda,\kappa} {\rm d}{\tau}            \nonumber\\
&+\int_0^t  \int_{\mathbb{R}^3}\Big\langle\nabla\Big(\frac{\widetilde\theta^{\mu,\lambda,\kappa}}{\sqrt{M}}\Delta_v(\sqrt{M}f)\Big),\nabla\widetilde{f}^{\mu,\lambda,\kappa}\Big\rangle{\rm d}x-3\int_0^t\!\int_{\mathbb{R}^3}\nabla(a^{\mu,\lambda,\kappa}\widetilde\theta^{\mu,\lambda,\kappa})\cdot \nabla\widetilde\theta^{\mu,\lambda,\kappa}{\rm d}x{\rm d}
\tau\nonumber\\
&+\int_0^t \int_{\mathbb{R}^3}\Big\langle\nabla\Big(\frac{\theta}{\sqrt{M}}\Delta_v(\sqrt{M}\widetilde{f}^{\mu,\lambda,\kappa})\Big),\nabla\widetilde{f}^{\mu,\lambda,\kappa}\Big\rangle{\rm d}x-3\int_0^t\!\int_{\mathbb{R}^3}\nabla(\widetilde{a}^{\mu,\lambda,\kappa} \theta)\cdot\nabla\widetilde{\theta}^{\mu,\lambda,\kappa} {\rm d}x{\rm d}
\tau\nonumber\\
&+\int_0^t\!\int_{\mathbb{R}^3} \nabla\Big( \frac{1}{1+\rho^{\mu,\lambda,\kappa}}\Big)\widetilde{\rho}^{{\mu,\lambda,\kappa}}\cdot \nabla \widetilde{u}^{\mu,\lambda,\kappa}{\rm d}x{\rm d}
\tau +\int_0^t\!\int_{\mathbb{R}^3} \nabla\Big( \frac{1}{1+\theta^{\mu,\lambda,\kappa}}\Big)\widetilde{\theta}^{{\mu,\lambda,\kappa}}\cdot \nabla \widetilde{u}^{\mu,\lambda,\kappa}{\rm d}x{\rm d}
\tau\nonumber\\
&+ \int_0^t\!\int_{\mathbb{R}^3}   \frac{\theta^{\mu,\lambda,\kappa}}{1+\theta^{\mu,\lambda,\kappa}} \nabla(\widetilde{b}^{{\mu,\lambda,\kappa}}-\widetilde{u}^{\mu,\lambda,\kappa}-a^{\mu,\lambda,\kappa}\widetilde{u}^{\mu,\lambda,\kappa}+\widetilde{a}^{\mu,\lambda,\kappa}u)\cdot \nabla \widetilde{u}^{\mu,\lambda,\kappa}{\rm d}x{\rm d}
\tau \nonumber\\
&+ \int_0^t\!\int_{\mathbb{R}^3}   \frac{{(\theta^{\mu,\lambda,\kappa})}^2+2\theta^{\mu,\lambda,\kappa}}{(1+\theta^{\mu,\lambda,\kappa})^2} \nabla(\sqrt{2}\widetilde{\omega}^{{\mu,\lambda,\kappa}}-\sqrt{3}\widetilde{\theta}^{\mu,\lambda,\kappa})\cdot \nabla \widetilde{\theta}^{\mu,\lambda,\kappa}{\rm d}x{\rm d}
\tau \nonumber\\
&+ \int_0^t\!\int_{\mathbb{R}^3}   \frac{{(\theta^{\mu,\lambda,\kappa})}^2+2\theta^{\mu,\lambda,\kappa}}{(1+\theta^{\mu,\lambda,\kappa})^2}\nabla(b^{\mu,\lambda,\kappa}\cdot\widetilde{u}^{\mu,\lambda,\kappa}+\widetilde{b}^{\mu,\lambda,\kappa}\cdot u+3\widetilde{a}^{\mu,\lambda,\kappa}\theta+3a^{\mu,\lambda,\kappa}\widetilde{\theta}^{\mu,\lambda,\kappa})\cdot\nabla\theta^{\mu,\lambda,\kappa}  {\rm d}x{\rm d}
\tau \nonumber\\
&+\int_0^t\!\int_{\mathbb{R}^3}\frac{1}{(1+\rho^{\mu,\lambda,\kappa})^2} \nabla\widetilde{\rho}^{{\mu,\lambda,\kappa}}\cdot\nabla\widetilde{G}_1 {\rm d}x{\rm d}\tau+\int_0^t\!\int_{\mathbb{R}^3}\frac{1}{(1+\theta^{\mu,\lambda,\kappa}) }  \nabla\widetilde{u}^{{\mu,\lambda,\kappa}}\cdot\nabla\widetilde{G}_2 {\rm d}x{\rm d}\tau\nonumber\\
&+\int_0^t\!\int_{\mathbb{R}^3}\frac{1}{(1+\theta^{\mu,\lambda,\kappa}) } \nabla\widetilde{u}^{{\mu,\lambda,\kappa}}\cdot\nabla\widetilde{G}_3 {\rm d}x{\rm d}\tau+\int_0^t\!\int_{\mathbb{R}^3} \frac{1}{(1+\theta^{\mu,\lambda,\kappa}) }\nabla\widetilde{u}^{{\mu,\lambda,\kappa}}\cdot\nabla\widetilde{G}_4 {\rm d}x{\rm d}\tau\nonumber\\
&+\mu\int_0^t\!\int_{\mathbb{R}^3}\frac{1}{(1+\theta^{\mu,\lambda,\kappa}) } \nabla\Big(\frac{\Delta u^{\mu,\lambda,\kappa}  }{1+\rho^{\mu,\lambda,\kappa}} \Big)\cdot\nabla\widetilde{u}^{\mu,\lambda,\kappa}{\rm d}x{\rm d}\tau\nonumber\\
&+(\mu+\lambda)\int_0^t\!\int_{\mathbb{R}^3} \frac{1}{(1+\theta^{\mu,\lambda,\kappa}) }\nabla\Big(\frac{\nabla {\rm div} u^{\mu,\lambda,\kappa}  }{1+\rho^{\mu,\lambda,\kappa}} \Big)\cdot\nabla\widetilde{u}^{\mu,\lambda,\kappa}{\rm d}x{\rm d}\tau\nonumber\\
&+2\mu\int_0^t\!\int_{\mathbb{R}^3} \frac{1}{(1+\theta^{\mu,\lambda,\kappa}) }\nabla\Big(\frac{  
|D( u^{\mu,\lambda,\kappa} )|^2 }{1+\rho^{\mu,\lambda,\kappa}} \Big)\cdot\nabla\widetilde{u}^{\mu,\lambda,\kappa}{\rm d}x{\rm d}\tau\nonumber\\
&+\lambda\int_0^t\!\int_{\mathbb{R}^3} \frac{1}{(1+\theta^{\mu,\lambda,\kappa}) }\nabla\Big(\frac{  
|{\rm div} u^{\mu,\lambda,\kappa}  |^2 }{1+\rho^{\mu,\lambda,\kappa}} \Big)\cdot\nabla\widetilde{u}^{\mu,\lambda,\kappa}{\rm d}x{\rm d}\tau \nonumber\\
&+\kappa\int_0^t\!\int_{\mathbb{R}^3}\frac{1}{(1+\theta^{\mu,\lambda,\kappa})^2 } \nabla\Big(\frac{\Delta \theta^{\mu,\lambda,\kappa}  }{1+\rho^{\mu,\lambda,\kappa}} \Big)\cdot\nabla\widetilde{\theta}^{\mu,\lambda,\kappa}{\rm d}x{\rm d}\tau\nonumber\\
&+ \int_0^t\!\int_{\mathbb{R}^3} \frac{1}{(1+\theta^{\mu,\lambda,\kappa})^2 } \nabla\widetilde{\theta}^{{\mu,\lambda,\kappa}} \cdot\nabla\widetilde{G}_5 {\rm d}x{\rm d}\tau+ \int_0^t\!\int_{\mathbb{R}^3} \frac{1}{(1+\theta^{\mu,\lambda,\kappa})^2 } \nabla\widetilde{\theta}^{{\mu,\lambda,\kappa}} \cdot\nabla\widetilde{G}_6 {\rm d}x{\rm d}\tau\nonumber\\
&+ \int_0^t\!\int_{\mathbb{R}^3} \frac{1}{(1+\theta^{\mu,\lambda,\kappa})^2 } \nabla\widetilde{\theta}^{{\mu,\lambda,\kappa}} \cdot\nabla\widetilde{G}_7 {\rm d}x{\rm d}\tau  +\int_0^t\!\int_{\mathbb{R}^3} \frac{1}{2}  \partial_t \Big(\frac{1}{(1+\rho^{\mu,\lambda,\kappa})^2} \Big) |\nabla\widetilde\rho^{\mu,\lambda,\kappa}|^2 {\rm d}x{\rm d}\tau\nonumber\\
&+ \int_0^t\!\int_{\mathbb{R}^3} \frac{1}{2}  \partial_t \Big(\frac{1}{(1+\theta^{\mu,\lambda,\kappa}) } \Big) |\nabla\widetilde u^{\mu,\lambda,\kappa}|^2 {\rm d}x{\rm d}\tau+ \int_0^t\!\int_{\mathbb{R}^3} \frac{1}{2}  \partial_t \Big(\frac{1}{(1+\theta^{\mu,\lambda,\kappa})^2 } \Big) |\nabla\widetilde \theta^{\mu,\lambda,\kappa}|^2 {\rm d}x{\rm d}\tau\nonumber\\ 
&+\frac{1}{2}   \bigg(\Big\|\frac{ \nabla\widetilde\rho^{\mu,\lambda,\kappa}}{1+\rho^{\mu,\lambda,\kappa}}(0)\Big\|_{L^2}^2+\Big\|\frac{\nabla \widetilde u^{\mu,\lambda,\kappa}}{\sqrt{1+\theta^{\mu,\lambda,\kappa}}}(0)\Big\|_{L^2}^{2}+\Big\|\frac{\nabla\widetilde\theta^{\mu,\lambda,\kappa}}{1+\theta^{\mu,\lambda,\kappa}}(0)\Big\|_{L^2}^2 +\|\nabla \widetilde f^{\mu,\lambda,\kappa}(0)\|_{L_{x,v}^{2}}^2\bigg)\nonumber\\
\equiv:&\sum_{i=1}^{32}\widetilde{J}_i+\frac{1}{2}   \Big(\Big\|\frac{ \nabla\widetilde\rho^{\mu,\lambda,\kappa}}{1+\rho^{\mu,\lambda,\kappa}}(0)\Big\|_{L^2}^2+\Big\|\frac{\nabla \widetilde u^{\mu,\lambda,\kappa}}{\sqrt{1+\theta^{\mu,\lambda,\kappa}}}(0)\Big\|_{L^2}^{2}+\Big\|\frac{\nabla\widetilde\theta^{\mu,\lambda,\kappa}}{1+\theta^{\mu,\lambda,\kappa}}(0)\Big\|_{L^2}^2 +\|\nabla \widetilde f^{\mu,\lambda,\kappa}(0)\|_{L_{x,v}^{2}}^2\Big) .
\end{align} 

Now we deal with  the terms $\widetilde{J}_i (i=1,\dots, 32)$ in turn. 
First, for the terms $\widetilde{J}_1, \widetilde{J}_2,\widetilde{J}_3,\widetilde{J}_4,\widetilde{J}_9$, 
and $\widetilde{J}_{10}$, 
thanks to Lemmas \ref{LA.1} and \ref{LA.2} and the macro-micro decomposition  \eqref{G2.4}, we have
\begin{align}
|\widetilde{J}_1|\lesssim&\,\int_0^t (\|\widetilde{f}^{\mu,\lambda,\kappa}\|_{L_v^2(L^3)}\|\nabla u\|_{L^6}+\|u\|_{L^\infty}\|\nabla\widetilde{f}^{\mu,\lambda,\kappa}\|_{L_{x,v}^2}) \|\nabla(\widetilde{a}^{\mu,\lambda,\kappa},\widetilde{b}^{\mu,\lambda,\kappa},\widetilde{\omega}^{\mu,\lambda,\kappa})\|_{L^2}  {\rm d}\tau  \nonumber\\
&+\int_0^t (\|\widetilde{f}^{\mu,\lambda,\kappa}\|_{L_v^2(L^3)}\|\nabla u\|_{L^6}+\|u\|_{L^\infty}\|\nabla\widetilde{f}^{\mu,\lambda,\kappa}\|_{L_{x,v}^2})  \|\nabla\{\mathbf{I}-\mathbf{P}\}\widetilde{f}^{\mu,\lambda,\kappa}\|_{\nu}  {\rm d}\tau  \nonumber\\
\lesssim&\,\|\widetilde{f}^{\mu,\lambda,\kappa}\|_{L_t^\infty(L_v^2(H^1))}\|\nabla u\|_{L_t^2(H^1)}\big(\|\nabla(\widetilde{a}^{\mu,\lambda,\kappa},\widetilde{b}^{\mu,\lambda,\kappa},\widetilde{\omega}^{\mu,\lambda,\kappa})\|_{L_t^2(L^2)}+\|\nabla\{\mathbf{I}-\mathbf{P}\}\widetilde{f}^{\mu,\lambda,\kappa}\|_{L_t^2(\nu)}\big)\nonumber\\
\lesssim&\, \big(\eps_0^\frac{1}{2}+\eps_1^\frac{1}{2}\big){\Xi} ^{\mu,\lambda,\kappa}(t),\label{G4.18}\\
|\widetilde{J}_2| 
\lesssim&\,\|\widetilde{f}^{\mu,\lambda,\kappa}\|_{L_t^\infty(L_v^2(H^2))}\|\nabla u\|_{L_t^2(H^1)}\big(\|\nabla(\widetilde{a}^{\mu,\lambda,\kappa},\widetilde{b}^{\mu,\lambda,\kappa},\widetilde{\omega}^{\mu,\lambda,\kappa})\|_{L_t^2(L^2)}+\|\nabla\{\mathbf{I}-\mathbf{P}\}\widetilde{f}^{\mu,\lambda,\kappa}\|_{L_t^2(\nu)}\big)\nonumber\\
\lesssim&\, \big(\eps_0^\frac{1}{2}+\eps_1^\frac{1}{2}\big){\Xi} ^{\mu,\lambda,\kappa}(t),\label{G4.18-1}\\
|\widetilde{J}_3|\lesssim&\,\int_0^t \|\nabla\widetilde{\theta}^{\mu,\lambda,\kappa}\|_{L^2}\big(\|\nabla \widetilde{b}^{\mu,\lambda,\kappa}\|_{L^2}\|\nabla u\|_{L^\infty}+\|\widetilde{b}^{\mu,\lambda,\kappa}\|_{L^6}\|\nabla^2 u\|_{L^3}\big){\rm d}\tau\nonumber\\
\lesssim&\, \|\widetilde{\theta}^{\mu,\lambda,\kappa}\|_{L_t^\infty(H^1)}\|\nabla\widetilde{b}^{\mu,\lambda,\kappa}\|_{L_t^2(L^2)}\|\nabla^2 u\|_{L_t^2(H^1)}\nonumber\\
\lesssim&\, \big(\eps_0^\frac{1}{2}+\eps_1^\frac{1}{2}\big){\Xi} ^{\mu,\lambda,\kappa}(t),\label{G4.18-2}\\
|\widetilde{J}_4|\lesssim&\,\int_0^t \|\nabla\widetilde{u}^{\mu,\lambda,\kappa}\|_{L^2}\big(\|\nabla \widetilde{a}^{\mu,\lambda,\kappa}\|_{L^2}\|\nabla u\|_{L^\infty}+\|\widetilde{a}^{\mu,\lambda,\kappa}\|_{L^6}\|\nabla^2 u\|_{L^3}\big){\rm d}\tau\nonumber\\
\lesssim&\, \|\widetilde{u}^{\mu,\lambda,\kappa}\|_{L_t^\infty(H^1)}\|\nabla\widetilde{a}^{\mu,\lambda,\kappa}\|_{L_t^2(L^2)}\|\nabla^2 u\|_{L_t^2(H^1)}\nonumber\\
\lesssim&\, \big(\eps_0^\frac{1}{2}+\eps_1^\frac{1}{2}\big){\Xi} ^{\mu,\lambda,\kappa}(t),\label{G4.18-3}\\
|\widetilde{J}_9|\lesssim&\,\int_0^t \big\|\nabla\big[\widetilde{\theta}^{\mu,\lambda,\kappa}(2\nabla_v f-vf) \big]\big\|_{L_{x,v}^2} \|\nabla  (2\nabla_v \widetilde{f}^{\mu,\lambda,\kappa}+v\widetilde{f}^{\mu,\lambda,\kappa})   \|_{L_{x,v}^2} {\rm d}\tau\nonumber\\
\lesssim&\, \int_0^t \|\nabla\widetilde\theta^{\mu,\lambda,\kappa}\|_{L^2}\|(2\nabla_v-v)  f\|_{L_v^2(L^\infty)} \|(2\nabla_v+v)\nabla \widetilde{f}^{\mu,\lambda,\kappa}\|_{L_{x,v}^2 }{\rm d}\tau \nonumber\\
&+\int_0^t  \| \widetilde\theta^{\mu,\lambda,\kappa}\|_{L^6}\|(2\nabla_v-v)  \nabla f\|_{L_v^2(L^3)} \|(2\nabla_v+v)\nabla \widetilde{f}^{\mu,\lambda,\kappa}\|_{L_{x,v}^2 }{\rm d}\tau \nonumber\\
\lesssim&\, \|\widetilde{\theta}^{\mu,\lambda,\kappa}\|_{L_t^\infty(H^1)}\big(\|\{\mathbf{I}-\mathbf{P}\}\nabla\widetilde{f}^{\mu,\lambda,\kappa}\|_{L_t^2(\nu)} +\|\nabla(\widetilde{a}^{\mu,\lambda,\kappa},\widetilde{b}^{\mu,\lambda,\kappa},\widetilde{\omega}^{\mu,\lambda,\kappa})\|_{L_t^2(L^2)} \big)\nonumber\\
&\times \Big(\sum_{1\leq |\alpha|\leq 2}\|\{\mathbf{I}-\mathbf{P}\}\partial^\alpha {f} \|_{L_t^2(\nu)}+\|\nabla(a,b,\omega)\|_{L_t^2(H^1)}\Big)\nonumber\\
\lesssim&\, \big(\eps_0^\frac{1}{2}+\eps_1^\frac{1}{2}\big){\Xi} ^{\mu,\lambda,\kappa}(t),\label{G4.18-4}\\
|\widetilde{J}_{10}|\lesssim&\,\int_0^t \|\nabla\widetilde{\theta}^{\mu,\lambda,\kappa}\|_{L^2}\big(\| {a}^{\mu,\lambda,\kappa}\|_{L^\infty}\|\nabla\widetilde{u}^{\mu,\lambda,\kappa}\|_{L^2}+\|\nabla a^{\mu,\lambda,\kappa}\|_{L^3}\|\widetilde{u}^{\mu,\lambda,\kappa}\|_{L^6}\big){\rm d}\tau\nonumber\\
\lesssim&\, \|\widetilde{\theta}^{\mu,\lambda,\kappa}\|_{L_t^\infty(H^1)}\|\nabla a^{\mu,\lambda,\kappa}\|_{L_t^2(H^1)}\|\nabla\widetilde{u}^{\mu,\lambda,\kappa}\|_{L_t^2(L^2)}\nonumber\\
\lesssim&\, \big(\eps_0^\frac{1}{2}+\eps_1^\frac{1}{2}\big){\Xi} ^{\mu,\lambda,\kappa}(t).\label{G4.18-5}
\end{align}

For the terms $\widetilde{J}_5,  \widetilde{J}_6, \widetilde{J}_7,  \widetilde{J}_8, \widetilde{J}_{11}$,
and $  \widetilde{J}_{12}$,  taking similar arguments  to the above, we   infer that
\begin{align}
|\widetilde{J}_5 |+|\widetilde{J}_6|
\lesssim&\,\|\widetilde{u}^{\mu,\lambda,\kappa}\|_{L_t^\infty( H^1)}\|\nabla f^{\mu,\lambda,\kappa}\|_{L_t^2(L_v^2(H^1))} \|\nabla(\widetilde{a}^{\mu,\lambda,\kappa},\widetilde{b}^{\mu,\lambda,\kappa},\widetilde{\omega}^{\mu,\lambda,\kappa})\|_{L_t^2(L^2)} \nonumber\\
&+\|\widetilde{u}^{\mu,\lambda,\kappa}\|_{L_t^\infty( H^1)}\|\nabla f^{\mu,\lambda,\kappa}\|_{L_t^2(L_v^2(H^1))} \|\nabla\{\mathbf{I}-\mathbf{P}\}\widetilde{f}^{\mu,\lambda,\kappa}\|_{L_t^2(\nu)} \nonumber\\
\lesssim&\, \big(\eps_0^\frac{1}{2}+\eps_1^\frac{1}{2}\big){\Xi} ^{\mu,\lambda,\kappa}(t),\label{G4.19-1}\\   
|\widetilde{J}_7|\lesssim&\,  \|\widetilde{\theta}^{\mu,\lambda,\kappa}\|_{L_t^\infty(H^1)}\|\nabla\widetilde{u}^{\mu,\lambda,\kappa}\|_{L_t^2(L^2)}\|\nabla^2 b^{\mu,\lambda,\kappa}\|_{L_t^2(H^1)} 
\lesssim  \big(\eps_0^\frac{1}{2}+\eps_1^\frac{1}{2}\big){\Xi} ^{\mu,\lambda,\kappa}(t),\label{G4.19-2}\\
|\widetilde{J}_8|\lesssim&\, \|\widetilde{u}^{\mu,\lambda,\kappa}\|_{L_t^\infty(H^1)}\|\nabla\widetilde{u}^{\mu,\lambda,\kappa}\|_{L_t^2(L^2)}\|\nabla^2 a^{\mu,\lambda,\kappa}\|_{L_t^2(H^1)} 
\lesssim  \big(\eps_0^\frac{1}{2}+\eps_1^\frac{1}{2}\big){\Xi} ^{\mu,\lambda,\kappa}(t),\label{G4.19-3}\\
|\widetilde{J}_{11}|\lesssim&\,\ \|\theta\|_{L_t^\infty(H^3 )}\big(\|\{\mathbf{I}-\mathbf{P}\}\nabla\widetilde{f}^{\mu,\lambda,\kappa}\|_{L_t^2(\nu)}^2 +\|\nabla(\widetilde{a}^{\mu,\lambda,\kappa},\widetilde{b}^{\mu,\lambda,\kappa},\widetilde{\omega}^{\mu,\lambda,\kappa})\|_{L_t^2(L^2)}^2 \big) \nonumber\\ 
\lesssim &\, \big(\eps_0^\frac{1}{2}+\eps_1^\frac{1}{2}\big){\Xi} ^{\mu,\lambda,\kappa}(t),\label{G4.19-4}\\
|\widetilde{J}_{12}| 
\lesssim&\, \|\widetilde{a}^{\mu,\lambda,\kappa}\|_{L_t^\infty(H^1)}\|\nabla \theta \|_{L_t^2(H^1)}\|\nabla\widetilde{u}^{\mu,\lambda,\kappa}\|_{L_t^2(L^2)} 
\lesssim  \big(\eps_0^\frac{1}{2}+\eps_1^\frac{1}{2}\big){\Xi} ^{\mu,\lambda,\kappa}(t).\label{G4.19}
\end{align} 

For the terms $\widetilde{J}_{13},\dots\widetilde{J}_{17}$,
it follows from Lemmas \ref{LA.1} and \ref{LNXL3.2} that
\begin{align}\label{LNXEG4.38}
&|\widetilde{J}_{13}+\widetilde{J}_{14}+\widetilde{J}_{15}+\widetilde{J}_{16}+\widetilde{J}_{17}|   \nonumber\\
\lesssim&\,\int_0^t \|\nabla\widetilde{u}^{\mu,\lambda,\kappa}\|_{L^2}\|(\widetilde{\rho}^{\mu,\lambda,\kappa},\widetilde{\theta}^{\mu,\lambda,\kappa})\|_{L^6} \|\nabla (\rho^{\mu,\lambda,\kappa},\theta^{\mu,\lambda,\kappa})\|_{L^3} {\rm d}\tau \nonumber\\
&+\int_0^t\|\nabla(\widetilde{u}^{\mu,\lambda,\kappa},\widetilde{\theta}^{\mu,\lambda,\kappa})\|_{L^2}\|\theta^{\mu,\lambda,\kappa}\|_{L^\infty}\|(\widetilde{a}^{\mu,\lambda,\kappa},\widetilde{b}^{\mu,\lambda,\kappa},\widetilde u^{\mu,\lambda,\kappa},\widetilde\theta^{\mu,\lambda,\kappa})\|_{L^6} \|\nabla( a^{\mu,\lambda,\kappa},b^{\mu,\lambda,\kappa},  u,\theta)\|_{L^3}                    {\rm d}\tau\nonumber\\
&+\int_0^t\|\nabla(\widetilde{u}^{\mu,\lambda,\kappa},\widetilde{\theta}^{\mu,\lambda,\kappa})\|_{L^2}\|\theta^{\mu,\lambda,\kappa}\|_{L^\infty}\|\nabla(\widetilde{a}^{\mu,\lambda,\kappa},\widetilde{b}^{\mu,\lambda,\kappa},\widetilde u^{\mu,\lambda,\kappa},\widetilde{\theta}^{\mu,\lambda,\kappa})\|_{L^2}\| ( a^{\mu,\lambda,\kappa},b^{\mu,\lambda,\kappa},u,\theta)\|_{L^\infty}                  {\rm d}\tau\nonumber\\
&+\int_0^t\|\nabla(\widetilde{u}^{\mu,\lambda,\kappa},\widetilde{\theta}^{\mu,\lambda,\kappa})\|_{L^2} \|\theta^{\mu,\lambda,\kappa}\|_{L^\infty}\|\nabla\big((\widetilde{b}^{\mu,\lambda,\kappa}-\widetilde u^{\mu,\lambda,\kappa}),(\sqrt{2}\widetilde{\omega}^{\mu,\lambda,\kappa}-\sqrt{3}\widetilde \theta^{\mu,\lambda,\kappa})\big)\|_{L^2}{\rm d}\tau \nonumber\\
\lesssim&\, \big(1+\eps_0^\frac{1}{2}\big)\| (\widetilde{\rho}^{\mu,\lambda,\kappa},\widetilde{u}^{\mu,\lambda,\kappa},\widetilde{\theta}^{\mu,\lambda,\kappa})\|_{L_t^\infty(H^1)}\|\nabla(\widetilde a^{\mu,\lambda,\kappa},\widetilde b^{\mu,\lambda,\kappa},\widetilde u^{\mu,\lambda,\kappa},\widetilde \theta^{\mu,\lambda,\kappa})\|_{L_t^2(L^2)} \nonumber\\
& \times  \|\nabla(  a^{\mu,\lambda,\kappa},  b^{\mu,\lambda,\kappa}, \rho^{\mu,\lambda,\kappa},  \theta^{\mu,\lambda,\kappa},  u ,  \theta )\|_{L_t^2(H^2)}\nonumber\\
&+\| ( \widetilde{u}^{\mu,\lambda,\kappa},\widetilde{\theta}^{\mu,\lambda,\kappa})\|_{L_t^\infty(H^1)}\|\nabla\theta^{\mu,\lambda,\kappa}\|_{L_t^2(H^2)}\|\nabla\big((\widetilde{b}^{\mu,\lambda,\kappa}-\widetilde u^{\mu,\lambda,\kappa}),(\sqrt{2}\widetilde{\omega}^{\mu,\lambda,\kappa}-\sqrt{3}\widetilde \theta^{\mu,\lambda,\kappa})\big)\|_{L_t^2(L^2)}\nonumber\\
\lesssim&\,  \big(\eps_0^\frac{1}{2}+\eps_1^\frac{1}{2}\big){\Xi} ^{\mu,\lambda,\kappa}(t).
\end{align}  

For the term $\widetilde{J}_{18}$, by applying integration by parts and Lemmas \ref{LA.1} and \ref{LNXL3.2}, we have  
\begin{align}\label{G4.20}
|\widetilde{J}_{18}|\lesssim&\,     \int_0^t\|\nabla\widetilde{\rho}^{\mu,\lambda,\kappa}\|_{L^2}\big(\|\nabla(\widetilde{\rho}^{\mu,\lambda,\kappa},\widetilde{u}^{\mu,\lambda,\kappa})\|_{L^2}\|\nabla (\rho,u)\|_{L^\infty}+\|(\widetilde{\rho}^{\mu,\lambda,\kappa},\widetilde{u}^{\mu,\lambda,\kappa})\|_{L^3}\|\nabla^2 (\rho,u)\|_{L^6}\big){\rm d}\tau\nonumber\\
&+\int_0^t\|\nabla\widetilde{\rho}^{\mu,\lambda,\kappa}\|_{L^2}^2  \big(\|\nabla u^{\mu,\lambda,\kappa}\|_{L^{\infty}}+\nabla \rho^{\mu,\lambda,\kappa}\|_{L^{\infty}}\|  u^{\mu,\lambda,\kappa}\|_{L^{\infty}}\big) {\rm d}\tau \nonumber\\
\lesssim&\, \|(\widetilde{\rho}^{\mu,\lambda,\kappa},\widetilde{u}^{\mu,\lambda,\kappa})\|_{L_t^\infty(H^1)} \|\nabla (\widetilde{\rho}^{\mu,\lambda,\kappa},\widetilde{u}^{\mu,\lambda,\kappa} ) \|_{L_t^2(L^2)}  \|\nabla ({\rho}^{\mu,\lambda,\kappa},{u}^{\mu,\lambda,\kappa},\rho,u)\|_{L_t^2(H^2)} \nonumber\\
\lesssim&\,  \big(\eps_0^\frac{1}{2}+\eps_1^\frac{1}{2}\big){\Xi} ^{\mu,\lambda,\kappa}(t).
\end{align}  
 For the terms $\widetilde{J}_{19}$, $\widetilde{J}_{20}$, and $\widetilde{J}_{21}$, similar to \eqref{G4.20}, we compute that
\begin{align}\label{G4.21}
 |\widetilde{J}_{19}+\widetilde{J}_{20}+J_{21}|\lesssim \big(\eps_0^\frac{1}{2}+\eps_1^\frac{1}{2}\big){\Xi} ^{\mu,\lambda,\kappa}(t).
\end{align}

For the term $\widetilde{J}_{22}$, analogous to \eqref{LNXG4.18}, we derive that 
\begin{align}\label{G4.22}
|\widetilde{J}_{22}|\lesssim&\, \mu \int_0^t\|\nabla^3u^{\mu,\lambda,\kappa}\|_{L^2} \|\nabla\widetilde{u}^{\mu,\lambda,\kappa}\|_{L^2}{\rm d}\tau+\mu\int_0^t \Big\|\nabla\Big( \frac{1}{1+\rho^{\mu,\lambda,\kappa}} \Big)\Big\|_{L^\infty}\|\nabla^2 u^{\mu,\lambda,\kappa}\|_{L^2}\|\nabla\widetilde{u}^{\mu,\lambda,\kappa}\|_{L^2}{\rm d}\tau\nonumber\\
\lesssim&\, \mu\|\nabla\widetilde{u}^{\mu,\lambda,\kappa}\|_{L_t^2(L^2)}\|\nabla^3 u^{\mu,\lambda,\kappa}\|_{L_t^2(L^2)}+\mu\|\rho^{\mu,\lambda,\kappa}\|_{L_t^\infty(H^3)}\|\nabla^2 u^{\mu,\lambda,\kappa}\|_{L_t^2(L^2)} \|\nabla\widetilde{u}^{\mu,\lambda,\kappa}\|_{L_t^2(L^2)}\nonumber\\
\lesssim&\, \big(\eps_0^\frac{1}{2}+\eps_1^\frac{1}{2}\big){\Xi} ^{\mu,\lambda,\kappa}(t)+\mu^2.
\end{align} 

Similarly, we have
\begin{align}\label{LNXL4.42}
|\widetilde{J}_{23}+\widetilde{J}_{24}+\widetilde{J}_{25}+\widetilde{J}_{26}| \lesssim&\,     \big(\eps_0^\frac{1}{2}+\eps_1^\frac{1}{2}\big){\Xi} ^{\mu,\lambda,\kappa}(t)+\mu^2+\lambda^2+\kappa^2.
\end{align} 

For the terms $\widetilde{J}_{27}, \widetilde{J}_{28}$ and $\widetilde{J}_{29}$, it follows from Lemma \ref{LA.1} and   integration by parts that  
\begin{align}
|\widetilde{J}_{27}|\lesssim&\,\int_0^t \|\nabla\widetilde\theta^{\mu,\lambda,\kappa}\|_{L^2}\big(\|\nabla \widetilde{\theta}^{\mu,\lambda,\kappa}\|_{L^2}\|\nabla u^{\mu,\lambda,\kappa}\|_{L^2}+\|\widetilde{\theta}^{\mu,\lambda,\kappa}\|_{L^3}\|\nabla^2 u^{\mu,\lambda,\kappa}\|_{L^6}\big){\rm d}\tau\nonumber\\
&+\int_0^t \|\nabla\widetilde\theta^{\mu,\lambda,\kappa}\|_{L^2}\big( \|\nabla  \widetilde{u}^{\mu,\lambda,\kappa}\|_{L^6}\|\nabla\theta\|_{L^3}+\|\widetilde{u}^{\mu,\lambda,\kappa}\|_{L^6}\|\nabla^2\theta\|_{L^3}\big){\rm d}\tau\nonumber\\
&+\int_0^t \|\nabla\widetilde\theta^{\mu,\lambda,\kappa}\|_{L^2}^2 \big(\|\nabla u^{\mu,\lambda,\kappa}\|_{L^\infty}+\|  u^{\mu,\lambda,\kappa}\|_{L^\infty}\|  \nabla\theta^{\mu,\lambda,\kappa}\|_{L^\infty}\big) {\rm d}\tau\nonumber\\
\lesssim&\, \| \widetilde{\theta}^{\mu,\lambda,\kappa}\|_{L_t^\infty(H^1)} \|\nabla(\widetilde{u}^{\mu,\lambda,\kappa},\widetilde{\theta}^{\mu,\lambda,\kappa})\|_{L_t^2(L^2)} \big( \|\nabla (u^{\mu,\lambda,\kappa},\theta^{\mu,\lambda,\kappa})\|_{L_t^2(H^2)}+\|\nabla\theta\|_{L_t^2(H^2)} \big)\nonumber\\
\lesssim&\,\big(\eps_0^\frac{1}{2}+\eps_1^\frac{1}{2}\big){\Xi} ^{\mu,\lambda,\kappa}(t),\label{G4.23-1}\\
|\widetilde{J}_{28}|\lesssim&\, \int_0^t  \|\nabla\widetilde{\theta}^{\mu,\lambda,\kappa}\|_{L^2}\big( \| b^{\mu,\lambda,\kappa}-u^{\mu,\lambda,\kappa}\|_{L^3}\|\widetilde{u}^{\mu,\lambda,\kappa}\|_{L^6}+\| \nabla(b^{\mu,\lambda,\kappa}-u^{\mu,\lambda,\kappa})\|_{L^3}\|\widetilde{u}^{\mu,\lambda,\kappa}\|_{L^6}\big){\rm d}\tau\nonumber\\
&+\int_0^t\|\nabla\widetilde{\theta}^{\mu,\lambda,\kappa}\|_{L^2}\big(\| b^{\mu,\lambda,\kappa}-u^{\mu,\lambda,\kappa}\|_{L^\infty}\|\nabla\widetilde{u}^{\mu,\lambda,\kappa}\|_{L^2}+\|\nabla u\|_{L^3}\| \widetilde{b}^{\mu,\lambda,\kappa}-\widetilde{u}^{\mu,\lambda,\kappa}\|_{L^6} \big){\rm d}\tau\nonumber\\
&+\int_0^t\|\nabla\widetilde{\theta}^{\mu,\lambda,\kappa}\|_{L^2}\big\{ \|u\|_{L^\infty}\|\widetilde{u}^{\mu,\lambda,\kappa}-\widetilde{b}^{\mu,\lambda,\kappa}\|_{L^2}\nonumber\\
 &\qquad \ \
+\|\nabla{a}^{\mu,\lambda,\kappa}\|_{L^6} \|\widetilde{\theta}^{\mu,\lambda,\kappa}\|_{L^6}(\|\theta\|_{L^6}+\|\theta^{\mu,\lambda,\kappa}
\|_{L^6})\big\}{\rm d}\tau\nonumber\\
&+\int_0^t  \|\nabla\widetilde{\theta}^{\mu,\lambda,\kappa}\|_{L^2} \|\nabla\rho\|_{L^6}\big(\|u\|_{L^6}\|\widetilde{b}^{\mu,\lambda,\kappa}-\widetilde{u}^{\mu,\lambda,\kappa}\|_{L^6}+\|\widetilde{a}^{\mu,\lambda,\kappa}\|_{L^6}(\|\theta\|_{L^2}+\|\theta\|_{L^6})\|\theta\|_{L^\infty}  \big){\rm d}\tau\nonumber\\
&+\int_0^t  \|\nabla\widetilde{\theta}^{\mu,\lambda,\kappa}\|_{L^2}\big\{\|\nabla \widetilde{a}^{\mu,\lambda,\kappa}\|_{L^2}\|\theta\|_{L^\infty}^2\nonumber\\
&\qquad \ \ 
+\|\nabla\rho\|_{L^\infty} \|a^{\mu,\lambda,\kappa}\|_{L^6}\|\widetilde{\theta}^{\mu,\lambda,\kappa}\|_{L^6}(\|\theta\|_{L^6}+\|\theta^{\mu,\lambda,\kappa}\|_{L^6})\big\}{\rm d}\tau\nonumber\\
&+\int_0^t \|\nabla\widetilde{\theta}^{\mu,\lambda,\kappa}\|_{L^2} \|a^{\mu,\lambda,\kappa}
\|_{L^6}\|\widetilde{\theta}^{\mu,\lambda,\kappa}\|_{L^6}(\|\theta\|_{L^6}
+\|\theta^{\mu,\lambda,\kappa}\|_{L^6}){\rm d}\tau\nonumber\\   
  &+\int_0^t\|\nabla\widetilde{\theta}^{\mu,\lambda,\kappa}\|_{L^2}^2
  \big(\|a^{\mu,\lambda,\kappa}\|_{L^\infty}\|(\theta,
  \theta^{\mu,\lambda,\kappa})\|_{L^\infty}+\|\nabla\rho\|_{L^3}\|\sqrt{2}
  \widetilde{\omega}^{\mu,\lambda,\kappa}
-\sqrt{3}\widetilde{\theta}^{\mu,\lambda,\kappa}\|_{L^6}\big) {\rm d}\tau\nonumber\\
&+\int_0^t\|\nabla\widetilde{\theta}^{\mu,\lambda,\kappa}\|_{L^2} \|\rho\|_{L^\infty}\|\nabla(\sqrt{2}\widetilde{\omega}^{\mu,\lambda,\kappa}-\sqrt{3}\widetilde{\theta}^{\mu,\lambda,\kappa})\|_{L^2}   {\rm d}\tau\nonumber\\
&+\int_0^t \|\nabla\widetilde{\theta}^{\mu,\lambda,\kappa}\|_{L^2}\|\nabla\rho\|_{L^6}\|(\widetilde{a}^{\mu,\lambda,\kappa},\widetilde{b}^{\mu,\lambda,\kappa},\widetilde{u}^{\mu,\lambda,\kappa},\widetilde{\theta}^{\mu,\lambda,\kappa})\|_{L^6} \|(u,\theta,a^{\mu,\lambda,\kappa},b^{\mu,\lambda,\kappa})\|_{L^6} {\rm d}\tau\nonumber\\
&+\int_0^t \|\nabla\widetilde{\theta}^{\mu,\lambda,\kappa}\|_{L^2}\|\rho\|_{L^\infty}\|\nabla(\widetilde{a}^{\mu,\lambda,\kappa},\widetilde{b}^{\mu,\lambda,\kappa},\widetilde{u}^{\mu,\lambda,\kappa},\widetilde{\theta}^{\mu,\lambda,\kappa})\|_{L^2}(\|(u,\theta,a^{\mu,\lambda,\kappa},b^{\mu,\lambda,\kappa})\|_{L^\infty}{\rm d}\tau\nonumber\\
&+\int_0^t \|\nabla\widetilde{\theta}^{\mu,\lambda,\kappa}\|_{L^2}\|\rho\|_{L^6}\| (\widetilde{a}^{\mu,\lambda,\kappa},\widetilde{b}^{\mu,\lambda,\kappa},\widetilde{u}^{\mu,\lambda,\kappa},\widetilde{\theta}^{\mu,\lambda,\kappa})\|_{L^6} \|\nabla(u,\theta, a^{\mu,\lambda,\kappa},b^{\mu,\lambda,\kappa})\|_{L^6}   {\rm d}\tau\nonumber\\
\lesssim&\,  \|\widetilde{\theta}^{\mu,\lambda,\kappa}\|_{L_t^\infty(H^1)}\big\{\|\nabla(\widetilde{a}^{\mu,\lambda,\kappa},\widetilde{b}^{\mu,\lambda,\kappa},\widetilde{u}^{\mu,\lambda,\kappa},\widetilde{\theta}^{\mu,\lambda,\kappa})\|_{L_t^2(L^2)}\nonumber\\
& +\|(\widetilde{b}^{\mu,\lambda,\kappa}-\widetilde{u}^{\mu,\lambda,\kappa},
\sqrt{2}\widetilde{\omega}^\kappa-\sqrt{3}\widetilde{\theta}^\kappa)\|_{L_t^2(H^1)}\big\}\nonumber\\
&  \times \big\{\|\nabla( {a} , {b} , {u} , {\theta} )\|_{L_t^2(H^2)}+\| {b}^{\mu,\lambda,\kappa} - {u}^{\mu,\lambda,\kappa}\|_{L_t^2(H^1)}+\|\sqrt{2} {\omega}^{\mu,\lambda,\kappa} -\sqrt{3} {\theta}^{\mu,\lambda,\kappa}\|_{L_t^2(H^1)}\nonumber\\
&  +\|\nabla(a^{\mu,\lambda,\kappa},b^{\mu,\lambda,\kappa},u^{\mu,\lambda,\kappa})\|_{L_t^2(H^2)}\big\}\nonumber\\
\lesssim&\, \big(\eps_0^\frac{1}{2}+\eps_1^\frac{1}{2}\big){\Xi} ^{\mu,\lambda,\kappa}(t),\label{G4.23-2}\\
|\widetilde{J}_{29}|\lesssim&\,\int_0^t \|\nabla\widetilde{\theta}^{\mu,\lambda,\kappa}\|_{L^2} (\|\widetilde{\rho}^{\mu,\lambda,\kappa}\|_{L^6}+\|\nabla\widetilde{\rho}^{\mu,\lambda,\kappa}\|_{L^2})\|\sqrt{2} {\omega}^{\mu,\lambda,\kappa} -\sqrt{3} {\theta}^{\mu,\lambda,\kappa} \|_{H^2} {\rm d}\tau\nonumber\\
&+\int_0^t \|\nabla\widetilde{\theta}^{\mu,\lambda,\kappa}\|_{L^2}(\|\widetilde{\rho}^{\mu,\lambda,\kappa}\|_{L^6}+\|\nabla\widetilde{\rho}^{\mu,\lambda,\kappa}\|_{L^2}) \|\nabla(a^{\mu,\lambda,\kappa},b^{\mu,\lambda,\kappa},u^{\mu,\lambda,\kappa})\|_{H^1} {\rm d}\tau\nonumber\\
&+\int_0^t \|\nabla\widetilde{\theta}^{\mu,\lambda,\kappa}\|_{L^2}\big(\|\nabla\widetilde{\rho}^{\mu,\lambda,\kappa}\|_{L^2}\|\nabla u^{\mu,\lambda,\kappa}\|_{L^\infty}^2+\|\widetilde{\rho}^{\mu,\lambda,\kappa}\|_{L^6}\|\nabla u^{\mu,\lambda,\kappa}\|_{L^6}\|\nabla^2 u^{\mu,\lambda,\kappa}\|_{L^6}\big){\rm d}\tau\nonumber\\
\lesssim&\,\big(1+\eps_0^\frac{1}{2}+\eps_1^\frac{1}{2}\big)\|\widetilde{\theta}^{\mu,\lambda,\kappa}\|_{L_t^\infty(H^1)} \|\nabla\widetilde{\rho}^{\mu,\lambda,\kappa}\|_{L_t^\infty(H^1)}  \|\nabla(a^{\mu,\lambda,\kappa},b^{\mu,\lambda,\kappa},u^{\mu,\lambda,\kappa})\|_{L_t^2(H^2)} ,\nonumber\\
&+\big(1+\eps_0^\frac{1}{2}+\eps_1^\frac{1}{2}\big)\|\widetilde{\theta}^{\mu,\lambda,\kappa}\|_{L_t^\infty(H^1)} \|\nabla\widetilde{\rho}^{\mu,\lambda,\kappa}\|_{L_t^\infty(H^1)}  \|\sqrt{2} {\omega}^{\mu,\lambda,\kappa} -\sqrt{3} {\theta}^{\mu,\lambda,\kappa} \|_{L_t^2(H^2)} ,\nonumber\\
\lesssim&\, \big(\eps_0^\frac{1}{2}+\eps_1^\frac{1}{2}\big){\Xi} ^{\mu,\lambda,\kappa}(t). \label{G4.23}
\end{align} 

Finally, for the remaining terms $\widetilde{J}_{30}$, $\widetilde{J}_{31}$ and $\widetilde{J}_{32}$, using Lemma \ref{LNXL3.2} and H\"{o}lder's inequality yields
\begin{align}\label{LNXG4.46}
|\widetilde{J}_{30}|+|\widetilde{J}_{31}|+|\widetilde{J}_{32}|\lesssim&\, \|\partial_t (\rho^{\mu,\lambda,\kappa},u^{\mu,\lambda,\kappa},\theta^{\mu,\lambda,\kappa})\|_{L_t^\infty(L^\infty)}  \|\nabla(\widetilde\rho^{\mu,\lambda,\kappa},\widetilde u^{\mu,\lambda,\kappa},\widetilde\theta^{\mu,\lambda,\kappa})\|_{L_t^2(L^2)}^2  \nonumber\\
\lesssim&\,\big(\eps_0^\frac{1}{2}+\eps_1^\frac{1}{2}\big){\Xi} ^{\mu,\lambda,\kappa}(t). 
\end{align} 

By incorporating the estimates  \eqref{G4.18}--\eqref{LNXG4.46} into \eqref{G4.17}, we ultimately derive the desired estimate \eqref{G4.16}.
\end{proof}

Next, we provide the estimate of $ \|\nabla(\widetilde{a}^{\mu,\lambda,\kappa},\widetilde{b}^{\mu,\lambda,\kappa},\widetilde{\omega}^{\mu,\lambda,\kappa}) \|_{L^2_t(L^2)}^2 $.

\begin{lem}\label{L4.3}
It holds that
\begin{align}\label{G4.24}
\int_0^t \|\nabla(\widetilde{a}^{\mu,\lambda,\kappa},\widetilde{b}^{\mu,\lambda,\kappa},\widetilde{\omega}^{\mu,\lambda,\kappa}
)\|_{L^2}^2{\rm d}\tau\leq  C {\Xi} ^{\mu,\lambda,\kappa}(t)+{\Xi} ^{\mu,\lambda,\kappa}(0),
\end{align}
where $C>0$ is a constant independent of   $\mu$, $\lambda$, $\kappa$ and   $t$.
\end{lem}
\begin{proof}
Similarly to \eqref{G3.21}, by combining the equations \eqref{G4.1}$_4$ and \eqref{G4.2}$_4$, we can derive the equations for $\widetilde{a}^{\mu,\lambda,\kappa}$, $\widetilde{b}^{\mu,\lambda,\kappa}$ and $\widetilde{\omega}^{\mu,\lambda,\kappa}$:

\begin{equation}\label{G4.25}
\left\{\begin{aligned}
&\partial_t \widetilde{a}^{\mu,\lambda,\kappa}+{\rm div} \widetilde{b}^{\mu,\lambda,\kappa}=0,\\
& \partial_t \widetilde{b}^{\mu,\lambda,\kappa}_i+\partial_i \widetilde{a}^{\mu,\lambda,\kappa}
+\frac{2}{\sqrt{6}} \partial_i \widetilde\omega^{\mu,\lambda,\kappa}+\sum_{j=1}^3 \partial_{j} 
\Gamma_{i, j}(\{\mathbf{I}-\mathbf{P}\} \widetilde{f}^{\mu,\lambda,\kappa}) \\
&\quad=-(\widetilde{b}^{\mu,\lambda,\kappa}_i-\widetilde{u}^{\mu,\lambda,\kappa}_i)
+\widetilde{u}^{\mu,\lambda,\kappa}_i a+u_i^{\mu,\lambda,\kappa}\widetilde{a}^{\mu,\lambda,\kappa}, \\
& \partial_t \widetilde\omega^{\mu,\lambda,\kappa}+\sqrt{2}(\sqrt{2} 
\widetilde\omega^{\mu,\lambda,\kappa}-\sqrt{3} \widetilde\theta^{\mu,
\lambda,\kappa})-\sqrt{6} (\widetilde{a}^\kappa \theta+{a}^{\mu,\lambda,
\kappa}\widetilde{\theta}^{\mu,\lambda,\kappa})+\frac{2}{\sqrt{6}} {\rm div} 
 \widetilde{b}^{\mu,\lambda,\kappa}     \\
&\quad=-\sum_{i=1}^3 \partial_{i} \Upsilon_i(\{\mathbf{I}-\mathbf{P}\} 
\widetilde{f}^{\mu,\lambda,\kappa})+\frac{2}{\sqrt{6}}(\widetilde{u}^{\mu,
\lambda,\kappa}\cdot b+u^{\mu,\lambda,\kappa}\cdot\widetilde{b}^{\mu,\lambda,\kappa}),\\
& \partial_j \widetilde{b}^{\mu,\lambda,\kappa}_i+\partial_i \widetilde{b}^{\mu,\lambda,\kappa}_j
- (\widetilde{u}^{\mu,\lambda,\kappa}_i b_j+u^{\mu,\lambda,\kappa}_i \widetilde{b}^{\mu,\lambda,\kappa}_j
+\widetilde{u}^{\mu,\lambda,\kappa}_j b_i+u^{\mu,\lambda,\kappa}_j\widetilde{b}^{\mu,\lambda,\kappa}_i )   \\
&\quad =\frac{2}{\sqrt{6}} \delta_{i j}\Big(\frac{2}{\sqrt{6}} {\rm div} 
  \widetilde{b}^{\mu,\lambda,\kappa}-\frac{2}{\sqrt{6}} (\widetilde{u}^{\mu,\lambda,\kappa} \cdot b
  +u^{\mu,\lambda,\kappa}\cdot\widetilde{b}^{\mu,\lambda,\kappa})+\sum_{i=1}^3 \partial_{ i} 
  \Upsilon_i(\{\mathbf{I}-\mathbf{P}\} \widetilde{f}^{\mu,\lambda,\kappa})\Big), \\
&\quad\quad -  \partial_t \Gamma_{i, j}(\{\mathbf{I}-\mathbf{P}\} f) 
 -\Gamma_{i, j}\big(\,\widetilde{\mathfrak{l}}^{\mu,\lambda,\kappa}
 +\widetilde{\mathfrak{r}}^{\mu,\lambda,\kappa}+\widetilde{\mathfrak{s}}^{\mu,\lambda,\kappa}\big)\\
& \frac{5}{3} (\partial_i \widetilde\omega^{\mu,\lambda,\kappa}
-\widetilde\omega^{\mu,\lambda,\kappa} u_i-\omega^{\mu,\lambda,\kappa} 
\widetilde{u}^{\mu,\lambda,\kappa}_i-\sqrt{6} \widetilde\theta^{\mu,
\lambda,\kappa} b_i-\sqrt{6} \theta^{\mu,\lambda,\kappa} \widetilde{b}^{\mu,\lambda,\kappa}_i )   \\
&\quad=\frac{2}{\sqrt{6}} \sum_{j=1}^3 \partial_{j} \Gamma_{i, j}
(\{\mathbf{I}-\mathbf{P}\} \widetilde{f}^{\mu,\lambda,\kappa})-\partial_t 
\Upsilon_i(\{\mathbf{I}-\mathbf{P}\} 
\widetilde{f}^{\mu,\lambda,\kappa})
+\Upsilon_i\big(\,\widetilde{\mathfrak{l}}^{\mu,\lambda,\kappa}
+\widetilde{\mathfrak{r}}^{\mu,\lambda,\kappa}+\widetilde{\mathfrak{s}}^{\mu,\lambda,\kappa}\big),
\end{aligned}\right.
\end{equation}
for $1\leq i,j\leq 3$, where $\widetilde{\mathfrak{l}}^{\mu,\lambda,\kappa}$, 
$\widetilde{\mathfrak{r}}^{\mu,\lambda,\kappa}$ and $\widetilde{\mathfrak{s}}^{\mu,\lambda,\kappa}$ are given by
\begin{align*} 
 \widetilde{\mathfrak{l}}^{\mu,\lambda,\kappa}:=&\,-v \cdot \nabla\{\mathbf{I}-\mathbf{P}\} 
 \widetilde{f}^{\mu,\lambda,\kappa}+\mathcal{L}\{\mathbf{I}-\mathbf{P}\} \widetilde{f}^{\mu,\lambda,\kappa},  \\
 \widetilde{\mathfrak{r}}^{\mu,\lambda,\kappa}:=&\,-(\widetilde{u}^{\mu,\lambda,\kappa} 
 \cdot \nabla_v\{\mathbf{I}-\mathbf{P}\} f^{\mu,\lambda,\kappa}
 +u \cdot \nabla_v\{\mathbf{I}-\mathbf{P}\} \widetilde{f}^{\mu,\lambda,\kappa})\\
&+\frac{1}{2} (\widetilde{u}^{\mu,\lambda,\kappa}\cdot v \{\mathbf{I}-\mathbf{P}\}f^{\mu,\lambda,
\kappa}+u \cdot v\{\mathbf{I}-\mathbf{P}\} \widetilde{f}^{\mu,\lambda,\kappa}),  \\
  \widetilde{\mathfrak{s}}^{\mu,\lambda,\kappa}:=&\,\frac{1}{2}\frac{\theta}{\sqrt{M}} 
  {\rm div}_v(\sqrt{M} (2\nabla_v-v)\{\mathbf{I}-\mathbf{P}\} \widetilde{f}^{\mu,\lambda,\kappa}),\\
&+\frac{1}{2}\frac{\widetilde\theta^{\mu,\lambda,\kappa}}{\sqrt{M}} {\rm div}_v(\sqrt{M} 
(2\nabla_v-v)\{\mathbf{I}-\mathbf{P}\} f^{\mu,\lambda,\kappa}).
\end{align*}

Define the temporal functional
\begin{align}\nonumber% \label{G4.26}
\widetilde{\mathfrak{E}}^{\mu,\lambda,\kappa}_0(t):= &   \sum_{i, j=1}^3 \int_{\mathbb{R}^3} 
 (\partial_j \widetilde{b}^{\mu,\lambda,\kappa}_i+\partial_i \widetilde{b}^{\mu,\lambda,\kappa}_j)  
 \Gamma_{i, j}(\{\mathbf{I}-\mathbf{P}\} \widetilde{f}^{\mu,\lambda,\kappa}) \mathrm{d} x \nonumber\\
&+  \sum_{i=1}^3 \int_{\mathbb{R}^3}   \partial_i \widetilde\omega^{\mu,\lambda,\kappa}  
 \Upsilon_i(\{\mathbf{I}-\mathbf{P}\} \widetilde{f}^{\mu,\lambda,\kappa}) \mathrm{d} x \nonumber\\
& +\frac{2}{21}  \int_{\mathbb{R}^3}   \widetilde{a}^{\mu,\lambda,\kappa}   
\Big(\frac{\sqrt{6}}{5} \sum_{i=1}^3 \partial_i \Upsilon_i(\{\mathbf{I}-\mathbf{P}\} 
\widetilde{f}^{\mu,\lambda,\kappa})-{\rm div}   \widetilde{b}^{\mu,\lambda,\kappa}\Big) \mathrm{d} x.
\end{align} 
By Young's inequality, we have
\begin{align}\label{G4.27}
|\widetilde{\mathfrak{E}}^{\mu,\lambda,\kappa}_0(t)|\lesssim  
\|\widetilde{f}^{\mu,\lambda,\kappa}\|_{L_v^2(H^1)}^2.    
\end{align}

Let us first estimate $\|\nabla \widetilde{b}^{\mu,\lambda,\kappa}\|_{L^2}$. 
By employing the energy method and using \eqref{G4.25}$_4$, we have
\begin{align}\label{G4.28}
& \sum_{i, j=1}^3  \|  \partial_j \widetilde{b}^{\mu,\lambda,\kappa}_i+\partial_i
 \widetilde{b}^{\mu,\lambda,\kappa}_j  \|_{L^2}^2-\sum_{i, j=1}^3 \int_{\mathbb{R}^3} \frac{2}{3} \delta_{i j}   ({\rm div} 
 \widetilde{b}^{\mu,\lambda,\kappa})\cdot  
  (\partial_j \widetilde{b}^{\mu,\lambda,\kappa}_i
  +\partial_i \widetilde{b}^{\mu,\lambda,\kappa}_j ) \mathrm{d} x \nonumber\\
  =\,& -\frac{\mathrm{d}}{\mathrm{d} t} 
  \sum_{i, j=1}^3 \int_{\mathbb{R}^3}  
  \big(\partial_j \widetilde{b}^{\mu,\lambda,\kappa}_i
  +\partial_i \widetilde{b}^{\mu,\lambda,\kappa}_j \big)  
   \Gamma_{i, j}(\{\mathbf{I}-\mathbf{P}\} \widetilde{f}^{\mu,\lambda,\kappa}) \mathrm{d} x \nonumber\\
&  +\sum_{i, j=1}^3 \int_{\mathbb{R}^3}  
\big(\partial_i \partial_t \widetilde{b}^{\mu,\lambda,\kappa}_j
+\partial_j \partial_t \widetilde{b}^{\mu,\lambda,\kappa}_i\big)  
 \Gamma_{i, j}(\{\mathbf{I}-\mathbf{P}\} \widetilde{f}^{\mu,\lambda,\kappa}) \mathrm{d} x \nonumber\\
&   +\sum_{i, j=1}^3 \int_{\mathbb{R}^3}  
 \big(\partial_i \widetilde{b}^{\mu,\lambda,\kappa}_j+\partial_j 
 \widetilde{b}^{\mu,\lambda,\kappa}_i \big)   \big(\widetilde{u}^{\mu,\lambda,\kappa}_i 
 b_j+u^{\mu,\lambda,\kappa}_i \widetilde{b}^{\mu,\lambda,\kappa}_j
 +\widetilde{u}^{\mu,\lambda,\kappa}_j b_i+u^{\mu,\lambda,\kappa}_j\widetilde{b}^{\mu,\lambda,
 \kappa}_i \big) {\rm d}x\nonumber \\
&   -\sum_{i, j=1}^3 \int_{\mathbb{R}^3}  
 \big(\partial_i \widetilde{b}^{\mu,\lambda,\kappa}_j+\partial_j 
 \widetilde{b}^{\mu,\lambda,\kappa}_i \big)   \delta_{i j}\sum_{i=1}^3 \partial_i 
 \Upsilon_i(\{\mathbf{I}-\mathbf{P}\} \widetilde{f}^{\mu,\lambda,\kappa}){\rm d}x\nonumber \\
&   +\sum_{i, j=1}^3 \int_{\mathbb{R}^3}  
 \big(\partial_i \widetilde{b}^{\mu,\lambda,\kappa}_j+\partial_j 
 \widetilde{b}^{\mu,\lambda,\kappa}_i \big)  \Big[\Gamma_{i, j}
 \big(\,\widetilde{\mathfrak{l}}^{\mu,\lambda,\kappa}+\widetilde{\mathfrak{r}}^{\mu,\lambda,
 \kappa}+\widetilde{\mathfrak{s}}^{\mu,\lambda,\kappa}\big)-\frac{2}{3}\delta_{i j} 
 \big(\widetilde{u}^{\mu,\lambda,\kappa} 
\cdot b+u^{\mu,\lambda,\kappa}\cdot
\widetilde{b}^{\mu,\lambda,\kappa}\big)\Big]{\rm d}x\nonumber\\
 =:\,&-\frac{\mathrm{d}}{\mathrm{d} t} \sum_{i, j=1}^3 
 \int_{\mathbb{R}^3}  \big(\partial_j \widetilde{b}^{\mu,\lambda,\kappa}_i
 +\partial_i \widetilde{b}^{\mu,\lambda,\kappa}_j \big)  
  \Gamma_{i, j}(\{\mathbf{I}-\mathbf{P}\} \widetilde{f}^{\mu,\lambda,\kappa})
   \mathrm{d} x+\widetilde{K}_1+\widetilde{K}_2+\widetilde{K}_3+\widetilde{K}_4,
\end{align}
where $\widetilde{K}_1$, $\widetilde{K}_2$, $\widetilde{K}_3$ and $\widetilde{K}_4$ denote the third, fourth, fifth and sixth lines of 
\eqref{G4.28}, respectively. And $\delta_{ij}$ is the Kronecker delta function.

According to \eqref{G4.25}$_2$ and Young’s inequality, the term $\widetilde{K}_1$  can be derived as follows:  
\begin{align*}
\widetilde{K}_1=&\, 2\sum_{i, j=1}^3 \int_{\mathbb{R}^3}   \partial_j \partial_t \widetilde{b}^{\mu,\lambda,\kappa}_i   \Gamma_{i, j}(\{\mathbf{I}-\mathbf{P}\} \widetilde{f}^{\mu,\lambda,\kappa}) \mathrm{d} x\nonumber\\
=&\,  -2\sum_{i, j=1}^3 \int_{\mathbb{R}^3}    \partial_t \widetilde{b}^{\mu,\lambda,\kappa}_i   \partial_j\Gamma_{i, j}(\{\mathbf{I}-\mathbf{P}\} \widetilde{f}^{\mu,\lambda,\kappa}) \mathrm{d} x\nonumber\\
=&\,2\sum_{i, j=1}^3 \int_{\mathbb{R}^3} \Big(\partial_i \widetilde{a}^{\mu,\lambda,\kappa}+\frac{2}{\sqrt{6}} \partial_i \widetilde\omega^{\mu,\lambda,\kappa}+\sum_{j=1}^3 \partial_{j} \Gamma_{i, j}(\{\mathbf{I}-\mathbf{P}\} \widetilde{f}^{\mu,\lambda,\kappa})\Big)\partial_j\Gamma_{i, j}(\{\mathbf{I}-\mathbf{P}\} \widetilde{f}^{\mu,\lambda,\kappa}) \mathrm{d} x\nonumber\\
&+2\sum_{i, j=1}^3 \int_{\mathbb{R}^3} \big((\widetilde{b}^{\mu,\lambda,\kappa}_i-\widetilde{u}^{\mu,\lambda,\kappa}_i)-\widetilde{u}^{\mu,\lambda,\kappa}_i a-u_i^{\mu,\lambda,\kappa}\widetilde{a}^{\mu,\lambda,\kappa}\big)\partial_j\Gamma_{i, j}(\{\mathbf{I}-\mathbf{P}\} \widetilde{f}^{\mu,\lambda,\kappa}) \mathrm{d} x\nonumber\\
\lesssim&\,\sigma\|\nabla(\widetilde{a}^{\mu,\lambda,\kappa},\widetilde\omega^{\mu,\lambda,\kappa})\|_{L^2}^2+ \|\nabla \{\mathbf{I}-\mathbf{P}\} \widetilde{f}^{\mu,\lambda,\kappa}\|_{L_{x,v}^2}^2+ \|\widetilde{b}^{\mu,\lambda,\kappa}-\widetilde{u}^{\mu,\lambda,\kappa}\|_{L^2}^2\nonumber\\
&+\|(\widetilde{a}^{\mu,\lambda,\kappa},\widetilde{u}^{\mu,\lambda,\kappa})\|_{L^6}^2(\|a\|_{L^3}^2+\|u^{\mu,\lambda,\kappa}\|_{L^3}^2).
\end{align*}
Here and below, $\sigma>0$ is a sufficiently small constant.
For the terms $\widetilde{K}_2$, $\widetilde{K}_2$, and $\widetilde{K}_2$, from Young's inequality and Lemma \ref{LA.1}, it follows that  
\begin{align*}
|\widetilde{K}_2+\widetilde{K}_3+\widetilde{K}_4|\leq&\,  \frac{1}{4}\sum_{i, j=1}^3  \|  \partial_j \widetilde{b}^{\mu,\lambda,\kappa}_i+\partial_i \widetilde{b}^{\mu,\lambda,\kappa}_j  \|_{L^2}^2 +C \|\widetilde{u}^{\mu,\lambda,\kappa} \cdot b+u^{\mu,\lambda,\kappa}\cdot\widetilde{b}^\kappa\|_{L^2}^2\nonumber\\
&+C\|\nabla \Upsilon_i(\{\mathbf{I}-\mathbf{P}\} \widetilde{f}^{\mu,\lambda,\kappa})\|_{L^2}^2\nonumber\\
&+C\|\widetilde{u}^{\mu,\lambda,\kappa}_i b_j+u^{\mu,\lambda,\kappa}_i \widetilde{b}^{\mu,\lambda,\kappa}_j+\widetilde{u}^{\mu,\lambda,\kappa}_j b_i+u^{\mu,\lambda,\kappa}_j\widetilde{b}^{\mu,\lambda,\kappa}_i \|_{L^2}^2\nonumber\\
&+C\big\|\Gamma_{i, j}\big(\,\widetilde{\mathfrak{l}}^{\mu,\lambda,\kappa}+\widetilde{\mathfrak{r}}^{\mu,\lambda,\kappa}+\widetilde{\mathfrak{s}}^{\mu,\lambda,\kappa}\big)\big\|_{L^2}^2\nonumber\\
\leq\,&\frac{1}{4}\sum_{i, j=1}^3  \|  \partial_j \widetilde{b}^{\mu,\lambda,\kappa}_i+\partial_i \widetilde{b}^{\mu,\lambda,\kappa}_j  \|_{L^2}^2 +C\|\nabla \{\mathbf{I}-\mathbf{P}\} \widetilde{f}^{\mu,\lambda,\kappa}\|_{L_{x,v}^2}^2\nonumber\\
&+C\|\{\mathbf{I}-\mathbf{P}\} \widetilde{f}^{\mu,\lambda,\kappa}\|_{L_v^2(H^1)}^2+C\|\nabla(\widetilde{u}^{\mu,\lambda,\kappa},\widetilde{\theta}^{\mu,\lambda,\kappa})\|_{L^2}^2   \|\{\mathbf{I}-\mathbf{P}\}   {f}^{\mu,\lambda,\kappa} \|_{L_v^2(H^1)} ^2\nonumber\\
&+C\|\nabla(u,\theta)\|_{L^2}^2 \|  \{\mathbf{I}-\mathbf{P}\}  \widetilde{f}^{\mu,\lambda,\kappa} \|_{L_v^2(H^1)}^2\nonumber\\
&+C\|(\widetilde{b}^{\mu,\lambda,\kappa},\widetilde{u}^{\mu,\lambda,\kappa})\|_{L^6}^2\big(\|b\|_{L^3}^2+\|u^{\mu,\lambda,\kappa}\|_{L^3}^2\big).
\end{align*}
Moreover, we  observe that   
\begin{gather*}
  \sum_{i, j=1}^3 \|\ \partial_i \widetilde{b}^{\mu,\lambda,\kappa}_j+\partial_j \widetilde{b}^{\mu,\lambda,\kappa}_i \|_{L^2}^2 =2 \|\nabla   \widetilde{b}^{\mu,\lambda,\kappa} \|_{L^2}^2+2 \| {\rm div} \widetilde{b}^{\mu,\lambda,\kappa}\|_{L^2}^2, \\
 -\sum_{i, j=1}^3 \frac{2}{3} \delta_{i j} \int_{\mathbb{R}^3}   {\rm div}   \widetilde{b}^{\mu,\lambda,\kappa}  (\partial_i \widetilde{b}^{\mu,\lambda,\kappa}_j+\partial_j \widetilde{b}^{\mu,\lambda,\kappa}_i ) \mathrm{d} x 
=-\frac{4}{3} \|  {\rm div}   \widetilde{b}^{\mu,\lambda,\kappa} \|_{L^2}^2.
\end{gather*}
Consequently, by substituting the aforementioned estimates into \eqref{G4.28},  we obtain 
\begin{align}\label{G4.29}
& \frac{\mathrm{d}}{\mathrm{d} t}   \sum_{i, j=1}^3 \int_{\mathbb{R}^3} \ \big(\partial_i \widetilde{b}^{\mu,\lambda,\kappa}_j+\partial_j \widetilde{b}^{\mu,\lambda,\kappa}_i \big)   \Gamma_{i, j}(\{\mathbf{I}-\mathbf{P}\} \widetilde{f}^{\mu,\lambda,\kappa}) \mathrm{d} x  +   \frac{3}{2} \|\nabla   \widetilde{b}^{\mu,\lambda,\kappa}\|_{L^2}^2+\frac{1}{6} \|  \nabla  \cdot \widetilde{b}^{\mu,\lambda,\kappa}\|_{L^2}^2  \nonumber\\
 &\quad \leq  C\sigma \|\nabla (\widetilde{a}^{\mu,\lambda,\kappa}, \widetilde\omega^{\mu,\lambda,\kappa})\|_{L^2}^2+C(\eps_0+\eps_1)\|\nabla(\widetilde{a}^{\mu,\lambda,\kappa},\widetilde{b}^{\mu,\lambda,\kappa},\widetilde{u}^{\mu,\lambda,\kappa},\widetilde{\theta}^{\mu,\lambda,\kappa})\|_{L^2}^2\nonumber\\
&\qquad\, +C \|\{\mathbf{I}-\mathbf{P}\} \widetilde{f}^{\mu,\lambda,\kappa}\|_{L_v^2 (H^1 )}^2+C\|\widetilde{b}^{\mu,\lambda,\kappa}-\widetilde{u}^{\mu,\lambda,\kappa}\|_{H^1}^2.  
\end{align}

Next, we proceed to estimate $\|\nabla\widetilde{\omega}^{\mu,\lambda,\kappa}\|_{L^2}$.
From \eqref{G4.25}$_5$, it follows that  
\begin{align}\label{G4.30}
 \|  \partial_i \widetilde\omega^{\mu,\lambda,\kappa} \|_{L^2}^2= &\, -\frac{3}{5} \frac{\mathrm{d}}{\mathrm{d} t} \int_{\mathbb{R}^3}  \partial_i \widetilde\omega^{\mu,\lambda,\kappa}   \Upsilon_i(\{\mathbf{I}-\mathbf{P}\} \widetilde{f}^{\mu,\lambda,\kappa}) \mathrm{d} x+\frac{3}{5} \int_{\mathbb{R}^3}   \partial_i \partial_t \widetilde\omega^{\mu,\lambda,\kappa}   \Upsilon_i(\{\mathbf{I}-\mathbf{P}\} \widetilde{f}^{\mu,\lambda,\kappa}) \mathrm{d} x \nonumber\\
& +\int_{\mathbb{R}^3}  \partial_i \widetilde\omega^{\mu,\lambda,\kappa}   \big(\widetilde\omega^{\mu,\lambda,\kappa} u_i+\omega^{\mu,\lambda,\kappa}\widetilde{u}^{\mu,\lambda,\kappa}_i+\sqrt{6} \widetilde{\theta}^{\mu,\lambda,\kappa} b_i+\sqrt{6}\theta^{\mu,\lambda,\kappa}\widetilde{b}^{\mu,\lambda,\kappa}_i \big) \mathrm{d} x\nonumber\\
& +\frac{\sqrt{6}}{5} \int_{\mathbb{R}^3}   \partial_i \widetilde\omega^{\mu,\lambda,\kappa}   \sum_{j=1}^3 \partial_j \Gamma_{i, j}(\{\mathbf{I}-\mathbf{P}\} \widetilde{f}^{\mu,\lambda,\kappa})  \mathrm{d} x\nonumber\\
&+ \frac{3}{5}\int_{\mathbb{R}^3}   \partial_i \widetilde\omega \Upsilon_i^{\mu,\lambda,\kappa}\big(\widetilde{\mathfrak{l}}^{\mu,\lambda,\kappa}+\widetilde{\mathfrak{r}}^{\mu,\lambda,\kappa}+\widetilde{\mathfrak{s}}^{\mu,\lambda,\kappa}\big) \mathrm{d} x \nonumber\\
\equiv: & -\frac{3}{5} \frac{\mathrm{d}}{\mathrm{d} t} \int_{\mathbb{R}^3}  \partial_i \widetilde\omega^{\mu,\lambda,\kappa}   \Upsilon_i(\{\mathbf{I}-\mathbf{P}\} \widetilde{f}^{\mu,\lambda,\kappa}) \mathrm{d} x+ \widetilde{K}_5+\widetilde{K}_6+\widetilde{K}_7+\widetilde{K}_8.
\end{align}
By leveraging \eqref{G4.25}$_3$, Young's inequality and Lemma \ref{LA.1}, we have
\begin{align*}
| \widetilde{K}_5|= &\,\Big| \int_{\mathbb{R}^3} \partial_i  \Big(\sqrt{6}\big( \widetilde{a}^{\mu,\lambda,\kappa} \theta+a^{\mu,\lambda,\kappa}\widetilde{\theta}^{\mu,\lambda,\kappa}\big)-\sqrt{2}\big(\sqrt{2} \widetilde\omega^{\mu,\lambda,\kappa}-\sqrt{3} \widetilde\theta^{\mu,\lambda,\kappa}\big) \Big)\Upsilon_i(\{\mathbf{I}-\mathbf{P}\} \widetilde{f}^{\mu,\lambda,\kappa}) \mathrm{d} x\Big| \\
&   +\Big|\int_{\mathbb{R}^3}\partial_i  \Big(\frac{2}{\sqrt{6}}\big( \widetilde{u}^{\mu,\lambda,\kappa} \cdot b+u^{\mu,\lambda,\kappa}\cdot\widetilde{b}^{\mu,\lambda,\kappa}\big)-\sum_i \partial_i \Upsilon_i\{\mathbf{I}-\mathbf{P}\} \widetilde{f}^{\mu,\lambda,\kappa}\Big)    \Upsilon_i(\{\mathbf{I}-\mathbf{P}\} \widetilde{f}^{\mu,\lambda,\kappa}) \mathrm{d} x\Big| \\
&+\Big| \int_{\mathbb{R}^3} \partial_i  \Big( \frac{2}{\sqrt{6}} {\rm div} \widetilde{b}^{\mu,\lambda,\kappa}\Big)\Upsilon_i(\{\mathbf{I}-\mathbf{P}\} \widetilde{f}^{\mu,\lambda,\kappa}) \mathrm{d} x\Big| \\
\lesssim \,& \sigma\|\nabla  \widetilde{b}^{\mu,\lambda,\kappa}\|_{L^2}^2+ \|\{\mathbf{I}-\mathbf{P}\} \widetilde{f}^{\mu,\lambda,\kappa}\|_{L_v^2 (H^1)}^2+ \|\sqrt{2} \widetilde\omega^{\mu,\lambda,\kappa}-\sqrt{3} \widetilde\theta^{\mu,\lambda,\kappa}\|_{H^1}^2\nonumber\\
&+(\eps_0+\eps_1)\|\nabla(\widetilde{b}^{\mu,\lambda,\kappa},\widetilde{u}^{\mu,\lambda,\kappa})\|_{L^2}^2.
\end{align*}
Similarly, for the terms $\widetilde{K}_6$, $\widetilde{K}_6$, and $\widetilde{K}_6$, through a direct calculation, we arrive at  
\begin{align*}
|\widetilde{K}_6+\widetilde{K}_7+\widetilde{K}_8|\leq&\,
\frac{2}{5} \| \partial_i \widetilde\omega^{\mu,\lambda,\kappa} \|_{L^2}^2+C\|\{\mathbf{I}-\mathbf{P}\} \widetilde{f}^{\mu,\lambda,\kappa}\|_{L_v^2 (H^1)}^2 \nonumber\\ &+C(\eps_0+\eps_1)\|\nabla(\widetilde{b}^{\mu,\lambda,\kappa},\widetilde{\omega}^{\mu,\lambda,\kappa},\widetilde{u}^{\mu,\lambda,\kappa},\widetilde\theta^{\mu,\lambda,\kappa})\|_{L^2}^2.   
\end{align*}
Inserting the above estimates on $\widetilde{K}_5$ and $\widetilde{K}_6+\widetilde{K}_7+\widetilde{K}_8$ into
\eqref{G4.30} gives rise to
\begin{align}\label{G4.31}
& \frac{\mathrm{d}}{\mathrm{d} t}   \sum_{ i=1}^3 \int_{\mathbb{R}^3}   \partial_i \widetilde\omega^{\mu,\lambda,\kappa}   \Upsilon_i(\{\mathbf{I}-\mathbf{P}\} \widetilde{f}^{\mu,\lambda,\kappa}) \mathrm{d} x+ \| \nabla  \widetilde\omega^{\mu,\lambda,\kappa} \|_{L^2}^2 \nonumber\\
&\quad \leq C\sigma\|\nabla b\|_{H^1}^2+C \|\{\mathbf{I}-\mathbf{P}\} \widetilde{f}^{\mu,\lambda,\kappa}\|_{L_v^2 (H^1 )}^2\nonumber\\
& \quad\quad  +C(\eps_0+\eps_1)\|\nabla(\widetilde{a}^{\mu,\lambda,\kappa},\widetilde{b}^{\mu,\lambda,\kappa},\widetilde\omega^{\mu,\lambda,\kappa},\widetilde{u}^{\mu,\lambda,\kappa},\widetilde\theta^{\mu,\lambda,\kappa})\|_{L^2}^2\nonumber\\
& \quad\quad   +C\|\sqrt{2} \widetilde\omega^{\mu,\lambda,\kappa}-\sqrt{3} \widetilde\theta^{\mu,\lambda,\kappa}\|_{H^1}^2. 
\end{align}

Finally, we   estimate $\|\nabla\widetilde{a}^{\mu,\lambda,\kappa}\|_{L^2}$.
From \eqref{G4.25}$_2$ and \eqref{G4.25}$_5$, one has
\begin{align}\label{G4.32}
\|  \partial_i \widetilde{a}^{\mu,\lambda,\kappa} \|_{L^2}^2
=\, & \frac{\mathrm{d}}{\mathrm{d} t} \int_{\mathbb{R}^3}   \partial_i \widetilde{a}^{\mu,\lambda,\kappa}  \Big(\frac{\sqrt{6}}{5} \Upsilon_i(\{\mathbf{I}-\mathbf{P}\} \widetilde{f}^{\mu,\lambda,\kappa})-\widetilde{b}^{\mu,\lambda,\kappa}_i\Big) \mathrm{d} x \nonumber\\
& -\int_{\mathbb{R}^3}   \partial_i \partial_t\widetilde{a}^{\mu,\lambda,\kappa}   \Big(\frac{\sqrt{6}}{5} \Upsilon_i(\{\mathbf{I}-\mathbf{P}\} \widetilde{f}^{\mu,\lambda,\kappa})-\widetilde{b}^{\mu,\lambda,\kappa}_i\Big) \mathrm{d} x \nonumber\\
& +\int_{\mathbb{R}^3}  \partial_i \widetilde{a}^{\mu,\lambda,\kappa}  \big( (\widetilde{u}^{\mu,\lambda,\kappa}_i-\widetilde{b}^{\mu,\lambda,\kappa}_i)+\widetilde{u}^{\mu,\lambda,\kappa}_i a+ {u}_i^{\mu,\lambda,\kappa}\widetilde{a}^{\mu,\lambda,\kappa} \big){\rm d}x\nonumber \\
& -\int_{\mathbb{R}^3}  \partial_i \widetilde{a}^{\mu,\lambda,\kappa}  \Big(  \frac{2}{\sqrt{6}} (\widetilde\omega^{\mu,\lambda,\kappa} u_i+\omega^{\mu,\lambda,\kappa} \widetilde{u}^{\mu,\lambda,\kappa}_i)+2 (\widetilde\theta^{\mu,\lambda,\kappa} b_i+\theta^{\mu,\lambda,\kappa}\widetilde{b}^{\mu,\lambda,\kappa}_i)\Big){\rm d}x\nonumber \\
& - \int_{\mathbb{R}^3}  \partial_i \widetilde{a}^{\mu,\lambda,\kappa}\Big(\frac{7}{5} \sum_{j=1}^3 \partial_j \Gamma_{i j}(\{\mathbf{I}-\mathbf{P}\} \widetilde{f}^{\mu,\lambda,\kappa})+\frac{\sqrt{6}}{5} \Upsilon_i\big(\,\widetilde{\mathfrak{l}}^{\mu,\lambda,\kappa}+\widetilde{\mathfrak{r}}^{\mu,\lambda,\kappa}
+\widetilde{\mathfrak{s}}^{\mu,\lambda,\kappa}\big)\Big) \mathrm{d}x\nonumber\\
\equiv:\,\,&\frac{\mathrm{d}}{\mathrm{d} t} \int_{\mathbb{R}^3}   \partial_i \widetilde{a}^{\mu,\lambda,\kappa}  \Big(\frac{\sqrt{6}}{5} \Upsilon_i(\{\mathbf{I}-\mathbf{P}\} \widetilde{f}^{\mu,\lambda,\kappa})-\widetilde{b}^{\mu,\lambda,\kappa}_i\Big) \mathrm{d}x+\widetilde{K}_9+\widetilde{K}_{10}+\widetilde{K}_{11}+\widetilde{K}_{12}. 
\end{align}
By employing \eqref{G4.25}$_1$ and applying Young's inequality, we obtain  
\begin{align*}
|\widetilde{K}_9 | =&\,\Big|\int_{\mathbb{R}^3}  \partial_i ({\rm div}   \widetilde{b}^{\mu,\lambda,\kappa})  \Big(\widetilde{b}^{\mu,\lambda,\kappa}_i-\frac{\sqrt{6}}{5} \Upsilon_i(\{\mathbf{I}-\mathbf{P}\} \widetilde{f}^{\mu,\lambda,\kappa})\Big) \mathrm{d} x\Big|\\
\leq&\, \int_{\mathbb{R}^3}   {\rm div}   \widetilde{b}^{\mu,\lambda,\kappa}   \partial_i \widetilde{b}^{\mu,\lambda,\kappa}_i \mathrm{d} x+\frac{1}{4} \|  {\rm div}   \widetilde{b}^{\mu,\lambda,\kappa} \|_{L^2}^2+C\|\{\mathbf{I}-\mathbf{P}\} \widetilde{f}^{\mu,\lambda,\kappa}\|_{L_v^2 (H^1 )}^2. 
\end{align*}
By a direct computation, we also have
\begin{align*}
|\widetilde{K}_{10}+\widetilde{K}_{11}+\widetilde{K}_{12}|
\leq \,&\frac{1}{4}\| \partial_i \widetilde{a}^{\mu,\lambda,\kappa}\|_{L^2}^2+C\big(\| \widetilde{u}^{\mu,\lambda,\kappa}_i a+u_i^{\mu,\lambda,\kappa}\widetilde{a}^{\mu,\lambda,\kappa}-\frac{2}{\sqrt{6}} (\widetilde
\omega^{\mu,\lambda,\kappa} u_i+\omega^{\mu,\lambda,\kappa}\widetilde{u}_i^{\mu,\lambda,\kappa})\|_{L^2}^2\big) \\
&+ C\big( \|  \widetilde{u}^{\mu,\lambda,\kappa}_i-\widetilde{b}^{\mu,\lambda,\kappa}_i\|_{L^2}^2+\sum_{j=1}^3 \|  \partial_j \Gamma_{i j}(\{\mathbf{I}-\mathbf{P}\} \widetilde{f}^{\mu,\lambda,\kappa})\|_{L_{x,v}^2}^2\big) \\
&+C\big(\|  \widetilde\theta^{\mu,\lambda,\kappa} b_i+\theta^{\mu,\lambda,\kappa}\widetilde{b}^{\mu,\lambda,\kappa}_i\|_{L^2}^2+\big\|  \Upsilon_i\big(\widetilde{\mathfrak{l}}^{\mu,\lambda,\kappa}+\widetilde{\mathfrak{r}}^{\mu,\lambda,\kappa}+\widetilde{\mathfrak{s}}^{\mu,\lambda,\kappa}\big) \big\|_{L ^2}^2 \big)\\
\leq\, &\frac{1}{4}\|  \partial_i \widetilde{a}^{\mu,\lambda,\kappa} \|_{L^2}^2+C\|\{\mathbf{I}-\mathbf{P}\} \widetilde{f}^{\mu,\lambda,\kappa}\|_{L_v^2 (H^1 )}^2+C\|\widetilde{b}^{\mu,\lambda,\kappa}-\widetilde{u}^{\mu,\lambda,\kappa}\|_{L^2}^2\nonumber\\
& + C(\eps_0+\eps_1)\|\nabla(\widetilde{a}^{\mu,\lambda,\kappa},\widetilde{b}^{\mu,\lambda,\kappa},\widetilde{\omega}^{\mu,\lambda,\kappa},\widetilde{u}^{\mu,\lambda,\kappa},\widetilde{\theta}^{\mu,\lambda,\kappa})\|_{L^2}^2.    
\end{align*}
Plugging the above estimates on $\widetilde{K}_9$ and $\widetilde{K}_{10} + \widetilde{K}_{11}+ \widetilde{K}_{12}$ into \eqref{G4.32} yields
\begin{align}\label{G4.33}
&\frac{\mathrm{d}}{\mathrm{d} t}   \int_{\mathbb{R}^3} \frac{2}{21}   \widetilde{a}^{\mu,\lambda,\kappa}  \Big(\frac{\sqrt{6}}{5}\sum_{i=1}^3 \Upsilon_i(\{\mathbf{I}-\mathbf{P}\} \widetilde{f}^{\mu,\lambda,\kappa})-{\rm div}   \widetilde{b}^{\mu,\lambda,\kappa}\Big) \mathrm{d} x+\frac{1}{14}   \|  \nabla  \widetilde{a}^{\mu,\lambda,\kappa}\|_{L^2}^2\nonumber \\
&\quad \leq \frac{5}{42} \| {\rm div}   \widetilde{b}^{\mu,\lambda,\kappa}\|_{L^2}^2+C\|\{\mathbf{I}-\mathbf{P}\} \widetilde{f}^{\mu,\lambda,\kappa}\|_{L_v^2 (H^1 )}^2+C\|\widetilde{b}^{\mu,\lambda,\kappa}-\widetilde{u}^{\mu,\lambda,\kappa}\|_{H^1}^2\nonumber\\
&\quad\quad\, +C(\eps_0+\eps_1)\|\nabla(\widetilde{a}^{\mu,\lambda,\kappa},\widetilde{b}^{\mu,\lambda,\kappa},\widetilde{\omega}^{\mu,\lambda,\kappa},\widetilde{u}^{\mu,\lambda,\kappa},\widetilde{\theta}^{\mu,\lambda,\kappa})\|_{L^2}^2.
\end{align}

Adding   \eqref{G4.29}, \eqref{G4.31}, and \eqref{G4.33}, we consequently get
\begin{align}\label{G4.34}
&\frac{{\rm d}}{{\rm d} t}\widetilde{\mathfrak{E}}^{\mu,\lambda,\kappa}_0(t)+\lambda_8 \|\nabla(\widetilde{a}^{\mu,\lambda,\kappa},\widetilde{b}^{\mu,\lambda,\kappa},\widetilde{\omega}^{\mu,\lambda,\kappa})\|_{L^2}^2\nonumber\\
&\quad\lesssim  \|\{\mathbf{I}-\mathbf{P}\} \widetilde{f}^{\mu,\lambda,\kappa}\|_{L_v^2 (H^1 )}^2+ \|\widetilde{b}^{\mu,\lambda,\kappa}-\widetilde{u}^{\mu,\lambda,\kappa}\|_{H^1}^2+\|\sqrt{2}\widetilde\omega^{\mu,\lambda,\kappa}-\sqrt{3}\widetilde{\theta}^{\mu,\lambda,\kappa}\|_{H^1}^2\nonumber\\
& \quad\quad+(\eps_0+\eps_1)\|\nabla( \widetilde{u}^{\mu,\lambda,\kappa},\widetilde{\theta}^{\mu,\lambda,\kappa})\|_{L^2}^2,   
\end{align}
for some $\lambda_8>0$.

Integrating \eqref{G4.34} over $[0,t]$ yields
\begin{align*}
&\widetilde{\mathfrak{E}}_0^{\mu,\lambda,\kappa}(t)+\|\nabla(\widetilde{a}^{\mu,\lambda,\kappa},\widetilde{b}^{\mu,\lambda,\kappa},\widetilde{\omega}^{\mu,\lambda,\kappa})\|_{L_t^2(L^2)}^2\nonumber\\
&\quad \lesssim\widetilde{\mathfrak{E}}_0^{\mu,\lambda,\kappa}(0)+\|\widetilde{b}^{\mu,\lambda,\kappa}-\widetilde{u}^{\mu,\lambda,\kappa}\|_{L_t^2(H^1)}^2
+\|\sqrt{2}\widetilde\omega^{\mu,\lambda,\kappa}-\sqrt{3}\widetilde{\theta}^{\mu,\lambda,\kappa}\|_{L_t^2(H^1)}^2\nonumber\\
&\qquad \, +\|\{\mathbf{I}-\mathbf{P}\} \widetilde{f}^{\mu,\lambda,\kappa}\|_{L_t^2(L_v^2 (H^1 ))}^2+(\eps_0+\eps_1)\|\nabla(\widetilde{u}^{\mu,\lambda,\kappa},\widetilde{\theta}^{\mu,\lambda,\kappa})\|_{L_t^2(L^2)}^2,
\end{align*}
which, together with \eqref{G4.7}, \eqref{G4.16}  and  \eqref{G4.27}, leads to  the desired estimate \eqref{G4.24}.
\end{proof}

Finally, we give the estimate of $\int^t_0 \| \nabla\widetilde\rho^{\mu,\lambda,\kappa}(\tau)\|_{L^2} {\rm d}\tau$.
\begin{lem}\label{L4.4}
It holds that
\begin{align}\label{G4.35}
\int_0^t \|\nabla \widetilde{\rho}^{\mu,\lambda,\kappa}
\|_{L^2}^2{\rm d}\tau\leq C\max\{\mu,\lambda,\kappa\}^2+C {\Xi} ^{\mu,\lambda,\kappa}(t)+{\Xi} ^{\mu,\lambda,\kappa}(0),
\end{align}
where $C>0$ is a constant independent of   $\mu$, $\lambda$, $\kappa$ and   $t$.
\end{lem}
\begin{proof}
First,   by applying Young's inequality, we have 
\begin{align}\label{G4.36}
\Big|\int_{\mathbb{R}^3} \widetilde{u}^{\mu,\lambda,\kappa} \cdot \nabla \widetilde{\rho}^{\mu,\lambda,\kappa} \, \mathrm{d}x\Big|  \lesssim \|\widetilde{u}^{\mu,\lambda,\kappa}\|_{L^2}^2+\|\nabla\widetilde{\rho}^{\mu,\lambda,\kappa}\|_{L^2}^2.    
\end{align}

It follows from \eqref{G4.1}$_2$ and \eqref{G4.2}$_2$ that
\begin{align}\label{G4.37}
 \|\nabla   \widetilde{\rho}^{\mu,\lambda,\kappa} \|_{L^2}^2
 =\, & -\int_{\mathbb{R}^3} \nabla    \widetilde{\rho}^{\mu,\lambda,\kappa}  \cdot \partial_t \widetilde{u}^{\mu,\lambda,\kappa} \mathrm{d} x\nonumber \\
 & -\int_{\mathbb{R}^3} \nabla    \widetilde{\rho}^{\mu,\lambda,\kappa}\cdot\Big(\nabla\widetilde{\theta}^{\mu,\lambda,\kappa}-\frac{\mu\Delta {u}^{\mu,\lambda,\kappa}}{1+\rho^{\mu,\lambda,\kappa}}-\frac{(\mu+\lambda)\nabla{\rm div} {u}^{\mu,\lambda,\kappa}}{1+\rho^{\mu,\lambda,\kappa}}-(\widetilde{b}^{\mu,\lambda,\kappa}-\widetilde{u}^{\mu,\lambda,\kappa})\Big)   \mathrm{d} x\nonumber \\
& -\int_{\mathbb{R}^3} \nabla   \widetilde{\rho}^{\mu,\lambda,\kappa}\cdot   ( \widetilde{a}^{\mu,\lambda,\kappa} u+a^{\mu,\lambda,\kappa}\widetilde{u}^{\mu,\lambda,\kappa}+\widetilde{u}^{\mu,\lambda,\kappa}\cdot\nabla u+u^{\mu,\lambda,\kappa}\cdot\nabla\widetilde{u}^{\mu,\lambda,\kappa} )\ \mathrm{d} x\nonumber\\
&+ \int_{\mathbb{R}^3}  \nabla   \widetilde{\rho}^{\mu,\lambda,\kappa}\cdot \widetilde{G}_8 \mathrm{d} x\nonumber\\  
\equiv:\,& \widetilde{R}_1 +\widetilde{R}_2+\widetilde{R}_3+\widetilde{R}_4,
\end{align}
where  
\begin{align*}
\widetilde{G}_8:=&\,\Big(\frac{1}{1+ \rho^{\mu,\lambda,\kappa}}-\frac{1}{1+\rho}    \Big)(\theta^{\mu,\lambda,\kappa}\nabla\rho^{\mu,\lambda,\kappa}+a^{\mu,\lambda,\kappa} u^{\mu,\lambda,\kappa})\nonumber\\
& -\frac{\rho}{1+\rho}\big((\widetilde{b}^{\mu,\lambda,\kappa}-\widetilde{u}^{\mu,\lambda,\kappa})-\nabla\widetilde{\rho}^{\mu,\lambda,\kappa}\big) \nonumber\\
&+\Big(\frac{\rho^{\mu,\lambda,\kappa}}{1+ \rho^{\mu,\lambda,\kappa}}-\frac{\rho}{1+\rho}    \Big)\big(\nabla\rho^{\mu,\lambda,\kappa}-(b^{\mu,\lambda,\kappa}-u^{\mu,\lambda,\kappa})\big)   \nonumber\\ 
&+\frac{1}{1+\rho}(\widetilde{a}^{\mu,\lambda,\kappa} u^{\mu,\lambda,\kappa}+a^{\mu,\lambda,\kappa}\widetilde{u}^{\mu,\lambda,\kappa}-\widetilde\theta^{\mu,\lambda,\kappa}\nabla\rho-\theta^{\mu,\lambda,\kappa}\nabla\widetilde{\rho}^{\mu,\lambda,\kappa}).
\end{align*} 
For the term $\widetilde{R}_1$, applying integration by parts and utilizing \eqref{G4.1}$_1$ and \eqref{G4.1}$_2$, we get  
\begin{align}\label{G4.38}
\widetilde{R}_1 =&\,-\frac{\mathrm{d}}{\mathrm{d} t} \int_{\mathbb{R}^3} \nabla    \widetilde{\rho}^{\mu,\lambda,\kappa} \cdot \widetilde{u}^{\mu,\lambda,\kappa} \mathrm{d} x+\int_{\mathbb{R}^3}  {\rm div}   \widetilde{u}^{\mu,\lambda,\kappa} \big ( (1+ {\rho}^{\mu,\lambda,\kappa}) {\rm div} \widetilde{u}^{\mu,\lambda,\kappa}+\widetilde{\rho}^{\mu,\lambda,\kappa}{\rm div}u \big) \mathrm{d} x \nonumber\\   
&+\int_{\mathbb{R}^3}  {\rm div}   \widetilde{u}^{\mu,\lambda,\kappa} \big (  \widetilde{u}^{\mu,\lambda,\kappa} \cdot \nabla \rho+u^{\mu,\lambda,\kappa}\cdot\nabla\widetilde{\rho}^{\mu,\lambda,\kappa}\big) \mathrm{d} x \nonumber\\  
\leq &\,-\frac{\mathrm{d}}{\mathrm{d} t} \int_{\mathbb{R}^3} \nabla    \widetilde{\rho}^{\mu,\lambda,\kappa} \cdot \widetilde{u}^{\mu,\lambda,\kappa} \mathrm{d} x+C\|\nabla\widetilde{u}^{\mu,\lambda,\kappa}\|_{L^2}^2+C\big(\eps_0^\frac{1}{2}+\eps_1^\frac{1}{2}\big)
\|\nabla(\widetilde{\rho}^{\mu,\lambda,\kappa},\widetilde{u}^{\mu,\lambda,\kappa})\|_{L^2}^2.
\end{align}
For the terms $\widetilde{R}_2, \widetilde{R}_3$, and $\widetilde{R}_4$, using Young's inequality, we perform a direct calculation to deduce  
\begin{align}
|\widetilde{R}_2 |  \leq&\, \frac{1}{4}\|\nabla\widetilde{\rho}^{\mu,\lambda,\kappa}\|_{L^2}^2+C\big(\|\nabla\widetilde{\theta}^{\mu,\lambda,\kappa}\|_{L^2}^2+ \mu^2\|\nabla^2{u}^{\mu,\lambda,\kappa}\|_{L^2}^2+(\mu+\lambda)^2\|\nabla{\rm div}u^{\mu,\lambda,\kappa}\|_{L^2}^2)\nonumber\\
&+C\|\widetilde{b}^{\mu,\lambda,\kappa}-\widetilde{u}^{\mu,\lambda,\kappa}\|_{L^2}^2,\label{G4.39-1}\\
|\widetilde{R}_3 |  \leq&\, \frac{1}{4}\|\nabla\widetilde{\rho}^{\mu,\lambda,\kappa}\|_{L^2}^2+C(\eps_0+\eps_1)\|\nabla (\widetilde{a}^{\mu,\lambda,\kappa},\widetilde{\rho}^{\mu,\lambda,\kappa},\widetilde{u}^{\mu,\lambda,\kappa})\|_{L^2}^2,\label{G4.39-2}\\
|\widetilde{R}_4 |\leq&\,  \frac{1}{4}\|\nabla\widetilde{\rho}^{\mu,\lambda,\kappa}\|_{L^2}^2+C(\eps_0+\eps_1)\big(\|\nabla (\widetilde{a}^{\mu,\lambda,\kappa},\widetilde{\rho}^{\mu,\lambda,\kappa}, \widetilde{u}^{\mu,\lambda,\kappa},\widetilde\theta^{\mu,\lambda,\kappa})\|_{L^2}^2+\|\widetilde{b}^{\mu,\lambda,\kappa}-\widetilde{u}^{\mu,\lambda,\kappa}\|_{L^2}^2 \big).\label{G4.39}
\end{align} 
Putting \eqref{G4.38}--\eqref{G4.39} into \eqref{G4.37} gives
\begin{align}\label{G4.40}
&\frac{\mathrm{d}}{\mathrm{d} t} \int_{\mathbb{R}^3} \nabla    \widetilde{\rho}^{\mu,\lambda,\kappa} \cdot \widetilde{u}^{\mu,\lambda,\kappa} \mathrm{d} x+\|\nabla\widetilde{\rho}^{\mu,\lambda,\kappa}\|_{L^2}\nonumber\\
&\quad\lesssim \|\nabla\widetilde{u}^{\mu,\lambda,\kappa}\|_{H^1}^2+\|\nabla\widetilde{\theta}^{\mu,\lambda,\kappa}\|_{L^2}^2+\|\widetilde{b}^{\mu,\lambda,\kappa}-\widetilde{u}^{\mu,\lambda,\kappa}\|_{L^2}^2 +(\eps_0+\eps_1)\|\nabla(\widetilde{a}^{\mu,\lambda,\kappa},\widetilde{u}^{\mu,\lambda,\kappa})\|_{L^2}^2.   
\end{align}
By integrating \eqref{G4.40} over $[0,t]$, and subsequently combining it with \eqref{G4.7}, \eqref{G4.16}, \eqref{G4.24}, and \eqref{G4.36}, we consequently derive the estimate \eqref{G4.35}.
\end{proof}

Applying similar arguments to those in the proof in 
Lemma \ref{L3.5}, taking $\nabla_v$ to \eqref{G4.3}$_4$ 
and subsequently employing the energy method, we   obtain  
\begin{align}\label{G4.41}&\sup_{\tau\in[0,t]}
\|\nabla_v\{\mathbf{I}-\mathbf{P}\}\widetilde{f}^{\mu,\lambda,\kappa}(\tau)\|_{L_{x,v}^2}^2
+ \int_0^t\|\nabla_v \{\mathbf{I}-\mathbf{P}\}\widetilde{f}^{\mu,\lambda,\kappa}(\tau)\|_{\nu}^2{\rm d}
\tau \nonumber\\
&\quad \leq C\sum_{|\alpha|\leq 1}\|\partial^\alpha\{\mathbf{I}-\mathbf{P}\}
\widetilde f^{\mu,\lambda,\kappa} \|_{\nu}^2
+C\|\nabla(\widetilde{a}^{\mu,\lambda,\kappa},\widetilde{b}^{\mu,\lambda,\kappa},
\widetilde{\omega}^{\mu,\lambda,\kappa})\|_{L^2}^2+{\Xi} ^{\mu,\lambda,\kappa}(0).\end{align}  
For the sake of brevity, the detailed derivation is omitted here. 
Subsequently, by combining \eqref{G4.7}, \eqref{G4.16}, \eqref{G4.24}, \eqref{G4.35},
 and \eqref{G4.41} up, we arrive at the following lemma:
 
\begin{lem}\label{L4.5}
It holds that
\begin{align}\nonumber% \label{G4.42}
&\sup_{\tau\in[0,t]}\|\nabla_v\{\mathbf{I}-\mathbf{P}\}
\widetilde{f}^{\mu,\lambda,\kappa}(\tau)\|_{L_{x,v}^2}^2 
+ \int_0^t\|\nabla_v \{\mathbf{I}-\mathbf{P}\}
\widetilde{f}^{\mu,\lambda,\kappa}(\tau)\|_{\nu}^2{\rm d}\tau\nonumber\\
&\quad\leq {\Xi} ^{\mu,\lambda,\kappa}(0)
+C\max\{\mu,\lambda,\kappa\}^2+C \big(\eps_0^\frac{1}{2} 
 +\eps_1^\frac{1}{2}\big){\Xi} ^{\mu,\lambda,\kappa}(t),
\end{align}
where $C>0$ is a constant independent of  $\mu$, $\lambda$,  $\kappa$ and    $t$.
\end{lem}

With Lemmas \ref{L4.1}--\ref{L4.5} in hand, we now proceed to prove Theorem \ref{T1.3}.  
\begin{proof}[Proof of Theorem \ref{T1.3}]
We deduce from Lemmas \ref{L4.1}--\ref{L4.5} that  
\begin{align*}  
{\Xi} ^{\mu,\lambda,\kappa}(t) \leq C \big(\eps_0^\frac{1}{2} 
+ \eps_1^\frac{1}{2}\big){\Xi} ^{\mu,\lambda,\kappa}(t) + C\max\{\mu,\lambda,\kappa\}^2,  
\end{align*}  
and by virtue of the smallness assumption on $\eps_0^\frac{1}{2}$ and $\eps_1^\frac{1}{2}$, we conclude that  
\begin{align*}  
{\Xi} ^{\mu,\lambda,\kappa}(t) \leq C\max\{\mu,\lambda,\kappa\}^2.  
\end{align*}  
Thus, the estimate \eqref{c2} is rigorously established. This completes the proof of Theorem \ref{T1.3}.    
\end{proof}

%\section{Time decay rates of solutions and their spatial derivatives for 
%\texorpdfstring{\(\mu=\lambda=\kappa=0\)}{kappa=0}}
%\texorpdfstring{\(\kappa>0\)}{\kappa>0}}
\section{Time decay rates of solutions and their spatial derivatives for 
 {\(\mu=\lambda=\kappa=0\)} }
%\texorpdfstring{\(\kappa>0\)}{\kappa>0}}

In this section, we aim to derive the time-decay rates for $(\rho, u, \theta, f)$ to the non-isentropic 
compressible Euler-VFP system \eqref{Q2}. To achieve this goal, we divide the total arguments into three parts:   analysis of 
the linearized system, higher-order Lyapunov-type inequality, analysis of the nonlinear system. 

% and nonlinear analysis thereby obtaining the time-decay rates of the decay estimates for $(\rho, u, \theta, f)$.
\subsection{Time decay estimates of  solutions to the linearized  problem}%  of \eqref{Q2}--\eqref{Q2-2}}
Below we investigate the time decay rates of the global classical solution 
to the linearized problem of \eqref{Q2}--\eqref{Q2-2} in $\mathbb{R}^{3}$ (We shall still denote the solution by 
$(\rho, u, \theta, f)$ for notation simplicity). To   this end, we first analyze the following linear system:  
\begin{equation}\label{G5.1}\left\{\begin{aligned}
&\partial_t \rho + {\rm div}u = 0, \\
&\partial_t u  + \nabla \theta + \nabla \rho - (b - u) = 0, \\
&\partial_t \theta + {\rm div}u - \sqrt{3}(\sqrt{2} \omega - \sqrt{3} \theta) = 0, \\
&\partial_t f + v \cdot \nabla f - u \cdot v \sqrt{M} - \theta (|v|^2 - 3) \sqrt{M} - \mathcal{L} f =  {S}_f,
\end{aligned}\right.
\end{equation}  
with the initial data  
\begin{equation}\label{G5.2}
(\rho, u, \theta, f)|_{t=0} = \big( \rho_0(x), u_0(x), \theta_0(x), f_0(x,v) \big), \quad x \in \mathbb{R}^3, \  v \in \mathbb{R}^3.
\end{equation}
In \eqref{G5.1}, the source term  ${S}_f$ takes the form  
\begin{equation*}
{S}_f = {\rm div}_v G - \frac{1}{2} v \cdot G + h,
\end{equation*}  
with  
\begin{equation*}
G = \big(G_1(t,x,v), G_2(t,x,v), G_3(t,x,v)\big) \in \mathbb{R}^3, \quad h = h(t,x,v) \in \mathbb{R},
\end{equation*}  
satisfying  
\begin{equation*}
\mathbf{P}_{0}G = 0, \quad \mathbf{P}_{1}G = 0, \quad \mathbf{P}h = 0.
\end{equation*}   

Now, we define $U=(\rho,u,\theta,f)$ and $U_0=(\rho_0,u_0,\theta_0,f_0)$ to simplify the presentation of the problem \eqref{G5.1}--\eqref{G5.2}. By employing Duhamel's principle, we get  
\begin{align}\label{G5.3}  
U(t)=\mathbb{A}(t)U_{0}+\int_{0}^{t}\mathbb{A}(t-s)\big( 0,0,0,{S}_f(s)\big){\rm d}s,  
\end{align}  
for all $t\geq 0$, where $\mathbb{A}(t)$ denotes the solution operator corresponding to 
the problem \eqref{G5.1}--\eqref{G5.2} with $ {S}_f=0$. 

Applying the operator $\phi_0(D_x)$ to \eqref{G5.3}, we derive  that
\begin{align}\label{G5.4}  
U^L(t) = \mathbb{A}(t)U^L_{0} + \int_{0}^{t}\mathbb{A}(t-s)\big(0,0,0,{S}_f^L(s)\big){\rm d}s,  
\end{align}  
where we have denoted $U^L(t):= \big(\rho^L(t), u^L(t), \theta^L(t), f^L(\theta)\big)$. We recall 
that $\phi_0(\cdot)$ is defined in \eqref{G2.2} and the low-frequency part of a function is given 
in \eqref{G2.1}.

Subsequently, we shall establish the following time-decay properties of $U(t)$ and $U^L(t)$.

\begin{thm}\label{T5.1}
Let $1\leq q\leq2$. For any $\alpha$, $\alpha^{\prime}$ with 
$\alpha^{\prime}\leq\alpha$ and $m=|\alpha-\alpha^{\prime}|$, the following estimates hold for 
  all $t\geq 0$:
\begin{align}\label{G5.5}
& \|\partial^{\alpha}\mathbb{A}(t)U_{0}\|_{\mathcal{Z}_{2}}
\lesssim (1+t )^{-\frac{3}{2}(\frac{1}{q}-\frac{1}{2})-\frac{m}{2}}
\big(\|\partial^{\alpha^{\prime}}U_{0}\|_{\mathcal{Z}_{q}}+\|\partial^{\alpha}U_{0}\|_{\mathcal{Z}_{2}}\big),\\ \label{G5.6}
&\Big\|\partial^{\alpha}\int_{0}^{t}\mathbb{A}(t-s)
\big( 0,0,0,S_f(s)\big){\rm d}s \Big\|_{\mathcal{Z}_{2}}^{2}
  \lesssim\int_0^t(1+t-s)^{-3(\frac{1}{q}-\frac{1}{2})-m}\nonumber\\
&\qquad \quad   \times \big(\big\|\partial^{\alpha^{\prime}}\big(G(s),\nu^{-1/2}h(s)\big)\big\|_{\mathcal{Z}_{q}}^{2}
+\big\|\partial^{\alpha}\big(G(s),\nu^{-1/2}h(s)\big)\big\|_{\mathcal{Z}_{2}}^{2}\big){\rm d}s, 
\end{align}
and
\begin{align}\label{G5.7}
 &\|\partial^{\alpha}\mathbb{A}(t)U_{0}^L\|_{\mathcal{Z}_{2}}
\lesssim  (1+t )^{-\frac{3}{2}(\frac{1}{q}-\frac{1}{2})-\frac{m}{2}}
 \|\partial^{\alpha^{\prime}}U_{0}\|_{\mathcal{Z}_{q}},  \qquad  \\ \label{G5.8}
&\Big\|\partial^{\alpha}\int_{0}^{t}\mathbb{A}(t-s)
\big( 0,0,0,S^L_f(s)\big){\rm d}s\Big\|_{\mathcal{Z}_{2}}^{2}\nonumber\\
&\qquad\quad   \lesssim\int_0^t(1+t-s)^{-3(\frac{1}{q}-\frac{1}{2})-m}   \big\|\partial^{\alpha^{\prime}}\big(G(s),\nu^{-1/2}h(s)\big) \big\|_{\mathcal{Z}_{q}}^{2}
 {\rm d}s. 
\end{align}
\end{thm}

\begin{proof}
Inspired by the hypercoercivity argument in \cite{SK-HMJ-1985}, we prove Theorem \ref{T5.1} by using Fourier analysis techniques.
Applying the Fourier transform with respect to $x$ to the system \eqref{G5.1}, we get
\begin{equation}\label{G5.9}\left\{
\begin{aligned}
&\partial_t \hat{\rho}+i \xi \cdot \hat{u}=0, \\
&\partial_t \hat{u} +i \xi \hat{\theta}+i \xi \hat{\rho}-(\hat{b}-\hat{u})=0, \\
&\partial_t \hat{\theta} +i \xi \hat{u}-\sqrt{3}(\sqrt{2} \hat{\omega}-\sqrt{3} \hat{\theta})=0 ,\\
&\partial_t \hat{f}+i v \cdot \xi \hat{f}-\hat{u} \cdot v \sqrt{M}-\hat{\theta} (|v|^2-3) \sqrt{M}
=\mathcal{L} \hat{f}+{\rm div}_v  \hat{G}-\frac{1}{2} v \cdot \hat{G}+\hat{h}.
\end{aligned}  \right.
\end{equation}
By taking the inner product of \eqref{G5.9}$_4$ with $\Bar{\hat{f}}$ and integrating the result over $\mathbb{R}^3$ in $v$, we arrive at  
\begin{align*}  
& \frac{1}{2} \partial_t\|\hat{f}\|_{L_v^2}^2 + \mathbf{Re} \int_{\mathbb{R}^3} \langle -\mathcal{L}\{\mathbf{I}-\mathbf{P}\} \hat{f}, \{\mathbf{I}-\mathbf{P}\} \hat{f} \rangle \mathrm{d} v + |\hat{b}|^2 + 2|\hat{\omega}|^2 - \mathbf{Re}(\hat{u}|\hat{b}) - \mathbf{Re}(\hat{\theta}|\sqrt{6} \hat{\omega}) \nonumber\\  
& \quad = \mathbf{Re} \int_{\mathbb{R}^3} \Big( {\rm div}_v \hat{G} - \frac{1}{2} v \cdot \hat{G} \Big) \{\mathbf{I}-\mathbf{P}\} \hat{f} \, \mathrm{d} v + \mathbf{Re} \int_{\mathbb{R}^3} (\hat{h}| \{\mathbf{I}-\mathbf{P}\} \hat{f}) \mathrm{d} v,  
\end{align*}  
where the derivation relies on the following key properties:  
\begin{gather*}  
{\rm div}_v G - \frac{1}{2} v \cdot G \perp {\rm Range}(\mathbf{P})  \ \ \text{and} \ \ \int_{\mathbb{R}^3} (\mathbf{P} \hat{f}|\hat{h}) \mathrm{d} v = 0,  
\end{gather*}  
owing to the facts that  
\begin{align*}  
\mathbf{P}_{0}G = 0, \quad \mathbf{P}_{1}G = 0, \quad \mathbf{P}h=0.  
\end{align*}
Then, by exploiting the coercivity of $\mathcal{L}$ in \eqref{G2.6}, applying integration by parts in $v$, and invoking   Cauchy-Schwarz's inequality, it further follows that  
\begin{align}\label{G5.10}
& \frac{1}{2} \partial_t \|\hat{f}\|_{L_v^2}^2 +\frac{1}{2}\bar\lambda|\{\mathbf{I}-\mathbf{P}\} \hat{f}|_\nu^2+|\hat{b}|^2\nonumber\\
 &\quad  +2|\hat{\omega}|^2- \mathbf{Re}(\hat{u}|\hat{b}) - \mathbf{Re}(\hat{\theta}|\sqrt{6} \hat{\omega}) \leq  C\big(\|\hat{G}\|_{L_v^2}^2+\|\nu^{-\frac{1}{2}} \hat{h}\|_{L_v^2}^2\big) .    
\end{align}

Multiplying \eqref{G5.1}$_1$--\eqref{G5.1}$_3$ by $\Bar{\hat{\rho}}$, $\Bar{\hat{u}}$, and $\Bar{\hat{\theta}}$, respectively, and taking the real part of them yield   
\begin{equation*} \left\{
\begin{aligned}
&\frac{1}{2} \partial_t|\hat{\rho}|^2+\mathbf{Re}( i \xi \hat{u}|\hat{\rho})=0, \\
 &\frac{1}{2} \partial_t|\hat{u}|^2+\mathbf{Re}( i \xi \hat{\theta}|\hat{u})+\mathbf{Re}( i \xi \hat{\rho}| \hat{u})+|\hat{u}|^2 -\mathbf{Re}(\hat{b}| \hat{u})=0,\\
& \frac{1}{2} \partial_t|\hat{\theta}|^2+\mathbf{Re}( i \xi \hat{u}| \hat{\theta})+3|\hat{\theta}|^2-\mathbf{Re}(\sqrt{6} \hat{\omega}|\hat{\theta})=0,
 \end{aligned}\right.
\end{equation*}
which, together with \eqref{G5.10} and   Cauchy-Schwarz's inequality, leads to
\begin{align}\label{G5.11}
& \frac{1}{2} \partial_t\big(|\hat{\rho}|^2
+|\hat{u}|^2+|\hat{\theta}|^2+\|\hat{f}\|_{L_v^2}^2\big)+\bar\lambda|\{\mathbf{I}-\mathbf{P}\} \hat{f}|_\nu^2
  \nonumber\\
&\quad+|\hat{b}-\hat{u}|^2+|\sqrt{2} \hat{\omega}-\sqrt{3} \hat{\theta}|^2\leq  C\big(\|\hat{G}\|_{L_v^2}^2+\|\nu^{-\frac{1}{2}} \hat{h}\|_{L_v^2}^2\big).
\end{align}

Next, we handle the estimates of $a$, $b$, and $\omega$. Similar to the derivation of \eqref{G3.21}, we conclude that
\begin{equation*} \left\{
\begin{aligned}
& \partial_t a+{\rm div}   b=0, \\
& \partial_t b_i+\partial_i a+\frac{2}{\sqrt{6}} 
\partial_i \omega+\sum_{j=1}^3 \partial_{j} \Gamma_{i, j}(\{\mathbf{I}-\mathbf{P}\} f)+(b_i-u_i)=0, \\
& \partial_t \omega+\sqrt{2}(\sqrt{2} \omega-\sqrt{3} \theta)
+\frac{2}{\sqrt{6}} {\rm div}  b+\sum_{i=1}^3 \partial_{i} \Upsilon_i(\{\mathbf{I}-\mathbf{P}\} f)=0, \\
& \partial_j b_i+\partial_i b_j-\frac{2}{\sqrt{6}} \delta_{i j}
\Big(\frac{2}{\sqrt{6}} {\rm div}   b+\sum_{i=1}^3 \partial_{ i} 
\Upsilon_i(\{\mathbf{I}-\mathbf{P}\} f)\Big)\\
&\quad =-\partial_t \Gamma_{i, j}(\{\mathbf{I}-\mathbf{P}\} f)
+\Gamma_{i, j} (\mathfrak{l}+{S}_f ), \\
& \frac{5}{3} \partial_i \omega-\frac{2}{\sqrt{6}} \sum_{j=1}^3 \partial_{j} 
\Gamma_{i, j}(\{\mathbf{I}-\mathbf{P}\} f)
=-\partial_t \Upsilon_i(\{\mathbf{I}-\mathbf{P}\} f)+\Upsilon_i(\mathfrak{l}+ {S}_f ),
\end{aligned}\right.
\end{equation*}
where $\mathfrak{l}$ is defined in the same way as   in \eqref{G3.22}.

Then,  by applying the  Fourier transform with respect to $x$ to the above system, one has
\begin{equation*} \left\{
\begin{aligned}
& \partial_t \hat{a}+i \xi \cdot \hat{b}=0, \\
& \partial_t \hat{b}_i+i \xi_i \hat{a}+\frac{2}{\sqrt{6}} i \xi_i \hat{\omega}+\sum_{j=1}^3 i \xi_j \Gamma_{i, j}(\{\mathbf{I}-\mathbf{P}\} \hat{f})+(\hat{b}_i-\hat{u}_i)=0, \\
& \partial_t \hat{\omega}+\sqrt{2}(\sqrt{2} \hat{\omega}-\sqrt{3} \hat{\theta})+\frac{2}{\sqrt{6}} i \xi \cdot \hat{b}+\sum_{i=1}^3 i \xi_i \Upsilon_i(\{\mathbf{I}-\mathbf{P}\} \hat{f})=0, \\
& i \xi_j \hat{b}_i+i \xi_i \hat{b}_j-\frac{2}{\sqrt{6}} \delta_{i j}\Big(\frac{2}{\sqrt{6}} i \xi \cdot \hat{b}+\sum_{i=1}^3 i \xi_i \Upsilon_i(\{\mathbf{I}-\mathbf{P}\} \hat{f})\Big)\\
&\quad =-\partial_t \Gamma_{i, j}(\{\mathbf{I}-\mathbf{P}\} \hat{f})+\Gamma_{i, j}\big(\hat{\mathfrak{l}}+\hat{ {S}_f}\big), \\
& \frac{5}{3} i \xi_i \hat{\omega}-\frac{2}{\sqrt{6}} \sum_{j=1}^3 i \xi_j 
\Gamma_{i, j}(\{\mathbf{I}-\mathbf{P}\} \hat{f})=-\partial_t \Upsilon_i(\{\mathbf{I}-\mathbf{P}\} \hat{f})
+\Upsilon_i\big(\hat{\mathfrak{l}}+\hat{{S}_f}\big) .
\end{aligned}\right.
\end{equation*}
By applying the energy method and performing a direct calculation, we obtain that there exists a positive constant $\lambda_{9} > 0$ such that 
\begin{align}\label{G5.12}
& \partial_t \mathbf{Re} \sum_{i, j=1}^3\big( i \xi_i \hat{b}_j+i \xi_j \hat{b_i} 
| \Gamma_{i, j}(\{\mathbf{I}-\mathbf{P}\} \hat{f})\big)
+\lambda_{9} (|\xi|^2|\hat{b}|^2+|\xi \cdot \hat{b}|^2 ) \nonumber\\
&\quad\leq \sigma|\xi|^2 (|\hat{a}|^2+|\hat{\omega}|^2 )+C (1+|\xi|^2 )|\{\mathbf{I}-\mathbf{P}\} \hat{f}|_{L_v^2}^2\nonumber\\
&\qquad+C\big(|\hat{b}-\hat{u}|^2+\|\hat{G}\|_{L_v^2}^2+\|\nu^{-\frac{1}{2}} \hat{h}\|_{L_v^2}^2\big),
\end{align}
and
\begin{align}\label{G5.13}
& \partial_t \mathbf{Re} \sum_{i=1}^3\big( i \xi_i \hat{\omega} 
| \Upsilon_i(\{\mathbf{I}-\mathbf{P}\} \hat{f})\big)+|\xi|^2|\hat{\omega}|^2  \nonumber\\
&\quad\leq \sigma|\xi|^2|\hat{b}|^2+C (1+|\xi|^2)|\{\mathbf{I}-\mathbf{P}\} \hat{f}|_{L_v^2}^2\nonumber\\
&\qquad  +C\big(|\sqrt{2} \hat{\omega}-\sqrt{3} \hat{\theta}|^2+\|\hat{G}\|_{L_v^2}^2+\|\nu^{-\frac{1}{2}} \hat{h}\|_{L_v^2}^2\big), \\ \label{G5.14}
& \partial_t \mathbf{Re}\Big(\hat{a} \Big\lvert\, i 
\frac{\sqrt{6}}{5} \sum_{j=1}^3 \xi_j \Upsilon_j(\{\mathbf{I}-\mathbf{P}\} 
\hat{f})-i \xi \cdot \hat{b}\Big)+\frac{3}{4}|\xi|^2|\hat{a}|^2  \nonumber\\
&\quad\leq \frac{5}{4}|\xi \cdot \hat{b}|^2 +C (1+|\xi|^2 )|\{\mathbf{I}-\mathbf{P}\} \hat{f}|_{L_v^2}^2\nonumber\\
& \qquad+C\big(|\hat{u}-\hat{b}|^2+\|\hat{G}\|_{L_v^2}^2+\|\nu^{-\frac{1}{2}} \hat{h}\|_{L_v^2}^2\big),  \\\label{G5.15}
& \partial_t \mathbf{Re}(\hat{u}|i \xi \hat{\rho})
+\frac{3}{4}|\xi|^2|\hat{\rho}|^2 \leq C|\xi|^2 (|\hat{\theta}|^2
+|\hat{u}|^2 )+C|\hat{b}-\hat{u}|^2+|\xi\cdot\hat{u}|^2 . 
\end{align}

Now, we define 
\begin{align}\label{G5.16}
 {\mathfrak{E}}_1(\hat{f}):= & \frac{1}{1+|\xi|^2}\Big(\sum_{i, j=1}^3
\big(i \xi_i \hat{b}_j+i \xi_j \hat{b}_i|\Gamma_{i, j}(\{\mathbf{I}-\mathbf{P}\} \hat{f})\big)+\sum_{i=1}^3\big( i \xi_i \hat{\omega}|\Upsilon_i(\{\mathbf{I}-\mathbf{P}\} \hat{f})\big)\Big)
\nonumber\\
& 
+\frac{\iota_1}{1+|\xi|^2}\Big(\hat{a} \Big\lvert\, i \frac{\sqrt{6}}{5} \sum_{j=1}^3 \xi_j \Upsilon_j(\{\mathbf{I}-\mathbf{P}\} 
\hat{f})-i \xi \cdot \hat{b} \Big) ,
\end{align}
where $0<\iota_1\ll 1$ is a constant.

From   \eqref{G5.12}--\eqref{G5.14}, it can be deduced that there exists a positive constant $\lambda_{10} > 0$ such that 
\begin{align}\label{G5.17}
& \partial_t \mathbf{Re} \mathfrak{E}_1(\hat{f})+\frac{\lambda_{10}|\xi|^2}{1+|\xi|^2} (|\hat{a}|^2+|\hat{b}|^2+|\hat{\omega}|^2 )\nonumber \\
&\quad\leq C\big(|\{\mathbf{I}-\mathbf{P}\} \hat{f}|_{L_v^2}^2+|\hat{b}-\hat{u}|^2+|\sqrt{2} \hat{\omega}-\sqrt{3} \hat{\theta}|^2\big)+C\big(\|\hat{G}\|_{L_v^2}^2+\|\nu^{-\frac{1}{2}} \hat{h}\|_{L_v^2}^2\big).
\end{align}
Let us define the time-frequency Lyapunov functional ${\mathcal E}_{M}\big(\hat{U}(t,\xi)\big)$ as  
\begin{align*}  
{\mathcal E}_{M}\big(\hat{U}(t,\xi)\big) := |\hat{\rho}|^2 + |\hat{u}|^2 + |\hat{\theta}|^2 + \|\hat{f}\|_{L_v^2}^2 + \iota_2 \mathbf{Re}\, \mathfrak{E}_1(\hat{f}) + \iota_3 \mathbf{Re} \frac{1}{1+|\xi|^2} (\hat{u}|i \xi \hat{\rho}),  
\end{align*}  
where $\iota_2$ and $\iota_3$ are positive constants satisfying $0 < \iota_2, \iota_3 \ll 1$. By Young's inequality, it yields  
\begin{align*}
{\mathcal E}_{M}\big(\hat{U}(t,\xi)\big)\backsim    |\hat{\rho}|^2 + |\hat{u}|^2 + |\hat{\theta}|^2 + \|\hat{f}\|_{L_v^2}^2. 
\end{align*}
Combining \eqref{G5.11}, \eqref{G5.15} and \eqref{G5.17} up, we infer that there exists a positive constant $\lambda_{11} > 0$ such that 
\begin{align}\label{G5.18}
\partial_{t}\mathcal{E}_{M}\big(\hat{U}(t,\xi)\big)+\frac{\lambda_{11}|\xi|^{2}}{1+|\xi|^{2}}
\mathcal{E}_{M}(\hat{U}(t,\xi))\leq C\big(\|\hat{G}\|_{L_{v}^{2}}^{2}
+\|\nu^{-\frac{1}{2}}\hat{h}\|_{L_{v}^{2}}^{2}\big).
\end{align}
Here, we have utilized the following inequalities: 
\begin{align*}  
|\xi|^2|\hat{b}|^2 + |\hat{b} - \hat{u}|^2 &\geq |\xi|^2|\hat{b}|^2 + |\xi|^2|\hat{b} - \hat{u}|^2 \geq \frac{1}{4}|\xi|^2|\hat{u}|^2, \qquad\qquad\quad\quad\, |\xi| \leq \frac{r_0}{2}, \\  
|\xi|^2|\hat{b}|^2 + |\hat{b} - \hat{u}|^2 &\geq \Big(\frac{r_0}{2}\Big)^2|\hat{b}|^2 + |\hat{b} - \hat{u}|^2 \geq \min\Big\{1, \Big(\frac{r_0}{2}\Big)^2\Big\}\frac{1}{4}|\hat{u}|^2, \quad |\xi| \geq \frac{r_0}{2},  
\end{align*}  
and
\begin{align*}  
2|\xi|^2|\hat{\omega}|^2 + |\sqrt{2}\hat{\omega} - \sqrt{3}\hat{\theta}|^2 \geq&\, 2|\xi|^2|\hat{\omega}|^2 + |\xi|^2|\sqrt{2}\hat{\omega} - \sqrt{3}\hat{\theta}|^2  \geq  \frac{1}{9}|\xi|^2|\hat{\theta}|^2, \quad\qquad\qquad\,\,\,  |\xi| \leq \frac{r_0}{2}, \\  
|\xi|^2|\hat{\omega}|^2 +  |\sqrt{2}\hat{\omega} - \sqrt{3}\hat{\theta}|^2   \geq&\, \Big(\frac{r_0}{2}\Big)^2|\hat{\omega}|^2 +  |\sqrt{2}\hat{\omega} - \sqrt{3}\hat{\theta}|^2  \geq \min\Big\{1,\Big(\frac{r_0}{2}\Big)^2\Big\}\frac{1}{9}|\hat{\theta}|^2,   \quad |\xi| \geq \frac{r_0}{2}, 
\end{align*} 
where $r_0$ is defined in \eqref{G2.2}.
It follows from \eqref{G5.18} and Gronwall's inequality that
\begin{align}\label{G5.19}
\mathcal{E}_{M}\big(\hat{U}(t,\xi)\big)\leq Ce^{\frac{-\lambda_{11}|\xi|^{2}t}{1+|\xi|^{2}}}
\mathcal{E}_{M}\big(\hat{U}(0,\xi)\big)+C\int_{0}^{t}e^{\frac{-\lambda_{11}|\xi|^{2}(t-s)}{1+|\xi|^{2}}}
\big(\|\hat{G}\|_{L_{v}^{2}}^{2}+\|\nu^{-\frac{1}{2}}\hat{h}\|_{L_{v}^{2}}^{2}\big){\rm d}s.
\end{align}

As pointed out  in \cite{SK-HMJ-1985} and \cite[Theorem 3.1]{DFT-2010-CMP}, the anticipated time-decay properties given by \eqref{G5.5}--\eqref{G5.8} can be directly derived from the aforementioned estimate. 
%
%For the sake of conciseness, we  only provide the proof of \eqref{G5.5} and
%\eqref{G5.7} below.
In fact, from \eqref{G5.19}, by setting $ {S}_f(s) = 0 $, we obtain   
\begin{align}\label{G5.20}
\|\partial^{\alpha}\mathbb{A}(t)U_{0}\|_{\mathcal{Z}_{2}}^2\lesssim & \, \int_{|\xi|\leq \frac{r_0}{2}} |\xi|^{2\alpha}e^{\frac{-\lambda_{11}|\xi|^{2}t}{1+|\xi|^{2}}}
\mathcal{E}_{M}\big(\hat{U}(0,\xi)\big)\mathrm{d}\xi\nonumber\\
& +\int_{|\xi|\geq \frac{r_0}{2}} |\xi|^{2\alpha}e^{\frac{-\lambda_{11}|\xi|^{2}t}{1+|\xi|^{2}}}
\mathcal{E}_{M}\big(\hat{U}(0,\xi)\big)\mathrm{d}\xi\nonumber\\
\lesssim&\,   (1+t )^{-3 (\frac{1}{q}-\frac{1}{2} )-m}
\|\partial^{\alpha^{\prime}}U_{0}\|_{\mathcal{Z}_{q}}^2+ e^{-ct}\|\partial^{\alpha}U_{0}\|_{\mathcal{Z}_{2}}^2,
\end{align}
for some constant $c>0$.

To achieve  this,  on the one hand, for the low-frequency part of our solution $\mathbb{A}(t)U_0$, we derive  
\begin{align}  \label{G5.21}
\int_{|\xi|\leq \frac{r_0}{2}} |\xi|^{2\alpha}e^{\frac{-\lambda_{11}|\xi|^{2}t}{1+|\xi|^{2}}}\mathcal{E}_{M}\big(\hat{U}(0,\xi)\big)\mathrm{d}\xi \lesssim&\, \int_{|\xi|\leq \frac{r_0}{2}} |\xi|^{2|\alpha-\alpha^{\prime}|}e^{\frac{-\lambda_{11}|\xi|^{2}t}{1+|\xi|^{2}}}|\xi|^{2\alpha^{\prime}}\mathcal{E}_{M}\big(\hat{U}(0,\xi)\big)\mathrm{d}\xi \nonumber \\  
\lesssim&\, \Big\||\xi|^{m} e^{\frac{-\lambda_{11}|\xi|^{2}t}{2+2|\xi|^{2}}}\Big\|_{L^{\frac{2q}{2-q}}}^2 \|\partial^{\alpha}U_{0}\|_{\mathcal{Z}_{q}}^2.  
\end{align}  
Here, we have applied   Hausdorff-Young's and Hölder's inequalities. By utilizing  a change of variables and the properties of   Gamma function, we obtain  
\begin{align}  \label{G5.22}
\Big\||\xi|^{m} e^{\frac{-\lambda_{11}|\xi|^{2}t}{2+2|\xi|^{2}}}\Big\|_{L^{\frac{2q}{2-q}}}^{\frac{2q}{2-q}} \lesssim&\, \int_{0}^{\infty} |x|^{\frac{m (\frac{2q}{2-q} )+1}{2}}e^{-\lambda_{11}tx (\frac{2q}{2-q} )}\mathrm{d}x \nonumber \\  
\lesssim&\, (1+t)^{\frac{-m (\frac{2q}{2-q} )-3}{2}} \int_{0}^{\infty} x^{\frac{m (\frac{2q}{2-q} )+3}{2}-1}e^{-x}\mathrm{d}x \nonumber \\  
\lesssim&\, (1+t)^{\frac{-m (\frac{2q}{2-q} )-3}{2}}.  
\end{align}  
Substituting \eqref{G5.22} into \eqref{G5.21}, we arrive at  
\begin{align}  \label{G5.23}
\int_{|\xi|\leq \frac{r_0}{2}} |\xi|^{2\alpha}e^{\frac{-\lambda_{11}|\xi|^{2}t}{1+|\xi|^{2}}}\mathcal{E}_{M}\big(\hat{U}(0,\xi)\big)\mathrm{d}\xi \lesssim&\,  (1+t )^{-m-3 (\frac{1}{q}-\frac{1}{2} )} \|\partial^{\alpha}U_{0}\|_{\mathcal{Z}_{q}}^2.  
\end{align}  
This indicates that  \eqref{G5.7} is valid.

On the other hand, for the high-frequency component of our solution $\mathbb{A}(t)U_0$, we observe that  
\begin{align}\label{G5.24}
\int_{|\xi|\geq \frac{r_0}{2}} |\xi|^{2\alpha}e^{\frac{-\lambda_{11}|\xi|^{2}t}{1+|\xi|^{2}}}\mathcal{E}_{M}\big(\hat{U}(0,\xi)\big)\mathrm{d}\xi \lesssim&\,   e^{-ct}\|\partial^{\alpha}U_{0}\|_{\mathcal{Z}_{2}}^2.  
\end{align}
By combining \eqref{G5.23} and \eqref{G5.24}, we conclude  that \eqref{G5.20} holds, thereby implying that \eqref{G5.5} also holds. 

Similarly, by analyzing the low- and high-frequency components of the integral $ \int_{0}^{t}\mathbb{A}(t-s)\big(0,0,0,\linebreak {S}_f^L(s)\big){\rm d}s$, we   further obtain \eqref{G5.6} and \eqref{G5.8}. Thus, we complete the proof of Theorem \ref{T5.1}.
\end{proof}

\subsection{Higher-order Lyapunov-type inequality }
We shall establish a higher order   Lyapunov-type inequality for the low-frequency decomposition of $(\rho, u, \theta, f)$ to the nonlinear problem \eqref{Q2}--\eqref{Q2-2} stated as follows.

% Now, we   state the following higher-order Lyapunov inequality:
\begin{prop}\label{P5.2}
For the classical solution $(\rho, u, \theta, f)$ to the problem  \eqref{Q2}--\eqref{Q2-2}, it holds that  
\begin{align}\label{G5.25}  
&\frac{\rm d}{{\rm d}t}\big( \|\nabla^k(\rho,u,\theta)\|_{L^2}^2+\|\nabla^k f\|_{L_{x,v}^2}^2\big) \nonumber\\  
&\quad+ \lambda_{12}\big(  \|\nabla^k(\rho,u,\theta)\|_{L^2}^2+\|\nabla^k f\|_{L_{x,v}^2}^2\big) \lesssim \|\nabla^k(a^L,b^L,\omega^L,\rho^L,\theta^L)\|_{L^2}^2,  
\end{align}  
for any $k=2,3,4$ and $0 \leq t < T$, where $\lambda_{12}>0$ is independent of $T$.
\end{prop}

The proof of Proposition \ref{P5.2}   relies on Lemmas \ref{L5.3}--\ref{L5.5} below. %, which will be presented  below.

\begin{lem}\label{L5.3}
For the  classical solution $(\rho, u,\theta,f)$ to the Cauchy problem \eqref{Q2}--\eqref{Q2-2}, it holds that 
\begin{align}\label{G5.26}
&\frac{\rm d}{{\rm d}t}\big(\|\nabla^k(\rho,u,\theta)\|_{L^2}^{2} 
+\|\nabla^k f\|_{L_{x,v}^2}^2\big)+\lambda_{13} \big(\|\nabla^k(b-u)\|_{L^2}^{2}
\nonumber\\
& \quad+\|\nabla^k(\sqrt{2}\omega-\sqrt{3}\theta)\|_{L^2}^2
+\|\{\mathbf{I}-\mathbf{P}\}\nabla^k f\|_{\nu}^{2}\big)
\leq C\eps_1^\frac{1}{2} \|\nabla^k(a,b,\rho,\omega)\|_{L^2}^2,
\end{align}
for any $k=2,3,4$ and $0 \leq t < T$, where $\lambda_{13}>0$ is independent of $T$. 
\end{lem}

\begin{rem}
To prove Lemma \ref{L5.3}, we only need to consider the cases where $|\alpha|=k$ 
\emph{(}$k=1,2,3,4$, respectively\emph{)} and $\mu=\lambda=\kappa=0$ in the equality
 \eqref{G3.10}. By performing calculations analogous to those in the 
 proof of Lemma \ref{L3.2}, we can derive \eqref{G5.26}. For the sake of brevity, 
 the detailed steps are omitted here. Nevertheless, it is important to emphasize 
 that we have utilized the dissipative structures of $b-u$ and $\sqrt{2}\omega-\sqrt{3}\theta$, i.e.,
\begin{align*}  
\|\nabla^k u\|_{L^2}^2 &\leq 2\big(\|\nabla^k(b-u)\|_{L^2}^2+\|\nabla^k b\|_{L^2}^2\big), \\  
\|\nabla^k \theta\|_{L^2}^2 &\leq \frac{2}{3}\big(\|\nabla^k(\sqrt{2}\omega-\sqrt{3}\theta)\|_{L^2}^2 + 2\|\nabla^k \omega\|_{L^2}^2\big).  
\end{align*}  
\end{rem}

Similar to \cite{Ww-CMS-2024,LNW-arXiv}, by employing the low-high frequency decomposition method, we aim to derive more refined energy estimates for the classical solutions to \eqref{Q2}--\eqref{Q2-2}. Specifically, we first focus on estimating the dissipation terms $(a^H, b^H, \omega^H)$. By applying the operator $\phi_1(D_x)$ (as defined in \eqref{G2.1}) to  \eqref{G3.21}, we obtain that  
\begin{equation}\label{G5.27}
\left\{\begin{aligned}
&\partial_t a^H+{\rm div} b^H=0,\\
& \partial_t b^H_i+\partial_i a^H+\frac{2}{\sqrt{6}} \partial_i \omega^H+\sum_{j=1}^3 \partial_{j} \Gamma_{i, j}(\{\mathbf{I}-\mathbf{P}\} f^H)+(b_i^H-u_i^H)= (u_i a)^H, \\
& \partial_t \omega^H+\sqrt{2}(\sqrt{2} \omega^H-\sqrt{3} \theta^H)+\frac{2}{\sqrt{6}} {\rm div}  b^H    +\sum_{i=1}^3 \partial_{i} \Upsilon_i(\{\mathbf{I}-\mathbf{P}\} f^H) \\
&\quad=\frac{2}{\sqrt{6}} (u \cdot b)^H+\sqrt{6} (a \theta)^H,     \\
& \partial_j b^H_i+\partial_i b^H_j -\frac{2}{\sqrt{6}} \delta_{i j}\Big(\frac{2}{\sqrt{6}} {\rm div}   b^H-\frac{2}{\sqrt{6}} (u \cdot b)^H+\sum_{i=1}^3 \partial_{ i} \Upsilon_i(\{\mathbf{I}-\mathbf{P}\} f^H)\Big) \\
&\quad = (u_i b_j)^H+(u_j b_i )^H  -\partial_t \Gamma_{i, j}(\{\mathbf{I}-\mathbf{P}\} f^H)+\Gamma_{i, j}\big(\mathfrak{l}^H+\mathfrak{r}^H+\mathfrak{s}^H\big), \\
& \frac{5}{3} \big(\partial_i \omega^H-(\omega u_i)^H-\sqrt{6} (\theta b_i)^H \big)-\frac{2}{\sqrt{6}} \sum_{j=1}^3 \partial_{j} \Gamma_{i, j}(\{\mathbf{I}-\mathbf{P}\} f^H)   \\
&\quad=-\partial_t \Upsilon_i(\{\mathbf{I}-\mathbf{P}\} f^H)+\Upsilon_i\big(\mathfrak{l}^H+\mathfrak{r}^H+\mathfrak{s}^H\big),
\end{aligned}\right.
\end{equation}
for $1\leq i,j\leq 3$, where the terms $\mathfrak{l}$, $\mathfrak{r}$ and $\mathfrak{s}$ are  defined
in the same way  as 
  in \eqref{G3.22}.
We then define the temporal functional $\mathfrak{E}_0^H(t)$ as  % follows:
\begin{align}\label{G5.28}
\mathfrak{E}_0^H(t) :=\, &  \sum_{i, j=1}^3 \int_{\mathbb{R}^3} \nabla^{k-1} (\partial_j b_i^H+\partial_i b_j^H ) \nabla^{k-1} \Gamma_{i, j}(\{\mathbf{I}-\mathbf{P}\} f^H) \mathrm{d} x \nonumber\\
& + \sum_{i=1}^3 \int_{\mathbb{R}^3}  \nabla^{k-1} \partial_i \omega^H  \nabla^{k-1} \Upsilon_i(\{\mathbf{I}-\mathbf{P}\} f^H) \mathrm{d} x \nonumber\\
& +\frac{2}{21} \int_{\mathbb{R}^3} \nabla^{k-1} a^H \nabla^{k-1}\Big(\frac{\sqrt{6}}{5} \sum_{i=1}^3 \partial_{i} \Upsilon_i(\{\mathbf{I}-\mathbf{P}\} f^H)-{\rm div}  b^H\Big) \mathrm{d} x .
\end{align} 

\begin{lem}\label{L5.4}
For  the classical solution $(\rho, u,\theta,f)$ to the Cauchy problem \eqref{Q2}--\eqref{Q2-2}, it holds that 
\begin{align}\label{G5.29}
& \frac{\mathrm{d}}{\mathrm{d} t} \mathfrak{E}_0^H(t)+\lambda_{14}\big(\|\nabla^k a^H\|_{L^2}^2+\|\nabla^k b^H\|_{L^2}^2+\|\nabla^k \omega^H\|_{L^2}^2 \big)\nonumber\\
&\quad \leq  C \big(\|\{\mathbf{I}-\mathbf{P}\}\nabla^k f\|_{L_{x,v}^2}^2+\|\nabla^k(b-u)\|_{L^2}^2+\|\nabla^k(\sqrt{2} \omega-\sqrt{3} \theta)\|_{L^2}^2 \big) \nonumber\\
& \qquad +C\eps_1\|\nabla^k(a,b,\omega)\|_{L^2}^2,
\end{align}
for any $k=2,3,4$ and $0 \leq t <T$, where $\lambda_{14}>0$ is independent of  $T$. 
\end{lem}

\begin{proof}
From \eqref{G5.27}, by  performing energy method, we arrive at 
\begin{align}\label{G5.30}
& \sum_{i, j=1}^3  \|\nabla^{k-1} (\partial_i b^H_j+\partial_j b^H_i ) \|_{L^2}^2-\sum_{i, j=1}^3 \int_{\mathbb{R}^3} \frac{2}{3} \delta_{i j} \nabla^{k-1} {\rm div}  b^H\cdot \nabla^{k-1} (\partial_i b^H_j+\partial_j b^H_i ) \mathrm{d} x \nonumber\\
=&\, -\frac{\mathrm{d}}{\mathrm{d} t} \sum_{i, j=1}^3 \int_{\mathbb{R}^3}\nabla^{k-1} (\partial_i b^H_j+\partial_j b^H_i) \nabla^{k-1} \Gamma_{i, j}(\{\mathbf{I}-\mathbf{P}\} f^H) \mathrm{d} x \nonumber\\
&   +\sum_{i, j=1}^3 \int_{\mathbb{R}^3} \nabla^{k-1} (\partial_i \partial_t b^H_j+\partial_j \partial_t b^H_i ) \nabla^{k-1} \Gamma_{i, j}(\{\mathbf{I}-\mathbf{P}\} f^H) \mathrm{d} x \nonumber\\
&   +\sum_{i, j=1}^3 \int_{\mathbb{R}^3} \nabla^{k-1} (\partial_i b^H_j+\partial_j b^H_i ) \nabla^{k-1}\Big[ (u_i b_j)^H+(u_j b_i )^H-\delta_{i j}\sum_i \partial_i \Upsilon_i(\{\mathbf{I}-\mathbf{P}\} f^H)\Big]\mathrm{d}x\nonumber \\
&  +\int_{\mathbb{R}^3}\nabla^{k-1} (\partial_i b^H_j+\partial_j b^H_i )\Big[-\frac{2}{3}\delta_{i j} (u \cdot b)^H+\Gamma_{i, j}\big(\mathfrak{l}^H+\mathfrak{r}^H+\mathfrak{s}^H\big)\Big] \mathrm{d} x \nonumber\\
\equiv:&\,-\frac{\mathrm{d}}{\mathrm{d} t} \sum_{i, j=1}^3 \int_{\mathbb{R}^3}\nabla^{k-1} (\partial_i b^H_j+\partial_j b^H_i) \nabla^{k-1} \Gamma_{i, j}(\{\mathbf{I}-\mathbf{P}\} f^H) \mathrm{d} x
+\sum_{i=1}^3I_i^H.
\end{align}
By virtue of \eqref{G5.27}$_2$, \eqref{G2.3}, Lemmas \ref{LA.1}--\ref{LA.2} 
and Young's inequality, the term ${I}^H_1$ is estimated as % follows:
\begin{align}\label{G5.31}
{I}^H_1=&\, 2\sum_{i, j=1}^3 \int_{\mathbb{R}^3}   \nabla^{k-1}\partial_j \partial_t {b}_i^H   \nabla^{k-1}\Gamma_{i, j}(\{\mathbf{I}-\mathbf{P}\} {f}^H) \mathrm{d} x\nonumber\\
=&\,  -2\sum_{i, j=1}^3 \int_{\mathbb{R}^3}   \nabla^{k-1} \partial_t {b}^H_i   \nabla^{k-1}\partial_j\Gamma_{i, j}(\{\mathbf{I}-\mathbf{P}\} {f}^H) \mathrm{d} x\nonumber\\
=&\,2\sum_{i, j=1}^3 \int_{\mathbb{R}^3} \Big( \frac{2}{\sqrt{6}} \nabla^{k-1}\partial_i \omega^H+\sum_{j=1}^3 \nabla^{k-1}\partial_{j} \Gamma_{i, j}(\{\mathbf{I}-\mathbf{P}\} {f}^H)\Big)\nabla^{k-1}\partial_j\Gamma_{i, j}(\{\mathbf{I}-\mathbf{P}\} {f}^H) \mathrm{d} x\nonumber\\
&+2\sum_{i, j=1}^3 \int_{\mathbb{R}^3} \nabla^{k-1}\big( \partial_i  {a}^H+({b}^H_i-{u}^H_i)-({u}_i a)^H\big)\nabla^{k-1}\partial_j\Gamma_{i, j}(\{\mathbf{I}-\mathbf{P}\} {f}^H) \mathrm{d} x\nonumber\\
\lesssim&\, \sigma\|\nabla^k({a}^H,\omega^H )\|_{L^2}^2+ \| \{\mathbf{I}-\mathbf{P}\} \nabla^k{f}\|_{L_{x,v}^2}^2+\|\nabla^k(b-u)\|_{L^2}^2+\|\nabla^{k-1} (a,u)\|_{L^6}^2 \|(a,u)\|_{L^3}^2\nonumber\\
\lesssim&\,\sigma \|\nabla^k({a}^H ,\omega^H )\|_{L^2}^2+ \| \{\mathbf{I}-\mathbf{P}\} \nabla^k{f}\|_{L_{x,v}^2}^2+\|\nabla^k(b-u)\|_{L^2}^2+\eps_1\|\nabla^{k} (a,u)\|_{L^2}^2. 
\end{align}
Similarly, for the terms $I_2^H$ and $I_3^H$,  one has % by   direct computations, one has
\begin{align}\label{G5.32}
&|I_2^H+I_3^H|\nonumber\\  
\leq \,& \frac{1}{4} \sum_{i, j=1}^3 \|\nabla^{k-1} (\partial_i b_j^H+\partial_j b_i^H ) \|_{L^2}^2+C     \|\nabla^{k-1} \big((u_i b_j)^H+(u_j b_i)^H \big) \|_{L_{x,v}^2}^2 \nonumber\\
&  +C \|\nabla^{k-1}(u \cdot b)^H \|_{L^2}^2+ \big\|\nabla^{k-1} \Gamma_{i, j}\big(\mathfrak{l}^H+\mathfrak{r}^H+\mathfrak{s}^H\big) \big\|_{L ^2}^2+C\|\nabla^{k-1} \partial_i \Upsilon_i(\{\mathbf{I}-\mathbf{P}\} f^H) \|_{L ^2}^2  \nonumber\\
\leq\,& \frac{1}{4} \sum_{i, j=1}^3\|\nabla^{k-1} (\partial_i b_j^H+\partial_j b_i^H ) \|_{L^2}^2+C\|\{\mathbf{I}-\mathbf{P}\} \nabla^k f\|_{L_{x,v}^2  }^2+\eps_1\|\nabla^k (b,u,\theta)\|_{L^2}^2. 
\end{align}
Besides, it is easy to observe that
\begin{gather} 
 \sum_{i, j=1}^3 \|\nabla^{k-1}(\partial_i  {b}^H _j+\partial_j  {b}^H _i) \|_{L^2}^2 =2 \|\nabla^{k}   {b}^H \|_{L^2}^2+2 \| \nabla^{k-1}{\rm div}  {b}^H \|_{L^2}^2,   \quad \label{G5.33-1}\\
-\sum_{i, j=1}^3 \frac{2}{3} \delta_{i j} \int_{\mathbb{R}^3}   \nabla^{k-1}{\rm div}    {b}^H   \nabla^{k-1} (\partial_i  {b}^H _j+\partial_j  {b}^H _i ) \mathrm{d} x 
=-\frac{4}{3} \|  \nabla^{k-1}{\rm div}    {b}^H  \|_{L^2}^2.\label{G5.33}
\end{gather}   
Inserting \eqref{G5.31}--\eqref{G5.33} into \eqref{G5.30}, we end up with
\begin{align}\label{G5.34}
& \frac{\mathrm{d}}{\mathrm{d} t}  \sum_{i, j=1}^3 \int_{\mathbb{R}^3}  \nabla^{k-1} (\partial_i b^H_j+\partial_j b^H_i )  \nabla^{k-1} \Gamma_{i, j}(\{\mathbf{I}-\mathbf{P}\} f^H) \mathrm{d} x  +  \frac{3}{2} \|\nabla^k b^H\|_{L^2}^2+\frac{1}{6} \|   \nabla^{k-1}{\rm div}b^H \|_{L^2}^2  \nonumber\\
&\quad \lesssim \sigma \|\nabla^k(a^H, \omega^H)\|_{L^2}^2+\| \{\mathbf{I}-\mathbf{P}\} \nabla^k{f}\|_{L_{x,v}^2}^2+\|\nabla^k(b-u)\|_{L^2}^2+\eps_1\|\nabla^{k} (a,b,\theta)\|_{L^2}^2. 
\end{align}

It follows from \eqref{G5.27}$_5$ that
\begin{align}\label{G5.35}
 &\|\nabla^{k-1} \partial_i \omega^H\|_{L^2}^2\nonumber\\
=&\,  -\frac{3}{5} \frac{\mathrm{d}}{\mathrm{d} t} \int_{\mathbb{R}^3} \nabla^{k-1} \partial_i \omega^H \nabla^{k-1} \Upsilon_i(\{\mathbf{I}-\mathbf{P}\} f^H) \mathrm{d} x+\frac{3}{5} \int_{\mathbb{R}^3} \nabla^{k-1} \partial_i \partial_t \omega^H \nabla^{k-1} \Upsilon_i(\{\mathbf{I}-\mathbf{P}\} f^H) \mathrm{d} x \nonumber\\
& +\int_{\mathbb{R}^3} \nabla^{k-1} \partial_i \omega^H \nabla^{k-1} \big((\omega u_i)^H+\sqrt{6} (\theta b_i)^H \big) \mathrm{d} x+\frac{3}{5}\int_{\mathbb{R}^3} \nabla^{k-1} \partial_i \omega \nabla^{k-1}
 \Upsilon_i\big(\mathfrak{l}^H+\mathfrak{r}^H+\mathfrak{s}^H\big) \mathrm{d} x \nonumber\\
& +\frac{\sqrt{6}}{5} \int_{\mathbb{R}^3} \nabla^{k-1} \partial_i \omega^H    \sum_{j=1}^3 \nabla^{k-1}\partial_j \Gamma_{i, j}(\{\mathbf{I}-\mathbf{P}\} f^H)  \mathrm{d} x\nonumber\\
\equiv:& -\frac{3}{5} \frac{\mathrm{d}}{\mathrm{d} t} \int_{\mathbb{R}^3} \nabla^{k-1} \partial_i \omega^H \nabla^{k-1} \Upsilon_i(\{\mathbf{I}-\mathbf{P}\} f^H) \mathrm{d}x+\sum_{i=1}^4J_i^H.
\end{align}
By using \eqref{G5.27}$_3$, \eqref{G2.3}, and Lemmas \ref{LA.1} and \ref{LA.2}, one derives  that 
\begin{align}\label{G5.36}
J_1^H
= &\, \int_{\mathbb{R}^3} \partial_i \nabla^{k-1}\Big(\sqrt{6} (a \theta)^H-\sqrt{2}(\sqrt{2} \omega^H-\sqrt{3} \theta^H)-\frac{2}{\sqrt{6}} {\rm div}  b^H\Big)\nabla^{k-1} \Upsilon_i(\{\mathbf{I}-\mathbf{P}\} f^H) \mathrm{d} x \nonumber\\
&  +\int_{\mathbb{R}^3}\partial_i \nabla^{k-1}\Big(\frac{2}{\sqrt{6}} (u \cdot b)^H-\sum_i \partial_i \Upsilon_i(\{\mathbf{I}-\mathbf{P}\} f^H)\Big)  \nabla^{k-1} \Upsilon_i(\{\mathbf{I}-\mathbf{P}\} f^H) \mathrm{d} x \nonumber\\
\lesssim&\,  \sigma\|\nabla^k b^H\|_{L^2}^2+\|\nabla^k (a,b,u,\theta)\|_{L^2}^2\|(a,b,u,\theta)\|_{L^\infty}^2+\|\nabla^k(\sqrt{2}\omega^H-\sqrt{3}\theta^H)\|_{L^2}^2\nonumber\\
&+ \|\{\mathbf{I}-\mathbf{P}\} \nabla^k f\|_{L_{x,v}^2  }^2\nonumber\\
\lesssim&\,\sigma\|\nabla^k b^H\|_{L^2}^2+\eps_1\|\nabla^k (a,b,u,\theta)\|_{L^2}^2+\|\nabla^k(\sqrt{2}\omega^H-\sqrt{3}\theta^H)\|_{L^2}^2+ \|\{\mathbf{I}-\mathbf{P}\} \nabla^k f\|_{L_{x,v}^2  }^2.
\end{align}
Similarly, for the terms $J_2^H$, $J_3^H$, and $J_4^H$, we have   %by a direct calculation, we have
\begin{align}\label{G5.37}
 |J_2^H+J_3^H+J_4^H| \leq&\, \frac{2}{5}\|\nabla^k \omega^H\|_{L^2}^2+C\|\nabla^k ( b,\omega,u,\theta)\|_{L^2}^2\|( b,\omega,u,\theta)\|_{L^\infty}^2+C \|\{\mathbf{I}-\mathbf{P}\} \nabla^k f\|_{L_{x,v}^2  }^2\nonumber\\
\leq&\,\frac{2}{5}\|\nabla^k \omega^H\|_{L^2}^2+C\eps_1\|\nabla^k ( b,\omega,u,\theta)\|_{L^2}^2+C\|\{\mathbf{I}-\mathbf{P}\} \nabla^k f\|_{L_{x,v}^2  }^2.
\end{align}
Inserting the estimates \eqref{G5.36} and \eqref{G5.37} into \eqref{G5.35} results in  
\begin{align}\label{G5.38}
& \frac{\mathrm{d}}{\mathrm{d} t}  \sum_{ i=1}^3 \int_{\mathbb{R}^3} \nabla^{k-1} \partial_i \omega^H \nabla^{k-1} \Upsilon_i(\{\mathbf{I}-\mathbf{P}\} f ^H)\mathrm{d} x+  \|\nabla^{k}  \omega^H \|_{L^2}^2 \nonumber\\
&\quad \lesssim  \sigma\|\nabla^k b^H\|_{L^2}^2+ \|\{\mathbf{I}-\mathbf{P}\} \nabla^k f\|_{L_{x,v}^2}^2+ \|\nabla^k(\sqrt{2} \omega-\sqrt{3} \theta)\|_{L^2}^2\nonumber\\
& \qquad +\eps_1\|\nabla^k(a,b,\omega,u,\theta)\|_{L^2}^2. 
\end{align}
From \eqref{G5.27}$_2$ and \eqref{G5.27}$_5$, one has
\begin{align}\label{G5.39}
\|\nabla^{k-1} \partial_i a^H \|_{L^2}^2=&\, \frac{\mathrm{d}}{\mathrm{d} t} \int_{\mathbb{R}^3} \nabla^{k-1} \partial_i a^H \nabla^{k-1}\Big(\frac{\sqrt{6}}{5} \Upsilon_i(\{\mathbf{I}-\mathbf{P}\} f^H)-b^H_i\Big) \mathrm{d} x \nonumber\\
& -\int_{\mathbb{R}^3} \nabla^{k-1} \partial_i\partial_t a^H \nabla^{k-1}\Big(\frac{\sqrt{6}}{5} \Upsilon_i(\{\mathbf{I}-\mathbf{P}\} f^H)-b^H_i\Big) \mathrm{d} x \nonumber\\
& +\int_{\mathbb{R}^3} \nabla^{k-1} \partial_i a \nabla^{k-1}\Big( (u^H_i-b^H_i)+(u_i a)^H-\frac{2}{\sqrt{6}} (\omega u_i)^H-2 (\theta b_i)^H\Big){\rm d}x\nonumber \\
& -\int_0^t \nabla^{k-1} \partial_i a \nabla^{k-1}\Big( \frac{7}{5} \sum_{j=1}^3 \partial_j \Gamma_{i j}(\{\mathbf{I}-\mathbf{P}\} f^H)+\frac{\sqrt{6}}{5} \Upsilon_i\big(\mathfrak{l}^H+\mathfrak{r}^H+\mathfrak{s}^H\big)\Big) \mathrm{d}x \nonumber\\
\equiv:\,&\frac{\mathrm{d}}{\mathrm{d} t} \int_{\mathbb{R}^3} \nabla^{k-1} \partial_i a^H \nabla^{k-1}\Big(\frac{\sqrt{6}}{5} \Upsilon_i(\{\mathbf{I}-\mathbf{P}\} f^H)-b^H_i\Big) \mathrm{d} x+\sum_{j=1}^3 K_j^H.
\end{align}
It follows from \eqref{G5.27}$_1$ and Young's inequality that
\begin{align}\label{G5.40}
K_1^H
=&\,\int_{\mathbb{R}^3} \nabla^{k-1} \partial_i ({\rm div} b^H) \nabla^{k-1} \Big(b^H_i-\frac{\sqrt{6}}{5} \Upsilon_i(\{\mathbf{I}-\mathbf{P}\} f^H)\Big) \mathrm{d} x\nonumber\\
\leq&\, \int_{\mathbb{R}^3} \nabla^{k-1} ({\rm div}   b^H) \nabla^{k-1} \partial_i b^H_i \mathrm{d} x+\frac{1}{4} \|\nabla^{k-1} {\rm div}   b^H \|_{L^2}^2+C\|\{\mathbf{I}-\mathbf{P}\} \nabla^k f\|_{L_{x,v}^2}^2.
\end{align}
Then, through  direct calculations, we get  the estimates of $K_2^H$ and $K_3^H$ as follows: 
\begin{align}\label{G5.41}
 |K_2^H+K_3^H |
\leq \,&\frac{1}{4}\|\nabla^{k-1} \partial_i a^H\|^2+C  \|\nabla^{k-1} (b^H_i-u^H_i ) \|_{L^2}^2+C \|\nabla^{k-1}  (\theta b_i)^H \|_{L^2}^2 \nonumber\\
& +C\|\nabla^{k-1} (\omega u_i)^H\|_{L^2}^2+C\big\|\nabla^{k-1} \Upsilon_i\big(\mathfrak{l}^H+\mathfrak{r}^H+\mathfrak{s}^H\big) \big\|_{L ^2}^2 \nonumber\\
&+C\sum_{j=1}^3 \|\nabla^{k-1} \partial_j \Gamma_{i j}(\{\mathbf{I}-\mathbf{P}\} f^H) \|_{L_{x,v}^2}^2+C \|\nabla^{k-1} (u_i a)^H\|_{L^2}^2\nonumber\\
\leq\, &\frac{1}{4} \|\nabla^k a^H \|_{L^2}^2+C\|\nabla^k(b-u)\|_{L^2}^2+C\|\{\mathbf{I}-\mathbf{P}\} \nabla^k f\|_{L_{x,v}^2}^2\nonumber\\
&+C\|\nabla^k (a, b,\omega,u,\theta)\|_{L^2}^2\|(a, b,\omega,u,\theta)\|_{L^\infty}^2\nonumber\\
\leq\, &\frac{1}{4} \|\nabla^k a^H \|_{L^2}^2+C\|\nabla^k(b-u)\|_{L^2}^2+C\|\{\mathbf{I}-\mathbf{P}\} \nabla^k f\|_{L_{x,v}^2}^2\nonumber\\
&+C\eps_1\|\nabla^k (a, b,\omega,u,\theta)\|_{L^2}^2.
\end{align}
By incorporating the estimates \eqref{G5.40} and \eqref{G5.41} into \eqref{G5.39}, we consequently obtain   
\begin{align*}
& \frac{\mathrm{d}}{\mathrm{d} t}   \int_{\mathbb{R}^3} \frac{2}{21} \nabla^{k-1} a^H \nabla^{k-1} \Big(\frac{\sqrt{6}}{5}\sum_{i=1}^3 \Upsilon_i(\{\mathbf{I}-\mathbf{P}\} f^H )-{\rm div}   b^H  \Big) \mathrm{d} x+\frac{1}{14}   \|\nabla^{k} a^H  \|_{L^2}^2\nonumber \\
&\quad \leq \frac{5}{42}   \|\nabla^{k-1} {\rm div}   b^H \|_{L^2}^2+C\|\nabla^{k}(b-u)\|_{L^2}^2\nonumber\\
& \qquad +C\|\{\mathbf{I}-\mathbf{P}\}\nabla^k f\|_{L_{x,v}^2}^2 +C\eps_1\|\nabla^k (a, b,\omega,u,\theta)\|_{L^2}^2.     
\end{align*}
This together with \eqref{G5.34} and \eqref{G5.38} leads to \eqref{G5.29}.
\end{proof}

Finally, we proceed to estimate $\|\nabla^k \rho^H\|_{L^2}$ for $k=2,3,4$.

\begin{lem}\label{L5.5}
For  the classical solution $(\rho, u,\theta,f)$ to the Cauchy problem \eqref{Q2}--\eqref{Q2-2}, it holds that 
\begin{align}\label{G5.42}
& -\frac{\mathrm{d}}{\mathrm{d} t}   \int_{\mathbb{R}^3} \nabla^{k-1} \rho^H \cdot \nabla^{k-1} {\rm div}u \mathrm{d} x+\lambda_{15}\|\nabla^{k}  \rho^H\|_{L^2}^2\nonumber \\
&\quad  \lesssim  \eps_1\|\nabla^{k} (a,\rho)\|_{L^2}^2+\|\nabla^{k}(b-u)\|_{L^2}^2+\|\nabla  ^{k}u\|_{L^2}^2+\|\nabla^{k} \rho^L\|_{L^2}^2+\|\nabla^k \theta\|_{L^2}^2 ,
\end{align}
for any $k=2,3,4$ and $0 \leq t < T$, where $\lambda_{15}>0$ is independent of  $T$. 
\end{lem}
\begin{proof}
By \eqref{Q2}$_2$, we have  
\begin{align}  \label{G5.43}
 \nabla \rho  =&\, -\partial_t u-\nabla\theta+h(\rho)(b - u)+h(\rho)\big(\rho\nabla\rho-u\cdot\nabla u-\theta\cdot\nabla\rho- au\big),  
\end{align}  
where  $h(\rho)$ is defined as in \eqref{define:g,h}, i.e., 
\begin{gather*}  
h(\rho) = \frac{1}{1+\rho}.
\end{gather*}  
Multiplying $\nabla^{k-1}$\eqref{G5.43} by $\nabla^k \rho^H$, utilizing the decomposition $\rho= \rho^L  + \rho^H $, and subsequently integrating   over $\mathbb{R}^3_x$, one obtains
\begin{align}\label{G5.44}
 \|\nabla^k\rho^H\|_{L^2}^2 
=\,&-\int_{\mathbb{R}^3} \nabla^k \rho^H\cdot \nabla^{k-1}\partial_t u{\rm d}x-\int_{\mathbb{R}^3} \nabla^k \rho^H\cdot \nabla^k \rho^L{\rm d}x\nonumber\\
&-\int_{\mathbb{R}^3}\nabla^k\rho^H\cdot\nabla^k \theta{\rm d}x+\int_{\mathbb{R}^3} \nabla^k \rho^H\cdot \nabla^{k-1}[h(\rho)(b-u)]{\rm d}x\nonumber\\
&+\int_{\mathbb{R}^3} \nabla^k \rho^H\cdot\nabla^{k-1}\big[h(\rho)(\rho\nabla\rho-u\cdot\nabla u-au-\theta\nabla\rho)\big]{\rm d}x\nonumber\\
\equiv:\,& \sum_{i=1}^5 R_i^H.
\end{align} 
For the first term $R_1^H$, by using \eqref{Q2}$_1$, integration by parts, Young's inequality, and Lemmas \ref{LA.1} and \ref{LA.2}, we have
\begin{align}\label{G5.45}
R_1^H=&\, -\frac{{\rm d}}{{\rm d}t}\int_{\mathbb{R}^3} \nabla^k \rho^H\cdot\nabla^{k-1} u{\rm d}x +\int_{\mathbb{R}^3} \nabla^k
\partial_t \rho^H\cdot\nabla^{k-1}u {\rm d}x\nonumber\\
=&\,-\frac{{\rm d}}{{\rm d}t}\int_{\mathbb{R}^3} \nabla^k \rho^H\cdot\nabla^{k-1} u{\rm d}x+\int_{\mathbb{R}^3}\nabla^{k-1}\big[(\rho+1){\rm div}u+u\cdot\nabla\rho      \big]^H\cdot\nabla^k u{\rm d}x\nonumber\\
\leq&\,-\frac{{\rm d}}{{\rm d}t}\int_{\mathbb{R}^3} \nabla^k \rho^H\cdot\nabla^{k-1} u{\rm d}x+C\|\nabla^{k-1} {\rm div}u\|_{L^2}^2+C\eps_1\|\nabla^k (\rho, u)\|_{L^2}^2.
\end{align}
Continuing from \eqref{G2.3}, by applying   Young's inequality and Lemmas \ref{LA.1} and \ref{LA.2}, we arrive at  
\begin{align}
|R_2^H|\leq&\, \frac{1}{6} \|\nabla^k \rho^H\|_{L^2}^2+C \|\nabla^k \rho^L\|_{L^2}^2, \label{G5.46-1} \\
|R^H_3|\leq &\,\frac{1}{6} \|\nabla^k \rho^H\|_{L^2}^2+C\|\nabla^k \theta\|_{L^2}^2.\label{G5.46}
\end{align}
In a similar manner to the above, we have
\begin{align}
|R^H_4| =&\, -\int_{\mathbb{R}^3} \nabla^{k-1}\rho^H \cdot\nabla^{k} [h(\rho)(b-u)]{\rm d}x\nonumber\\
\leq&\, C\|\nabla^{k-1}\rho^H\|_{L^2}\big( \|\nabla^k(b-u)\|_{L^2}\|h(\rho)\|_{L^\infty}+\|b-u\|_{L^\infty}\|\nabla^k h(\rho)\|_{L^2}    \big)\nonumber\\
\leq&\, \frac{1}{6}\|\nabla^k \rho^H\|_{L^2}^2+C\|\nabla^k(b-u)\|_{L^2}^2+C\eps_1 \|\nabla^k \rho\|_{L^2}^2,\label{G5.47-1}\\
|R^H_5|\leq&\,  \frac{1}{6}\|\nabla^{k} \rho^H\|_{L^2}^2+C\|\nabla^{k-1}\rho\|_{L^6}^2\big(\|(\rho,u,\theta)\|_{L^6}^2\|\nabla(\rho,u,\theta)\|_{L^6}^2+\|a\|_{L^6}^2\|u\|_{L^6}^2 
 \big)\nonumber\\
&+C\|h(\rho)\|_{L^\infty}^2\big(\|\nabla^{k-1}(\rho,u,\theta)\|_{L^6}^2\|\nabla(\rho,u,\theta)\|_{L^3}^2+\|\nabla^k(\rho,u,\theta)\|_{L^2}^2\|(\rho,u,\theta)\|_{L^\infty}^2  \big) \nonumber\\
&+C\|h(\rho)\|_{L^\infty}^2\|\nabla^{k-1}(a,u)\|_{L^6}^2\|(a,u)\|_{L^3}^2\nonumber\\
\leq&\, \frac{1}{6}\|\nabla^{k} \rho^H\|_{L^2}^2+C\eps_1 \|\nabla^k(a,\rho,u,\theta)\|_{L^2}^2.\label{G5.48}
\end{align}
Putting the estimates \eqref{G5.45}--\eqref{G5.48} into \eqref{G5.44}, we ultimately obtain the desired estimate \eqref{G5.42}.
\end{proof}

Now, we present the  proof of Proposition \ref{P5.2} as follows.  
\begin{proof}[Proof of Proposition \ref{P5.2}]
To begin with, we define the following energy and dissipation functionals:   
\begin{align}\label{G5.49}
\mathcal{E}_2(t):=\,&  \|\nabla^k(\rho,u,\theta)\|_{L^2}^{2} 
+\|\nabla^k f\|_{L_{x,v}^2}^2  +\tau_4\mathfrak{E}_0^H(t)
+\tau_5\int_{\mathbb{R}^3}\nabla^{k-1}\rho^H\cdot\nabla^{k-1}{\rm div}u{\rm d}x,\\ \label{G5.50}
\mathcal{D}_{1}(t):=\,&\|\nabla^k(b-u)\|_{L^2}^{2}
+\|\nabla^k(\sqrt{2}\omega-\sqrt{3}\theta)\|_{L^2}^2\nonumber\\
&+\|\{\mathbf{I}-\mathbf{P}\}\nabla^k f\|_{\nu}^{2}+\|\nabla^k(a^H,b^H,\omega^H,\rho^H)\|_{L^2}^2,
\end{align}
where $0<\tau_4$, $\tau_5<1$ are small constants.
From the definitions of $\mathfrak{E}_0^H(t)$ in \eqref{G5.28} and $\mathcal{E}_2(t)$ in \eqref{G5.49}, and by applying Young's inequality together with \eqref{G2.3}, we conclude that    
\begin{align}\label{G5.51}
 \mathcal{E}_2(t)\backsim   \|\nabla^k(\rho,u,\theta)\|_{L^2}^{2} +\|\nabla^k f\|_{L_{x,v}^2}^2,
\end{align}
for $k=2,3,4$.

Adding \eqref{G5.26} to $\tau_4\times$ \eqref{G5.29} and $\tau_5\times$ \eqref{G5.42}, we deduce that
\begin{align}\label{G5.52}
\frac{{\rm d}}{{\rm d}t}\mathcal{E}_2(t)+\lambda_{16} \mathcal{D}_{1}(t)\lesssim \|\nabla^k(a^L,b^L,\omega^L,\rho^L)\|_{L^2}^2,   
\end{align}
where $\lambda_{16} > 0$ satisfies $\lambda_{16} < \min\{\lambda_{13}, \tau_4\lambda_{14}, \tau_5\lambda_{15}\}$.
Here, we have used the following properties of the dissipative structures of $b-u$ and $\sqrt{2}\omega-\sqrt{3}\theta$:
\begin{align*}
\|\nabla^k u\|_{L^2}^2 \leq&\, 2\big(\|\nabla^k(b-u)\|_{L^2}^2+\|\nabla^k b\|_{L^2}^2\big)\\
\|\nabla^k \theta\|_{L^2}^2 \leq&\, \frac{2}{3}\big(\|\nabla^k(\sqrt{2}\omega-\sqrt{3}\theta)\|_{L^2}^2 + 2\|\nabla^k \omega\|_{L^2}^2\big), 
\end{align*}
and the properties of \eqref{G2.1} and \eqref{G2.3}:
\begin{align*}
\|\nabla^k (a,b,\omega,\rho)\|_{L^2}^2\leq&\, C\|\nabla^k (a^H,b^H,\omega^H,\rho^H)\|_{L^2}^2+C\|\nabla^k (a^L,b^L,\omega^L,\rho^L)\|_{L^2}^2.    
\end{align*}

Adding $\lambda_{16}  \|\nabla^k (a^L,b^L,\omega^L,\rho^L)\|_{L^2}^2$ to both  sides
of \eqref{G5.52}, we find that 
\begin{align}\label{G5.53}
 \frac{{\rm d}}{{\rm d}t}\mathcal{E}_2(t)+\lambda_{17} \mathcal{E}_2(t)\lesssim \|\nabla^k(a^L,b^L,\omega^L,\rho^L)\|_{L^2}^2,   
\end{align}
for some $\lambda_{17}>0$. 

Combining \eqref{G5.51} and \eqref{G5.53}, we consequently obtain \eqref{G5.25}. Therefore, the proof of Proposition \ref{P5.2} is completed.
\end{proof}

\subsection{Time   decay estimates of  solutions to the nonlinear problem}
Now, we focus on deriving the time-decay rates of $(\rho, u, \theta, f)$ to the nonlinear problem  \eqref{Q2}--\eqref{Q2-2} under the assumption stated in Theorem \ref{T1.4}.   By utilizing Theorem \ref{T5.1},
and the higher-order Lyapunov-type inequality in Proposition \ref{P5.2},
we deduce the optimal time-decay rates of classical solutions and their
gradients to the nonlinear problem \eqref{Q2}--\eqref{Q2-2}.

In what follows, we aim to establish the optimal time-decay estimates of $\|\nabla^k(\rho,u,\theta)\|_{L^2}^{2} +\|\nabla^k f\|_{L_{x,v}^2}^2$ with $k=0, 1 $. To achieve this, we rewrite the solution $U(t)$ to the problem  \eqref{Q2}--\eqref{Q2-2} as   
\begin{align}\label{G5.54}  
U(t)=\mathbb{A}(t)U_{0}+\int_{0}^{t}\mathbb{A}(t-s)\big( S_\rho(s),S_u(s), S_\theta(s), S_f(s)\big) {\rm d}s,
\end{align}  
where the  source term $S_f$ takes the form  
\begin{align*}  
S_f := &\,  -u\cdot\nabla_vf+\frac{1}{2}u\cdot vf + (|v|^2-3)\sqrt{M}\theta+ \frac{\theta}{\sqrt{M}}\Delta_v(\sqrt{M}f)\nonumber\\
=\,&-u\cdot\nabla_v\{\mathbf{I}-\mathbf{P}_0-\mathbf{P}_1\}f+\frac{1}{2} u\cdot v\{\mathbf{I}-\mathbf{P}_0-\mathbf{P}_1\}f- u\cdot\nabla_v  \mathbf{P}_0 f\nonumber\\
&-u\cdot\nabla_v \mathbf{P}_1 f+ \frac{1}{2}u\cdot v \mathbf{P}_0 f+\frac{1}{2}u\cdot v \mathbf{P}_1 f+ (|v|^2-3)\sqrt{M}\theta\nonumber\\ &+\mathbf{P}_2\Big(\frac{\theta}{\sqrt{M}}\Delta_v(\sqrt{M}f)\Big)
+\{\mathbf{I}-\mathbf{P}_2\}\Big(\frac{\theta}{\sqrt{M}}\Delta_v(\sqrt{M}f)\Big)\nonumber\\
=:&\, {\rm div}_v G-\frac{1}{2}v\cdot G+h+au\cdot v\sqrt{M}-u\cdot b\sqrt{M}+a\theta(|v|^2-3)\sqrt{M}+u\cdot v(b\cdot v\sqrt{M}),
\end{align*}
with  
\begin{align*}  
G = -u \{\mathbf{I} - \mathbf{P}_{\mathbf{0}} - \mathbf{P}_{\mathbf{1}}\} f, \quad h = \frac{1}{2} \theta \{\mathbf{I} - \mathbf{P}_2\} \Big[\frac{1}{\sqrt{M}} \nabla_v \cdot \big(\sqrt{M} (2 \nabla_v f - v f)\big)\Big].  
\end{align*}  
And the other source terms $S_\rho$, $S_u$, and $S_\theta$ are given by 
\begin{align*}  
S_\rho :=&\, -\rho {\rm div} u - u \cdot \nabla \rho, \\  
S_u :=&\, -u \cdot \nabla u + \frac{1}{1+\rho} \big(\rho \nabla \rho - \theta \nabla \rho + \rho (b-u)\big) - \frac{a u}{1+\rho}, \\  
S_\theta :=&\, -u \cdot \nabla \theta - \theta {\rm div} u - \frac{\sqrt{3}}{1+\rho} \big(\rho (\sqrt{2} \omega - \sqrt{3} \theta)\big) - \frac{1}{1+\rho} \big(u \cdot (b-u)\big)- \frac{u \cdot b}{1+\rho} \\  
&\quad  - \frac{3a \theta}{1+\rho}+ \frac{a |u|^2 }{1+\rho}  .  
\end{align*} 
 
Thus, by applying $\nabla^l\phi_0(D_x)$ to \eqref{G5.54}, 
we decompose $\nabla^l U^L(t)$ into the following form:
\begin{align}\label{G5.55}  
 {\nabla^l} U^L(t) 
=&\,\mathbb{A}(t){\nabla^l}U^L_0+\int_0^t\mathbb{A}(t-s)\Big(0,0,0,{\nabla^l}
\Big({\rm div}_v G-\frac{1}{2}v\cdot G\Big)^L\Big)\mathrm{d}s\nonumber\\
&+\int_0^t\mathbb{A}(t-s) \big( 0,0,0,{\nabla^l}
\big(u\cdot av\sqrt{M}-u\cdot b\sqrt{M}\nonumber\\
& \ \ \ \ \ \ +u\cdot v(b\cdot v\sqrt{M})
+a\theta(|v|^2-3)\sqrt{M} \big)^L\big)\mathrm{d}s \nonumber\\
&+\int_0^t\mathbb{A}(t-s)(0,0,0,{\nabla^l}h^L)\mathrm{d}s
+\int_0^t\mathbb{A}(t-s)({\nabla^l}S^L_\rho(s),0,0,0)\mathrm{d}s\nonumber\\
&+\int_0^t\mathbb{A}(t-s)(0,{\nabla^
l}S^L_u(s),0,0)\mathrm{d}s+\int_0^t\mathbb{A}(t-s)(0,0,{\nabla^l}S^L_\theta(s),0)\mathrm{d}s\nonumber\\
\equiv:&\,\sum_{j=1}^7L_j(t). 
\end{align}  

Now, we handle with the terms $ L_i(t) (i=1,\dots, 7)$ in turn. 
\emph{Without the assumption that the quantity $\|(\rho_0, u_0, \theta_0, f_0)\|_{\mathcal{Z}_1}$  
  is small}, below we first establish slower decay rates for $(\rho, u, \theta, f)$ 
  and their third-order spatial derivatives in $\mathcal{Z}_2$-norm,
   by solely leveraging the boundedness of the $\mathcal{Z}_1$-norm of $(\rho, u, \theta, f)$.
We define
\begin{align}\label{G5.56}
\mathcal{E}_{\infty}(t):=\sup_{0\leq s\leq t}\,&\bigg\{\sum_{l=0}^1(1+s)^{\frac{3}{4}+l}\big(\|\nabla^l(\rho,u,\theta)\|_{L^2}^2 +\|\nabla^l f\|_{L_{x,v}^2}^2\big)\nonumber\\
&+\sum_{l=2}^4(1+s)^{\frac{7}{4}}\big(\|\nabla^l(\rho,u,\theta)\|_{L^2}^2+\|\nabla^l f\|_{L_{x,v}^2}^2 \big)   \bigg\}.
\end{align}
For any $l=0,1$, by invoking Theorem \ref{T5.1}, setting $q=\frac{4}{3}$ in \eqref{G5.7}, and applying the interpolation inequality
and Lemma \ref{LA.3}, we obtain   
\begin{align}\label{G5.57}
\|L_1(t)\|_{\mathcal{Z}_2}\lesssim (1+t)^{-\frac{3}{8}-\frac{l}{2}}\|U_0\|_{\mathcal{Z}_{\frac{4}{3}}}\lesssim (1+t)^{-\frac{3}{8}-\frac{l}{2}}\|U_0\|_{\mathcal{Z}_{1}}^\frac{1}{2}\|U_0\|_{\mathcal{Z}_{2}}^\frac{1}{2}.    
\end{align}
Regarding the term $L_2(t)$, for any $l=0,1$, by applying $\mathbf{P}_0G=0$, $\mathbf{P}_1 G=0$ and substituting $q=1$, $m=l$ as well as $q=\frac{3}{2}$, $m=0$ into \eqref{G5.8} respectively, we
compute
\begin{align}\label{G5.58}
\| L_2(t)\|_{\mathcal{Z}_2}^2\lesssim\,& \int_0^{\frac{t}{2}}(1+t-s)^{-\frac{3}{2}-l}
\| u\{\mathbf{I}-\mathbf{P}_0-\mathbf{P}_1\}f\|_{\mathcal{Z}_1}^2\mathrm{d}s\nonumber\\
&+ \int_{\frac{t}{2}}^t(1+t-s)^{-\frac{1}{2}}
\|\nabla^l ( u\{\mathbf{I}-\mathbf{P}_0-\mathbf{P}_1\}f)\|_{\mathcal{Z}_{\frac{3}{2}}}^2\mathrm{d}s\nonumber\\
\lesssim &\, \mathcal{E}_{ \infty}^2(t)\int_0^{\frac{t}{2}}(1+t-s)^{-\frac{3}{2}-l}
(1+s)^{-\frac{3}{2}}\mathrm{d}s +\mathcal{E}_{ \infty}^2(t)\int_{\frac{t}{2}}^t(1+t-s)^{-\frac{1}{2}}(1+s)^{-\frac{5}{2}-l}\mathrm{d}s\nonumber\\
\lesssim &\,(1+t)^{-\frac{3}{2}-l}\mathcal{E}_{\infty}^2(t).
\end{align}
Similarly, taking $q=1$, $m=l$ and $q=\frac{3}{2}$, $m=0$ in \eqref{G5.7} in Theorem \ref{T5.1} respectively, and  using Lemmas \ref{LA.1}--\ref{LA.3} along with the inequality $v^k\sqrt{M}\lesssim 1$ for any integer $k\geq 0$,  we derive that 
\begin{align}\label{G5.59}
&\| L_3(t)\|_{\mathcal{Z}_2}\nonumber\\
\lesssim&\, \int_0^{\frac{t}{2}}(1+t-s)^{-\frac{3}{4}-\frac{l}{2}}\|u\cdot av\sqrt{M}-u\cdot b\sqrt{M}+u\cdot v(b\cdot v\sqrt{M})+a\theta(|v|^2-3)\sqrt{M}\|_{\mathcal{Z}_1}\mathrm{d}s\nonumber\\
&+ \int_{\frac{t}{2}}^{t}(1+t-s)^{-\frac{1}{4}}\big\|\nabla^l\big(u\cdot av\sqrt{M}-u\cdot b\sqrt{M}+u\cdot v(b\cdot v\sqrt{M})+a\theta(|v|^2-3)\sqrt{M}\big)\big\|_{\mathcal{Z}_{\frac{3}{2}}}\mathrm{d}s\nonumber\\
\lesssim&\, \int_0^{\frac{t}{2}}(1+t-s)^{-\frac{3}{4}-\frac{l}{2}}  \|(a,b,u,\theta)\|_{L^2}^2       
  \mathrm{d}s +\int_{\frac{t}{2}}^t(1+t-s)^{-\frac{1}{4}} \|(a,b,u,\theta)\|_{L^\infty}  \|\nabla^l (a,b,u,\theta)\|_{L^2}\mathrm{d}s\nonumber\\
\lesssim& \,\mathcal{E}_{ \infty}(t)\int_0^{\frac{t}{2}}(1+t-s)^{-\frac{3}{4}-\frac{l}{2}}(1+s)^{-\frac{3}{4}}
  \mathrm{d}s + \mathcal{E}_{ \infty}(t)\int_{\frac{t}{2}}^{t}(1+t-s)^{-\frac{1}{4}}(1+s)^{-\frac{5}{4}-\frac{l}{2}}
  \mathrm{d}s\nonumber\\
\lesssim&\, (1+t)^{-\frac{1}{2}-\frac{l}{2}}\mathcal{E}_{ \infty}(t), 
\end{align}
for $l=0,1$.

For the term $ L_4(t) $, it follows from \eqref{G5.8} in Theorem \ref{T5.1} and the condition $\mathbf{P}h=0$ that  
\begin{align}\label{G5.60}
\| L_4(t)\|_{\mathcal{Z}_2}^2   \lesssim&\,  \int_0^{\frac{t}{2}}(1+t-s)^{-\frac{3}{2}-l}
\Big\|\nu^{-\frac{1}{2}}  \theta \{\mathbf{I} - \mathbf{P}_2\} \Big[\frac{1}{\sqrt{M}} \nabla_v \cdot \big(\sqrt{M} (2 \nabla_v f - v f)\big)\Big]\Big\|_{\mathcal{Z}_1}^2\mathrm{d}s\nonumber\\
&+  \int_{\frac{t}{2}}^{t}(1+t-s)^{-\frac{1}{2}-l}
\Big\|\nu^{-\frac{1}{2}}  \theta \{\mathbf{I} - \mathbf{P}_2\} \Big[\frac{1}{\sqrt{M}} \nabla_v \cdot \big(\sqrt{M} (2 \nabla_v f - v f)\big)\Big]\Big\|_{\mathcal{Z}_{\frac{3}{2}}}^2\mathrm{d}s\nonumber\\
\lesssim&\,  (1+t)^{-\frac{3}{2}-l}\mathcal{E}_{\infty}(t)\int_0^{\frac{t}{2}}
\mathcal{D}(\rho,u,\theta,f)(s)\mathrm{d}s + (1+t)^{-\frac{5}{4}-l}\mathcal{E}_{\infty}(t)\int_{\frac{t}{2}}^t\mathcal{D}(\rho,u,\theta,f)(s)\mathrm{d}s\nonumber\\
\lesssim\,& \mathcal{E}_0(1+t)^{-\frac{5}{4}-l}\mathcal{E}_{\infty}(t),
\end{align}
for $l=0,1$.

For the remaining terms $L_5(t)$, $L_6(t)$, and $L_7(t)$, taking $q=1$, $m=0$ and $q=\frac{3}{2}$, $m=0$ in  \eqref{G5.7} in Theorem \ref{T5.1}, respectively, we get
\begin{align}\label{G5.61}
&\|L_5(t)\|_{\mathcal{Z}_2}+\|L_6(t)\|_{\mathcal{Z}_2}+\|L_7(t)\|_{\mathcal{Z}_2} \nonumber\\
\lesssim &\,
\int_0^{\frac{t}{2}}(1+t-s)^{-\frac{3}{4}-\frac{l}{2
}}\big\|\big(S_\rho (s),S_u(s),S_\theta(s)\big)\big\|_{\mathcal{Z}_1}\mathrm{d}s\nonumber\\
&+\int_{\frac{t}{2}}^{t}(1+t-s)^{-\frac{1}{4}-\frac{l}{2}}\big\|\big(S_\rho (s),S_u(s),S_\theta(s)\big)\big\|_{\mathcal{Z}_{\frac{3}{2}}}\mathrm{d}s\nonumber\\
\lesssim\,&\int_0^{\frac{t}{2}}(1+t-s)^{-\frac{3}{4}-\frac{l}{2}}\big(\|(a,b,\omega,\rho,u,\theta)\|_{H^1}\|\nabla (a,b,\omega,\rho,u,\theta)\|_{H^1}+\|(a,b,u,\theta)\|_{L^2}^2 \big)\mathrm{d}s\nonumber\\
&+\int_{\frac{t}{2}}^t(1+t-s)^{-\frac{1}{4}-\frac{l}{2}} \|(a,b,\omega,\rho,u,\theta)\|_{L^6}\|\nabla(a,b,\omega,\rho,u,\theta)\|_{H^1} \mathrm{d}s\nonumber\\
&+\int_{\frac{t}{2}}^t(1+t-s)^{-\frac{1}{4}-\frac{l}{2}} \|(a,b,u,\theta)\|_{L^6}\|(a,b,u,\theta)\|_{L^3} \mathrm{d}s\nonumber\\
\lesssim&\,\mathcal{E}_{ \infty}(t)\int_0^{\frac{t}{2}}(1+t-s)^{-\frac{3}{4}-\frac{l}{2}}(1+s)^{-\frac{3}{4}}\mathrm{d}s +\mathcal{E}_{ \infty}(t)\int_{\frac{t}{2}}^t(1+t-s)^{-\frac{1}{4}-\frac{l}{2}}(1+s)^{-\frac{5}{4}}\mathrm{d}s\nonumber\\
\lesssim&\, (1+t)^{-\frac{1}{2}-\frac{l}{2}}\mathcal{E}_{ \infty}(t),
\end{align}
for $l=0,1$. 

By plugging the estimates  \eqref{G5.57}--\eqref{G5.61} into \eqref{G5.55}, 
we   therefore infer that  
\begin{align}\label{G5.62}
\|\nabla^l U^L(t)\|_{\mathcal{Z}_2}\lesssim (1+t)^{-\frac{3}{8}-\frac{l}{2}}
\big(   \|U_0\|_{\mathcal{Z}_1}^{\frac{1}{2}}\|U_0\|_{\mathcal{Z}_2}^{\frac{1}{2}}
+\mathcal{E}^{\frac{1}{2}}_0\mathcal{E}^{\frac{1}{2}}_{ \infty}(t)+\mathcal{E}_{ \infty}(t)\big),  
\end{align}
for $l=0,1$.

From \eqref{G5.25} in Proposition \ref{P5.2} and Gronwall's inequality, one gets
\begin{align}\label{G5.63}
\|\nabla^k(\rho,u,\theta)(t)\|_{L^2}^2+\|\nabla^k f(t)\|_{L_{x,v}^2}^2
\lesssim\, & e^{-\lambda_{12}t}\big( \|\nabla^k(\rho,u,\theta)(0)\|_{L^2}^2
+\|\nabla^k f(0)\|_{L_{x,v}^2}^2\big) \nonumber\\
& +\int_{0}^te^{-\lambda_{12}(t-s)} \|\nabla^k U^L(s)\|_{\mathcal{Z}_2}^2\mathrm{d}s,  
\end{align}
for $k=2,3,4$.
Then, the third inequality in \eqref{G2.3} implies that for any $k=2,3,4$,  
\begin{align}  \label{G5.64}
\|\nabla^k U^L(t)\|_{\mathcal{Z}_2}^2 \lesssim \|\nabla^l U^L(t)\|_{\mathcal{Z}_2}^2,  
\end{align}  
for any $l=0,1$.  

Putting   \eqref{G5.64} into \eqref{G5.63} gives rise to
\begin{align}\label{G5.65}
 \|\nabla^k(\rho,u,\theta)(t)\|_{L^2}^2+\|\nabla^k f(t)\|_{L_{x,v}^2}^2
\lesssim&\, e^{-\lambda_{12}t}\big( \|\nabla^k(\rho,u,\theta)(0)\|_{L^2}^2
+\|\nabla^k f(0)\|_{L_{x,v}^2}^2\big) \nonumber\\
& +\int_{0}^te^{-\lambda_{12}(t-s)} \|\nabla^l U^L(s)\|_{\mathcal{Z}_2}^2\mathrm{d}s,  
\end{align}
for any $k=2,3,4$ and $l=0,1$.

By combining \eqref{G5.62} with \eqref{G5.65}, we further obtain  
\begin{align}\label{G5.66}
\|\nabla^k U\|_{\mathcal{Z}_2}^2\lesssim (1+t)^{-\frac{7}{4}}\big(\|U_0\|_{\mathcal{H}^4}^2+\|U_0\|_{\mathcal{Z}_1}\|U_0\|_{\mathcal{Z}_2}
+\mathcal{E}_0\mathcal{E}_{ \infty}(t)+\mathcal{E}_{ \infty}^2(t)    \big),
\end{align}
for any $k=2,3,4$.

On the other hand, for any $l=0,1$, from \eqref{G2.3} and \eqref{G5.66}, it holds that
\begin{align*}
\|\nabla^l U(t)\|_{\mathcal{Z}_2}^2\lesssim &\,\|\nabla^l U^L(t)\|_{\mathcal{Z}_2}^2+\|\nabla^l U^H(t)\|_{\mathcal{Z}_2}^2\nonumber\\
\lesssim\,& \|\nabla^l U^L(t)\|_{\mathcal{Z}_2}^2+\|\nabla^k U(t)\|_{\mathcal{Z}_2}^2\nonumber\\
\lesssim\,&(1+t)^{-\frac{3}{4}-l}\big(\|U_0\|_{\mathcal{H}^4}^2
+\|U_0\|_{\mathcal{Z}_1}\|U_0\|_{\mathcal{Z}_2}+\mathcal{E}_0\mathcal{E}_{\infty}(t)
+\mathcal{E}_{\infty}^2(t)    \big),
\end{align*}
which together with \eqref{G5.66} and the definition of
$\mathcal{E}_{\infty}(t)$ in \eqref{G5.56} leads to
\begin{align*}
\mathcal{E}_{\infty}(t)\lesssim  \|U_0\|_{\mathcal{H}^4}^2+\|U_0\|_{\mathcal{Z}_1}\|U_0\|_{\mathcal{H}^4}+\mathcal{E}_0\mathcal{E}_{ \infty}(t)
+\mathcal{E}_{ \infty}^2(t).   
\end{align*}

Thanks to the smallness of $\|U_0\|_{\mathcal{H}^3}$ and $\mathcal{E}_0$, from Lemma \ref{LA.4}, we arrive at 
\begin{align*}
\mathcal{E}_{\infty}(t)\lesssim  \|U_0\|_{\mathcal{H}^4}^2+\|U_0\|_{\mathcal{Z}_1}\|U_0\|_{\mathcal{H}^4},
\end{align*}
which implies that 
\begin{align}\label{decay-a}
\|f\|_{L_{x,v}^2}+\|(\rho,u,\theta)\|_{L^2} \lesssim&\, (1+t)^{-\frac{3}{8}}\big(\|U_0\|_{\mathcal{H}^4}+\|U_0\|_{\mathcal{Z}_1}^{\frac{1}{2}}\|U_0\|_{\mathcal{H}^4}^{\frac{1}{2}}   \big),  \\\label{decay-b}
\|\nabla ( \rho,u,f)\|_{\mathcal{H}^3}\lesssim&\,(1+t)^{-\frac{7}{8}}\big(\|U_0\|_{\mathcal{H}^4}+\|U_0\|_{\mathcal{Z}_1}^{\frac{1}{2}}\|U_0\|_{\mathcal{H}^4}^{\frac{1}{2}}   \big).
\end{align}

\subsection{Optimal time-decay rates of  solutions and their gradients }
Note that the decay rates obtained in \eqref{decay-a} and \eqref{decay-b} are 
slower compared to the optimal time-decay rates of classical solutions for the nonlinear problem \eqref{Q2}--\eqref{Q2-2}. To achieve the improved results stated in Theorem \ref{T1.4}, we present more precise energy estimates for $(\rho, u, \theta, f)$. Subsequently, we define  
\begin{align}\label{G5.69}  
\mathcal{E}_{1,\infty}(t) := \sup_{0 \leq s \leq t} & \bigg\{ \sum_{l=0}^1 (1+s)^{\frac{3}{2}+l} \big( \| \nabla^l f \|_{L_{x,v}^2}^2 + \| \nabla^l (\rho, u,\theta) \|_{L^2}^2 \big) \nonumber \\  
& + \sum_{l=2}^4(1+s)^{\frac{5}{2}} \big( \| \nabla^l f \|_{L_{x,v}^2}^2 + \| \nabla^l (\rho, u,\theta) \|_{L^2}^2 \big) \bigg\}.  
\end{align}

We aim to provide a more precise energy estimate for the nonlinear terms associated with $ L_1(t),\dots,L_7(t) $ in \eqref{G5.55}.

Taking $q=1,m=l$ in \eqref{G5.7}, we get
\begin{align}\label{G5.70}
\| L_1(t)\|_{\mathcal{Z}_2}\lesssim (1+t)^{-\frac{3}{4}-\frac{l}{2}}
\|U_0\|_{\mathcal{Z}_{1}},
\end{align}
for $l=0,1$. Similar to \eqref{G5.58}, we have
\begin{align}\label{G5.71}
\| L_2(t)\|_{\mathcal{Z}_2}^2\lesssim\,& \int_0^{\frac{t}{2}}(1+t-s)^{-\frac{3}{2}-l}
\| u\{\mathbf{I}-\mathbf{P}_0-\mathbf{P}_1\}f\|_{\mathcal{Z}_1}^2\mathrm{d}s\nonumber\\
&+ \int_{\frac{t}{2}}^t(1+t-s)^{-\frac{1}{2}}
\|\nabla^l ( u\{\mathbf{I}-\mathbf{P}_0-\mathbf{P}_1\}f)\|_{\mathcal{Z}_{\frac{3}{2}}}^2\mathrm{d}s\nonumber\\
\lesssim &\,(1+t)^{-\frac{3}{2}-l}\big(\|U_0\|_{\mathcal{H}^4}^2+\|U_0\|_{\mathcal{Z}_1}\|U_0\|_{\mathcal{H}^4}   \big)\mathcal{E}_{1,\infty}(t),    
\end{align}
for $l=0,1$.
Following a similar approach to \eqref{G5.59}, we derive that 
\begin{align}\label{G5.72}
\| L_3(t)\|_{\mathcal{Z}_2}
\lesssim&\, \int_0^{\frac{t}{2}}(1+t-s)^{-\frac{3}{4}-\frac{l}{2}}\|(a,b,u,\theta)\|_{L^2}^2 
  \mathrm{d}s\nonumber\\
&+ \int_{\frac{t}{2}}^t(1+t-s)^{-\frac{1}{4}}\|\nabla^l(a,b,u,\theta)\|_{L^2}\|\nabla(a,b,u,\theta)\|_{L^2}\mathrm{d}s\nonumber\\
\lesssim&\, \big(\|U_0\|_{\mathcal{H}^4}+\|U_0\|_{\mathcal{Z}_1}^{\frac{1}{2}}\|U_0\|_{\mathcal{H}^4}^{\frac{1}{2}}   \big)\mathcal{E}^{\frac{1}{2}}_{\infty}(t)\int_0^{\frac{t}{2}}(1+t-s)^{-\frac{3}{4}-\frac{l}{2}}(1+s)^{-\frac{9}{8}}
  \mathrm{d}s\nonumber\\
&+ \big(\|U_0\|_{\mathcal{H}^4}+\|U_0\|_{\mathcal{Z}_1}^{\frac{1}{2}}\|U_0\|_{\mathcal{H}^4}^{\frac{1}{2}}   \big)\mathcal{E}^{\frac{1}{2}}_{\infty}(t)\int_{\frac{t}{2}}^{t}(1+t-s)^{-\frac{1}{4}}(1+s)^{-\frac{13}{8}-\frac{l}{2}}
  \mathrm{d}s\nonumber\\
\lesssim&\, (1+t)^{-\frac{3}{4}-\frac{l}{2}}\big(\|U_0\|_{\mathcal{H}^4}+\|U_0\|_{\mathcal{Z}_1}^{\frac{1}{2}}\|U_0\|_{\mathcal{H}^4}^{\frac{1}{2}}   \big)\mathcal{E}^{\frac{1}{2}}_{1,\infty}(t),
\end{align}
for $l=0,1$.
Similar to \eqref{G5.60}, one easily observes that   
\begin{align}\label{G5.73}
\| L_4(t)\|_{\mathcal{Z}_2}^2   \lesssim&\,  \int_0^{\frac{t}{2}}(1+t-s)^{-\frac{3}{2}-l}
\Big\|\nu^{-\frac{1}{2}}  \theta \{\mathbf{I} - \mathbf{P}_2\} \Big[\frac{1}{\sqrt{M}} \nabla_v \cdot \big(\sqrt{M} (2 \nabla_v f - v f)\big)\Big]\Big\|_{\mathcal{Z}_1}^2\mathrm{d}s\nonumber\\
&+  \int_{\frac{t}{2}}^{t}(1+t-s)^{-\frac{1}{2}-l}
\Big\|\nu^{-\frac{1}{2}}  \theta \{\mathbf{I} - \mathbf{P}_2\} \Big[\frac{1}{\sqrt{M}} \nabla_v \cdot \big(\sqrt{M} (2 \nabla_v f - v f)\big)\Big]\Big\|_{\mathcal{Z}_{\frac{3}{2}}}^2\mathrm{d}s\nonumber\\
\lesssim&\, (1+t)^{-\frac{3}{2}-l}\mathcal{E} _{1,\infty}(t)\int_0^{\frac{t}{2}}
\mathcal{D}(\rho,u,\theta,f)(s)\mathrm{d}s \nonumber\\ &+(1+t)^{-\frac{5}{2}}\mathcal{E}_{1,\infty}(t)\int_{\frac{t}{2}}^t\mathcal{D}(\rho,u,\theta,f)(s)\mathrm{d}s\nonumber\\
\lesssim&\, \mathcal{E}(0)(1+t)^{-\frac{3}{2}-l}\mathcal{E}_{1,\infty}(t),
\end{align}
for $l=0,1$.

For the remaining terms $L_5(t)$, $L_6(t)$, and $L_7(t)$ with $l=0,1$, similar to \eqref{G5.61}, it follows from \eqref{G5.7} in Theorem \ref{T5.1}  that  
\begin{align}\label{G5.74}
&\|L_5(t)\|_{\mathcal{Z}_2}+\|L_6(t)\|_{\mathcal{Z}_2}+\|L_7(t)\|_{\mathcal{Z}_2} \nonumber\\
\lesssim &\,
   \int_0^{\frac{t}{2}}(1+t-s)^{-\frac{3}{4}-\frac{l}{2}}\big\|\big(S_\rho(s),S_u(s), S_{\theta}(s)\big)\big\|_{\mathcal{Z}_1}\mathrm{d} s \nonumber\\
&+  \int_{\frac{t}{2}}^{t}(1+t-s)^{-\frac{1}{4}-\frac{l}{2}}\Big(\| S_\rho(s)\|_{\mathcal{Z}_{\frac{3}{2}}}+\Big\|S_u(s)+\frac{\rho}{1+\rho}(b-u)+\frac{a u}{1+\rho} \Big\|_{\mathcal{Z}_{\frac{3}{2}}}\Big)\mathrm{d} s\nonumber\\
&+  \int_{\frac{t}{2}}^{t}(1+t-s)^{-\frac{1}{4}-\frac{l}{2}}\Big\|S_{\theta}(s)+\frac{\sqrt{3} \rho}{1+\rho}(\sqrt{2} \omega-\sqrt{3} \theta)-\frac{(1+a)|u|^2-3 a \theta-2 u \cdot b}{1+\rho} \Big\|_{\mathcal{Z}_{\frac{3}{2}}}\mathrm{d} s\nonumber\\
&+  \int_{\frac{t}{2}}^{t}(1+t-s)^{-\frac{1}{4}}\Big\|\nabla^l\Big(\frac{\rho}{1+\rho}(b-u)+\frac{a u}{1+\rho} \Big)\Big\|_{\mathcal{Z}_{\frac{3}{2}}}\mathrm{d} s\nonumber\\
&+  \int_{\frac{t}{2}}^{t}(1+t-s)^{-\frac{1}{4}}\Big\|\nabla^l\Big(\frac{\sqrt{3} \rho}{1+\rho}(\sqrt{3} \theta-\sqrt{2} \omega)\Big)\Big\|_{\mathcal{Z}_{\frac{3}{2}}}\mathrm{d} s\nonumber\\
&+  \int_{\frac{t}{2}}^{t}(1+t-s)^{-\frac{1}{4}}\Big\|\nabla^l\Big(\frac{(1+a)|u|^2-3 a \theta-2 u \cdot b}{1+\rho} \Big)\Big\|_{\mathcal{Z}_{\frac{3}{2}}}\mathrm{d} s\nonumber\\ 
\lesssim&\, \big(\|U_0\|_{\mathcal{Z}_{1}}^\frac{1}{2}\|U_0\|_{\mathcal{H}^4}^\frac{1}{2}+\|U_0\|_{\mathcal{H}^4}\big)\mathcal{E}_{1,\infty}^{\frac{1}{2}} (t) \int_0^{\frac{t}{2}}(1+t-s)^{-\frac{3}{4}-\frac{l}{2}}(1+s)^{-\frac{9}{8}}\mathrm{d}s\nonumber\\
&+ \big(\|U_0\|_{\mathcal{Z}_{1}}^\frac{1}{2}\|U_0\|_{\mathcal{H}^4}^\frac{1}{2}+\|U_0\|_{\mathcal{H}^4}\big)\mathcal{E}_{1,\infty}^{\frac{1}{2}} (t)\int_{\frac{t}{2}}^{t}(1+t-s)^{-\frac{1}{4}-\frac{l}{2}}(1+s)^{-\frac{17}{8}}\mathrm{d}s\nonumber\\
&+ \big(\|U_0\|_{\mathcal{Z}_{1}}^\frac{1}{2}\|U_0\|_{\mathcal{H}^4}^\frac{1}{2}+\|U_0\|_{\mathcal{H}^4}\big)\mathcal{E}_{1,\infty}^{\frac{1}{2}} (t)\int_{\frac{t}{2}}^{t}(1+t-s)^{-\frac{1}{4}}(1+s)^{-\frac{13}{8}-\frac{l}{2}}\mathrm{d}s\nonumber\\
\lesssim& \,(1+t)^{-\frac{3}{4}-\frac{l}{2}}\big(\|U_0\|_{\mathcal{Z}_{1}}^{\frac{1}{2}}\|U_0\|_{\mathcal{H}^4}^\frac{1}{2}+\|U_0\|_{\mathcal{H}^4}\big)\mathcal{E}_{1,\infty}^{\frac{1}{2}} (t).
\end{align} 

Collecting all the above estimates \eqref{G5.70}--\eqref{G5.74} together,
 we consequently obtain  
\begin{align}\label{G5.75}
\|\nabla^l U^L(t)\|_{\mathcal{Z}_2}\lesssim (1+t)^{-\frac{3}{4}-\frac{l}{2}} 
\big(\|U_0\|_{\mathcal{H}^4}
+\|U_0\|_{\mathcal{Z}_1}^{\frac{1}{2}}\|U_0\|_{\mathcal{H}^4}^{\frac{1}{2}}+\mathcal{E}^{\frac{1}{2}}_0  
 \big)\mathcal{E}^{\frac{1}{2}}_{1,\infty}(t),   
\end{align}
for $l=0,1$.

Similar to \eqref{G5.66}, by using \eqref{G2.3}, we   infer that
\begin{align}\label{G5.76}
\|\nabla^k U(t)\|_{\mathcal{Z}_2}\lesssim (1+t)^{-\frac{5}{4}}
 \big(\|U_0\|_{\mathcal{H}^4}+\|U_0\|_{\mathcal{Z}_1}^{\frac{1}{2}}
 \|U_0\|_{\mathcal{H}^4}^{\frac{1}{2}}+\mathcal{E}^{\frac{1}{2}}_0  
  \big)\mathcal{E}^{\frac{1}{2}}_{1,\infty}(t),   
\end{align}
for any $k=2,3,4$.
Utilizing the inequalities  \eqref{G2.3} and the estimate \eqref{G5.76} leads to  
\begin{align}\label{G5.77}
\|\nabla^l U(t)\|_{\mathcal{Z}_2}^2\lesssim \,
& (\|\nabla^l U^L(t)\|_{\mathcal{Z}_2}^2+\|\nabla^l U^H(t)\|_{\mathcal{Z}_2}^2)\nonumber\\
\lesssim\,&  (\|\nabla^l U^L(t)\|_{\mathcal{Z}_2}^2+\|\nabla^k U(t)\|_{\mathcal{Z}_2}^2)\nonumber\\
\lesssim\,& (1+t)^{-\frac{3}{2}-l}\big(\big(\|U_0\|_{\mathcal{H}^4}^2+\|U_0\|_{\mathcal{H}^4}\|U_0\|_{\mathcal{Z}_1} 
+\mathcal{E}_0    \big)\mathcal{E}_{1,\infty}(t)+\|U_0\|_{\mathcal{H}^4\cap{\mathcal{Z}_1}}^2  \big),
\end{align}
for any $k=2,3,4$ and $l=0,1$. Then,
\eqref{G5.77} in combination with   the definition of $\mathcal{E}_{1,\infty}(t)$   in \eqref{G5.69} implies that  
\begin{align*}
\mathcal{E}_{1,\infty}(t)\lesssim 
 \big(\|U_0\|_{\mathcal{H}^4}^2+\|U_0\|_{\mathcal{Z}_1}\|U_0\|_{\mathcal{H}^4}+\mathcal{E}_0 
    \big)\mathcal{E}_{1,\infty}(t)+\|U_0\|_{\mathcal{H}^4\cap{\mathcal{Z}_1}}^2  .   
\end{align*}
Owing to the smallness of $ \|U_0\|_{\mathcal{H}^3} $ and $ \mathcal{E}_0 $, we derive  that
\begin{align*}  
\mathcal{E}_{1,\infty}(t) \lesssim \|U_0\|_{\mathcal{H}^4 \cap \mathcal{Z}_1}^2,  
\end{align*}  
which guarantees that the inequalities in  \eqref{d1} hold.

\subsection{End of the proof of  Theorem \ref{T1.4}}
\begin{proof}[Proof of Theorem \ref{T1.4} \emph{(}continued\emph{)}]
To end the proof of  Theorem \ref{T1.4},  we first need to investigate  
the  time-decay rates of classical solutions in $L^p$ space with  $p\in [2,+\infty]$. In fact, 
based on the decay rates obtained in   \eqref{d1}, we have 
\begin{align*}
\|(\rho,u,\theta)\|_{L^2}\lesssim \,& (1+t)^{-\frac{3}{4}}, \qquad
 \|(\rho,u,\theta)\|_{L^6}\lesssim   \|\nabla(\rho,u,\theta)\|_{L^2}\lesssim (1+t)^{-\frac{5}{4}},   \\
 \|f\|_{L_v^2(L^{2})}\lesssim \,&  (1+t)^{-\frac{3}{4}},\qquad\quad\, 
 \| f\|_{L_v^2(L^{6})}\lesssim    \|\nabla f\|_{L_v^2(L^{2})}\lesssim (1+t)^{-\frac{5}{4}}.      
\end{align*}

For $p\in [2,6]$, using the interpolation inequality, it holds that
\begin{align*}
 \|(\rho,u,\theta)\|_{L^p}\lesssim \,
 &  \|(\rho,u,\theta)\|_{L^2}^{\zeta}\|(\rho,u,\theta)\|_{L^6}^{{1-\zeta}}
 \lesssim  (1+t)^{-\frac{3}{2}(1-\frac{1}{p})} , \\
 \|f\|_{L_v^2(L^{p})}\lesssim \,&  
 \|f\|_{L_v^2(L^{2})}^{\zeta}\|f\|_{L_v^2(L^{6})}^{{1-\zeta}}\lesssim  (1+t)^{-\frac{3}{2}(1-\frac{1}{p})},
\end{align*}
where $\zeta= ({6-p})/ {2p} \in [0,1]$.

Using   Gagliardo-Nirenberg's inequality, we derive  that
\begin{align*} 
\|(\rho,u,\theta)\|_{L^{\infty}} \lesssim \,& \|\nabla(\rho,u,\theta)\|_{H^{1}} 
\lesssim (1+t)^{-\frac{5}{4}}, \\ 
\|f\|_{L_v^2(L^{\infty})} \lesssim \,& \|\nabla f\|_{L_v^2(H^{1})} 
\lesssim (1+t)^{-\frac{5}{4}}. 
\end{align*}  

For $p \in [6,\infty]$, by applying the interpolation inequality, it follows that  
\begin{align*} 
\|(\rho,u,\theta)\|_{L^p} &\lesssim \|(\rho,u,\theta)\|_{L^6}^{\zeta'} 
\|(\rho,u,\theta)\|_{L^{\infty}}^{1-\zeta'} \lesssim (1+t)^{-\frac{5}{4}}, \\ 
\|f\|_{L_v^2(L ^{p})} &\lesssim \|f\|_{L_v^2(L ^{6})}^{\zeta'} 
\|f\|_{L_v^2(L ^{\infty})}^{1-\zeta'} \lesssim (1+t)^{-\frac{5}{4}}, 
\end{align*}  
where $\zeta' = {6}/{p} \in [0,1]$. This guarantees that \eqref{d2} holds.

 \medskip 
 Finally, we prove the remaining inequality \eqref{NJKd3}. 
 It follows  from \eqref{Q2}$_2$, \eqref{Q2}$_3$ and 
\eqref{G3.21}$_1$ \linebreak  –\eqref{G3.21}$_3$ that
\begin{align}\label{NJKG3.158}
 \partial_t (b-u)+2(b-u) 
=\,&  -\nabla a-\frac{2}{\sqrt{6}}\nabla\omega-{\rm div}
\big\langle\{ v\otimes v-{\rm Id}\} \sqrt{M},\{\mathbf{I}-\mathbf{P}\}f               
    \big\rangle+2au\nonumber\\
&+\nabla \rho+\nabla\theta  +u\cdot\nabla u
+h(\rho)\theta\nabla\rho +g(\rho)  \big(  \nabla\rho +(b-u)+au\big)\nonumber\\
\equiv:\,& \mathcal{Q}_1, 
\end{align}
and
\begin{align}\label{NJKG3.159}
&\partial_t(\sqrt{2}\omega-\sqrt{3}\theta)+5(\sqrt{2}\omega-\sqrt{3}\theta)\nonumber\\ 
=\,& \sqrt{2}\Big(\sqrt{6}a\theta-\frac{2}{\sqrt{6}}{\rm div}b
+\frac{2}{\sqrt{6}}u\cdot b-{\rm div} \Big\langle \frac{v(|v|^2-3)}{\sqrt{6}}\sqrt{M}, 
\{\mathbf{I}-\mathbf{P}\}f   \Big\rangle   \Big)\nonumber\\
&+\sqrt{3}{\rm div}u+\sqrt{3}\theta{\rm div}u
+\sqrt{3} u\cdot\nabla\theta-\sqrt{3} u\cdot b-3\sqrt{3}a\theta\nonumber\\
&-\sqrt{3} h(\rho) \big( -u\cdot(b-u) +a|u|^2 
+\sqrt{3}(\sqrt{2}\omega-\sqrt{3}\theta)-b\cdot u-3a\theta      \big)  \nonumber\\
\equiv:\,& \mathcal{Q}_2.
\end{align}
The equations \eqref{NJKG3.158} and \eqref{NJKG3.159} can be represented as % follows:
\begin{align}\label{NJKG3.160}
b-u=&\, e^{-2t}(b_0-u_0)+\int_{0}^t e^{-2(t-\tau)} \mathcal{Q}_1{\rm d}\tau,    
\end{align}
and
\begin{align}\label{NJKG3.161}
\sqrt{2}\omega-\sqrt{3}\theta=&\, e^{-5t}(\sqrt{2}\omega_0-\sqrt{3}\theta_0)
+\int_{0}^t e^{-5(t-\tau)} \mathcal{Q}_2{\rm d}\tau.   
\end{align} 
Taking the $L^2$-norm of both sides of \eqref{NJKG3.160} and \eqref{NJKG3.161}, and applying 
the decay estimates in \eqref{d2}, we obtain   
\begin{align}\label{NJKG3.162}
\|(b-u)(t)\|_{L^2}\lesssim&\, e^{-2t} \|b_0-u_0\|_{L^2}+\int_0^t e^{-2(t-\tau)}
\big(\|\nabla(a,\rho,\theta,\omega)\|_{L^2}+\|\nabla f\|_{L_{x,v}^2}\big)  {\rm d}\tau\nonumber\\
&+\int_0^t e^{-2t}\big(\|au\|_{L^2}+\|u\cdot\nabla u\|_{L^2}
+\|\theta\nabla\rho\|_{L^2}+\|\rho\|_{L^\infty}\|b-u\|_{L^2}  \big){\rm d}\tau \nonumber\\
\lesssim&\, e^{-2t}+\int_0^t e^{-2(t-\tau)}(1+\tau)^{-\frac{5}{4}}{\rm d}\tau 
+\Big(\int_0^t e^{-4(t-\tau)}(1+\tau)^{-\frac{5}{2}}{\rm d}\tau\Big)^\frac{1}{2}\|b-u\|_{L_t^2(L^2)}\nonumber\\
\lesssim&\, (1+t)^{-\frac{5}{4}},
\end{align}
and
\begin{align}\label{NJKG3.163}
\|(\sqrt{2}\omega-\sqrt{3}\theta)(t)\|_{L^2}\lesssim&\, e^{-5t} \|\sqrt{2}\omega_0
-\sqrt{3}\theta_0\|_{L^2}+\int_0^t e^{-5(t-\tau)}\big(\|\nabla(u,b)\|_{L^2}
+\|\nabla f\|_{L_{x,v}^2}\big)  {\rm d}\tau\nonumber\\
&+\int_0^t e^{-5t}\big(\|b\cdot
 u\|_{L^2}+\|a\theta\|_{L^2}+\|\theta{\rm div} u\|_{L^2}+\|u\cdot\nabla\theta\|_{L^2} 
 \big){\rm d}\tau \nonumber\\
&+\int_0^t\big(  \|a\|_{L^6}\|u\|_{L^6}^2+\|u\cdot(b-u)\|_{L^2}  
+\|\rho\|_{L^\infty}\| \sqrt{2}\omega-\sqrt{3}\theta \|_{L^2} \big)
  {\rm d}\tau\nonumber\\
\lesssim&\, e^{-2t}+\int_0^t e^{-2(t-\tau)}(1+\tau)^{-\frac{5}{4}}{\rm d}\tau \nonumber\\
&+\Big(\int_0^t e^{-4(t-\tau)}(1+\tau)^{-\frac{5}{2}}{\rm d}\tau\Big)^\frac{1}{2}
\|\sqrt{2}\omega-\sqrt{3}\theta\|_{L_t^2(L^2)}\nonumber\\
\lesssim&\, (1+t)^{-\frac{5}{4}}.
\end{align}

We now proceed to derive the time-decay estimate  for the microscopic
 component $\{\mathbf{I}-\mathbf{P}\}f$ in \eqref{NJKd3}. To this end, we consider the equation 
 \eqref{Q2}$_4$, from which we obtain  
\begin{align}\label{NJKG3.164}
&\partial_t \{\mathbf{I}-\mathbf{P}\}f-\mathcal{L}{\{\mathbf{I}-\mathbf{P}\}}f
+v\cdot\nabla\{\mathbf{I}-\mathbf{P}\}f-\frac{1}{2}u\cdot v\mathbf{P}f+u\cdot\nabla_v\mathbf{P}f\nonumber\\
&\quad +\partial_t {\mathbf{P}f}+v\cdot\nabla {\mathbf{P}f}+(b-u)v\sqrt{M}
+\frac{1}{\sqrt{3}}(|v|^2-3)\sqrt{M}(\sqrt{2}\omega-\sqrt{3}\theta)\nonumber\\
&\qquad =  \frac{1}{2}u\cdot v\{\mathbf{I}-\mathbf{P}\}f
-u\cdot \nabla_v\{\mathbf{I}-\mathbf{P}\}f+\frac{\theta}{\sqrt{M}} \Delta_v(\sqrt{M}  \mathbf{P}f)\nonumber\\
& \qquad \ \ \, +\frac{\theta}{\sqrt{M}} \Delta_v(\sqrt{M} \{\mathbf{I}- \mathbf{P}\}f).
\end{align}
Taking the $L^2_{x,v}$ inner product of \eqref{NJKG3.164} and utilizing the
 property \eqref{G2.6} and Young's inequality, we  arrive at
\begin{align*} 
&\frac{1}{2}  \frac{\rm d}{{\rm d}t} \|\{\mathbf{I}-
\mathbf{P}\}f\|_{L_{x,v}^2}^2+\bar\lambda \|\{\mathbf{I}-\mathbf{P}\}f\|_{\nu}^2\nonumber\\
&\quad \lesssim \|\partial_t{\mathbf{P}f} \|_{L_{x,v}^2}^2
+\big(1+\|(u,\theta)\|_{H^2}^2\big)\|\nabla(a,b,\omega)\|_{L^2}^2+\|b-u\|_{L^2}^2
\nonumber\\
&\qquad+\|\sqrt{2}\omega-\sqrt{3}\theta\|_{L^2}^2 
+\eps_1\bar\lambda\|\{\mathbf{I}-\mathbf{P}\}f\|_{\nu}^2+\|u\|_{H^2}\|\{\mathbf{I}-\mathbf{P}\}f\|_{\nu}^2,
\end{align*}
which, together with the fact $ \|\{\mathbf{I}-\mathbf{P}\}f\|_{L_{x,v}^2} 
\lesssim\|\{\mathbf{I}-\mathbf{P}\}f\|_{\nu}^2$, leads to
\begin{align}\label{NJKG3.165}
  \|\{\mathbf{I}-\mathbf{P}\}f\|_{L_{x,v}^2}^2 
 \lesssim&\, \int_0^t e^{-2(t-\tau)}\big( \|\partial_t{\mathbf{P}f} 
 \|_{L_{x,v}^2}^2+\|\nabla(a,b,\omega)\|_{L^2}^2\nonumber\\
 & \quad + \|b-u\|_{L^2}^2
 +\|\sqrt{2}\omega-\sqrt{3}\theta\|_{L^2}^2  \big){\rm d}\tau %\nonumber\\
  +e^{-2\bar\lambda t}  \|\{\mathbf{I}-\mathbf{P}\}f_0\|_{L_{x,v}^2}^2\nonumber\\
 \lesssim&\, \int_0^t e^{-2(t-\tau)}(1+\tau)^{-\frac{5}{2}}{\rm d}\tau
 +e^{-2\bar\lambda t}+ \int_0^t e^{-2(t-\tau)}  \|\partial_t{\mathbf{P}f} \|_{L_{x,v}^2}^2{\rm d}\tau\nonumber\\
 \lesssim&\,(1+t)^{-\frac{5}{2}} + \int_0^t e^{-2(t-\tau)}  \|\partial_t{\mathbf{P}f} \|_{L_{x,v}^2}^2{\rm d}\tau.
\end{align}

Through a direct computation, it follows from \eqref{G3.21}$_1$–\eqref{G3.21}$_3$ that  
\begin{align}\label{NJKG3.166}
\|\partial_t{\mathbf{P}f} \|_{L_{x,v}^2}\lesssim&\, \|\partial_t a\|_{L^2}+ 
  \|\partial_t b\|_{L^2} +\|\partial_t \omega\|_{L^2}\nonumber\\
\lesssim\,&\|\nabla (a,b,\omega)\|_{L^2}+\|\nabla f\|_{L_{x,v}^2}+\|b-u\|_{L^2}\nonumber\\
& +\|\sqrt{2}\omega-\sqrt{3}\theta\|_{L^2}+\|au\|_{L^2}+\|u\cdot b\|_{L^2}\nonumber\\
\lesssim\,& (1+t)^{-\frac{5}{4}},
\end{align}

Substituting \eqref{NJKG3.166} into \eqref{NJKG3.165}, we obtain 
\begin{align}\label{NJKG3.167}
 \|\{\mathbf{I}-\mathbf{P}\}f\|_{L_{x,v}^2} \lesssim (1+t)^{-\frac{5}{4}}.   
\end{align}
By putting the estimates \eqref{NJKG3.162}, \eqref{NJKG3.163}, 
and \eqref{NJKG3.167} together, we obtain the desired  inequality \eqref{NJKd3}.
 Thus, the proof of Theorem \ref{T1.4} is completed.  
\end{proof}

\begin{rem}
The term $L_4(t)$ in \eqref{G5.55}, which involves $\Delta_v f$, 
introduces significant challenges in the aforementioned estimates.
 Specifically,  $L_4(t)$ leads to energy dissipation, thereby 
 hindering the derivation of optimal decay rates for the second-order 
 and third-order derivatives of $(\rho, u, \theta, f)$. This phenomenon 
is strikingly different from \cite{Ww-CMS-2024} where the  
compressible isentropic NS-VFP system \eqref{H2}  was studied.
\end{rem}

\section{Periodic domain case} 

In this section, we investigate the non-isentropic compressible 
 NS-VFP system \eqref{Q1}   in the periodic domain $\mathbb{T}^3$. 
We first state the global existence result to the   problem 
\eqref{Q1}--\eqref{Q1-1} in $\mathbb{T}^3$. 
\begin{thm}\label{T1.5}
% Let $\Omega = \mathbb{T}^3$.
Assume that the initial data $(\rho_0^{\mu,\lambda,\kappa},
u_0^{\mu,\lambda,\kappa},\theta_0^{\mu,\lambda,\kappa},f_0^{\mu,\lambda,\kappa})$
satisfies $F_0^{\mu,\lambda,\kappa}=M+\sqrt{M}f_0^{\mu,\lambda,\kappa}\geq 0$, 
$(\rho_0^{\mu,\lambda,\kappa},u_0^{\mu,\lambda,\kappa},\theta_0^{\mu,\lambda,\kappa})
\in H^4(\mathbb{T}^3)$ and $f_0^{\mu,\lambda,\kappa}\in H^4_{x,v}(\mathbb{T}^3\times\mathbb{R}^3)$.
Assume that there exists a  small constant $\eps_2>0$ independent of $\mu$, $\lambda$ and $\kappa$  such that
\begin{align*}
\|(\rho^{{\mu,\lambda,\kappa}}_0,u^{{\mu,\lambda,\kappa}}_0,
\theta^{{\mu,\lambda,\kappa}}_0)\|_{H^4(\mathbb{T}^3)}
+\|f^{\mu,\lambda,\kappa}_0\|_{H^4_{x,v}(\mathbb{T}^3\times\mathbb{R}^3)}\leq \eps_2,   
\end{align*}
 and
\begin{gather}
\int_{\mathbb{T}^{3}}a_{0}^{{\mu,\lambda,\kappa}}\mathrm{d}x=0,
\quad\int_{\mathbb{T}^3} 
 \rho_{0}^{{\mu,\lambda,\kappa}}\mathrm{d}x=0,
 \quad\int_{\mathbb{T}^3}\big(b_0^{\mu,\lambda,
 \kappa}+(1+\rho_0^{\mu,\lambda,\kappa})u_0^{\mu,\lambda,\kappa}\big){\rm d}x=0, \label{e1-1}\\
\int_{\mathbb{T}^{3}}\Big((1+\rho^{{\mu,\lambda,\kappa}}_{0})
\Big(\theta_0^{\mu,\lambda,\kappa}+\frac{1}{2}|u_0^{\mu,\lambda,\kappa}|^2\Big)
+\frac{\sqrt{6}}{2}\omega^{\mu,\lambda,\kappa}_0\Big)\mathrm{d}x=0,\label{e1-2}
\end{gather}
with
\begin{align*}
a_0^{\mu,\lambda,\kappa}=&\,\int_{\mathbb{T}^3}\sqrt{M}f_0^{\mu,\lambda,\kappa}(x,v){\rm d}v,\\
b_0^{\mu,\lambda,\kappa}=&\,\int_{\mathbb{T}^3}v\sqrt{M}f_0^{\mu,\lambda,\kappa}(x,v){\rm d}v,\\
\ \omega_0^{\mu,\lambda,\kappa}=
&\,\int_{\mathbb{T}^3}\frac{|v|^2-3}{\sqrt{6}}\sqrt{M}f_0^{\mu,\lambda,\kappa}(x,v){\rm d}v.
\end{align*}
Then, the  problem \eqref{Q1}--\eqref{Q1-1} in $\mathbb{T}^3$ admits a unique global
 classical solution $(\rho^{{\mu,\lambda,\kappa}}, u^{{\mu,\lambda,\kappa}},\theta^{{\mu,\lambda,\kappa}},
  \linebreak f^{{\mu,\lambda,\kappa}})$  satisfying 
   $\rho^{{\mu,\lambda,\kappa}}=1+\varrho^{\mu,\lambda,\kappa}>0$, 
   $F^{\mu,\lambda,\kappa}=M+\sqrt{M}f^{\mu,\lambda,\kappa}\geq 0$, and
\begin{align*}
&\sup_{t\geq 0}\big\{\|(\rho^{{\mu,\lambda,\kappa}},u^{{\mu,\lambda,\kappa}},
\theta^{{\mu,\lambda,\kappa}} )(t)\|_{H^{4}(\mathbb{T}^3) }^2
+\|f^{\mu,\lambda,\kappa}(t)\|_{H_{x,v}^4(\mathbb{T}^3) }^2\big\}\nonumber\\
& \quad +\int_{0}^{t} \mathcal{D}(\rho^{\mu,\lambda,\kappa},
u^{\mu,\lambda,\kappa},\theta^{\mu,\lambda,\kappa},f^{\mu,\lambda,\kappa})(\tau){\rm d}\tau
 +\mu\int_0^{t}\|\nabla u^{\mu,\lambda,\kappa} (\tau)\|_{H^4(\mathbb{T}^3)}^2{\rm d}\tau\nonumber\\
&\qquad
+(\mu+\lambda)\int_0^{t}\|{\rm div} u^{\mu,\lambda,\kappa} (\tau)\|_{H^4(\mathbb{T}^3)}^2{\rm d}\tau
 +\kappa\int_0^{t}\|\nabla \theta^{\mu,\lambda,\kappa} (\tau)\|_{H^4(\mathbb{T}^3)}^2{\rm d}\tau \nonumber\\
&\qquad\quad \leq C_4\big(\|(\rho^{{\mu,\lambda,\kappa}}_0,u^{{\mu,\lambda,\kappa}}_0,
\theta^{{\mu,\lambda,\kappa}}_0)\|_{H^4(\mathbb{T}^3)}^2
+\|f^{\mu,\lambda,\kappa}_0\|_{H^4_{x,v}(\mathbb{T}^3\times\mathbb{R}^3)}^2\big),
\end{align*}
and
\begin{align*}
& \|(\rho^{\mu,\lambda,\kappa},u^{\mu,\lambda,\kappa},\theta^{\mu,\lambda,\kappa})(t)\|_{H^4(\mathbb{T}^3) }
+\|f^{\mu,\lambda,\kappa}(t)\|_{H^4_{x,v}(\mathbb{T}^3\times\mathbb{R}^3) }\nonumber\\
& \quad \leq C_4e^{-\zeta_1 t}\big(\|(\rho^{{\mu,\lambda,\kappa}}_0,u^{{\mu,\lambda,\kappa}}_0,
\theta^{{\mu,\lambda,\kappa}}_0)\|_{H^4(\mathbb{T}^3)}
+\|f^{\mu,\lambda,\kappa}_0\|_{H^4_{x,v}(\mathbb{T}^3\times\mathbb{R}^3)}\big),
\end{align*}
for all $t\geq0$, where $\zeta_1,C_4 > 0$ are two positive constants 
 independent of $\mu$, $\lambda$, $\kappa$ and   $t$.
\end{thm}

% Next, we proceed to prove Theorem \ref{T1.5}.
\begin{proof}% [Proof of Theorem \ref{T1.5}] 
For notations simplification, we shall  omit  the superscripts 
$\lambda$, $\mu$ and $\kappa$ on the unknowns in  the problem \eqref{Q1}--\eqref{Q1-1}.  
 Through   direct computations, we  derive the following conservation quantities from the system \eqref{LN-2}:
\begin{equation*}
\left\{
\begin{aligned}
&\frac{\mathrm{d}}{\mathrm{d} t} \int_{\mathbb{T}^3} \varrho \mathrm{d} x=0,\\
& \frac{\mathrm{d}}{\mathrm{d} t}\Big(\int_{\mathbb{T}^3} \varrho u \mathrm{d} x
+\iint_{\mathbb{T}^3\times\mathbb{R}^3} v F \mathrm{d} x \mathrm{d} v\Big)=0, \\
& \frac{\mathrm{d}}{\mathrm{d} t}\Big(\int_{\mathbb{T}^3} \varrho\Big(C_\textsl{v} 
\vartheta+\frac{1}{2}|u|^2\Big)
 \mathrm{d} x
+\iint_{\mathbb{T}^3\times\mathbb{R}^3} \frac{|v|^2}{2} F \mathrm{d} x \mathrm{d} v \Big)=0,\\
& \frac{\mathrm{d}}{\mathrm{d}t}\iint_{\mathbb{T}^3\times\mathbb{R}^3}F\mathrm{d}x=0,
\end{aligned}\right.
\end{equation*}
where $\varrho=1+\rho$, $\vartheta =1+\theta$ and $F=M+\sqrt{M}f$. Hence, under the assumptions
 \eqref{e1-1} and \eqref{e1-2}, we derive that the following identities hold for any $t \geq 0$:  
\begin{gather}
 \int_{\mathbb{T}^3} a \mathrm{d} x=0, \quad
 \int_{\mathbb{T}^3} \rho \mathrm{d} x=0, \quad
\int_{\mathbb{T}^3}\big(b+(1+\rho) u\big) \mathrm{d} x=0,\,\,\label{G6.1-1}\\
 \int_{\mathbb{T}^3}(1+\rho)\Big(\theta+\frac{1}{2}|u|^2\Big)+\frac{\sqrt{6}}{2}
  \omega \mathrm{d} x=0.\label{G6.1-2}
\end{gather}

Since the proofs of local existence and uniqueness for the classical solution 
$(\rho, u, \theta, f)$ are analogous to those in the $\mathbb{R}^3$ case, we omit 
the details here for the sake of brevity. The next key step involves 
deriving the uniform a priori estimates of classical solutions with 
exponential decay. By combining   Poincar\'{e}'s inequality with \eqref{G6.1-1} and \eqref{G6.1-2},
 we deduce that 
\begin{align}
\|a\|_{L^2(\mathbb{T}^3)} \lesssim&\, \|\nabla a\|_{L^2(\mathbb{T}^3)},\quad
\|\rho\|_{L^2(\mathbb{T}^3)}\lesssim \|\nabla \rho\|_{L^2(\mathbb{T}^3)}, \label{G6.2}\\
\|u+b\|_{L^2(\mathbb{T}^3)}  \lesssim&\,\|b+u+\rho u\|_{L^2(\mathbb{T}^3)}
+\|\rho u\|_{L^2(\mathbb{T}^3)} \nonumber\\
\lesssim&\, \|\nabla(b+u+\rho u)\|_{L^2(\mathbb{T}^3)}
+\|u\|_{L^{\infty}(\mathbb{T}^3)}\|\rho\|_{L^2(\mathbb{T}^3)} \nonumber\\
\lesssim&\,  \|\nabla(b, u)\|_{L^2(\mathbb{T}^3)}
+ \|u\|_{H^2{(\mathbb{T}^3)}}\|\nabla \rho\|_{L^2(\mathbb{T}^3)}
+ \|\rho\|_{H^2{(\mathbb{T}^3)}}\|\nabla u\|_{L^2(\mathbb{T}^3)}\nonumber\\
\lesssim&\,  \|\nabla(b,\rho,u)\|_{L^2(\mathbb{T}^3)},\label{G6.4}\\
\Big\|\frac{\sqrt{6}} {2} \omega+\theta\Big\|_{L^2(\mathbb{T}^3)}  
\lesssim&\,\Big\|(1+\rho)\Big(\theta+\frac{1}{2}|u|^2\Big)
+\frac{\sqrt{6}}{2} \omega\Big\|_{L^2(\mathbb{T}^3)}
+\|\rho \theta\|_{L^2(\mathbb{T}^3)}+ \big\|\rho|u|^2 \big\|_{L^2(\mathbb{T}^3)}\nonumber \\
 \lesssim&\, \big (\|\theta\|_{H^2{(\mathbb{T}^3)}}
 +\|u\|_{H^2{(\mathbb{T}^3)}}^2 \big)\|\nabla \rho\|_{L^2(\mathbb{T}^3)}
 +\big(1+\|u\|_{H^2{(\mathbb{T}^3)}}\big)\|\nabla (\omega,\theta)\|_{L^2(\mathbb{T}^3)}\nonumber \\
&+ \|\rho\|_{H^2{(\mathbb{T}^3)}}\big(1+\|(\rho,u)\|_{H^2{(\mathbb{T}^3)}}\big)
\|\nabla (u,\theta)\|_{L^2(\mathbb{T}^3)} \nonumber\\
\lesssim&\,\|\nabla(\rho,u,\omega,\theta)\|_{L^2(\mathbb{T}^3)}, \label{G6.5}\\
\|\omega\|_{L^2(\mathbb{T}^3)}+\|\theta\|_{L^2(\mathbb{T}^3)} 
 \lesssim&\,   \Big\|\frac{\sqrt{6}}{2}   \omega
 +\theta\Big\|_{L^2(\mathbb{T}^3)}+\|\sqrt{2} \omega-\sqrt{3} \theta\|_{L^2(\mathbb{T}^3)}  \nonumber\\
\lesssim&\, \|\nabla(\rho,u,\omega,\theta)\|_{L^2(\mathbb{T}^3)}
+\|\sqrt{2} \omega-\sqrt{3} \theta\|_{L^2(\mathbb{T}^3)},\label{NJK3.1622}\\
\|u\|_{L^2(\mathbb{T}^3)}+\|b\|_{L^2(\mathbb{T}^3)}
\lesssim&\, \|b-u\|_{L^2(\mathbb{T}^3)}+\|u+b\|_{L^2(\mathbb{T}^3)} \nonumber\\
\lesssim&\,  \|b-u\|_{L^2(\mathbb{T}^3)}+\|\nabla(b,\rho,u)\|_{L^2(\mathbb{T}^3)}.\label{NJK3.163}
\end{align}

By combining \eqref{G6.2}--\eqref{NJK3.163} with Lemma \ref{LA.1}, 
we   obtain the control of $L^p$-norms of $(\rho,u,\theta,f)$ ($2 \leq p \leq \infty$):
\begin{align}\label{G6.6}
&\|(\rho,u,\theta)\|_{L^p(\mathbb{T}^3)}+\|f\|_{L_v^2(L^p(\mathbb{T}^3))} \nonumber\\
&\quad \lesssim   \|\nabla (a,b,\omega,\rho, u, \theta )\|_{H^1(\mathbb{T}^3)} 
+ \|b-u\|_{L^2(\mathbb{T}^3)}+ \|\sqrt{2} \omega-\sqrt{3} \theta\|_{L^2(\mathbb{T}^3)} .    
\end{align}

Thus, by carrying the computations similar to those in Lemmas \ref{L3.1}--\ref{L3.5}, we derive that
\begin{align}\label{G6.7}  
\frac{d}{dt}\mathcal{E}(t) + \lambda_{18}\mathcal{D}(\rho,u,\theta,f) \leq 0, \quad 0 < t < T,  
\end{align}  
for some $\lambda_{18} > 0$  (uniform in $\mu,\lambda,$ and $\kappa$).
 Here, $\mathcal{E}(t)$ is defined  as same as that in \eqref{G3.33}, 
 and $\mathcal{D}(\rho,u,\theta,f)$ is specified by \eqref{D0} 
 for the case   $  \mathbb{R}^3$. These definitions combining \eqref{G6.7}  
 establish the uniform a-priori estimates and substantiate 
 the global existence claim presented in Theorem \ref{T1.5}. 
 For the sake of conciseness, the details are omitted here.

Furthermore, we rigorously establish the coercivity of 
the dissipation and derive the exponential decay rate. 
Specifically, by leveraging the inequalities \eqref{G6.6} 
and \eqref{G6.7}, we deduce that $\mathcal{E}(t) 
\lesssim \mathcal{D}(\rho,u,\theta,f)(t)$, and there exists
 a positive constant $\lambda_{19} > 0$ such that  
\begin{align*}  
\frac{d}{dt}\mathcal{E}(t) + \lambda_{19}\mathcal{E}(t) \leq 0,  
\end{align*}  
for all $t \geq 0$. Applying Gronwall's inequality, we consequently 
obtain the desired exponential decay estimate of 
$\mathcal{E}(t) \backsim \|(\rho,u,\theta,f)\|_{H^4(\mathbb{T}^3)}^2 
+ \|f\|_{H_{x,v}^4(\mathbb{T}^3\times\mathbb{R}^3)}^2$.
\end{proof}

\smallskip
Next, we state  the global convergence and stability results to the  problem \eqref{Q1}--\eqref{Q1-1} 
in $\mathbb{T}^3$ as follows. 

\begin{thm}\label{T1.6}
%Let $\Omega=\mathbb{T}^3$. 
Assume that the initial data $(\rho_0 ,u_0 ,\theta_0,f_0)$
satisfies $F_0=M+\sqrt{M}f_0\geq 0$, $(\rho_0,u_0,\theta_0)\linebreak \in H^4(\mathbb{T}^3)$
and $f_0\in H_{x,v}^4(\mathbb{T}^3\times\mathbb{R}^3)$, and 
$\{(\rho_0^{\mu,\lambda,\kappa},u_0^{\mu,\lambda,\kappa},
\theta_0^{\mu,\lambda,\kappa},f_0^{\mu,\lambda,\kappa})\}_{0<\mu,\lambda,\kappa<1}$ is
a sequence such that $(\rho^{\mu,\lambda,\kappa}_0,u^{\mu,\lambda,
\kappa}_0,\theta^{\mu,\lambda,\kappa}_0)\rightarrow (\rho_0,u_0,
\theta_0)$ in $H^4(\mathbb{T}^3)$ and $f_0^\kappa\rightarrow f_0$ in
$H^4_{x,v}(\mathbb{T}^3\times\mathbb{R}^3)$, as $\mu\rightarrow0$, 
$\lambda\rightarrow0$ 
and $\kappa \rightarrow 0$. And there exists a constant $\eps_3>0$
independent of $\kappa$ such that if 
\begin{align*}
\|(\rho _0,u _0,\theta _0 )\|_{H^4(\mathbb{T}^3)} +\|f_0 \|_{H_{x,v}^4(\mathbb{T}^3\times\mathbb{R}^3)} \leq   \eps_3, 
\end{align*}
then for the global solution $(\rho^{\mu,\lambda,\kappa},u^{\mu,\lambda,\kappa},
\theta^{\mu,\lambda,\kappa},f^{\mu,\lambda,\kappa})$ of the   problem 
\eqref{Q1}--\eqref{Q1-1} in $\mathbb{T}^3 $ subject to the initial data 
 $(\rho_0^{\mu,\lambda,\kappa},u_0^{\mu,\lambda,\kappa},\theta_0^{\mu,\lambda,\kappa},f_0^{\mu,\lambda,\kappa})$ 
 obtained in Theorem \ref{T1.5}, as $\mu\rightarrow 0$, $\lambda\rightarrow0$ and $\kappa\rightarrow 0$, 
 there exists a limit $(\rho,u,\theta,f)$ such that, up to a subsequence, it holds 
\begin{equation*} 
\left\{
\begin{aligned}
(\rho^{\mu,\lambda,\kappa},u^{\mu,\lambda,\kappa},\theta^{\mu,\lambda,\kappa} )
 &\rightharpoonup (\rho,u,\theta )\quad  \,\, \text{weakly-$*$~ ~~in\ \ } 
 &&L^{\infty}(\mathbb{R}^{+};H^4(\mathbb{T}^3)),& \\
f^{\mu,\lambda,\kappa} &\rightharpoonup  f\quad  \,\, \quad\quad\,\,\,\, 
\text{weakly-$*$~ ~~in \   } &&L^{\infty}(\mathbb{R}^{+};H_{x,v}^4(\mathbb{T}^3\times \mathbb{R}^3)),& \\
(\rho^{\mu,\lambda,\kappa},u^{\mu,\lambda,\kappa},\theta^{\mu,\lambda,\kappa} ) 
&\rightarrow (\rho,u,\theta )\quad  \,\, \text{strong~ ~~   in \ \ } 
&&C_{\rm{loc}}(\mathbb{R}^{+};H^3(\mathbb{T}^3)).&
\end{aligned}
\right.
\end{equation*}
 The limit $(\rho,u,\theta,f)$ is the unique global solution to the   problem \eqref{Q2}--\eqref{Q2-2} 
  in $\mathbb{T}^3 $ associated with the initial data $(\rho_0,u_0,\theta_0,f_0)$. 
  Furthermore, $(\rho,u,\theta,f)$ satisfies $\varrho=1+\rho>0$,
 $F=M+\sqrt{M}f \geq 0$ and
 \begin{align*} 
&\sup_{t\geq 0}\big(\|(\rho,u,\theta)(t)\|_{H^{4}(\mathbb{T}^3)}^2
+\|f(t)\|_{H_{x,v}^4(\mathbb{T}^3\times\mathbb{R}^3)}^2)
\nonumber\\
&\quad  +\int_{0}^{t}\mathcal{D}(\rho,u,\theta,f)(\tau){\rm d}\tau
\leq C_5\big(\|(\rho_{0},u_{0},\theta_{0})\|_{H^{4}(\mathbb{T}^3)}^2
 +\|f_0\|_{H_{x,v}^4(\mathbb{T}^3\times\mathbb{R}^3)}^2\big),
\end{align*}
and
\begin{gather*}
\|(\rho,u,\theta)(t)\|_{H^4(\mathbb{T}^3) }+\|f(t)\|_{H^4_{x,v}(\mathbb{T}^3\times\mathbb{R}^3) }
\leq C_5e^{-\zeta_2 t}\big(\|(\rho_0,u_0,\theta_0)\|_{H^4(\mathbb{T}^3)} 
+\|f_0\|_{H^4_{x,v}(\mathbb{T}^3\times\mathbb{R}^3)} \big),
\end{gather*}
for all $t\geq 0$, where $\zeta_2,C_5> 0$ are two positive constants 
 independent of $\mu$, $\lambda$, $\kappa$ and $t$.
\end{thm}

\begin{thm}\label{T1.7}
% Let $\Omega=\mathbb{T}^3$.
Let $(\rho^{\mu,\lambda,\kappa},u^{\mu,\lambda,\kappa},\theta^{\mu,\lambda,\kappa},
f^{\mu,\lambda,\kappa})$ and $(\rho,u,\theta,f)$ be  global solutions to the 
 problems of \eqref{Q1}--\eqref{Q1-1} and \eqref{Q2}--\eqref{Q2-2}  in $\mathbb{T}^3 $ 
 obtained in Theorem \ref{T1.5} and Theorem \ref{T1.6}  respectively, associated with the initial data 
  $(\rho_0^{\mu,\lambda,\kappa},u_0^{\mu,\lambda,\kappa},\theta_0^{\mu,\lambda,\kappa},
  f_0^{\mu,\lambda,\kappa})$ 
and $(\rho_0,u_0,\theta_0,f_0)$. Assume further that 
\begin{align*}
\|(\rho_0^{\mu,\lambda,\kappa}-\rho_0,u_0^{\mu,\lambda,\kappa}-u_0,
\theta_0^{\mu,\lambda,\kappa}-\theta_0)\|_{H^1(\mathbb{T}^3)}+\|f_0^{\mu,\lambda,\kappa}
-f_0\|_{H_{x,v}^1(\mathbb{T}^3\times\mathbb{R}^3)}\leq \max\{\mu,\lambda,\kappa\} .
\end{align*}
Then it holds for all $t\geq0$ that
\begin{align*} 
& \sup_{t\geq 0}\big\{\|(\rho^{\mu,\lambda,\kappa}-\rho,u^{\mu,\lambda,
\kappa}-u,\theta^{\mu,\lambda,\kappa}-\theta)(t)\|_{H^1(\mathbb{T}^3)}\nonumber\\
&\quad +\|(f^{\mu,\lambda,\kappa}
-f)(t)\|_{H_{x,v}^1(\mathbb{T}^3\times\mathbb{R}^3)}\big\}\leq C_6 \max\{\mu,\lambda,\kappa\},
\end{align*}
for some constant $C_6$ independent of $\mu$, $\lambda$, $\kappa$ and   $t$. 
\end{thm}

The proofs of Theorems \ref{T1.6} and \ref{T1.7} closely mirror to those of Theorems 
\ref{T1.2} and \ref{T1.3}. For the sake of brevity, we  opt  to omit their detailed proofs here.\hfill $ \square$

\bigskip 
{\bf Acknowledgements:}
Li and Ni are supported by NSFC (Grant No. 12331007).  And Li is also supported by
 the ``333 Project'' of Jiangsu Province.
Xin is  in part supported by Zheng Ge Ru Foundation, Hong Kong RGC Earmarked Research 
Grants CUHK-14301421, CUHK-14302819, and CUHK-14300819.

\vspace{2mm}

\textbf{Conflict of interest.} The authors do not have any possible conflicts of interest.

\vspace{2mm}

\textbf{Data availability statement.}
 Data sharing is not applicable to this article as no data sets were generated or analyzed 
 during the current study.

%\newpage 

\bibliographystyle{plain}

\end{document}